\documentclass{amsart}
 \pdfoutput=1 

\usepackage{amssymb,mathrsfs}
\usepackage{microtype}
\usepackage{hyphenat}
\usepackage{tikz}
\usepackage[section]{rrthm}
\usepackage[colorlinks]{hyperref}
\usepackage[lite,alphabetic]{amsrefs}
\usepackage{bookmark}

\hypersetup{
 linkcolor = blue,
 urlcolor = blue
}
\usetikzlibrary{external,positioning,scopes,matrix,chains,fit,calc,arrows}

\tikzstyle{every picture}+=[remember picture]
\tikzstyle{math node}=[execute at begin node={\ifmmode\else$\fi},
                       execute at end node={\ifmmode$\fi}]
\tikzstyle{normal math node}=[math node, execute at begin node=\textstyle]
\tikzstyle{small math node}=[math node, execute at begin node=\scriptstyle]
\tikzstyle{very small math node}=[math node, execute at begin node=\scriptscriptstyle]
\tikzstyle{object}=[normal math node]
\tikzstyle{object alias}=[small math node]

\tikzstyle{math arrow}=[->, every node/.style = small math node, auto]
\tikzstyle{equal}=[math arrow,-,double distance = 1.25pt, semithick,
                   every node/.style = {font = \tiny}]
\tikzstyle{iso}=[math arrow, to path = {-- node[normal math node,sloped,
                 auto = false, allow upside down]
                   {\cong} (\tikztotarget) \tikztonodes}]
\tikzstyle{iso arrow}=[math arrow, 
                       to path = {-- node[below,sloped] {\sim} (\tikztotarget) \tikztonodes}]
\tikzstyle{dashed arrow}=[math arrow, dashed]

\tikzstyle{math arrows}=[every path/.append style = math arrow]
\tikzstyle{iso arrows}=[every path/.append style = iso arrow]
\tikzstyle{dashed arrows}=[every path/.append style = dashed arrow]
\tikzstyle{equals}=[every path/.append style = equal]
\tikzstyle{isos}=[every path/.append style = iso]
\tikzstyle{math nodes}=[every node/.append style = math node]
\tikzstyle{small math nodes}=[every node/.append style = small math node]

\newcommand{\N}{\mathbb{N}} 
\newcommand{\Z}{\mathbb{Z}} 
\newcommand{\Q}{\mathbb{Q}} 
\newcommand{\C}{\mathbb{C}} 
\newcommand{\lb}{[\mspace{-2mu}[} 
\newcommand{\rb}{]\mspace{-2mu}]} 
\newcommand{\lp}{(\mspace{-3.5mu}(} 
\newcommand{\rp}{)\mspace{-3.5mu})} 
\newcommand{\incl}{\hookrightarrow} 
\newcommand{\farg}{\bullet} 
\newcommand{\OO}{\sh{O}} 
\newcommand{\KK}{\mathcal{K}}
\newcommand{\id}{\mathrm{id}} 
\newcommand{\Aff}{\mathbb{A}} 
\newcommand{\Prj}{\mathbb{P}} 
\newcommand{\DD}{\mathbb{D}} 
\newcommand{\Gm}{\avar{G_m}} 
\newcommand{\invlim}{\varprojlim\nolimits} 
\newcommand{\dirlim}{\varinjlim\nolimits} 
\renewcommand{\emptyset}{\varnothing} 
\newcommand{\op}{{\mathrm{op}}} 
\renewcommand{\epsilon}{\varepsilon}
\mathchardef\mhyphen="2D
\newcommand{\ULA}{\text{ULA}}
\newcommand{\Psiun}{\Psi^{\mathrm{un}}}
\newcommand{\Phiun}{\Phi^{\mathrm{un}}}
\newcommand{\SC}{\text{sc}}
\newcommand{\Perv}{\cat{Perv}}
\newcommand{\Sph}{\cat{Sph}}

\newcommand{\avar}[1]{\operatorname{\mathbf{#1}}}
\newcommand{\sh}[1]{\mathcal{#1}} 
\newcommand{\stack}[1]{\mathscr{#1}} 
\newcommand{\cat}[1]{\mathbf{#1}} 
\newcommand{\pideal}[1]{\mathfrak{#1}} 
\newcommand{\abs}[1]{\left\lvert #1 \right\rvert} 
\newcommand{\card}[1]{\##1} 
\renewcommand{\bar}[1]{\overline{#1}} 
\renewcommand{\hat}[1]{{\widehat{#1}}} 
\newcommand{\vect}[1]{\vec{#1}} 

\newcommand{\genby}[1]{\langle #1 \rangle} 
\newcommand{\map}[3]{#1 \colon #2 \to #3} 
\renewcommand{\tilde}[1]{\widetilde{#1}}
\newcommand{\csheaf}[1]{\underline{#1}} 
\newcommand{\lie}[1]{\mathfrak{#1}}

\newcommand{\on}{\operatorname}
\DeclareMathOperator{\shHom}{\sh{H}om} 
\DeclareMathOperator{\stHom}{\stack{H}\hspace{-1.5pt}om} 
\DeclareMathOperator*{\tensor}{\otimes}

\DeclareMathOperator{\tbtimes}{\tilde{\boxtimes}}
\DeclareMathOperator*{\overtimes}{\otimes}
\newcommand{\xtimes}{\overtimes^!}

\DeclareRobustCommand{\SkipTocEntry}[5]{}

\let\chapter\part
\clonecounter{chapter}{part} 

\numberwithin{section}{part}
\numberwithin{equation}{part}

\title[Twisted factorizable Satake equivalence]{Twisted geometric Satake equivalence via gerbes\texorpdfstring{\\}{ }on the factorizable grassmannian}
\author{Ryan Cohen Reich}
\date{\today}
\address{U.\ Michigan Mathematics Department, 530 Church Street, Ann Arbor MI 48109}
\email{ryancr@umich.edu}

\begin{document}

\begin{abstract}
 The \emph{geometric Satake equivalence} of Ginzburg and Mirkovi\'c--Vilonen, for a complex reductive group $G$, is a realization of the tensor category of representations of its Langlands dual group ${}^L G$ as a category of ``spherical'' perverse sheaves on the \emph{affine grassmannian} $\on{Gr}_G = G(\C\lp t\rp)/G(\C\lb t \rb)$.  Since its original statement it has been generalized in two directions: first, by Gaitsgory, to the \emph{Beilinson--Drinfeld} or \emph{factorizable grassmannian}, which for a smooth complex curve $X$ is a collection of spaces over the powers $X^n$ whose general fiber is isomorphic to $\on{Gr}_G^n$ but with the factors ``fusing'' as they approach points with equal coordinates, allowing a more natural description of the structures and properties even of the Mirkovi\'c--Vilonen equivalence.  The second generalization, due recently to Finkelberg--Lysenko, considers perverse sheaves twisted in a suitable sense by a root of unity, and obtains the category of representations of a 
group other than the Langlands dual.  This latter result can be considered as part of ``Langlands duality for quantum groups''.
 
 In this work we obtain a result simultaneously generalizing all of the above.  We consider the general notion of twisting by a gerbe and define the natural class of ``factorizable'' gerbes by which one can twist in the context of the Satake equivalence.  These gerbes are almost entirely described by the quadratic forms on the weight lattice of $G$.  We show that a suitable formalism exists such that the methods of Mirkovi\'c--Vilonen can be applied directly in this general context virtually without change and obtain a Satake equivalence for twisted perverse sheaves.  In addition, we present new proofs of the properties of their structure as an abelian tensor category.
\end{abstract}

\maketitle
\tableofcontents

\addtocontents{toc}{\SkipTocEntry}
\chapter*{Introduction}

The main result (\ref{main theorem}) of this paper is a synthesis of the geometric Satake equivalence, as proved in by Mirkovi\'c--Vilonen in \cite{MV_Satake}, the twisted Satake equivalence, as proved in Finkelberg--Lysenko \cite{FL_twistedSatake}, and Gaitsgory's factorizable Satake equivalence \cite{G_deJong}*{Theorem 2.6}.  A historical overview of the Satake equivalence up until its proof in the first citation can be found there; here, we will give a brief description of how this paper's techniques are understood by the author and how and why they differ from those of the latter two citations.

Let $G$ be a connected, reductive algebraic group over $\C$; the restriction to the complex field is dictated by the central role played by its \emph{root data}, namely, the quadruple $(X^*, X_*, \Psi^*, \Psi_*)$ consisting of the weights (characters) and coweights of some maximal torus of $G$ and the roots and coroots contained in these lattices.  The basic Satake equivalence describes the category of representations of the \emph{Langlands dual group} ${}^L G$ whose root data is $(X_*, X^*, \Psi_*, \Psi^*)$; the twisted Satake equivalence generalizes this duality operation by including the data of a certain integer $q$ (the manner of which is described in \ref{s:relation to the result of Finkelberg--Lysenko}) and arriving at a ``dual'' group $\check{G}_q$.  In both cases, the description is as a certain tensor category of perverse sheaves on the \emph{affine grassmannian} $\on{Gr}_G$; the details of this correspondence can of course be found in \ref{c:relative twisted geometric Satake}.  As for $G$, we 
define $\on{Gr}_G$ over $\C$ and will use topological concepts from the classical topology.  We assume some familiarity with the affine grassmannian and the proof of the Satake equivalence in this introduction.

Over time, the Satake equivalence has come to be seen less as a theorem about representation theory as it occurs in certain cohomology groups and more as a geometric theorem about the structure of the affine grassmannian.  The single most important event in that shift was probably Mirkovi\'c--Vilonen's use of the \emph{factorizable grassmannian} to state and prove the fusion product formula (\ref{convolution properties}\ref{en:convolution is perverse}).  The factorizable grassmannian (\ref{factorizable grassmannian}) was elevated by Gaitsgory to the fundamental setting of the theorem and in this paper we use it exclusively except for the specific technical computations of \ref{s:absolute twisted Satake root data} (and the supporting preceding section).

We introduce two main technical tools in this paper.  By far the more prevalent is that of \emph{symmetric factorizable (sf) gerbes} (\ref{sf gerbe}), which replace the twisting integer $q$ of Finkelberg--Lysenko.  There is a simple explanation for why gerbes should appear at all: the manner of twisting in that paper is by considering sheaves on a $\Gm$-bundle $\det_G$ over $\on{Gr}_G$ that satisfy a form of ``twisted decent'': their restrictions to each fiber are not constant but are local systems (representations of $\pi_1(\Gm) \cong \Z$) of monodromy $q$.  If $\det_G$ is trivialized on an open set $U \subset \on{Gr}_G$, such sheaves are unnaturally identifiable with sheaves on $U$, but such data on $\on{Gr}_G$ does not glue.  Instead, it can be seen (over a sufficiently fine covering) to constitute a \emph{\v{C}ech $2$-cocycle} with values in $\C^*$, and this concept of twisted gluing corresponds to twisting by a gerbe (\ref{s:twisting categories by a gerbe}).

The advantage of using gerbes on the grassmannian itself to twist is that there is a clear extension of the formalisms of sheaves to gerbe-twisted sheaves.  It is also that it clarifies the aspects of the proof of the Satake equivalence that are ``natural'' and those that are ``noncanonical''; for example, global cohomology is noncanonical from the gerbe perspective.  This precipitates a complete adoption of the geometric point of view, and in fact, there are no cohomology computations anywhere in this paper.  The ultimate validation of this point of view is in Theorems \plainref{sf gerbes exact sequence} and \plainref{existence of equivariance}, which show that gerbes can be classified and given the essential structure of $G(\smash{\hat{\OO}})$-equivariance (see \ref{arc and loop groups}) purely on the basis of the geometric property of factorizability.  We find it significant that these theorems closely mirror the Satake equivalence for sheaves; together, they appear to constitute a Satake equivalence for 
gerbes.

Our other technical tool is the use of \emph{ULA perverse sheaves} (\ref{ULA}) through the essential \ref{devissage}, which is based on the very useful nearby-cycles gluing theorem of Beilinson.  Through it we are able to give short and conceptual proofs of all the main ingredients of the Satake equivalence: convolution product (\ref{convolution properties}), commutativity constraint (via the fusion product), compatibility of the fiber functor (\ref{fiber functor properties}), and also the generalization from the ``absolute'' (on $\on{Gr}_G$) theorem (\ref{absolute twisted satake}) to the ``relative'' (factorizable) theorem (\ref{main theorem}).  We are also pleased to have been able to demonstrate the semisimplicity theorem on $\on{Gr}_G$ (\ref{absolute semisimple}) without the use of \cite{L_singularities}*{\S11}, which was previously the only known method.

The present article was prepared on the basis of my thesis, for which I give my deepest gratitude to my advisor, Dennis Gaitsgory.  His insights into the use of nearby cycles (literally and figuratively) led to the two fundamental technical results of this work: \ref*{2-line bundle of a Cartier divisor} and \ref*{devissage}, as well as numerous other essential contributions.


\chapter{Prelude on gerbes}
\label{c:prelude on gerbes}

As we use gerbes in an essential and sometimes elaborate way throughout this paper, we will need to introduce their basic properties in this chapter before moving on.

\section{Categorification of cohomology}
\label{s:categorification of cohomology}

Let $X$ be a topological space; except for \ref{s:construction of gerbes}, it could also be an arbitrary site.  Let $\sh{G}$ be a sheaf of abelian groups on $X$, and fix the notation
\begin{equation*}
 \stack{H}^1(X, \sh{G}) = X/\sh{G}
\end{equation*}
for the fibered category with descent (i.e.\ ``stack'', or ``sheaf of categories'') of $\sh{G}$-torsors on $X$; for a detailed description of basic facts about stacks, we recommend the chapter by Vistoli in \cite{FGA}.  Since $\sh{G}$ is abelian, this stack is in fact an abelian group stack, or ``Picard stack''.  We assume that a stack comes with a splitting (choice of pullbacks); since all our fibered categories will be of sheaf-theoretic origin, this will always be true.

\begin{theorem}{defn}{definition of gerbe}
 Let $\stack{G}$ be a stack with an action of the stack $\stack{H}^1(X, \sh{G})$.  We say that this action is a \emph{gerbe} for $\sh{G}$ if for some (in fact, every) cover of $X$ by open sets $U$ in which each $U$ admits a section $s \in \stack{G}_U$, the map of stacks
 \begin{equation*}
  \stack{H}^1(U, \sh{G}) \xrightarrow{\sh{T} \mapsto \sh{T} \cdot s} \stack{G}|_U
 \end{equation*}
 is an equivalence.  We will denote by $\cat{H}^2(X, \sh{G})$ the $2$-category of gerbes on $X$, in which morphisms are $\sh{G}$-equivariant maps of stacks.  We call the \emph{trivial gerbe} $\stack{G}^0$ the regular action of $\stack{H}^1(X, \sh{G})$ on itself.
\end{theorem}
An even more general theory of gerbes is given in Giraud's thesis \cite{Gir_Cohomologie}, but we will sketch the relevant aspects in this setting.  Actual consultation of this (much more complete) work is not necessary to understand the present one, and its focus on non-abelian cohomology is actually somewhat orthogonal to the ends to which we put our gerbes.

The following proposition is immediate from the definition.

\begin{theorem}{prop}{gerbes are a 2-groupoid}
 Any map of gerbes is an equivalence. \qed
\end{theorem}

\begin{theorem}{prop}{gerbe tensor product}
 There is a \emph{tensor product} of gerbes $(\stack{G}_1, \stack{G}_2) \mapsto \stack{G}_2 \otimes
 \stack{G}_2$ which is associative, commutative, unital, and admits inverses all up to natural
 equivalence.
\end{theorem}

\begin{proof}
 We define $\stHom_\sh{G}(\stack{G}_1, \stack{G}_2)$ to be the stack whose sections on a neighborhood $U$ are the category of $\stack{H}^1(X, \sh{G})$-equivariant (and by definition, cartesian) functors from $\stack{G}_1$ to $\stack{G}_2$, and whose morphisms are natural transformations of such functors. It admits an $\stack{H}^1(X, \sh{G})$-action on the second variable (equivalently, the inverse action on the first variable), which is easily seen to be trivialized wherever the $\stack{G}_i$ both are.  We set
 \begin{align*}
  \stack{G}^{-1} = \stHom_\sh{G}(\stack{G}, \stack{G}^0)
  &&
  \stack{G}_1 \otimes \stack{G}_2
   = \stHom_\sh{G}(\stack{G}_1^{-1}, \stack{G}_2).
 \end{align*}
 One checks that $\stack{G}^{-1}$ has the same trivializations as $\stack{G}$ but with the inverse gluing data, and thus that $(\stack{G}^{-1})^{-1} \cong \stack{G}$. This immediately gives $\stack{G}^0$ as both a left and a right identity.  For inverses, both $\stack{G}^{-1} \otimes \stack{G} = \stHom_\sh{G}(\stack{G}, \stack{G})$ and $\stack{G} \otimes \stack{G}^{-1} = \stHom_\sh{G}(\stack{G}^{-1}, \stack{G}^{-1})$ are globally trivial.  More generally, from this interpretation of the inverse we see that $\stHom(\stack{G}_1, \stack{G}_2) = \stHom(\stack{G}_1^{-1}, \stack{G}_2^{-1})$, so that inversion is an involution which is both a homomorphism and (by definition) an anti-homomorphism, proving commutativity of $\otimes$.  To prove associativity one argues directly, which is straightforward if tedious.
\end{proof}

This proof is valid without modification for the definition of the tensor product of two torsors as well, which gives the Picard category structure on sections of $\stack{H}^1(X, \sh{G})$.

Suppose we have a map of sheaves of groups $\map{\phi}{\sh{G}}{\sh{H}}$, inducing a change-of-group functor which we denote $\map{{}^1 \phi}{\stack{H}^1(X, \sh{G})}{\stack{H}^1(X, \sh{H})}$.  A similar $2$-functor exists for gerbes:
\begin{equation}
 \label{eq:gerbe change of group}
 \map{{}^2 \phi}{\cat{H}^2(X, \sh{G})}{\cat{H}^2(X,\sh{H})},
\end{equation}
where denoting $\stack{H}^0_\phi = \stack{H}^1(X, \sh{H})$ with the action of $\stack{H}^1(X, \sh{G})$ given by ${}^1 \phi$, we have
\begin{equation*}
 {}^2 \phi(\stack{G})
  = \stHom_\sh{G}(\stack{G}^{-1}, \stack{H}^0_\phi)
  = \stack{G} \otimes \stack{H}^0_\phi.
\end{equation*}
Here, $\stack{H}^1(X, \sh{H})$ acts on this via the second factor; since $\stack{G}$ is locally trivial, ${}^2 \phi(\stack{G})$ is locally $\stHom_\sh{G}(\stack{G}^0, \stack{H}^0) \cong \stack{H}^0$, so is also locally trivial, hence an $\sh{H}$-gerbe.  As for $\otimes$, this can be taken as an inductive definition of ${}^1 \phi$ as well.  Since $\stack{H}^0$ is the unit in $\cat{H}^2(X, \sh{H})$ (so idempotent), we see by associativity of $\otimes$ that ${}^2\phi$ is a tensor homomorphism.

\begin{theorem}{lem}{base change triviality}
 For a $\sh{G}$-gerbe $\stack{G}$, we have ${}^2 \phi(\stack{G}) \cong \stack{H}^0$ if and only if there exists a $\stack{G}$-equivariant map $\stack{G} \to \stack{H}^1(X, \sh{H})$.
\end{theorem}

\begin{proof}
 This is literally the criterion for ${}^2 \phi(\stack{G}^{-1})$ to have a global section, hence be trivial, and this occurs if and only if the inverse gerbe ${}^2 \phi(\stack{G})$ is trivial.
\end{proof}

The tensor product has a more natural description in terms of the group product:

\begin{theorem}{lem}{gerbe product relation}
 The following equivalences hold:
 \begin{enumerate}
  \item \label{en:gerbe product category}
  $\cat{H}^2(X, \sh{G} \times \sh{H}) \cong \cat{H}^2(X, \sh{G}) \times \cat{H}^2(X, \sh{H})$. 
  
  \item \label{en:gerbe product relation}
  Writing $\map{m}{\sh{G} \times \sh{G}}{\sh{G}}$ for the product, we have $\stack{G}_1 \otimes \stack{G}_2 = {}^2 m(\stack{G}_1 \times \stack{G}_2)$.
 \end{enumerate}
\end{theorem}

\begin{proof}
 \mbox{For \ref{en:gerbe product category}}:
 Let $\map{\on{pr}_1, \on{pr}_2}{\sh{G} \times \sh{H}}{\sh{G}, \sh{H}}$ respectively be the projections.  We claim that for gerbes $\stack{G}, \stack{H}$ on the right, $\stack{G} \times \stack{H}$ is a gerbe for $\sh{G} \times \sh{H}$.  Assuming this for the moment, both directions of the equivalence and their two composition isomorphisms are
 \begin{align*}
  \begin{gathered}
   \stack{P} \mapsto (\on{{}^2 pr}_1 \stack{P}, \on{{}^2 pr}_2 \stack{P}) \\
   \stack{P} \to \on{{}^2 pr}_1 \stack{P} \times \on{{}^2 pr}_2 \stack{P}
  \end{gathered}
  &&
  \begin{gathered}
   (\stack{G}, \stack{H}) \mapsto \stack{G} \times \stack{H} \\
   \stack{G} \to \on{{}^2 pr}_1 (\stack{G} \times \stack{H}), \text{ etc.}
  \end{gathered}
 \end{align*}
 where each component of the one on the left comes from the trivialization in $\stHom(\stack{P} \otimes \stack{P}^{-1}, \sh{G}^0)$ and the one on the right comes from the same for $\stack{G}$ after forgetting $\stack{H}$.

 To prove the deferred claim, apply the preceding construction to torsors rather than gerbes and then establish the claim directly; this shows that the lemma holds for $\stack{H}^1(X, \sh{G} \times \sh{H})$.  Now replace $\sh{G}$ and $\sh{H}$ by $\stack{H}^1(X, \sh{G})$ and $\stack{H}^1(X, \sh{H})$ and repeat the argument to prove the claim for gerbes.

 \mbox{For \ref{en:gerbe product relation}}:  We have
 \begin{equation}
  \label{eq:gerbe product pushforward}
  {}^2 m(\stack{G}_1 \times \stack{G}_2)
   = \stHom(\stack{G}_1^{-1} \times \stack{G}_2^{-1}, \stack{G}^0_m)
 \end{equation}
 and we identify this with $\stHom(\stack{G}_1^{-1}, \stack{G}_2)$ by defining a map which, for any homomorphism $\phi$ in \ref{eq:gerbe product pushforward} defines a map $\map{\psi}{\stack{G}_1^{-1}}{\stack{G}_2}$ sending a section $s$ of $\stack{G}_1$ to the unique section $t$ of $\stack{G}_2$ such that $\phi(s,t) = \sh{T}^0$ is the trivial torsor in $\stack{G}^0 = \stack{H}^1(X, \sh{G})$.  One checks that this is $\stack{H}^1(X, \sh{G})$-antiequivariant in $s$ and equivariant in $t$ and so gives an isomorphism of gerbes.
\end{proof}

The change-of groups functors are homomorphic in any way that makes sense:

\begin{theorem}{prop}{gerbe change of group composition}
 We have ${}^2(\psi \circ \phi) = {}^2\phi \circ {}^2 \psi$ for composable maps $\map{\phi}{\sh{K}}{\sh{G}}$ and $\map{\psi}{\sh{G}}{\sh{Q}}$, and ${}^2(\phi \cdot \psi) = {}^2 \phi \otimes {}^2 \psi$ for maps $\map{\phi, \psi}{\sh{G}}{\sh{H}}$.
\end{theorem}

\begin{proof}
 The first is another variant of associativity of $\stHom$ and the second is a consequence of $\phi \cdot \psi = m \circ (\phi, \psi)$, where $\map{(\phi, \psi)}{\sh{G}}{\sh{H} \times \sh{H}}$, the previous lemma, and the first claim of this proposition.
\end{proof}

There is a categorified long exact sequence of cohomology up to degree $2$.  Here, a sequence $\cat{A} \to \cat{B} \to \cat{C}$ of $2$-Picard categories is exact if the induced sequence on equivalence classes of objects is exact and the induced sequence $\cat{Aut}(1_\cat{A}) \to \cat{Aut}(1_\cat{B}) \to \cat{Aut}(1_\cat{C})$ is an exact sequence of Picard categories (which is defined inductively in the same way).

\begin{theorem}{prop}{gerbe exact sequence}
 If $1 \to \sh{K} \to \sh{G} \to \sh{Q} \to 1$ is a short exact sequence of sheaves of abelian groups, then we have a long exact sequence of groups/Picard categories/$2$-Picard categories, where $\cat{H}^1(X, \sh{G})$ is the category of $\sh{G}$-torsors on $X$:
 \begin{multline*}
  1 \to \sh{K}(X) \xrightarrow{\phi} \sh{G}(X) \xrightarrow{\psi} \sh{Q}(X) \\
    \overset{\delta^1}{\longrightarrow} \cat{H}^1(X, \sh{K})
    \xrightarrow{{}^1 \phi} \cat{H}^1(X, \sh{G})
    \xrightarrow{{}^1 \psi} \cat{H}^1(X, \sh{Q}) \\
    \overset{\delta^2}{\longrightarrow} \cat{H}^2(X, \sh{K})
    \xrightarrow{{}^2 \phi} \cat{H}^2(X, \sh{G})
    \xrightarrow{{}^2 \psi} \cat{H}^2(X, \sh{Q})
 \end{multline*}
\end{theorem}

\begin{proof}
 Since $\psi \circ \phi = 1$, we also have ${}^1 \phi \circ {}^1 \psi = 1$ and ${}^2 \phi \circ {}^2 \psi = 1$.

 We define $\smash{\delta^2(\sh{T})}$ for a $\sh{Q}$-torsor $\sh{T}$ to be the sheaf of categories whose category of sections over a neighborhood $U$ has objects the set of pairs $\smash{(\tilde{\sh{T}}, f)}$ where $\smash{\tilde{\sh{T}}}$ is a $\sh{G}$-torsor and $f$ is an isomorphism of ${}^1 \psi(\smash{\tilde{\sh{T}}})$ with $\sh{T}$; its morphisms are maps of the $\smash{\tilde{\sh{T}}}$-term inducing a commutative triangle after application of ${}^2 \psi$.  Evidently, $\stack{H}^1(X, \sh{K})$ acts on $\delta^2(\sh{T})$ through ${}^2 \phi$, since ${}^2 \psi$ annihilates its image, and since $\sh{T}$ is locally trivial we apply exactness at $\cat{H}^1(X, \sh{G})$ (inductively) to see that this in fact defines an $\sh{H}$-gerbe.  Clearly we have $\delta^2 \circ {}^1 \psi = 1$; we also have ${}^2 \phi \circ \delta^2 = 1$, by producing a map $\delta^2(\sh{T}) \to \stack{G}^0_\phi$ sending $(\smash{\tilde{\sh{T}}}, f)$ to $\smash{\tilde{\sh{T}}}$, a $\stack{G}$-torsor.
 
 Except for exactness at $\cat{H}^2(X, \sh{G})$, which we prove in a separate lemma, this proposition is not important to the rest of the paper, so we omit the further details.
\end{proof}

\begin{theorem}{lem}{gerbe base change trivialization}
 Every trivialization of ${}^2 \psi(\stack{G})$ is induced by an isomorphism $\stack{G} = {}^2 \phi (\stack{K})$ for a $\sh{K}$-gerbe $\stack{K}$, and likewise for torsors.
\end{theorem}

\begin{proof}
 If we are given a map $\map{f}{\stack{G}}{\stack{Q}^0_\psi}$, let $\stack{K}$ be the fiber $f^{-1}(\sh{Q}^0)$.  This is evidently a sheaf of categories; $\stack{H}^1(X, \sh{K})$ acts on it through ${}^2 \phi$ because ${}^2 \psi$ kills the image of that map, and if $\stack{G}$ is trivial then (inductively) by the torsor version of the lemma this fiber is just $\stack{H}^1(X, \sh{K})$, so $\stack{K}$ is a $\sh{K}$-gerbe.  By definition, it admits a map to $\stack{G}$, which induces an equivalence ${}^2 \phi(\stack{K}) = \stack{G}$, as desired.
\end{proof}

Further evidence of the ``correctness'' of this theory is given by the connection with cohomology:

\begin{theorem}{prop}{gerbes and cech cohomology}
 There is a natural identification of equivalence classes of objects in $\cat{H}^2(X, \sh{G})$ with sheaf \v{C}ech cohomology classes in $H^2(X, \sh{G})$ (and likewise for $H^1$) which agrees with the product operations, change of group, and long exact sequences.
\end{theorem}

Since this proposition is not essential for the rest of the paper, we will not give the complete details of the proof.

\begin{proof}
 Here is how the cohomology class of a gerbe is obtained: for any gerbe $\stack{G}$, choose an open cover of $X$ by neighborhoods $U_i$ such that $\stack{G}|_U$ is trivial, and pick trivializations $t_i$; on $U_{ij} = U_i \cap U_j$, we have $t_{ij} = t_j^{-1} t_i \in \cat{H}^1(U_{ij}, \sh{G})$, and hence the differences are locally trivial, so we refine the $U_i$ to assume that they are totally trivial; we pick trivializations. Then we have on $U_{ijk} = U_i \cap U_j \cap U_k$ two trivializations of $t_{ijk} = t_{ij} t_{jk} t_{ki}$: the given one and the natural one coming from cancellation.  They thus differ by an element of $\sh{G}(U_{ijk})$ which is easily seen to be a \v{C}ech 2-cocycle.  One checks that this cocycle depends on the choices of trivializations and the cover $U_i$ only up to coboundary, so gives a well-defined cohomology class.
   
 Conversely, given an open cover and a 2-cocycle $t_{ijk}$, for any neighborhood $V$ let $\stack{G}_V$ be the category of data consisting of $\sh{G}$-torsors $\sh{T}_i$ on the $V \cap U_i$ and isomorphisms $\map{\phi_{ij}}{\sh{T}_i|_{U_{ij}}}{\sh{T}_j|_{U_{ij}}}$ such that $\phi_{ij} \circ \phi_{jk} \circ \phi_{ki} = t_{ijk}$ as automorphisms of $\sh{T}_i|_{U_{ijk}}$; it is easy to check that this is a $\sh{G}$-gerbe and that this construction inverts the one of the previous paragraph.  The remaining claims can all be proved using similar reasoning by choosing sufficiently fine trivializations of the gerbes involved.
\end{proof}

Our remaining goal for this section is to prove the existence of pullback $2$-functors of gerbes along maps of spaces, as for sheaves.  This entails a great deal of $2$-categorical technicalities, some of which we will leave to the reader.

We may rephrase the above lemma in a more inductive way using the concept of a ``\v{C}ech 1-cocycle with values in torsors''.  By definition, such a thing (on $X$, for a sheaf of groups $\sh{G}$) is given by an open cover $\{U_i\}$ of $X$ together with torsors $\sh{T}_{ij} \in \cat{H}^1(U_{ij}, \sh{G})$ and a trivialization $t_{ijk}$ of the ``$2$-coboundary'' $\sh{T}_{ijk} = \sh{T}_{ij} \otimes \sh{T}_{jk} \otimes \sh{T}_{ki}$ on the triple intersection $U_{ijk}$, satisfying the cocycle condition on quadruple intersections.  Such data is equivalent (using the section-picking arguments of the lemma) to a single \v{C}ech 2-cocycle, and we are motivated to introduce the following definition:

\begin{theorem}{defn}{torsor cocycle 2-category}
 We set $\cat{H}^1(X, \stack{H}^1(X,\sh{G}))$ to be the $2$-category such that
 \begin{itemize}
  \item
  Its objects are 1-cocycles $(\{U_i\}, \{\sh{T}_{ij}\}, \{t_{ijk}\})$ in $\sh{G}$-torsors.

  \item
  A 1-morphism from $(\{U_i\}, \{\sh{T}_{ij}\}, \{t_{ijk}\})$ to $(\{U'_{i'}\}, \{\sh{T}'_{i'j'}\}, \{t'_{i'j'k'}\})$ is a common refinement $\{U''_{i''}\}$ of $\{U_i\}$ and $\{U'_{i'}\}$ together with torsors $\sh{S}_{i''}$ on the $U_{i''}$ and isomorphisms of torsors $s_{ij}$ as follows.  For simplicity of notation we assume $U_i = U'_{i'} = U''_{i''}$; then $s_{ij}$ is an isomorphism of $\sh{T'}_{ij} \otimes \sh{S}_i$ with $\sh{T}_{ij} \otimes \sh{S}_j$ on $U_{ij}$, compatible with the $t_{ijk}$ and $t'_{ijk}$.  (In other words, it is a ``1-cohomology'' in torsors.)
 
  \item
  A 2-morphism between two 1-morphisms 
  \begin{equation*}
   (\{U_i\}, \{\sh{S}_i\}, \{s_{ij}\}) \text{ and } (\{U'_{i'}\}, \{\sh{S}'_{i'}\}, \{s'_{i'j'}\})
  \end{equation*}
  is a common refinement of $\{U_i\}$ and $\{U'_{i'}\}$ together with maps of torsors (using the above notational convention) $r_i \colon \sh{S}_i \to \sh{S}'_i$ compatible with the $s_{ij}$ and $s'_{ij}$.
 \end{itemize}
\end{theorem}

\begin{theorem}{cor}{gerbes and torsor cocycles}
 We have $\cat{H}^2(X, \sh{G}) \cong \cat{H}^1(X, \stack{H}^1(X, \sh{G}))$.
\end{theorem}

\begin{proof}
 We replace $\cat{H}^2(X, \sh{G})$ with the 2-category $\cat{C}$ of triples $(\stack{G}, \{U_i\}, \{t_i\})$, where $\stack{G}$ is a gerbe, $\{U_i\}$ is an open cover, and $t_i$ is a trivialization of $\stack{G}|_{U_i}$.  The morphisms are morphisms of $\stack{G}$ alone, so $\cat{C}$ is equivalent to $\cat{H}^2(X, \sh{G})$.  It admits a natural 2-functor $F$ to $\cat{H}^1(X, \stack{H}^1(X, \sh{G}))$ sending such a triple to
 \begin{equation*}
  F(\stack{G}, \{U_i\}, \{t_i\})
   = (\{U_i\}, \{\sh{T}_{ij} = t_j^{-1} t_i\}, \{t_{ijk}\}),
 \end{equation*}
 where $t_{ijk}$ is the natural trivialization of $t_i^{-1} t_k t_k^{-1} t_j t_j^{-1} t_i$.  One can check that this extends to a tensor $2$-functor, essentially surjective by gluing (the comments following \ref{gerbes and cech cohomology}).  Since two objects are equivalent (and 1-morphisms isomorphic) if and only if they admit a map, to show that $F$ is fully faithful it suffices to show the following two facts:
 \begin{enumerate}
  \item
  If $F(\stack{G}, \{U_i\}, \{t_i\})$ is equivalent to the identity object, then $\stack{G}$
  is trivial.  This follows by gluing compatible trivializations of $\stack{G}$.
  
  \item
  The map $\cat{Aut}(\stack{G}^0, \{X\}, \{\id\}) \to \cat{Aut}(\{X\}, \{\sh{T}^0\}, \{\id\})$ induced by $F$ is an equivalence of categories.  This follows by gluing torsors. \qedhere
 \end{enumerate}
\end{proof}

Finally, one can go further still.  Suppose $f \colon Y \to X$ is any map of spaces; then, as we can pull back torsors (being sheaves), we have a pullback functor
\begin{equation*}
 f^* \colon \cat{H}^1(X, \stack{H}^1(X, \sh{G})) \to \cat{H}^1(Y, \stack{H}^1(Y, f^* \sh{G}))
\end{equation*}
and therefore a pullback functor
\begin{equation}
 \label{eq:pullback of gerbes}
 f^* \colon \cat{H}^2(X, \sh{G}) \to \cat{H}^2(Y, f^* \sh{G})
\end{equation}
which, by definition, agrees with the pullback of cohomology classes.

\section{Twisting categories by a gerbe}
 \label{s:twisting categories by a gerbe}

We have already defined the tensor product of two gerbes and the change of groups along a homomorphism in terms of the ``stack of homomorphisms''.  In general, suppose $\sh{G}$ is a sheaf of abelian groups and that $\stack{H}^1(X, \sh{G})$ acts on a sheaf of categories $\stack{F}$; then for any $\sh{G}$-gerbe $\stack{G}$, we may form
\begin{equation*}
 \stack{F}(\stack{G}) = \stack{G} \otimes \stack{F} = \stHom(\stack{G}^{-1}, \stack{F}),
\end{equation*}
the stack of \emph{$\stack{G}$-twisted objects of $\stack{F}$}.  Sometimes, if we denote $\cat{F} = \stack{F}_X$, we will write $\stack{G} \otimes \cat{F}$ by abuse of notation to mean $\stack{F}(\stack{G})_X$.

We can give two concrete descriptions of a section $s \in \stack{F}(\stack{G})_Y$ of a twisted sheaf of categories:
\begin{enumerate}
 \item \label{en:torsor trivialization}
 Using \ref{gerbes and torsor cocycles}:
 Suppose that $\stack{G}$ is trivialized on an open cover $U_i$; then a section $s \in \stack{F}(\stack{G})_X$ is given by sections $s_i \in \stack{F}_{U_i}$ together with isomorphisms $s_j \cong \sh{T} \otimes s_i$ on $U_{ij}$ (using $\otimes$ to denote the action of a torsor on sections of $\stack{F}$) satisfying the $1$-cocycle condition with respect to the cocycle constraint of $\sh{T}$.

 \item \label{en:cocycle trivialization}
 Using \ref{gerbes and cech cohomology}:
 Suppose that the $U_i$ are chosen so that all the transition torsors are trivial and denote by $t_{ijk} \in \sh{G}_{U_{ijk}}$ the \v{C}ech $2$-cocycle corresponding to $\stack{G}$; then $s$ is given by $s_i \in \stack{F}_{U_i}$ and isomorphisms $s_i \cong s_j$ whose coboundary on $U_{ijk}$ is equal to the action of $t_{ijk}$.
\end{enumerate}
From either of these descriptions, it is clear that if we have a group homomorphism $\map{\phi}{\sh{G}}{\sh{H}}$ and an $\sh{H}$-action on $\stack{F}$, then for any $\sh{G}$-gerbe $\stack{G}$ we have a functor $\stack{F}(\stack{G}) \to \stack{F}({}^2 \phi\,\stack{G})$, where the first twisting is with respect to $\sh{G}$ and the second with respect to $\sh{H}$.

Despite the intuitive appeal of \ref{en:cocycle trivialization}, it is not as useful as the coarser \ref{en:torsor trivialization} because sometimes, one does not have the freedom to refine $U$ to the point that the torsors are trivial.  For example, the following construction/theorem, whose proof is implicit in the statement:

\begin{theorem}{prop}{twisted pushforward functor}
 Suppose that $\stack{F}$ has \emph{pushforwards}: for any $\map{g}{Z}{Y}$ the pullback map $\map{g^*}{\stack{F}_Y}{\stack{F}_Z}$ admits a right adjoint such that $g^* g_* \to \id$ is an isomorphism of functors, and $g_*$ satisfies a projection formula with respect to the action of $\stack{H}^1(X, \sh{G})$: for any torsor $\sh{T}$ on $Y$ and $t \in \stack{F}_Z$, the natural map
 \begin{equation*}
  \sh{T} \otimes g_* t \to g_*(g^* \sh{T} \otimes t)
 \end{equation*}
 is an isomorphism.  Let $\map{f}{Y}{X}$ and suppose that the $U_i^0$ cover $X$, with $U_i$ the pullback cover of $Y$; let $\stack{G}$ be specified on the $U_i^0$ as in \ref{en:torsor trivialization} above.  For any $s \in \stack{F}(\stack{G})_Y$ specified with respect to the $U_i$, there exists a pushforward $f_*(s) \in \stack{F}(\stack{G})_X$, specified by the $f_*(s_i)$, extending to a pushforward functor of the above description. \qed
\end{theorem}

It is with the introduction of the geometric functors that large and small topologies make an appearance; by the difference we mean that large topologies generally include \emph{all} objects in the site and merely specify covers, whereas the small topologies also restrict the objects.  Large topologies are useful for theoretical purposes but not in practice, and we will not use them outside this prologue.  Our reason for using them here is that the existence of pullbacks of sheaves of categories is a theorem of type ``small to large''; it says that a sheaf of categories, and in particular a gerbe on a small topology over an object $X$ naturally induces a unique one on the corresponding large topology.  Then the geometric functors have the following interpretation:

Let $\map{f}{Y}{X}$ be a map and $\stack{G}$ be a gerbe on the small topology of $X$, while $\stack{F}$ is a sheaf of categories on the large topology (we have in mind here that of perverse sheaves).  We will call $\stack{F}|_X$ and $\stack{F}|_Y$ its restrictions to the small sites.  We employ the above extension of $\stack{G}$ to the large topology in order to form $\stack{F}(\stack{G})$, and in this language we find that there is a pullback functor of sections twisted on the small categories
\begin{equation*}
 \map{f^*}{\stack{F}|_X(\stack{G})_X}{\stack{F}|_Y(f^* \stack{G})_Y}
\end{equation*}
and possibly a pushforward functor going the other way.

The stack $\stack{F}$ may enjoy many categorical properties, among them its categories of sections being abelian or triangulated.  We will term a sheaf of such categories one in which the pullback maps respect all the structures imposed on the categories; for example, a sheaf of abelian categories is one in which all pullbacks are exact, and one of triangulated categories has triangulated pullbacks.  This necessitates the following comments.

\subsection*{Sheaves of abelian categories}
   
The usual abelian categories of sheaves are \emph{not} sheaves of abelian categories in the large topologies, because general pullbacks tend only to be right exact.  This is true in particular for perverse sheaves (where pullbacks are not even right exact).  However, we will only ever need the large topology for the theoretical purposes of this section, and in the rest of the work the classical site whose objects over a space $X$ are simply open sets in $X$ will suffice, and for such a small topology, pullbacks \emph{are} exact.  If, as in the case of this particular sheaf of abelian categories, tensoring with a $\sh{G}$-torsor is exact, then $\stack{F}(\stack{G})$ is again a sheaf of abelian categories in the obvious way.

\subsection*{Twisted derived category}
   
The fibered category of derived categories of sheaves on $X$ is not a stack because it does not enjoy descent (also, there is ambiguity as to what its pullbacks are, since there are two kinds). Some of our constructions in \ref{c:relative twisted geometric Satake} require mention of functors that are \emph{a priori} merely derived, so in order to handle this rigorously we introduce the sheaf $\stack{D}$ of categories to be the \emph{sheafification} of the aforementioned fibered category; the operation of ``associated stack'' is given in \cite{Gir_Cohomologie}*{\S II.2} and associates with a fibered category $\stack{F}$ the universal stack whose sections are locally isomorphic to those of $\stack{F}$.

Then $\stack{D}$ has not only $f^*$ but also $f^!$, as well as the corresponding lower-indexed adjoints.  If in the topology we use, all open covers are of a class of morphisms for which $f^* = f^!$, then it is a stack for both structures of fibered category; this will be the case either for the classical or the \'etale topologies, though we only use the former.  By definition, the duality functor is either a map $\map{\DD}{(\stack{D}^*)^\op}{\stack{D}^!}$ or in the other direction, where the superscript indicates which kind of pullbacks ($f^!$ or $f^*$) are the cartesian morphisms in the fibered category. In this situation, we will be taking $\sh{G}$ to be the constant sheaf on a complex variety with values in the multiplicative group of some field of characteristic zero, so its torsors are local systems, and tensor product with these is respected by $\DD$. We conclude that all of $f^*$, $f^!$, $f_*$, $f_!$, and $\DD$ have twisted analogues.
   
\subsection*{Twisted t-structures}

Objects of $\stack{D}(\stack{G})_X$ are not important for us insofar as the ones we will use will always turn out to be twisted perverse sheaves. In order to effect this, we will need to impose the perverse t-structure on $\stack{D}(\stack{G})$, which will simply be piecewise: a section $s$ is in ${}^p\stack{D}(\stack{G})_X^{\leqslant 0}$ if $u_i^* s$ is, where the $\map{u_i}{U_i}{X}$ are a trivializing open cover for $\stack{G}$, and likewise will be in ${}^p\stack{D}(\stack{G})_X^{\geqslant 0}$ if $u_i^! s$ is (or equivalently, if $\DD(s) \in {}^p \stack{D}(\stack{G})_X^{\leqslant 0}$). This is sensible since twisting by a local system respects perversity; the core of this t-structure is the twisted category of perverse sheaves (that is, taking core and twisting by $\stack{G}$ commute as operations).

\section{Equivariance and multiplicativity}
\label{s:equivariance and multiplicativity}

The two most important structures on gerbes, for us, are equivariance and multiplicativity.  As before, let $X$ be any space, $Y \to X$ any map, and let $G$ be an $X$-group with multiplication map $\map{m}{G \times_X G}{G}$.  Since all our spaces are relative to $X$, we will drop the subscripts and just write $\times$.

\subsection*{The definitions}
   
An equivariant gerbe is the analogue of a $G$-invariant map between spaces with a $G$-action.

\begin{theorem}{defn}{equivariant gerbe}
 Let $\map{a}{G \times Y}{Y}$ be a $G$-action, let $\stack{G}$ be a gerbe on $Y$, and let $\map{\on{pr}}{G \times Y}{Y}$ be the projection map. A \emph{$G$-equivariance structure} for $\stack{G}$ is the following data:
 \begin{enumerate}
  \item (Equivariance) An equivalence of gerbes on $G \times Y$:
  \begin{equation*}
   \map{\on{Eq}}{a^* \stack{G}}{\on{pr}^* \stack{G}}.
  \end{equation*}

  \item (Compatibilities) Isomorphisms
  \begin{gather*}
   \map{\on{As}}{\on{Eq}_{11} = \bigl(\on{pr}_{G \times Y}^* \on{Eq}\bigr)
                                \circ \bigl((G \times a)^* \on{Eq}\bigr)}
                                {(m \times Y)^* \on{Eq}} \\
   \map{\on{Id}}{\on{id}}{(1 \times Y)^* \on{Eq}}
  \end{gather*}
  referring to the following diagrams: for $\on{As}$ (over $G \times G \times Y$),
  \begin{equation}
   \label{eq:equivariant gerbe}
   \tikzsetnextfilename{equivariant-gerbe}
   \begin{tikzpicture}
    [every matrix/.style = {nodes = {anchor = center}, row sep = 1.2em},
     node distance = 6em and 5em,math nodes, baseline = (current bounding box)]
    \matrix (upper left block)
     {
      & 
        \node[object alias] (upper left -- above)  {(G \times a)^* a^* \stack{G}}; \\
        \node[object alias] (upper left -- left)   {(m \times Y)^* a^* \stack{G}};
      & \node[object]       (upper left)           {a_2^* \stack{G}};              \\
     };
    \matrix (upper right block) [right = of upper left block]
     {
        \node[object alias] (upper right -- above) {(G \times a)^* \on{pr}^* \stack{G}}; \\
        \node[object]       (upper right)          {a_{0,1}^* \stack{G}};
      & \node[object alias] (upper right -- right) {\on{pr}_{G \times Y}^* a^* \stack{G}}; \\
     };
    \matrix (lower right block) [below = of upper right block]
     {
        \node[object]       (lower right)          {\on{pr}^* \stack{G}};
      & \node[object alias] (lower right -- right) {\on{pr}_{G \times Y}^* \on{pr}^* \stack{G}}; \\
        \node[object alias] (lower right -- below) {(m \times Y)^* \on{pr}^* \stack{G}};         \\
     };
    { [math arrows]
     \draw (upper left)  -- (upper right);
     \draw (upper right) -- (upper right |- lower right.north);
     \draw (upper left)  -- coordinate (middle) (lower right);
    }
    { [isos]
     \path (upper left -- left)   to (upper left)  to (upper left -- above);
     \path (upper right -- above) to (upper right) to (upper right -- right);
     \path (lower right -- below) to (lower right) to (lower right -- right);
    }
    { [dashed arrows]
     \draw (upper left -- above)
           -- node {(G \times a)^* \on{Eq}}
           (upper left -- above -| upper right -- above.west);
     \draw (upper right -- right) 
             -- node {\on{pr}_{G \times Y}^* \on{Eq}}
           (upper right -- right |- lower right -- right.north);
     \draw (upper left -- left) 
            to node[swap] {(m \times Y)^* \on{Eq}}
           (lower right -- below.north west);
    }
    \path (upper right) [every node/.style = math node]
     to node[sloped,allow upside down,
             scale=2,xscale=1.5,transform shape,
             label=above:\on{As}]
       (middle arrow) {\Rightarrow}
     (middle);
   \end{tikzpicture}
  \end{equation}
  and for $\on{Id}$,
  \begin{equation*}
   \map{(1 \times Y)^* \on{Eq}}
    {(1 \times Y)^* a^* \stack{G} = \stack{G}}
    {(1 \times Y)^* \on{pr}^* \stack{G} = \stack{G}}.
  \end{equation*}

  \item (Cocycle condition) $\on{As}$ is required to satisfy the 2-cocycle condition on $G \times G \times G \times X$, which we will not write but which is an identity on $\on{As}$ corresponding to the two ways of using  associativity to rewrite the parenthesized expression:
  \begin{alignat*}{3}
   ((wx)y)z &={} &(w(xy))z =&\; w((xy)z)&&= w(x(yz)) \\
   ((wx)y)z &={} &     (wx)(&yz)      &&= w(x(yz)).
  \end{alignat*}
 \end{enumerate}
\end{theorem}

Likewise, we may speak of the gerbe analogue of a group homomorphism:

\begin{theorem}{defn}{multiplicative gerbe}
 Let $\stack{M}$ be a gerbe on $G$ itself; then a \emph{multiplicative structure} for $\stack{M}$ is an equivalence
 \begin{equation*}
  \map{\on{Mul}}{m^* \stack{M}}{\stack{M} \boxtimes \stack{M}}
 \end{equation*}
 with associativity and identity constraints as for equivariance.  This structure is \emph{commutative} if in addition, denoting by $\map{\on{sw}}{G \times G}{G \times G}$ the coordinate swap, we are given an isomorphism between
 \begin{equation*}
  \on{sw}^*(\stack{M} \boxtimes \stack{M})
   \xrightarrow{\on{sw}^* \on{Mul}^{-1}}
  \on{sw}^* m^* \stack{M} = m^* \stack{M}
   \xrightarrow{\on{Mul}}
  \stack{M} \boxtimes \stack{M}
 \end{equation*}
 and the natural equivalence of the first and last terms.

 Given such an $\stack{M}$, the structure of \emph{twisted equivariance} on a gerbe $\stack{G}$ on $Y$ is an equivalence
 \begin{equation*}
  \map{\alpha}{a^* \stack{G}}{\stack{M} \boxtimes \stack{G}}
 \end{equation*}
 with an associativity constraint $\on{As}$ and identity constraint $\on{Id}$ as before, where $\on{As}$ and $\on{Id}$ make reference to the associativity and identity constraints of both $\on{Mul}$ and of the tensor product of gerbes. By definition, an equivariance structure is just twisted equivariance with respect to the trivial multiplicative gerbe (see \ref{trivial equivariance and multiplicativity}), and any multiplicative gerbe is also twisted equivariant on both sides, or equivalently, $G \times G^\op$-twisted-equivariant.
\end{theorem}

This definition can be understood concretely in the situation where $G$ is a discrete group relative to a base space $X$, and we consider it as its group of connected components: ``$G = X \times G$''.

\begin{theorem}{defn}{multiplicative lattice gerbe}
 Let $G$ be a group (in sets).  A \emph{multiplicative gerbe} on $X \times G$ is the data of, for each $g \in G$, a gerbe $\stack{G}^g$ on $X$; for each pair $g, h$ an equivalence $\stack{G}^{gh} \cong \stack{G}^g \otimes \stack{G}^h$; for each triple $g$, $h$, $k$, we have an isomorphism of
 \begin{equation*}
  \stack{G}^g \otimes (\stack{G}^h \otimes \stack{G}^k)
  \cong \stack{G}^{g  h  k} \cong
  (\stack{G}^g \otimes \stack{G}^h) \otimes \stack{G}^k
 \end{equation*}
 with the natural associativity of gerbe products, satisfying the cocycle condition.  If $G$ is commutative, we also require an isomorphism of equivalences between the one:
 \begin{equation*}
  \stack{G}^g \otimes \stack{G}^h \cong \stack{G}^{gh} =
  \stack{G}^{hg} \cong \stack{G}^h \otimes \stack{G}^g
 \end{equation*}
 and the natural commutativity of the tensor product of gerbes.
\end{theorem}

The most obvious examples of these structures are of course the trivial ones.  For later use we state the definition in a manner which is in fact not entirely trivial.  Recall that all our spaces are over a fixed base $X$.

\begin{theorem}{defn}{trivial equivariance and multiplicativity}
 We say that a gerbe $\stack{G}$ on $Y$ is \emph{$X$-trivialized} if it is given descent data to $X$; i.e.\ if we denote by $p$ the structure map to $X$, we have an equivalence $\stack{G} \cong p^* \stack{G}_X$ constituting 2-descent data.  Any $X$-trivialized gerbe $\stack{M}$ on $G$ has a natural multiplicative structure, called the \emph{trivial multiplicative structure}, as follows: if $\stack{M}_X = \stack{M}|_{1_X}$ is the gerbe on $X$ descending $\stack{M}$ (the restriction to the identity section), then on $G \times_X G$ we have (overloading the notation $p$): \begin{equation*}
 m^* \stack{M}
   \cong p^* \stack{M}_X
   \cong p^* \stack{M}_X \boxtimes_X p^* \stack{M}_X
   \cong \stack{M} \boxtimes_X \stack{M},
 \end{equation*}
 with $\boxtimes_X$ denoting the outer tensor product relative to the fiber product $G \times_X G$. Similar measures express pullbacks to higher fibered Cartesian powers of $G$ and fill out the multiplicative structure.  Likewise, we have the \emph{trivial equivariance structure} on a $X$-trivialized gerbe $\stack{G}$ on $Y$.
\end{theorem}

\begin{theorem}{lem}{trivial equivariance multiplicative torsor}
 Let $\stack{G}$ be a $X$-trivial gerbe on $Y$; then the $G$-equivariance structures on $\stack{G}$ are identified with multiplicative torsors on the $Y$-group $G \times_X Y$.
\end{theorem}

\begin{proof}
 Referring to the notation of \ref{equivariant gerbe}, both $a^* \stack{G}$ and $\on{pr}^* \stack{G}$ are identified with $p^* \stack{G}_X$, where $p\colon G \times_X Y \to X$ is the structure map.  Therefore, the equivalence $\on{Eq}$ is identified with a torsor on $G \times_X Y$.  The isomorphism $\on{As}$ on $G \times_X G \times_X Y$, which is the same as $(G \times_X Y) \times_Y (G \times_X Y)$, then expresses
 \begin{equation*}
  \on{Eq} \boxtimes_Y \on{Eq} \cong m_{G \times_X Y}^* \on{Eq},
 \end{equation*}
 where $m_{G \times_X Y}$ is the same as $m \times Y$ but considered as the group operation on $G \times_X Y$.  Likewise, the unwritten cocycle condition expresses associativity of this multiplicative structure.
\end{proof}

\subsection*{Descent properties}

Outside of this chapter, we will use equivariance only as ``that which gives descent'':

\begin{theorem}{lem}{gerbe descent along a torsor}
 Suppose $\sh{G}$ is a sheaf of groups on $Y$, and let $\pi \colon \tilde{Y} \to Y$ be a $G$-torsor. Then $G$-equivariant $\pi^*(\sh{G})$-gerbes on $\tilde{Y}$ are equivalent to $\sh{G}$-gerbes on $Y$.
\end{theorem}

\begin{proof}
 One direction is simple: if $\stack{G}$ is a gerbe on $Y$, then the natural equivalences of the pullback make $\pi^* \stack{G}$ an equivariant gerbe on $\tilde{Y}$.  
 
 Conversely, we will descend gerbes in the form given by \ref{gerbes and torsor cocycles}.  We claim that if $\stack{G}$ is a $G$-equivariant gerbe on $\tilde{Y}$, then there is an open cover $\{U_i\}$ of $Y$ and trivializations of $\stack{G}|_{\pi^{-1}(U_i)}$ whose transition torsors are themselves $G$-equivariant.  Knowing this, the lemma follows from the descent of $G$-equivariant sheaves (or the analogous argument for torsors).
 
 To prove the claim, we begin with any cover $\{U_i\}$ trivializing $\tilde{Y}$, and pick trivializations $\pi^{-1}(U_i) \cong G \times U_i$.  Then the restrictions $\stack{G}|_{1 \times U_i}$, being gerbes on the $U_i$, are locally trivial and we may refine the cover so that they are trivial.  It therefore suffices to prove that:
 \begin{enumerate}
  \item \label{en:gerbe descent claim 1}
  If $\tilde{Y} = G \times Y$ is the trivial torsor, then $\stack{G} = \pi^*(\stack{G}|_{1 \times Y})$.  Furthermore, if $\stack{G}|_{1 \times Y}$ is trivial, then $\stack{G}$ is equivariantly trivial.
  
  \item \label{en:gerbe descent claim 2}
  Writing $V_i = \tilde{Y}|_{U_i}$ and $V_j = \tilde{Y}|_{U_j}$, if $\stack{G}|_{V_i}$, $\stack{G}|_{V_j}$ are equivariantly trivialized then the transition 1-morphism is an equivariant $\sh{G}$-torsor.
  
  \item \label{en:gerbe descent claim 3}
  Given $U_i$, $U_j$, and $U_k$, if $\stack{G}$ is trivialized over all of these opens and if the transition torsors are also equivariantly trivialized, then the $2$-morphism (section of $\sh{G}$) witnessing the cocycle condition is $G$-invariant.
 \end{enumerate}
 
 We consider \ref{en:gerbe descent claim 1}.  When $\tilde{Y} \cong G \times Y$ is trivialized, the action of $G$ is by multiplication $m$ on the first factor, and on $G \times G \times Y$ we have the equivariance datum $\on{Eq} \colon \on{pr}_{G \times Y}^* \stack{G} \cong (m \times \id)^* \stack{G}$.  If we restrict to the slice $G \times 1 \times Y$ we get an equivalence between $\pi^*(\stack{G}|_{1 \times Y})$ and $\stack{G}$, as desired.
 
 Suppose that $\stack{G}|_{1 \times Y}$ is trivial, so that this is a trivialization of $\stack{G}$; we claim that it extends to a trivialization of the equivariance structure.  Note that now $\on{Eq}$ is an automorphism of the trivial gerbe on $G \times G \times Y$, hence a $\sh{G}$-torsor.  Indeed, considering the first of the (Compatibilities) of \ref{equivariant gerbe} restricted to $G \times G \times (1 \times Y) \subset G \times G \times \tilde{Y}$, we get $\on{Eq} \otimes \on{Eq} \cong \on{Eq}$, so $\on{Eq}$ is trivialized.  By definition, the unit map $\on{Un}$ is also trivialized.
 
 Now we consider \ref{en:gerbe descent claim 2}.  Denote by $\phi$ the transition torsor from $\stack{G}|_{V_i} \cong \stack{G}^0$ to $\stack{G}|_{V_j} \cong \stack{G}^0$, so that on $G \times \smash{\tilde{Y}}$, the pullback $\on{pr}_{\tilde Y}^* \stack{G}$ is glued from two trivial gerbes on $G \times V_i$ and $G \times V_j$ by $\on{pr}_{\tilde Y}^* \phi$, and likewise for $a^* \stack{G}$. By assumption, the equivariance data $\on{Eq}$ is trivial on both $G \times V_i$ and $G \times V_j$, and therefore consists merely of a $2$-isomorphism identifying $\on{Eq} \circ a^* \phi$ and $\on{pr}_{V_i \cap V_j}^* \phi \circ \on{Eq}$, where the $\on{Eq}$ terms may be neglected.  This is the equivariance structure for $\phi$.  The proof for \ref{en:gerbe descent claim 3} is similar. 
\end{proof}

The lemma can be rephrased informally as ``a $G$-equivariant gerbe on $Y$ is the same as a gerbe on the stack $Y/G$''.  More precisely, we make the following definition:

\begin{theorem}{defn}{gerbe on stack}
 Let $G$ act on $Y$; a \emph{gerbe on the quotient stack} $Y/G$ is the following data: for every space $S$, $G$-torsor $\tilde{S}$, and $G$-map $s \colon \tilde{S} \to Y$, a gerbe $\stack{G}_{\tilde{S},s}$ on $S$. Furthermore, if $f \colon T \to S$ and we write $\tilde{T} = f^* \tilde{S} = \tilde{S} \times_S T$ with its induced map $t \colon \tilde{T} \to Y$, then we must have equivalences $f^* \stack{G}_{\tilde{S}, s} \cong \stack{G}_{\tilde{T}, t}$. If $U \to T \to S$ is a triple of maps, then both equivalences from $S$ to $U$ must be isomorphic, and these isomorphisms must satisfy a cocycle condition for triple nestings.
\end{theorem}

\begin{theorem}{cor}{equivariant gerbe descends}
 Let $G$ act on $Y$ and let $\stack{G}$ be a gerbe on $Y$.  A structure of $G$-equivariance on $\stack{G}$ is equivalent to giving a gerbe on the quotient stack as in \ref{gerbe on stack}, such that if for $S$ and $s \colon \tilde{S} \to Y$ as there, writing $\pi \colon \tilde{S} \to S$ for the torsor structure map, we also have:
 \begin{enumerate}
  \item Equivalences $s^* \stack{G} \cong \pi^* \stack{G}_{\tilde{S},s}$.
  
  \item For $f \colon T \to S$ and $\tilde{T} = f^* \tilde{S}$ as above (and denoting $f \colon \tilde{T} \to \tilde{S}$), 2-isomorphisms making the triangle of maps of gerbes on $\tilde{T}$ commute.  These isomorphisms must satisfy the cocycle condition for a triple of maps. \qed
 \end{enumerate}
\end{theorem}

All of the notions of equivariance, (commutative) multiplicativity, triviality, and descent are related via the following proposition:

\begin{theorem}{prop}{multiplicative gerbe descent}
 Let $\tilde{G}$ be a group, $G$ a subgroup, $\stack{M}$ a multiplicative gerbe on $\tilde{G}$. Then to give a multiplicative trivialization of $\stack{M}|_G$ is the same as to give an $\stack{M}$-twisted equivariant gerbe $\stack{G}$ on $Y = \tilde{G}/G$ descending $\stack{M}$:
 \begin{center}
  \tikzsetnextfilename{multiplicative-gerbe-descent}
  \begin{tikzpicture}
   \matrix [matrix of math nodes] (objects) [anchor = base, row sep = 3em, column sep = 3.5em]
   {
    \tilde{G} \times \tilde{G} & \tilde{G} \\
    \tilde{G} \times    Y      &    Y      \\
   };
   { [small math nodes, anchor = base, x = 1em, y = 1ex]
     \node (gerbes-1-1) at ($(objects-1-1.north west) + (-1,1)$)
           {m^* \stack{M} \;\cong\; \stack{M} \,\boxtimes\, \stack{M}};
     \node (gerbes-1-2) at ($(objects-1-2.north east) + ( 1,1)$)
           {\stack{M}};
     \node (gerbes-2-1) at ($(objects-2-1.south west -| gerbes-1-1) + (0,-1.5)$)
           {a^* \stack{G} \;\cong\; \stack{M} \,\boxtimes\, \stack{G}};
     \node (gerbes-2-2) at ($(objects-2-2.south east -| gerbes-1-2) + (0,-1.5)$)
           {\stack{G}};
   }
   { [math arrows]
     \draw (objects-1-1.mid east) -- node {m} (objects-1-2.mid west);
     \draw (objects-1-2)          --          (objects-2-2);
     \draw (objects-1-1)          --          (objects-2-1);
     \draw (objects-2-1.mid east) -- node {a} (objects-2-2.mid west);
   }
   { [dashed arrows]
     \draw (gerbes-2-2) -- (gerbes-1-2);
     \draw (gerbes-2-1) -- (gerbes-1-1);
   }
  \end{tikzpicture}
 \end{center}
\end{theorem}

\begin{proof}
 The trivialization of $\stack{M}|_G$ gives the data of $G$-equivariance for $\stack{M}$, via
 \begin{equation*}
  m|_{G \times \tilde{G}}^* \stack{M}
   \cong \stack{M}|_G \boxtimes \stack{M}
   \cong \stack{G}^0 \boxtimes \stack{M}
   \cong \on{pr}_{\tilde{G}}^* \stack{M}.
 \end{equation*}
 Thus, by \ref{gerbe descent along a torsor}, $\stack{M}$ descends to a gerbe $\stack{G}$ on $Y = \tilde{G}/G$.  By associativity of multiplication, the multiplicative structure on $\stack{M}$ descends to twisted equivariance of $\stack{G}$ over $\smash{\tilde{G}} \times Y$.  Note that this reduces to just $G$-equivariance when the action is restricted to $G \subset \tilde{G}$, since $\stack{M}|_G$ is explicitly trivialized.  Conversely, if $\stack{G}$ is an equivariant gerbe on $Y$ whose equivariance structure descends the multiplicative structure of $\stack{M}$, then over $1 = 1_{\tilde{G}}/G \in Y$, the isomorphism
 \begin{equation*}
  \stack{M}|_G \boxtimes \stack{G}|_1
   \cong a|_{G \times 1}^* \stack{G}|_1
   \cong \on{pr}_{G \times 1}^* \stack{G}|_1
   \cong \stack{G}|_1
 \end{equation*}
 cancels to a trivialization of $\stack{M}|_G$.
\end{proof}

\subsection*{Equivariant twisted objects}

Now we complicate the picture by describing the same structures on sections of a sheaf of
categories twisted by a gerbe.

\begin{theorem}{defn}{equivariant object}
 Let $G$ act on $Y$, let $\stack{G}$ be an equivariant gerbe on $Y$, and let $\stack{F}$ be a sheaf
 of categories of $Y$ (the latter in the large topology); we define the notion of
 \emph{$G$-equivariance of an object} $s \in \stack{F}(\stack{G})_Y$ to be the following.  Recall
 the pullback functors
 \begin{gather*}
  \map{a^*}{\stack{F}|_Y(\stack{G})_Y}{\stack{F}|_{G \times Y}(a^* \stack{G})_{G \times Y}} \\
  \map{\on{pr}^*}{\stack{F}|_Y(\stack{G})_Y}{\stack{F}|_{G \times Y}(\on{pr}^* \stack{G})_{G \times
  Y}}
 \end{gather*}
 and an equivalence $\map{\on{Eq}}{a^* \stack{G}}{\on{pr}^* \stack{G}}$. The equivariant structure
 on $s$ is then:
 \begin{enumerate}
  \item (Equivariance) An isomorphism
  \begin{equation*}
   \map{\epsilon}{(\on{Eq} \otimes \id)(a^* s)}{\on{pr}^* s}
  \end{equation*}
  of sections of $\on{pr}^* \stack{G} \otimes \stack{F}|_{G \times Y}$.

  \item (Identity condition) $\epsilon$ must agree with the unit:
  \begin{equation*}
   (1 \otimes Y)^* \epsilon \circ (\on{Id} \otimes \id) = \id
  \end{equation*}
  as maps from $s$ to itself.

  \item (Cocycle condition) The following diagram must commute:
  \begin{center}
   \tikzsetnextfilename{equivariant-twisted-object-cocycle-condition}
   \begin{tikzpicture}
    \matrix [matrix of math nodes] (objects) [row sep = 3em, column sep = 3em]
    {
       (\on{Eq}_{11} \tensor \on{id})(a_2^* s)
     & (\on{pr}_{G \times Y}^* \on{Eq} \tensor \on{id})(a_{0,1}^* s) \\
       ((m \times Y)^* \on{Eq} \tensor \on{id})(a_2^* s)
     & \on{pr}^* s \\
    };
    { [iso arrows]
      \draw (objects-1-1) to node {(G \times a)^* \epsilon}         
            (objects-1-1 -| objects-1-2.west);
      \draw (objects-1-2) to node {\on{pr}_{G \times Y}^* \epsilon} (objects-2-2);
      \draw (objects-1-1) to node {\on{As} \tensor \on{id}}         (objects-2-1);
      \draw (objects-2-1) to node {(m \times Y)^* \epsilon}         (objects-2-2);
    }
   \end{tikzpicture}
  \end{center}
  where the vertices are all sections of $\on{pr}^* \stack{G} \otimes \stack{F}|_{G \times G \times
  Y}$:
  \begin{center}
   \tikzsetnextfilename{equivariant-twisted-object-cocycle-condition-(objects)}
   \begin{tikzpicture}
     [row sep = 1.2em, node distance = 4em,
      every matrix/.style = {nodes = {anchor = center}},
      baseline = (current bounding box)]
    \matrix (upper left block)
    {
       \node[object]       (upper left element)      
         {a_2^* s};
     & \node[object alias] (upper left top alias)
         {(G \times a)^* a^* s}; \\
       \node[object alias] (upper left left alias)
         {(m \times Y)^* a^* s};
     & \node[object]       (upper left sheaf)
         {a_2^* \stack{G} \tensor \stack{F}_{G \times G \times Y}}; \\
    };
    \matrix (upper right block) [right = of upper left block]
    {
       \node[object alias] (upper right top alias)
         {(G \times a)^* \on{pr}^* s};
     & \node[object]       (upper right element)
         {a_{0,1}^* s}; \\
       \node[object]       (upper right sheaf)
         {a_{0,1}^* \stack{G} \tensor \stack{F}_{G \times G \times Y}};
     & \node[object alias] (upper right right alias)
         {\on{pr}_{G \times Y}^* a^* s}; \\
    };
    \matrix (lower right block) [below = of upper right block]
    {
       \node[object]       (lower right sheaf)
         {\on{pr}^* \stack{G} \tensor \stack{F}_{G \times G \times Y}};
     & \node[object alias] (lower right right alias)
         {\on{pr}_{G \times Y}^* \on{pr}^* s}; \\
       \node[object alias] (lower right bottom alias) 
         {(m \times Y)^* \on{pr}^* s};
     & \node[object]       (lower right element)
         {\on{pr}^* s}; \\
    };
    { [math arrows]
      \draw (upper left sheaf.mid east)  -- (upper right sheaf.mid west);
      \draw (upper right sheaf) -- (upper right sheaf |- lower right sheaf.north);
      \draw (upper left sheaf)  -- coordinate (middle) (lower right sheaf);
    }
    \path (upper right sheaf) -- node {\ref{eq:equivariant gerbe}} (middle);
    { [dashed arrows]
      \draw (upper left top alias) --
        node {(G \times a)^* \epsilon}
          (upper left top alias -| upper right top alias.west);
      \draw (upper right right alias) --
        node {\on{pr}_{G \times Y}^* \epsilon}
          (upper right right alias |- lower right right alias.north);
      \draw (upper left left alias) to
        node[swap] {(m \times Y)^* \epsilon}
          (lower right bottom alias.north west);
    }
    { [equals]
      \draw (upper left left alias)   -- (upper left element)  -- (upper left top alias);
      \draw (upper right top alias)   -- (upper right element) -- (upper right right alias);
      \draw (lower right right alias) -- (lower right element) -- (lower right bottom alias);
    }
    { [every path/.style = |->]
      \draw (upper left element)  -- (upper left sheaf);
      \draw (upper right element) -- (upper right sheaf);
      \draw (lower right element) -- (lower right sheaf);
    }
   \end{tikzpicture}
  \end{center}
 \end{enumerate}
\end{theorem}

\subsection*{Constructions and restrictions}
   
Equivariance data becomes dramatically simplified when the group in question acts transitively.

\begin{theorem}{lem}{constant sheaf multiplicativity}
 Let $X$ be a connected space, $G$ a connected $X$-group, $A$ an abelian group, $\sh{M}$ an $A$-torsor on $G$ trivialized over the unit section $\{1\} = 1_X$; then $\sh{M}$ has at most one multiplicative structure.  Furthermore, if $\map{m}{G \times_X G}{G}$ and we have an isomorphism $m^* \sh{M} \cong \sh{M} \boxtimes_X \sh{M}$, for this to constitute a multiplicative structure it suffices that its restriction to $\{1\} \times G$ be the identity.
\end{theorem}

\begin{proof}
 Any multiplicative structure on $\sh{M}$ is an isomorphism $\mu \colon m^* \sh{M} \cong \sh{M} \boxtimes_X \sh{M}$ which is the identity map over $G \times \{1\}$ and $\{1\} \times G$.  Clearly the product of two multiplicative torsors is again multiplicative, as is the inverse, so it suffices to show that if $\sh{M}$ is trivial, then $\mu$ is the identity map, identifying both sides with the trivial $A$-torsor $\sh{T}^0$ on $G \times_X G$.  Picking some section of $m^* \sh{M}$ we have a noncanonical identification of $\mu$ with an element of $A$ since $G$ is connected, and since $\mu|_{\{1\} \times G} = 1$, that element is $1$, as desired.

 For the second claim, we must verify associativity: that both the induced isomorphisms of $(m \times \id)^* m^* \sh{M} = (\id \times m)^* m^* \sh{M}$ with $\sh{M}^{\boxtimes 3}$ are equal. Let $\sh{T}$ be the difference between the two torsors, thus given two trivializations which must necessarily differ by some $a \in A$.  Restricting to $\{1\} \times G \times \{1\}$ we find that both are equal, so $a = 1$, as desired.
\end{proof}

\begin{theorem}{lem}{transitive equivariant objects}
 Let $X$ be a connected space, $G$ a connected $X$-group acting relatively transitively on the $X$-space $Y$; let $k$ be a field and $\stack{G}$ be a $k^*$-gerbe on $Y$ and $\sh{F}$ a $\stack{G}$-twisted locally constant sheaf of $k$-vector spaces on $Y$.  If $\stack{G}$ and $\sh{F}$ are $G$-equivariant, and if for some section $X \subset Y$ we have $\on{Stab}_G(X)$ connected, then both $\stack{G}$ and $\sh{F}$ are $X$-trivial.
\end{theorem}

\begin{proof}
 Picking $X \subset Y$ and taking $H = \on{Stab}_G(X)$, we have $Y \cong G/H$; let $\map{\pi}{G}{G/H}$ be the quotient.  Then $\pi^* \stack{G}$ is an $G$-equivariant gerbe on $G$ and is thus $X$-trivial by \ref{gerbe descent along a torsor}; likewise, $\pi^* \sh{F}$ is $X$-trivial.  To show that these trivializations descend to $Y$, it suffices to show that the natural $H$-equivariance structures on $\pi^* \stack{G}$ and $\pi^* \sh{F}$ are the trivial ones.

 By \ref{trivial equivariance multiplicative torsor}, the $H$-equivariance of $\pi^* \stack{G}$ is the same as a multiplicative $k^*$-torsor $\on{Eq}$ on the $G$-relative group $H \times_X G$ (where $G$ is considered without its group structure).  Likewise, the equivariance of $\pi^* \sh{F}$ is an isomorphism
 \begin{equation}
  \label{eq:equivariant torsor}
  m_H^*(\pi^* \sh{F}) \cong \on{Eq} \otimes \on{pr}_G^* (\pi^* \sh{F}).
 \end{equation}
 In the notation of \ref{trivial equivariance and multiplicativity}, both pullbacks are $p^* \sh{F}_X$; since the action of $k^*$-torsors on locally constant sheaves of $k$-vector spaces is faithful, this isomorphism makes $\on{Eq}$ a trivial torsor.  It is therefore \textit{a fortiori} $G$-trivial (on $H \times_X G$) so by \ref{constant sheaf multiplicativity} it is multiplicatively trivial since $H$ is connected.  Consequently, $\pi^* \stack{G}$ is indeed $H$-equivariantly trivial, as desired.
 
 Since $\on{Eq}$ is trivial, in \ref{eq:equivariant torsor} the isomorphism is given by a section of $\on{GL}_n(k)$ ($n$ being the rank of $\sh{F}$) over $H \times G$ which is the identity on $\{1\} \times G$. Since $H$ is connected, this section is equal to 1.  Thus, $\sh{F}$ is $X$-trivially $K$-equivariant, as desired.
\end{proof}

We can apply descent along torsors to the construction of ``twisted pullbacks''.

\begin{theorem}{defn}{twisted pullback of gerbe}
 Let $G$ act on $Y$ and let $\tilde{X} \to X$ be a $G$-torsor; set $\tilde{X}_Y = \tilde{X} \times^G Y = (\tilde{X} \times Y)/G$, so $\tilde{X} \times Y$ is a $G$-torsor over $\tilde{X}_Y$. The \emph{twisted pullback} of a $G$-equivariant gerbe $\stack{G}$ on $Y$ is the gerbe on $\tilde{X}_Y$ descending $\on{pr}_Y^* \stack{G}$ from $\tilde{X} \times Y$.  When $\tilde{X}$ is understood, we will denote the twisted pullback by $\tilde{\stack{G}}$.  The same procedure defines the twisted pullback of an equivariant section of some sheaf of categories twisted by $\stack{G}$.
\end{theorem}

Equivariance of gerbes and their twisted objects is compatible with pullback along equivariant maps, as formalized in the following proposition.  

\begin{theorem}{prop}{pullback of equivariant objects}
 Let $G$ act on $X$ and $H$ on $Y$, let $\map{\phi}{H}{G}$ be a homomorphism, let $\map{f}{X}{Y}$ be $(G,H)$-equivariant, let $\stack{G}$ be a $G$-equivariant gerbe on $X$, and let $\stack{F}$ be a sheaf of categories on $X$ with an action of $\stack{H}^1(Y,\sh{G})$.  Then $f^* \stack{G}$ has an $H$-equivariance structure, and for any $G$-equivariant object $s \in \stack{F}|_X(\stack{G})_Y$, there is an associated structure of $H$-equivariance on $f^* s \in f^*\stack{F}|_Y(f^* \stack{G})_Y$. \qed
\end{theorem}

\section{Twisting structures by a gerbe}
\label{s:twisting structures by a gerbe}
 
We now consider the interaction between twisting a sheaf of categories and structures defined on that sheaf. The basic theme is that if a sheaf of categories possesses a structure, a necessary and sufficient condition for it to pass to the twisting by some gerbe is that it be preserved by the action of torsors.

\subsection*{Twisted pairings}

One simple example is the twisting of a pairing: suppose we have three sheaves of categories $\stack{F}_1, \stack{F}_2, \stack{F}_3$ and a pairing (e.g.\ tensor product)
\begin{equation*}
 \map{P}{\stack{F}_1 \times \stack{F}_2}{\stack{F}_3}.
\end{equation*}

\begin{theorem}{defn}{G-biequivariant pairing}
 Let the $\stack{F}_i$ have actions of $\stack{H}^1(X, \sh{G})$; then the structure of \emph{$\sh{G}$-{}bi\-equi\-vari\-ance} is, for any sections $s \in (\stack{F}_1)_U, t \in (\stack{F}_2)_U$, and any $\sh{G}$-torsor $\sh{T}$, natural isomorphisms
 \begin{align*}
  P(\sh{T} \otimes s, t) \cong \sh{T} \otimes P(s,t) &&
  P(s, \sh{T} \otimes t) \cong \sh{T} \otimes P(s,t).
 \end{align*}
 Thus, for any $\sh{G}$-torsors $\sh{T}_1, \sh{T}_2$, we have
 \begin{equation*}
  P(\sh{T}_1 \otimes s, \sh{T}_2 \otimes t) = (\sh{T}_1 \otimes \sh{T}_2) \otimes P(s,t).
 \end{equation*}
\end{theorem}

If $\stack{G}_1$ and $\stack{G}_2$ are two $\sh{G}$-gerbes, then the action of $\stack{H}^1(X, \sh{G})$ on their tensor product is through the tensor product of torsors, so we have the following result:

\begin{theorem}{prop}{twisted pairing}
 If $P$ is $\sh{G}$-biequivariant, it admits a twisting:
 \begin{equation*}
  \map{\tilde{P}}{\stack{F}_1(\stack{G}_1) \times \stack{F}_2(\stack{G}_2)}
      {\stack{F}_3(\stack{G}_1 \otimes \stack{G}_2)}. \qed
 \end{equation*}
\end{theorem}

This construction is a special case of a more general operation in which we have two sheaves of groups $\sh{G}_i$ whose torsors act on the $\stack{F}_i$ ($i = 1, 2$), while $\stack{F}_3$ has an action of
\begin{equation*}
 \stack{H}^1(X, \sh{G}_1 \times \sh{G}_2) 
  \cong \stack{H}^1(X, \sh{G}_1) \times \stack{H}^1(X, \sh{G}_2)
\end{equation*}  
and for which the above pairing of $\stack{F}_1$ and $\stack{F}_2$ is biequivariant.  Then we obtain a twisted pairing, where the $\stack{G}_i$ are now gerbes for the $\sh{G}_i$:
\begin{equation}
 \label{eq:general twisted pairing}
 \tilde{P} \colon \stack{F}_1(\stack{G}_1) \times \stack{F}_2(\stack{G}_2)
           \to    \stack{F}_3(\stack{G}_1 \times \stack{G}_2),
\end{equation}
invoking \ref{gerbe product relation}\ref{en:gerbe product category}.  \ref{twisted pairing} can be recovered when $\sh{G}_1 = \sh{G}_2$ by pushing forward along the product map, by \ref{gerbe product relation}\ref{en:gerbe product relation}.

A special case of this construction is in the relative situation of a map $\map{f}{Y}{X}$, where we have a sheaf of categories $\stack{F}$ in the large topology on $X$ with a tensor product and a gerbe $\stack{G}$ on $Y$:

\begin{theorem}{cor}{outer product of twisted sections}
 In the above setup, we may form products $s \otimes f^* t$, where $s \in \stack{F}(\stack{G})_Y$ and $t \in \stack{F}_X$.
\end{theorem} 
 
\begin{proof}
 Let $\map{\Gamma_f}{Y}{Y \times X}$ be the graph of $f$, so $s \otimes f^* t = \Gamma_f^*(s \boxtimes t)$, where $s \boxtimes t$ makes sense as an object twisted by $\stack{G} \boxtimes \stack{G}^0$ on $Y \times X$, and $\Gamma_f^* (\stack{G} \boxtimes \stack{G}^0) = \stack{G} \otimes f^* \stack{G}^0 \cong \stack{G}$.
\end{proof}

\subsection*{Twisted equivariant pairings}
   
The contents of this section are unlikely to appear motivated until at least the advent of \ref{multiplicative factorizable}, and probably not until \ref{variant tensor category}. However, they embody an important (if obscure) principle which we might state as follows.  Suppose $S$ is a type of structure (for example, grading by some group together with a graded tensor product) and suppose that it can be specified on a category, or sheaf of categories, in many \textit{a priori} inequivalent ways all different, however, only in the morphisms specified by the structure (rather than the functors).  Then these differences can be reproduced via twisting by trivial gerbes having structure $S$.

Let $G$ be an $X$-group and let $\stack{F}$ be a sheaf of categories on $G$; as usual, we denote $\map{\on{pr}_1, \on{pr}_2, m}{G \times_X G}{G}$. 

\begin{theorem}{defn}{G-equivariant pairing}
 A \emph{$G$-equivariant pairing} on $\stack{F}$ is a map of stacks
 \begin{equation*}
  \map{\otimes}{\on{pr}_1^* \stack{F} \times \on{pr}_2^* \stack{F}}{m^* \stack{F}};
 \end{equation*}
 i.e.\ a pairing $(\on{pr}_1^* s, \on{pr}_2^* t) \mapsto m^*(s \otimes t)$ for any sections $s, t \in \stack{F}_G$.  This should be associative in a manner compatible with the associativity of $G$. If $G$ is commutative, then we have a natural associated notion of \emph{commutative $G$-equivariant pairing} in which we impose a commutativity constraint isomorphism on this pairing. Explicitly, it is an isomorphism
 \begin{equation*}
  m^*(s \otimes t) \cong m^* (t \otimes s)
 \end{equation*}
 for every $s, t \in \stack{F}_G$, compatible with associativity in triple products.
\end{theorem}

An example of such a pairing would be, as suggested above, a $G$-graded tensor product if $G$ happens to be a discrete (abelian) group.  Now let $\stack{G}$ be a (possibly commutative) multiplicative $\sh{G}$-gerbe on $G$; if $\stack{F}$ has an action of $\stack{H}^1(X, \sh{G})$ and if the above pairing is $\sh{G}$-biequivariant (in the sense of \ref{G-biequivariant pairing}) then we have a twisting
\begin{equation}
 \label{eq:twisted equivariant pairing}
 \on{pr}_1^* \stack{F}(\on{pr}_1^* \stack{G}) \times \on{pr}_2^* \stack{F}(\on{pr}_2^* \stack{G})
  \to m^* \stack{F}(\on{pr}_1^* \stack{G} \otimes_X \on{pr}_2^* \stack{G})
  \cong m^* \stack{F}(m^* \stack{G}).
\end{equation}
Thus, the stack $\stack{F}(\stack{G})$ has a (commutative) multiplicative $G$-equivariant pairing as well.  If $\stack{G}$ is \emph{trivial} as a multiplicative gerbe, then its \emph{commutative} multiplicative structure consists of the additional data of a single isomorphism which is identified with a section of $\sh{G}(G \times G)$.

\begin{theorem}{prop}{twisted G-equivariant pairing}
 The twisting of a commutative $G$-equivariant pairing $\otimes$ on $\stack{F}$ by a commutative trivially multiplicative gerbe $\stack{G}$ is equivalent to multiplying the commutativity constraint of $\otimes$ by the above section of $\sh{G}(G \times G)$. \qed
\end{theorem}

\section{Construction of gerbes from algebraic geometry}
\label{s:construction of gerbes}

Our main source of concretely-specified gerbes will be those given as ``group-valued Cartier divisors''.  Let $Y$ be a scheme over $\C$; we say that a closed subspace $Z \subset Y$ has \emph{local tubular neighborhoods} if it admits a covering by open sets $V$ in $Y$ such that each $V$ can be written as the product $C \times \C$, where $C$ is contractible, with $Z \cap V = C \times \{0\}$.  This notion clearly requires the ``analytic'' topology, which is why we restrict to schemes over $\C$.  For brevity, we will refer to $Z$ as a ``smooth divisor''.

\subsection*{Order of a ``meromorphic'' map of gerbes}

\begin{theorem}{lem}{2-line bundle of a Cartier divisor}
 Let $D \subset Y$ be a connected, smooth divisor and let $U$ be its complement.  Let $A$ be an abelian group (written multiplicatively) also denoting the constant sheaf with group $A$, let $\stack{G}$ and $\stack{H}$ be $A$-gerbes on $Y$, and let $\map{f}{\stack{G}|_U}{\stack{H}|_U}$ be a map of gerbes; we refer to this as a \emph{meromorphic map} from $\stack{G}$ to $\stack{H}$.  Then $f$ has an associated element $\on{ord}_D(f) \in A$, the \emph{$A$-valued order} of $f$ about $D$, with the following properties:
 \begin{enumerate}
  \item \label{en:order local}
  $\on{ord}_D(f)$ is local on $D$ in that if $Z \subset Y$ is any subspace transverse to $D$ and $C$ is a component of $Z \cap D$, then $\on{ord}_D(f) = \on{ord}_C(f|_{Z \cap U})$.
  
  \item \label{en:order multiplicative}
  If $\map{g}{\stack{H}|_U}{\stack{K}|_U}$, then $\on{ord}_D(g \circ f) = \on{ord}_D(g) \on{ord}_D(f)$, and if $\map{f'}{\stack{G}'|_U}{\stack{H}'|_U}$, then $\on{ord}_D(f \otimes f') = \on{ord}_D(f) \on{ord}_D(f')$.
  
  \item \label{en:order gerbe}
  For any $a \in A$, there exists a unique gerbe $\OO(D)^{\log a}$ and trivialization $\sh{G}^0 \to \OO(D)^{\log a}$ with order  $a$; for any map $f$ with $\on{ord}_D(f) = a$, the induced map $\stack{F}|_U \otimes \OO(D)^{\log a} \to \stack{H}|_U$ with order $1$ extends to an equivalence $\stack{H} \cong \stack{F} \otimes \OO(D)^{\log a}$.
 \end{enumerate}
\end{theorem}

The notation $\OO(D)^{\log a}$ is chosen to match point \ref{en:order multiplicative}, since we then have (formally) $\OO(D)^{\log a} \otimes \OO(D)^{\log b} = \OO(D)^{\log a + \log b} = \OO(D)^{\log (ab)}$, as claimed there.

\begin{proof}
 Let $V = C \times \C$ be a tubular neighborhood on $D$, so $D \cap V = C \times \{0\}$ and $U \cap V = C \times \C^*$, and let $\map{t}{\stack{G}^0|_V}{\stack{G}|_V}$, $\map{u}{\stack{G}^0|_V}{\stack{H}|_V}$ be trivializations.  Since $V$ is contractible, in fact there are no nontrivial automorphisms ($A$-torsors) of $\stack{G}^0|_V$, so this choice is canonical and, most notably, compatible with refinement of $V$.  We have two maps $\map{u t^{-1}, f}{\stack{G}|_{U \cap V}}{\stack{H}|_{U \cap V}}$ which, thus, differ by the action of some $A$-torsor $\sh{T}$ on $U \cap V$.  Since $C$ is contractible, this is the same as giving just an $A$-torsor on $\C^*$, and hence an element of $A$ which we define to be $\on{ord}_D(f)$.  Since $t$ and $u$ are compatible with refinement of $V$, this number is well-defined and locally constant on $D$, hence constant since $D$ is connected.
 
 To show \ref{en:order local}, we note that $\on{ord}_D(f)$ depends only on the transverse portion of $V$.  Point \ref{en:order multiplicative} is clear, since tensor product of maps is the same as tensor product of their corresponding torsors $\sh{T}$ (from the previous paragraph), and likewise for composition.  For point \ref{en:order gerbe}, it is immediate from the definition that any map of order $1$ extends uniquely, since this is true of the torsors $\sh{T}$.  We must construct $\OO(D)^{\log a}$, for which there are several options.
 
 We give the most direct here.  Given any tubular neighborhood $V$, let $\OO(D)^{\log a}|_V$ be the trivial gerbe and let its trivialization $t_V$ on $U \cap V$ be given by the torsor $\sh{T}_V$ on $\C^*$ with order $a$ about $0$, extended to a torsor on $C \times \C^*$.  For two neighborhoods $V_1$ and $V_2$, on $(V_1 \cap V_2) \cap U$, the composition $t_{V_2} t_{V_1}^{-1}$ has order $1$ about $D$ and therefore extends to an equivalence $\phi$ of $\OO(D)^{\log a}|_{V_1}$ with $\OO(D)^{\log a}|_{V_2}$ on $V_1 \cap V_2$.  By definition, the $\phi$'s have a cocycle constraint (a $2$-isomorphism in $\cat{H}^2(V_1 \cap V_2 \cap V_3, A)$) on triple intersections which themselves form a $2$-cocycle, so we may glue the $\OO(D)^{\log a}|_V$ over all $V$ (including those which do not intersect $U$ at all) to get $\OO(D)^{\log a}$ on $Y$ and a trivialization $t$ on $U$ with order $a$, by definition.
\end{proof}

There is a generalization of \ref{2-line bundle of a Cartier divisor} to reducible divisors, although the proof does not generalize since such divisors do not have local tubular neighborhoods at the crossings.

\begin{theorem}{lem}{reducible order}
 Let $D = \bigcup_i D_i$ be a union (not necessarily disjoint) of connected, smooth divisors, $f$ as before.  Then $f$ has orders $a_i$ about the $D_i$ inducing an equivalence $\stack{H} \cong \stack{G} \otimes \bigotimes_i \OO(D_i)^{\log a_i}$.
\end{theorem}

\begin{proof}
 We proceed by induction on the number $n$ of components of $D$.  If $n = 1$, then this is just \ref{2-line bundle of a Cartier divisor}.  In general, let $Y = X \setminus D_n$, which has the divisor $D' = \bigcup_{i < n} D_i$, and so by induction, the lemma holds on $Y$.  Thus, we have
 \begin{equation*}
  \stack{H}|_Y \cong \stack{G}|_Y \otimes \bigotimes_{i < n} \OO_Y(D_i \setminus D_n)^{\log a_i}
 \end{equation*}
 for some $a_i$. We write $\OO_Y$ rather than $\OO$ to emphasize that the latter gerbes are specific to $Y$; however, we have $\OO(D_i)^{\log a_i}|_Y \cong \OO_Y(D_i \setminus D_n)^{\log a_i}$. Thus, the above equation can be written
 \begin{equation*}
  \stack{H}|_Y \cong \Bigl(\stack{G} \otimes \bigotimes_{i < n} \OO(D_i)^{\log a_i}\Bigr)\Bigr|_Y
 \end{equation*}
 and so again by \ref{2-line bundle of a Cartier divisor}, we deduce the existence of some order $a_n$ and the lemma follows.
\end{proof}

It should be noted that the success of this lemma is due to the fact that $a_i$ is constant along $D_i$, and so extends from the complement of several points.  In the context of a possible theory of higher gerbes, the analogue of \ref{2-line bundle of a Cartier divisor} should be expected to hold for ``orders'' valued in other higher gerbes (but not as high) but the above lemma will fail since these orders will not necessarily extend.

\subsection*{Geometric construction of the twisting gerbe}

An alternative and much nicer construction of $\OO(D)^{\log a}$ can be extracted from the line bundle $\OO(D)$, when $D$ is a Cartier divisor.  Let $L$ be the total space of $\OO(D)$, so that each fiber of $L$ over a point of $Y$ is isomorphic to $\Gm$ or, topologically, to $\C^*$.  For each $V \subset Y$, let $\OO(D)^{\log a}_V$ be the category of $A$-torsors on $L|_V$ such that, on each fiber, they have monodromy $a$.  More precisely, let $\sh{L}_a$ be the local system on $\Gm$ with fiber $A$ and monodromy $a$, which is a multiplicative, or ``character'', sheaf in the sense of \ref{multiplicative gerbe}: if $\map{m}{\Gm \times \Gm}{\Gm}$, we have
\begin{equation*}
 m^*(\sh{L}_a) \cong \sh{L}_a \boxtimes \sh{L}_a,
\end{equation*}
compatibly with the group structure of $m$.  We consider $L$ as a $\Gm$-torsor with action map $\map{c}{\Gm \times L}{L}$ and define $\OO(D)^{\log a}_V$ to be the category of twisted equivariant $A$-torsors $\sh{T}$ on $L|_V$: as in \ref{multiplicative gerbe}, we are given
\begin{equation*}
 c^* \sh{T} \cong \sh{L}_a \boxtimes \sh{T},
\end{equation*}
compatibly with the unit section of $\Gm \times L$ and the multiplicativity of $\sh{L}_a$.

\begin{theorem}{prop}{line bundle order}
 Suppose $D$ has multiplicity $1$ as a connected Cartier divisor.  Then the above prescription defines an $A$-gerbe ``$\OO(D)^{\log a}$'' with a trivialization on $U$ of order $a$.
\end{theorem}

\begin{proof}
 If $\OO(D)$ is trivialized on $V$, then we have $L|_V \cong \Gm \times V$; then for any $A$-torsor $\sh{T}$ on $V$, we can take $\tilde{\sh{T}} = \sh{L}_a \boxtimes \sh{T}$, where $\sh{L}_a$ is the unique $A$-torsor on $\Gm$ with order $a$.  Thus, ``$\OO(D)^{\log a}$'' has local sections on $Y$.  It is obvious that we can glue sections on $Y$, since we can glue them as $A$-torsors on $L$, so ``$\OO(D)^{\log a}$'' is a sheaf of categories. The action of $\stack{H}^1(Y, A)$ is given by: for every $A$-torsor $\sh{S}$ on $V$ and every $a$-order $A$-torsor $\sh{T}$ on $L|_V$, we set $\sh{S} \cdot \sh{T} = \smash{\on{pr}_Y^*}(\sh{S}) \otimes \sh{T}$, which is still $a$-twisted since the pullback is plain equivariant. If $\sh{T}_1$, $\sh{T}_2$ are two sections of ``$\OO(D)^{\log a}_V$'', then $\sh{T}_2 \otimes \sh{T}_1^{-1}$ is $\Gm$-equivariant and, hence, is of the form $\on{pr}_Y^*(\sh{S})$ for some $\sh{S}$ on $V$; then $\sh{T}_2 = \sh{S} \cdot \sh{T}_1$, and this identification is obviously free, so ``$\OO(D)
^{\log a}$'' is a gerbe.
 
 We identify $\OO(D)$ with the sheaf of meromorphic functions on $Y$ with poles only at $D$ and of order at most $1$; thus, on any open set $V$ on which $D$ is given by the equation $z = 0$, $\OO(D)$ admits a trivialization $\map{t}{\OO}{\OO(D)}$ sending $1$ to $z^{-1}$.  Let $f$ be the trivialization on $U$ sending $1$ to $1$.  These trivializations give rise to two isomorphisms $\map{\tau,\phi}{\Gm \times V}{L|_V}$; here, $\tau^{-1} \phi = (z, \id)$ is multiplication of the first coordinate by the parameter $z$ on $\Gm$.  Let $\tilde{\sh{T}} = \sh{L}_a \boxtimes \sh{T}$ be a section of ``$\OO(D)^{\log a}$'' relative to the trivialization $\tau$; then relative to $\phi$, $\tilde{\sh{T}} = z^* \sh{L}_a \boxtimes \sh{T}$.  But $z^* \sh{L}_a = \sh{L}_a \otimes \sh{L}_a$, so if we call the induced trivializations of ``$\OO(D)^{\log a}$'' $T$ and $F$, then $F$ differs from $T$ by the action of $\sh{L}_a$ and hence $F$ has order $a$ (and is defined on all of $U$), as desired.
\end{proof}

Excluding the trivialization, this construction also furnishes a gerbe $\sh{L}^{\log a}$ for every line bundle $\sh{L}$ on $Y$ and every $a \in A$, which will be our main (even only) source of interesting gerbes.  We note the following property, obvious from the definition and the computation in the preceding proof:

\begin{theorem}{prop}{log power functorial}
 The construction $\sh{L} \mapsto \sh{L}^{\log a}$ is a $2$-functor from line bundles to $A$-gerbes.  (This means that given maps $\sh{L}_1 \to \sh{L}_2 \to \sh{L}_3$ of line bundles, the map $\sh{L}_1^{\log a} \to \sh{L}_3^{\log a}$ is given an isomorphism with the composition of the other two, and this isomorphism is compatible with associativity.)  We have $(\sh{L}^n)^{\log a} \cong \sh{L}^{\log a^n}$ and, for any map $\map{f}{Z}{Y}$, we have $(f^* \sh{L})^{\log a} \cong f^* \sh{L}^{\log a}$.  In fact, we have $(\sh{L} \otimes \sh{M})^{\log a} \cong \sh{L}^{\log a} \otimes \sh{M}^{\log a}$ for any line bundles $\sh{L}$, $\sh{M}$.
\end{theorem}

\subsection*{Order and equivariant gerbes}

We note the corollary of the proof of \ref{2-line bundle of a Cartier divisor}:

\begin{theorem}{cor}{order and covering maps}
 Let $\map{c}{\tilde{Y}}{Y}$ be a map which is a $d$-fold covering map away from $\tilde{D} = c^{-1}(D)$, and let $\map{\tilde{f}}{c^* \stack{G}|_{\tilde{U}}}{c^* \stack{H}|_{\tilde{U}}}$ be the pullback of $f$.  If $c|_{\tilde{D}}$ is an isomorphism, then $\on{ord}_{\tilde{D}}(\tilde{f}) = \on{ord}_D(f)^d$; if $c$ extends to a covering map of $Y$, then for every component $E$ of $\tilde{D}$, we have $\on{ord}_E(\tilde{f}) = \on{ord}_D(f)$.
\end{theorem}

\begin{proof}
 These facts are true of $A$-torsors on $\C^*$.
\end{proof}

This corollary is descriptive, in that the gerbes on $Y$ must exist.  Morally, such gerbes are equivalent to gerbes on $\tilde{Y}$ with suitable descent data, and we will need the following simple special case.  To state it, we need to make a definition that will figure prominently later.

\begin{theorem}{defn}{S_2 equivariance}
 Let $S_2$ act on a space $Y$, and let $s$ be the nonidentity element.  The structure of \emph{$S_2$-equivariance} for an $A$-gerbe $\stack{G}$ on $Y$, for an action of $S_2$ on $Y$, is the the data of an equivalence $\map{\phi}{\stack{G}}{s^* \stack{G}}$ and a $2$-isomorphism $\map{i}{\phi}{s^* \phi^{-1}}$ such that $s^* i \circ i = \id$ as a $2$-automorphism of $\phi$.
 
 Let $\map{f}{\stack{G}}{\stack{H}}$ be a map of $S_2$-equivariant gerbes ($\stack{H}$ with structure $\psi$, $j$).  The structure of \emph{$S_2$-equivariance} on it is a $2$-isomorphism $\map{c}{\psi \circ f}{s^* f \circ \phi}$ of maps $\stack{G} \to s^* \stack{H}$ such that in the diagram below:
 \begin{equation}
  \label{eq:S_2 map equivariance}
  \tikzsetnextfilename{S_2-equivariance}
  \begin{tikzpicture}[row sep = 4em, column sep = 4em, baseline = (current bounding box)]
   \matrix [matrix of math nodes] (objects)
   {
    \stack{H} & s^* \stack{H} & \stack{H} \\
    \stack{G} & s^* \stack{G} & \stack{G} \\
   };
   { [math arrows]
    \draw (objects-1-1) -- node[swap] {\psi}     (objects-1-2);
    \draw (objects-1-2) -- node[swap] {s^* \psi} (objects-1-3);
    \draw (objects-2-1) -- node       {\phi}     (objects-2-2);
    \draw (objects-2-2) -- node       {s^* \phi} (objects-2-3);
    \draw               (objects-2-1) -- node               {f}     (objects-1-1);
    \draw[auto = false] (objects-2-2) -- node[fill = white] {s^* f} (objects-1-2);
    \draw               (objects-2-3) -- node               {f}     (objects-1-3);
    \draw (objects-1-1) to[bend left = 15]  node       {\id} (objects-1-3);
    \draw (objects-2-1) to[bend right = 15] node[swap] {\id} (objects-2-3);
   }
  \end{tikzpicture}
 \end{equation}
 the $2$-automorphism of $f$ obtained by applying $j$, $c$, $s^* c$, and $i$ to go from the topmost path to the bottommost one, is the identity.
 
 If $S_2$ acts trivially on $Y$, then we define the \emph{trivial $S_2$-equivariance} structure on $\stack{G}$ to be that with $\phi = \id$ and $i = \id$; a \emph{trivialization} of any $S_2$-equivariance structure is an equivalence with this one.
\end{theorem}

\begin{theorem}{lem}{covering order quotient}
 Let $s$ be an involution on $Y$ with fixed-point set a connected smooth divisor $D$ with and
 complement $U$, and let $\stack{G}$ be an $S_2$-equivariant gerbe on $Y$
 together with a trivialization of this structure on $\stack{G}|_D$.  Let $\stack{H}$ be
 another $S_2$-equivariant gerbe and $\map{f}{\stack{G}|_U}{\stack{H}|_U}$ an $S_2$-equivariant
 meromorphic map between them with $\on{ord}_D(f) = a$.  Then the choice of an element $b \in A$
 with $b^2 = a$ is equivalent to the choice of a trivialization of the $S_2$-equivariance on
 $\stack{H}|_D$.
\end{theorem}

The statement of this lemma is motivated by the observation that if $Y/S_2$ existed and $\stack{G}$
descended, then it would have an $S_2$-equivariance structure, and since $Y \to Y/S_2$ would be
isomorphism on $D$, this structure would be trivial there.  Likewise, an $S_2$-equivariant gerbe
$\stack{H}$ would descend only if the same held for it, and then the $S_2$-equivariant
meromorphic map $f$ would descend as well and we could apply \ref{order and covering maps} to
get the claimed order.

\begin{proof}[Proof of Lemma]
 Suppose we had such a trivialization for $\stack{H}|_D$.  Since order is defined locally on $D$, we assume that $Y = \C$ and $s$ is rotation by $\pi$, fixing the origin $D = \{0\}$.  Then we may also assume that $\stack{G}$ and $\stack{H}$ are the trivial gerbes, so that $\phi$ and $\psi$ are identified with $A$-torsors on $\C$, hence also trivial, and so $i$ and $j$ are identified with square roots of $1$ in $A$; the trivializations of the $S_2$-structures over $D$ give $i = j = 1$.  Likewise, $f$ is identified (by definition) with the $A$-torsor $\sh{L}_a$ having monodromy $a$, so $c$ is an isomorphism of $\sh{L}_a$ with $s^* \sh{L}_a$ and the commutativity of diagram \ref{eq:S_2 map equivariance} becomes the identity $s^*c \circ c = \id$.
 
 We claim that such an isomorphism $c$ gives a square root of $a$.  Indeed, consider the germ of $c$ acting as $(\sh{L}_a)_1 \to (s^* \sh{L}_a)_1 = (\sh{L}_a)_{-1}$.  These two stalks are already identified by means of a counterclockwise path $\gamma$ from $1$ to $-1$ in $\C$, so fixing this identification, $c$ is just some element $b \in A$.  Likewise, $s^* c$ is multiplication by $b$ when compared to the isomorphism $(\sh{L}_a)_{-1} \cong (\sh{L}_a)_1$ induced by $s(\gamma)$. Thus, $s^* c \circ c = b^2$ when compared to $s(\gamma) \circ \gamma$, whose associated automorphism of $(\sh{L}_a)_1$ is, by definition, multiplication by $a$.  Since $s^* c \circ c = \id$, we must have $b^2 = a$.
 
 Conversely, suppose we have the square root $b$, and let $\OO(D)^{\log a}$ have the $S_2$-equivariance structure extending that of the trivial gerbe on $U$, using \ref{2-line bundle of a Cartier divisor}\ref{en:order gerbe}.  Using $b$ and the logic from the previous paragraph, we extract a trivialization of the $S_2$-equivariance structure of $\OO(D)^{\log a}$ on $D$.  Then the isomorphism $\stack{G} \otimes \OO(D)^{\log a} \to \stack{H}$ induced by $f$ is in fact an isomorphism of $S_2$-equivariant gerbes; in particular, the trivialization of that structure on the left-hand side over $D$ induces a trivialization of the equivariance structure on $\stack{H}|_D$, as desired.
\end{proof}

\chapter[sf gerbes and quadratic forms]
        {Symmetric factorizable gerbes and their classification by quadratic forms}
\label{c:symmetric factorizable gerbes and classification}

This chapter is devoted to investigating properties of gerbes on a certain kind of space, one with the property of \emph{factorizability}.  The principal goal is to show that the natural data on a gerbe which makes it compatible with this structure renders the category of all such gerbes easily described by simple algebraic objects.  In this chapter, we restrict our attention to those properties of factorizable gerbes which can be related directly to the topology of the underlying factorizable space, while in the next, we will add additional data and derive additional structure.

\section{The affine grassmannian and factorizability}
\label{s:the affine grassmannian and factorizability}

We will use the following notation throughout the rest of the chapter or even the rest of the paper.  The scheme $X$ will be a smooth complex algebraic curve (not necessarily complete).  For any scheme $S$ and any $S$-point $\vect{x} \in X^n(S)$, let $\bar{x} \subset X_S = S \times X$ be the scheme-theoretic union of the graphs of its coordinates, a union of Cartier divisors since $X$ is smooth; we will not need the scheme structure, however.

\begin{theorem}{defn}{factorizable grassmannian}
 Let $G$ be a complex algebraic group.  Taken over all integers $n$, the following functors from schemes to sets constitute the \emph{relative} or \emph{factorizable} or \emph{Beilinson--Drinfeld affine grassmannian}:
 \begin{equation*}
  \on{Gr}_{G,X^n}(S) = \left\{(\vect{x}, \sh{T}, \phi) \middle\vert
    \begin{gathered}
     \text{$\vect{x} \in X^n(S)$, $\sh{T}$ is a $G$-torsor on $X_S$}, \\
     \text{$\phi$ is a trivialization of $\sh{T}$ on $X_S \setminus \bar{x}$}
    \end{gathered}
    \right\}.
 \end{equation*}
 \emph{Only} in this definition, by $G$-torsor, we mean local triviality in the \'etale topology.  By definition, $\on{Gr}_{G,X^n}$ is naturally a functor on schemes over $X^n$. We will write $\on{Gr}_n$ when neither $G$ nor $X$ need be mentioned.
\end{theorem}

As mere presheaves of sets, these functors satisfy the following well-known theorem (see e.g.\ \cite{BD_quantization}*{Theorem 4.5.1}):

\begin{theorem*}{prop}
 The functors $\on{Gr}_{G,X^n}$ are representable by ind-schemes; they are ind-projective if and only if $G$ is reductive.
\end{theorem*}

We denote, as is usual, the ordinary affine grassmannian $G\lp t \rp/G\lb t \rb$ by $\on{Gr}_G$; it is obtained as the fiber of $\on{Gr}_{G,X}$ over any point of $X$ (but not canonically).  In fact, it is easy to see that $\on{Gr}_{G,X^n}$ is a bundle over the complement of all the diagonals in $X^n$, where the fiber is $\on{Gr}_G^n$; this will be a special case of \ref{grassmannian is factorizable}.

We will return to the grassmannian of a general group later.  Now, let $T$ be a complex algebraic torus; since $T$ is a commutative group, its torsors admit a commutative group operation which we denote by the symbol $\otimes$; tensor product of torsors makes $\on{Gr}_{T,X^n}$ a group functor over $X^n$. In fact, we have a pairing over $X^{n + m}$:
\begin{equation}
 \label{eq:torus grassmannian multiplication}
 \on{Gr}_{T,X^n} \times \on{Gr}_{T,X^m} \to \on{Gr}_{T,X^{n + m}}.
\end{equation}

We can give a more detailed description of the grassmannian of a torus. We denote the weight and coweight lattices of $T$ respectively by $\Lambda^T = X^*(T)$ and $\Lambda_T = X_*(T)$; we use the notation $\map{\langle \farg, \farg \rangle}{\Lambda^T \otimes \Lambda_T}{\Z}$ for the natural pairing.  In the second part of the following proposition, we refer to certain \emph{diagonals} in $X^n$ corresponding to partitions $p$ of $\{1, \dots, n\}$ into subsets $p_i$:
\begin{equation*}
 \Delta^n_p = \bigcap_i \{x_k = x_l \mid k, l \in p_i\}.
\end{equation*}
We will also sometimes write $\Delta^n_{n_1, n_2, \dots}$, in which case we mean that $p$ should have parts $\{n_1, n_2, \dots\}$ and singletons consisting of all unmentioned indices.  That is, $\Delta^3_{1,2} = \{x_1 = x_2\}$, and not $\{x_2 = x_3\}$, as might be imagined if the 2 is taken to mean ``group the last two numbers of $\{1, 2, 3\}$''.

\begin{theorem}{prop}{torus grassmannian properties}
 The complex points $\on{Gr}_{T,X^n}(\C)$ have the following structure as ind-complex varieties:
 \begin{enumerate}
  \item \label{en:torus grassmannian 1}
  For $\lambda \in \Lambda_T$ and $x \in X(\C)$, let $\OO(\lambda x)$ be the $T$-torsor on $X$ together with a trivialization $\phi_{\lambda,x}$ on $X \setminus \{x\}$ characterized by the following  property: for any $\mu \in \Lambda^T = \on{Hom}(T, \Gm)$, the line bundle ${}^1 \mu\, \OO(\lambda  x)$ together with its trivialization ${}^1 \mu\,(\phi_{\lambda, x})$ is isomorphic to $\OO(\langle  \mu, \lambda  \rangle x)$.  The map $X(\C) \times \Lambda_T \to \on{Gr}_{T,X}(\C)$ given by
  \begin{equation*}
   (x, \lambda) \mapsto (x, \OO(\lambda x), \phi_{\lambda, x})
  \end{equation*}
  is an isomorphism of groups over $X(\C)$.
  
  \item \label{en:torus grassmannian n}
  The multiplication \ref{eq:torus grassmannian multiplication} identifies the irreducible components of $\on{Gr}_{T,X^n}(\C)$ with the products of those of $\on{Gr}_{T,X}(\C)$; they are thus indexed by $(\Lambda_T)^n$ and isomorphic to $X^n(\C)$.   Let $p$ be a partition of $\{1,\dots,n\}$ into parts $p_j = \{i_1, \dots, i_{n_j}\}$ ($1 \leq i_k \leq n$), and let $\Delta^n_p$ be the corresponding diagonal of $X^n$.  Then the intersection of two components  $\on{Gr}_n^{\lambda_1, \dots, \lambda_n}$, $\on{Gr}_n^{\mu_1, \dots, \mu_n}$ is $\Delta^n_p$ (identifying $X^n$ with both components) if and only if for each $j$, we have $\lambda_{i_1} + \dots + \lambda_{i_{n_j}} = \mu_{i_1} + \dots + \mu_{i_{n_j}}$. \qed
 \end{enumerate}
\end{theorem}

We will only ever use the topological structure of the grassmannians, so we will not mention the
$\C$-points subsequently.  The structure exhibited in the above proposition is formalized the the
following definition:

\begin{theorem}{prop}{grassmannian is factorizable}
 The $\on{Gr}_n$ are \emph{factorizable}:
 \begin{enumerate}
  \item \label{en:factorizable diagonal}
  For $n, m \in \N$, let $p$ be a partition of $[1, n]$ into $m$ parts and $\Delta^n_p$ be the corresponding copy of $X^m$ inside $X^n$.  Then there are isomorphisms
  \begin{equation*}
   \on{Gr}_n|_{\Delta^n_p} \cong \on{Gr}_m
  \end{equation*}
  which are compatible with refinement of the partition $p$;
 
  \item \label{en:factorizable open}
  Let $p$ be a partition as above and suppose its parts $p_i$ have sizes $n_i$; let $X^n_p$ be the open subset of $X^n$ consisting of coordinates $\vect{x} = (x_1, \dots, x_n)$ such that if $x_i = x_j$, then $i, j$ are in the same part of $p$.  Then there are isomorphisms
  \begin{equation*}
   \on{Gr}_n|_{X^n_p} \cong \left.\left(\prod \on{Gr}_{n_i}\right)\right|_{X^n_p}
  \end{equation*}
  compatible with refinement of the partition $p$ (together with, of course, further restrictions to finer $X^n_p$'s).  Furthermore, these isomorphisms are compatible with those above when restricting both \emph{to} some diagonal, and \emph{away from} others, in either order.
  
  \item \label{en:factorizable symmetry}
  For any $n$, we have an equivariance structure for the action of the symmetric group $S_n$ on $X^n$ which is compatible with both of the above classes of isomorphisms.
 \end{enumerate}
\end{theorem}

The details of the second point will be relevant later, so we describe them fully.

\begin{proof}
 Points \ref{en:factorizable diagonal} and \ref{en:factorizable symmetry} are obvious, so we consider only \ref{en:factorizable open}.  Suppose that we have $m$ groups $p_i$ of coordinates' indices such that their graphs are disjoint: $\Gamma(x_k) \cap \Gamma(x_l) = \emptyset$ for $x_k \in p_i$ and $x_l \in p_j$ with $i \neq j$.  Let the coordinate groups be denoted $\vect{x}_i$, let $\bar{x}_i$ be the union of the graphs of the coordinates in each group, let $U_i$ be their complements in $X_S$, and let $V_i = \bigcap_{j \neq i} U_j$.  We establish mutually inverse isomorphisms of $\on{Gr}_n|{X^n_p}$ with $\prod_i \on{Gr}_{n_i}$ as follows:
 \begin{itemize}
  \item 
  For $(\vect{x}, \sh{T}, \phi) \in \on{Gr}_n|_{X^n_p}$, let $\sh{T}_i$ be the $G$-torsor on $X_S$ with $\sh{T}_i|_{U_i} = \sh{T}^0$ the trivial torsor and $\sh{T}_i|_{V_i} \cong \sh{T}|_{V_i}$, with the isomorphism on $U_i \cap V_i = X_S \setminus \bar{x}$ given by $\phi$; thus, there is a natural trivialization $\phi_i$ of $\sh{T}_i|_{U_i}$.  Then $(\vect{x}_i, \sh{T}_i, \phi_i) \in \on{Gr}_{n_i}(S)$, as desired.
  
  \item Given points $(\vect{x}_i, \sh{T}_i, \phi_i) \in \on{Gr}_{n_i}(S)$ with all the $\vect{x}_i$ disjoint from one another, let $\sh{T}$ be the $G$-torsor with $\sh{T}|_{V_i} \cong \sh{T}_i|_{V_i}$ for each $i$, where on $V_i \cap V_j = \bigcap_k U_k$, the isomorphisms are given by $\phi_i \phi_j^{-1}$, which are clearly compatible on triple intersections.  There is a natural trivialization $\phi$ of $\sh{T}$ on $\bigcap V_i = \bigcap U_i = X^n_p$ induced by any of the $\phi_i$'s. Taking $\vect{x}$ to  be the union of the $\vect{x}_i$, we have $(\vect{x}, \sh{T}, \phi) \in \on{Gr}_n|{X^n_p}(S)$, as desired.
 \end{itemize}
 It is easy to see that these constructions are inverse to one another.
\end{proof}

When $G = T$ is a torus, the gluing of point \ref{en:factorizable open} can be realized by taking the tensor product of $T$-torsors as well.  That is, given points $(\vect{x}_i, \sh{T}_i, \phi_i) \in \on{Gr}_{m_i}(S)$, we can simply define $\sh{T} = \bigotimes_i \sh{T}_i$ and, on $X_S \setminus \bar{x}$, trivialize it via $\phi = \bigotimes_i \phi_i|_{X_S \setminus \bar{x}}$.  This shows that the description of \ref{torus grassmannian properties} is indeed the same as the factorizable structure just given.

An example of point \ref{en:torus grassmannian n} is when $\Lambda_T = \Z$ and $n = 3$; then we can say that $\on{Gr}_3^{0,4,-1}$ and $\on{Gr}_3^{2,2,-1}$ intersect in the divisor $\Delta^3_{1,2}$ where $x_1 = x_2$, while $\on{Gr}_3^{1,1,1}$ intersects either of them only in the diagonal $\Delta^3_{1,2,3}$ where $x_1 = x_2 = x_3$.  To illustrate the factorizable structure, we generalize this to say that the $\on{Gr}_3^{a,b,c}$ are divided into groups, indexed by $(a + b, c)$, which intersect over $\Delta^3_{1,2}$; denote the irreducible components of the fiber $\on{Gr}_3|_{\Delta^3_{1,2}}$ by $C^{d,e}$.  If we identify $\Delta^3_{1,2} \cong X^2$ in the obvious way, then $\Delta^3_{1,2} \cap \Delta^3_{1,2,3} \cong \Delta$ is the diagonal in $X^2$, and the $C^{d,e}$ are divided into groups, indexed by $d + e$, which intersect over $\Delta$.  This is exactly the pattern of intersections of the components $\on{Gr}_2^{d,e}$, which are furthermore isomorphic to $C^{d,e}$. We dwell on this point because it will 
become important later.

\section{Gerbes and factorizability}
\label{s:gerbes and factorizability}

In this section, we make the main definition, that of a \emph{symmetric factorizable gerbe} for an
abelian group $A$ on the $\on{Gr}_n$.  Our goal is to associate with every such gerbe $\stack{G} =
\{\stack{G}_n\}$ a quadratic form $Q = Q(\stack{G})$ on the coweight lattice $\Lambda_T$, where
$\map{Q}{\Lambda_T}{A}$ (in particular, when $A$ is the multiplicative group of a field). We will
construct specific examples of such gerbes and show how most such forms (in particular, the Killing
form) can be obtained explicitly.

Here is the main definition of this work.

\begin{theorem}{defn}{sf gerbe}
 Let $A$ be a multiplicative abelian group.  A \emph{symmetric factorizable gerbe} over $A$, or \emph{sf $A$-gerbe}, is a collection of $A$-gerbes $\stack{G}_n$ on $\on{Gr}_n$ with the following properties reminiscent of \ref{grassmannian is factorizable}:
 \begin{enumerate}
  \item \label{en:factorizable gerbe closed}
  For a partition $p$ of $n$ into $m$ parts, we have equivalences
  \begin{equation*}
   \stack{G}_n|_{\Delta^n_p} \cong \stack{G}_m;
  \end{equation*}
  these come with compatibility $2$-isomorphisms for refinements of partitions which  are themselves compatible with further refinement.
  
  \item \label{en:factorizable gerbe open}
  If $p$ has parts of size $n_i$, then we have equivalences
  \begin{equation*}
   \stack{G}_n|_{X^n_p} \cong \Bigl(\bigotimes_i \stack{G}_{n_i}\Bigr)\Bigr|_{X^n_p}
  \end{equation*}
  which also come with compatibility $2$-isomorphisms which are suitably compatible with refinement.
  
  \item \label{en:factorizable gerbe symmetry}
  $\stack{G}_n$ has a structure of $S_n$-equivariance which is compatible with both of the above equivalences and their compatibilities.
 \end{enumerate}
\end{theorem}

The vague restrictions imposed above can be made precise as follows: let $p$ be a partition of $\{1, \dots, n\}$ into $m$ parts, let $q$ be a refinement of $p$ into $l$ parts, and let $r$ be a refinement of $q$ into $k$ parts.  Then for \ref{en:factorizable gerbe closed}, we are given the above equivalences ($1$-morphisms) together with $2$-isomorphisms between the following $1$-morphisms:
\begin{align*}
 \stack{G}_n|_{\Delta^n_p}|_{\Delta^m_q}
  \cong \stack{G}_m|_{\Delta^m_q}
  \cong \stack{G}_l &&
 \stack{G}_n|_{\Delta^n_q} \cong \stack{G}_l
\end{align*}
and likewise for restriction through $q$ and then $r$.  Furthermore, the equivalence
\begin{equation*}
 \stack{G}_n|_{\Delta^n_r} \cong \stack{G}_k
\end{equation*}
can be realized in four ways: (1) directly, considering $r$ as a partition of $\{1, \dots, n\}$; (2) passing first through $p$ and then considering $r$ as a refinement of it; (3) passing first through $q$ and then considering $r$ as a refinement of it; (4) passing through $p$, $q$, and $r$ in order.  Let these numbers denote the resulting $1$-morphism obtained by composition of the maps described, so that we have a diagram of $2$-isomorphisms
\begin{center}
 \tikzsetnextfilename{factoriable_gerbe_closed}
 \begin{tikzpicture}
  \node (1) {(1)};
  \node (2) [above right = of 1] {(2)};
  \node (3) [below right = of 1] {(3)};
  \node (4) [below right = of 2] {(4)};
  { [math arrows]
   \draw (1) -- (2);
   \draw (2) -- (4);
   \draw (1) -- (3);
   \draw (3) -- (4);
  }
 \end{tikzpicture}
\end{center}
which we require to be commutative.  For \ref{en:factorizable gerbe open}, we need to observe that as a refinement of $p$, $q$ induces a partition of each part of $p$; i.e.\ if $p$ has parts $p_i$ (of size $n_i$, in the notation of the definition) then $q$ gives a partition of each set $p_i$ into subsets $q_{ij}$ of size $n_{ij}$, giving an overall partition of $\{1, \dots, n\}$ into parts of size $n_{i'}$.  Then we have two equivalences
\begin{align*}
 \stack{G}_n|_{X^n_p}|_{X^n_q}
  \cong \Bigl(\bigotimes_i \stack{G}_{n_i}\Bigr)\Bigr|_{X^n_p}\Bigr|_{X^n_q}
  \cong \Bigl(\bigotimes_i \bigotimes_j \stack{G}_{n_{ij}}\Bigr)\Bigr|_{X^n_q}, &&
 \stack{G}_n|_{X^n_q} \cong \Bigl(\bigotimes_{i'} \stack{G}_{n_{i'}}\Bigr)\Bigr|_{X^n_q}
\end{align*}
where both pairs of termina are naturally identified; the factorization data then requires a $2$-isomorphism between these $1$-morphisms, and that these $2$-isomorphisms are compatible in the same way as above.  Finally, \ref{en:factorizable gerbe symmetry} is actually a special case of \ref{en:factorizable gerbe open} where the parts of $p$ are reordered singletons, and the fact that this is an $S_n$-action is encoded in the above morphisms and compatibilities.

The most important special case of this definition is when $n = 2$; then, we have the following translations of the three points:
\begin{enumerate}
 \item \label{en:n = 2 factorizable closed}
 $\stack{G}_2|_\Delta \cong \stack{G}_1$;
 
 \item \label{en:n = 2 factorizable open}
 $\stack{G}_2|_{X^2 \setminus \Delta} \cong (\stack{G}_1 \boxtimes \stack{G}_1)|_{X^2
 \setminus \Delta}$;
 
 \item \label{en:n = 2 factorizable symmetry}
 There is an $S_2$-equivariance structure (as in \ref{S_2 equivariance}) on $\stack{G}_2$ whose restriction to $\Delta$ is trivialized and whose restriction to $X^2 \setminus \Delta$, via \ref{en:n = 2 factorizable open}, is made isomorphic to the natural commutativity of the tensor product.
\end{enumerate}

It should be noted that since $S_n$ is generated by transpositions, the compatibility of the general case of point \ref{sf gerbe}\ref{en:factorizable symmetry} with \ref{en:factorizable gerbe open} can be reduced to point \ref{en:n = 2 factorizable symmetry} above.

\section{Classification: the torus case}
\label{s:split exact sequence}

Let $\stack{T}_n$ be an sf gerbe for $T$.  Using \ref{torus grassmannian properties}, we see that $\on{Gr}_2$ breaks up into irreducible components $\on{Gr}_2^{\lambda, \mu}$ isomorphic to $X^2$. Then on each, the above part \ref{en:n = 2 factorizable symmetry} of the factorizable structure is largely redundant: by \ref{2-line bundle of a Cartier divisor}\ref{en:order gerbe}, the $S_2$-equivariance is that unique equivariance structure extending the natural one on $\stack{G}_1 \boxtimes \stack{G}_1$ over $X^2 \setminus \Delta$, using the equivalence of \ref{en:n = 2 factorizable open}.  Thus, the only provision remaining is that this structure be trivialized on $\Delta$.

For each $\lambda \in \Lambda_T$, on $\on{Gr}_2^{\lambda,\lambda} = X^2$ we have a factorization isomorphism
\begin{equation*}
 \map{f^{\lambda,\lambda}}
     {(\stack{T}_1^\lambda \boxtimes \stack{T}_1^\lambda)|_{X^2 \setminus \Delta}}
     {\stack{T}_2^{\lambda,\lambda}|_{X^2 \setminus \Delta}}
\end{equation*}
as above.  This is in fact a meromorphic map of gerbes on $X^2$ having some order $a^{\lambda,\lambda}$; by \ref{covering order quotient} and the above comments on point \ref{en:n = 2 factorizable symmetry}, we have a uniquely specified square root $b^\lambda$ of $a^{\lambda,\lambda}$.  More generally, on the component $\on{Gr}_2^{\lambda, \mu} \cong X^2$, the factorization equivalence
\begin{equation*}
 \map{f^{\lambda,\mu}}
     {(\stack{T}_1^\lambda \boxtimes \stack{T}_1^\mu)|_{X^2 \setminus \Delta}}
     {\stack{T}_2^{\lambda,\mu}|_{X^2 \setminus \Delta}}
\end{equation*}
has some order $a^{\lambda, \mu}$.

\begin{theorem}{defn}{quadratic and bilinear forms}
 The quadratic form $Q = Q(\{\stack{T}_n\})$ associated with the sf gerbe $\stack{T}_n$ is the map $\map{Q}{\Lambda_T}{A}$ with $Q(\lambda) = b^\lambda$ above.  The bilinear form $\kappa = \kappa(\{\stack{T}_n\})$ associated with $\stack{T}_n$ is the map $\map{\kappa}{\Lambda_T \times \Lambda_T}{A}$ with $\kappa(\lambda,\mu) = a^{\lambda, \mu}$.
\end{theorem}

The notation $\kappa$ is chosen because this form is closely related to the Killing form on $\Lambda$ when $G$ is taken to be a more general reductive group.  The terminology ``quadratic form'' and ``bilinear form'' is chosen because we will presently prove that these functions are exactly that.  Before beginning the proof, we observe that
\begin{equation}
 \label{eq:order gerbe restriction}
 \OO_{X^2}(\Delta)^{\log a}|_\Delta
  = \sh{T}_X^{\log a}.
\end{equation}
Indeed, $\Delta^* \OO_{X^2}(\Delta) \cong \sh{T}_X$, where $\sh{T}_X$ is the tangent bundle on $X$.

\begin{theorem}{lem}{bilinear form multiplicativity}
 We have $\stack{T}_1^{\lambda + \mu} \cong \stack{T}_1^\lambda \otimes \stack{T}_1^\mu \otimes \sh{T}_X^{\log \kappa(\lambda,\mu)}$.
\end{theorem}

\begin{proof}
 We restrict to $\Delta$ the defining equivalence $\stack{T}_2^{\lambda, \mu} \cong \stack{T}_1^\lambda \boxtimes \stack{T}_1^\mu \otimes \OO(\Delta)^{\log \kappa(\lambda,\mu)}$.
\end{proof}

\begin{theorem}{prop}{form is bilinear}
 $\kappa$ is a symmetric bilinear form on $\Lambda_T$.
\end{theorem}

\begin{proof}
 We will show that $\kappa(\lambda + \mu, \nu) = \kappa(\lambda, \nu)\kappa(\mu,\nu)$ and that $\kappa$ is symmetric.  The latter is easy: the action of $S_2$ on $\on{Gr}_2$ switches the irreducible components $\on{Gr}_2^{\lambda, \mu}$ and $\on{Gr}_2^{\mu, \lambda}$ and the $S_2$-equivariance of the factorization equivalence \ref{en:n = 2 factorizable open} ensures that the order is the same on both.
 
 For the first, we consider the factorization equivalence on $\on{Gr}_3^{\lambda, \mu, \nu}$,
 \begin{equation}
  \label{eq:bilinear order}
  (\stack{T}_1^\lambda \boxtimes \stack{T}_1^\mu \boxtimes \stack{T}_1^\nu)|_{X^3_p}
   \to \stack{T}_3^{\lambda, \mu,\nu}|_{X^3_p},
 \end{equation}
 where $p$ is the partition $\{1,2,3\} = \{1\} \cup \{2\} \cup \{3\}$, and compute its order about the divisors $\Delta_{2,3}$, $\Delta_{1,3}$, and $\Delta_{1,2}$.  For $\Delta_{2,3}$, now let $q$ be the partition $\{1,2,3\} = \{1\} \cup \{2,3\}$ and consider the decomposition of \ref{grassmannian is factorizable}\ref{en:factorizable open}, restricted to $\on{Gr}_3^{\lambda, \mu, \nu}$:
 \begin{equation*}
  (\on{Gr}_1^{\lambda} \times \on{Gr}_2^{\mu,\nu})|_{X^3_q} = \on{Gr}_3^{\lambda,\mu,\nu}|_{X^3_q}
 \end{equation*}
 This decomposition takes place over $X \times X^2$ and in particular the diagonal $\Delta_{2,3}$ lies below the second factor.  We get the more specific factorization equivalence
 \begin{equation*}
  (\stack{T}_1^\lambda \boxtimes \stack{T}_2^{\mu,\nu})|_{X^3_q}
    \to \stack{T}_3^{\lambda,\mu,\nu}|_{X^3_q}
 \end{equation*}
 which, therefore, has order $a^{\mu,\nu} = \kappa(\mu,\nu)$ about $\Delta_{2,3}$; by the compatibilities of \ref{sf gerbe}\ref{en:factorizable gerbe open}, this is the same as the order of \ref{eq:bilinear order} about $\Delta_{2,3}$. The same argument works for $\Delta_{1,3}$, giving order $\kappa(\lambda,\nu)$, and for $\Delta_{1,2}$, giving order $\kappa(\lambda,\mu)$. We find that
 \begin{equation*}
  \stack{T}_3^{\lambda,\mu,\nu}
   \cong \stack{T}_1^\lambda \boxtimes \stack{T}_1^\mu \boxtimes \stack{T}_1^\nu
         \otimes \OO(\Delta_{13})^{\log \kappa(\lambda, \nu)}
         \otimes \OO(\Delta_{23})^{\log \kappa(\mu, \nu)}
         \otimes \OO(\Delta_{12})^{\log \kappa(\lambda,\mu)}.
 \end{equation*}
 We restrict to $\Delta_{12} \cong X^2$ and use equation \ref{eq:order gerbe restriction} and
 \ref{bilinear form multiplicativity}:
 \begin{align*}
  \stack{T}_2^{\lambda + \mu, \nu}
   &\cong (\stack{T}_1^\lambda \otimes \stack{T}_1^\mu) \boxtimes \stack{T}_1^\nu
         \otimes \OO(\Delta)^{\log \kappa(\lambda,\nu) \kappa(\mu,\nu)}
         \otimes \on{pr}_1^* \sh{T}_X^{\kappa(\lambda,\mu)} \\
   &\cong \stack{T}_1^{\lambda + \mu} \boxtimes \stack{T}_1^\nu
         \otimes \OO(\Delta)^{\log \kappa(\lambda,\nu) \kappa(\mu,\nu)}.
 \end{align*}
 This equivalence extends the factorization equivalence defined on $X^2 \setminus \Delta$ and the notation includes an equality of monodromies, so we have $\kappa(\lambda + \mu,\nu) = \kappa(\lambda,\nu) \kappa(\mu,\nu)$.
\end{proof}

\begin{theorem}{prop}{form is quadratic}
 $Q$ is a quadratic form on $\Lambda_T$.  More precisely, its associated bilinear form is $\kappa$, in that we have
 \begin{equation*}
  Q(\lambda + \mu) = Q(\lambda)Q(\mu) \kappa(\lambda, \mu).
 \end{equation*}
\end{theorem}

\begin{proof}
 We analyze the factorization equivalence
 \begin{equation}
  \label{eq:quadratic order}
  (\stack{T}_1^\lambda \boxtimes \stack{T}_1^\mu \boxtimes \stack{T}_1^\lambda
    \boxtimes \stack{T}_1^\mu)|_{X^4_p} \to \stack{T}_4^{\lambda,\mu,\lambda,\mu}|_{X^4_p}
 \end{equation}
 where $p$ is the partition $\{1,2,3,4\} = \{1\} \cup \{2\} \cup \{3\} \cup \{4\}$.  There are six exceptional divisors where any two of the four coordinates come together, and we compute the order of this equivalence about each one, then restrict appropriately as in the proof of \ref{form is bilinear}.
 
 The cases of the divisors $\Delta_{1,2}$, $\Delta_{1,4}$, $\Delta_{2,3}$, and $\Delta_{3,4}$ are essentially the same so we do the first one and state the results for the others.  For $\Delta_{1,2}$ we form the partition $q \colon \{1,2,3,4\} = \{1,2\} \cup \{3\} \cup \{4\}$ and consider the more specific factorization equivalence (through which \ref{eq:quadratic order} factors),
 \begin{equation*}
  (\stack{T}_2^{\lambda,\mu} \boxtimes \stack{T}_1^\lambda \boxtimes \stack{T}_2^\mu)|_{X^4_q}
   \to \stack{T}_4^{\lambda,\mu,\lambda,\mu}|_{X^4_q}
 \end{equation*}
 which evidently shows that \ref{eq:quadratic order} has order $\kappa(\lambda, \mu)$ about $\Delta_{1,2}$.  The same analysis shows that in fact we get this order for all four of the named divisors.
 
 The two remaining divisors $\Delta_{1,3}$ and $\Delta_{2,4}$ give, respectively, monodromies of $\kappa(\lambda, \lambda)$ and $\kappa(\mu,\mu)$ by the same reasoning, but we prefer to write them as $Q(\lambda)^2$ and $Q(\mu)^2$ in preparation for what is to come.  Thus, \ref{eq:quadratic order} becomes
 \begin{multline}
  \label{eq:new quotient order}
  \stack{T}_4^{\lambda,\mu,\lambda,\mu}
   \cong \stack{T}_1^\lambda \boxtimes \stack{T}_1^\mu \boxtimes \stack{T}_1^\lambda
         \boxtimes \stack{T}_1^\mu
         \otimes \OO(\Delta_{1,2})^{\log \kappa(\lambda,\mu)}
         \otimes \OO(\Delta_{3,4})^{\log \kappa(\lambda,\mu)} \\
         \otimes \OO(\Delta_{2,3} \cup \Delta_{1,4})^{\log \kappa(\lambda,\mu)}
         \otimes \OO(\Delta_{1,3})^{\log Q(\lambda)^2}
         \otimes \OO(\Delta_{2,4})^{\log Q(\mu)^2}.
 \end{multline}
 We restrict to the diagonal $\Delta_{1,2} \cap \Delta_{3,4} \cong \on{Gr}_2^{\lambda + \mu, \lambda + \mu} \cong X^2$, by \ref{torus grassmannian properties}.  Then all but $\Delta_{1,2}$ and $\Delta_{3,4}$ become the main diagonal $\Delta$, while we have
 \begin{equation*}
  \OO(\Delta_{1,2})^{\log \kappa(\lambda,\mu)}|_{\Delta_{1,2} \cap \Delta_{3,4}}
   = \OO(\Delta_{1,2})^{\log \kappa(\lambda,\mu)}|_{\Delta_{1,2}}|_{\Delta_{3,4}}
   = \on{pr}_1^* \sh{T}_X^{\kappa(\lambda,\mu)}
 \end{equation*}
 and likewise we have $\OO(\Delta_{3,4})^{\log \kappa(\lambda,\mu)}|_{\Delta_{1,2} \cap \Delta_{3,4}} \cong \on{pr}_2^* \sh{T}_X^{\kappa(\lambda,\mu)}$.  Therefore, after restricting and applying \ref{bilinear form multiplicativity}, equation \ref{eq:new quotient order} becomes
 \begin{equation}
  \label{eq:restricted quadratic order}
  \stack{T}_2^{\lambda + \mu, \lambda + \mu}
   \cong \stack{T}_1^{\lambda + \mu} \boxtimes \stack{T}_1^{\lambda + \mu}
         \otimes \OO(\Delta)^{\log Q(\lambda)^2 Q(\mu)^2 \kappa(\lambda,\mu)^2}.
 \end{equation}
 As in \ref{form is bilinear}, this is an equality of orders, and so by definition we have
 \begin{equation*}
  Q(\lambda + \mu)^2 = Q(\lambda)^2 Q(\mu)^2 \kappa(\lambda,\mu)^2.
 \end{equation*}
 In order to take the square root, we simply apply \ref{covering order quotient} to the coordinate-swapping action of $S_2$ along with provision \ref{sf gerbe}\ref{en:factorizable gerbe symmetry}, as follows: consider the larger action of $S_2$ on $X^4$ which swaps the first and last \emph{pairs} of coordinates; it exchanges $\Delta_{1,4}$ with $\Delta_{2,3}$ and fixes both $\Delta_{1,3}$ and $\Delta_{2,4}$, which are the four divisors whose corresponding monodromies appear in equation \ref{eq:restricted quadratic order}.  Thus, the square root of $\kappa(\lambda, \mu)^2$ extracted is indeed $\kappa(\lambda,\mu)$ (the order about either one of the first pair) and those of $Q(\lambda)^2$ and $Q(\mu)^2$ are respectively $Q(\lambda)$ and $Q(\mu)$, since that is how $Q$ is defined.  This completes the proof.
\end{proof}

Now that we have proved that $Q$ is a quadratic form, we can come full circle and show that it effectively determines the entire factorizable structure of the $\stack{T}_n$.  To state this, we make the definition:

\begin{theorem}{defn}{sf category}
 Let $\cat{H}^2_\text{sf}(T,X,A)$ be the $2$-Picard category whose objects are sf $A$-gerbes on $\on{Gr}_{T,X^n}$.  Let $Q(\Lambda_T, A)$ be the group of $A$-valued quadratic forms on $\Lambda_T$.
\end{theorem}

In this language, \ref{form is quadratic} together with the defintion of $Q$ shows that we have a homomorphism $\cat{H}^2_\text{sf}(T,X,A) \to Q(\Lambda_T, A)$.  This extends to the following theorem:

\begin{theorem}{prop}{torus exact sequence}
 The above homomorphism fits into a split short exact sequence of $2$-Picard categories,
 \begin{equation*}
  1 \to \cat{Hom}(\Lambda_T, \cat{H}^2(X,A))
    \to \cat{H}^2_\text{sf}(T,X,A)
    \to Q(\Lambda_T, A)
    \to 1.
 \end{equation*}
\end{theorem}

Here, the first term $\cat{Hom}(\Lambda_T, \cat{H}^2(X,A))$ refers to the category of homomorphisms from $\Lambda$ into the $2$-Picard category $\cat{H}^2(X,A)$, or in other words of ``commutative multiplicative $A$-gerbes'', as given in \ref{multiplicative lattice gerbe}.  The concept of an exact sequence of 2-Picard categories was described before \ref{gerbe exact sequence}.

\begin{proof}
 The easiest part of the proposition is that the automorphism category of the trivial sf gerbe is identified with that of commutative multiplicative $A$-torsors for $\Lambda_T$.  Indeed, it is (by definition) the category of ``sf $A$-torsors'' on $\on{Gr}_T$, but since all the diagonals in $X^n$ are Zariski-closed, the factorization isomorphisms extend across them and become a multiplicative structure, commutative by virtue of $S_n$-equivariance.  Likewise for their 2-morphisms (i.e.\ sf maps to $A$).

 Suppose we are given such a multiplicative gerbe $\{\stack{T}^\lambda\}$ and a quadratic form $Q$; we will construct an sf gerbe $\{\stack{T}_n\}$ in a manner compatible with multiplication.  We set $\kappa$ to be the bilinear form defined by $Q$ as in \ref{form is quadratic}.  First, let
 \begin{equation}
  \label{eq:gerbe construction}
  \stack{T}_1^\lambda = \stack{T}^\lambda \otimes \sh{T}_X^{\log Q(\lambda)}.
 \end{equation}
 This defines $\stack{T}_1$ by defining it on each connected component of $\on{Gr}_{T,X}$.  In general, to define $\stack{T}_n$ we need only define its restriction to each component copy $\on{Gr}_{T,X^n}^{\lambda_1, \dots, \lambda_n} \cong X^n$ and give isomorphisms of these restrictions on the diagonals described in \ref{torus grassmannian properties}.  Here, we take
 \begin{equation}
  \label{eq:torus gerbe diagonal orders}
  \stack{T}_n^{\lambda_1, \dots, \lambda_n}
   = \stack{T}_1^{\lambda_1} \boxtimes \dots \boxtimes \stack{T}_1^{\lambda_n}
     \otimes \bigotimes_{i,j} \OO(\Delta_{i,j})^{\log \kappa(\lambda_i, \lambda_j)},
 \end{equation}
 where as in the previous proofs, $\Delta_{i,j}$ is the divisor where $x_i = x_j$.  We show how to define the factorization equivalences for $n = 2$: there, we take
 \begin{multline}
  \label{eq:diagonal factorization construction}
  \stack{T}_2^{\lambda,\mu}|_\Delta
   = \stack{T}_1^\lambda \otimes \stack{T}_1^\mu
     \otimes \OO(\Delta)^{\log \kappa(\lambda,\mu)}|_\Delta
   = \stack{T}^\lambda \otimes \stack{T}^\mu
     \otimes \sh{T}_X^{\log Q(\lambda)Q(\mu)\kappa(\lambda,\mu)} \\
   \cong \stack{T}^{\lambda + \mu} \otimes \sh{T}_X^{\log Q(\lambda + \mu)}
   = \stack{T}_1^{\lambda + \mu},
 \end{multline}
 and let the equivalences
 \begin{equation*}
  (\stack{T}_1^\lambda \boxtimes \stack{T}_1^\mu)|_{X^2 \setminus \Delta}
  \cong \stack{T}_2^{\lambda + \mu}|_{X^2 \setminus \Delta}
 \end{equation*}
 be the natural ones induced by the trivialization of $\OO(\Delta)^{\log \kappa(\lambda,\mu)}$.  As noted after \ref{sf gerbe}, this suffices to define the $S_2$-equivariance of $\stack{T}_2$ as well: it is the unique equivalence $\map{\phi}{\stack{T}_2}{s^* \stack{T}_2}$ (where $s$ is the involution defined by swapping coordinates in $X^2$) extending the natural one on tensor products using the above factorization on $X^2 \setminus \Delta$; this is the same as taking the product of the natural symmetry of the tensor product with the $S_2$-equivariance of $\OO(\Delta)^{\log \kappa(\lambda, \mu)}$ induced by the invariance of the trivial gerbe on $X^2 \setminus \Delta$.  We define its trivialization on $\Delta$ to be that arising (in reference to the equations \ref{eq:diagonal factorization construction}) from the equivalence of $S_2$-equivariant gerbes on $X^2$:
 \begin{equation*}
  \OO(\Delta)^{\log \kappa(\lambda,\mu)} \otimes \OO(\Delta)^{\log Q(\lambda)}
   \otimes \OO(\Delta)^{\log Q(\mu)}
   \cong \OO(\Delta)^{\log Q(\lambda + \mu)}.
 \end{equation*}
 With these definitions, it is clear that $\stack{T}_2$ defines the bilinear form $\kappa$ and quadratic form $Q$. As an example of the higher factorizations, we verify the diagonal restrictions for $\stack{T}_3$:
 \begin{equation*}
  \stack{T}_3^{\lambda, \mu, \nu}|_{\Delta_{1,2}}
   = (\stack{T}_1^\lambda \otimes \stack{T}_1^\mu) \boxtimes \stack{T}_1^\nu
     \otimes \OO(\Delta_{1,2})^{\log \kappa(\lambda,\mu)}|_{\Delta_{1,2}}
     \otimes \OO(\Delta)^{\log \kappa(\lambda, \nu)\kappa(\mu,\nu)}.
 \end{equation*}
 Since $\OO(\Delta_{1,2})^{\log \kappa(\lambda,\mu)}|_{\Delta_{1,2}} = \on{pr}_1^* \sh{T}_X^{\log \kappa(\lambda,\mu)}$, the first factor becomes $\stack{T}_1^{\lambda + \mu}$ as in \ref{eq:diagonal factorization construction}, while for the remaining $\OO(\Delta)$, we simplify the exponent using the bilinearity of $\kappa$.  Similar considerations give the factorizations away from diagonals.
 
 It is easy to see that the construction $(\{\stack{T}^\lambda\}, Q) \mapsto \{\stack{T}_n\}$ is a homomorphism 
 \begin{equation*}
  \cat{Hom}(\Lambda_T, \cat{H}^2(X, A)) \times Q(\Lambda_T, A) \to \cat{H}^2_\text{sf}(T,X,A).
 \end{equation*}
 The resulting sequence
 \begin{equation*}
  1 \to \cat{Hom}(\Lambda_T, \cat{H}^2(X, A)) \to \cat{H}^2_\text{sf}(T,X,A)
    \to Q(\Lambda_T, A) \to 1
 \end{equation*}
 is evidently exact at both the left and (by this construction) the right.  If we have an sf gerbe $\stack{T}_n$ with $Q(\stack{T}_n) = 1$ the trivial form, then \ref{bilinear form multiplicativity} shows that the $\stack{T}_1^\lambda$ are themselves a commutative multiplicative gerbe on $X \times \Lambda$, and so $\stack{T}_n$ is in fact defined by this multiplicative gerbe by the above construction.  This gives exactness in the middle and completes the proof.
\end{proof}

\section{Multiplicative factorizable gerbes}
\label{s:multiplicative factorizable gerbes}

It is evident that the concepts of factorizability and multiplicativity are closely related.  In this section we combine the two and discuss their interactions.

\begin{theorem}{defn}{sf multiplicative gerbe}
 An sf $A$-gerbe $\stack{T}_n$ on $\on{Gr}_{T,X^n}$ forms a \emph{multiplicative sf gerbe} if the factorization equivalences of \ref{sf gerbe}\ref{en:factorizable gerbe open} extend across the exceptional diagonals using the multiplicative structure \ref{eq:torus grassmannian multiplication} of $\on{Gr}_{T,X^n}$.
\end{theorem}

Note that, by \ref{2-line bundle of a Cartier divisor}, if this extension exists it is unique, so this is indeed a condition rather than a structure.  Also note that, denoting by $m_n$ the multiplication on $\on{Gr}_{T,X^n}$, the induced equivalences $\stack{T}_n \boxtimes \stack{T}_n \cong m_n^* \stack{T}_n$ are the structure of a commutative multiplicative $A$-gerbe (in particular, $\stack{T}_1$ is such a gerbe as made explicit above).  This definition is inspired by the following phenomenon for sections of $A$:

\begin{theorem}{lem}{sf multiplicative functions}
 Let $f \colon \Lambda_T \to A$ be a function.  Then the following are equivalent:
 \begin{enumerate}
  \item \label{en:sf mult homomorphism}
  $f$ is a homomorphism.
 
  \item \label{en:sf mult factorizable}
  There exists a unique factorizable collection $f_n \colon \on{Gr}_{T,X^n} \to A$ of locally constant functions such that $f_1|\on{Gr}_{T,X}^\lambda = f(\lambda)$.
  
  \item \label{en:sf mult multiplicative}
  The above $f_n$ are actually multiplicative factorizable.
 \end{enumerate}
\end{theorem}

\begin{proof}
 \mbox{}
 
 \mbox{$\ref{en:sf mult homomorphism} \implies \ref{en:sf mult factorizable}$}:\kern0.5em We set $f_n|\on{Gr}_{T,X^n}^{\lambda_1, \dots, \lambda_n} = f(\lambda_1 + \dots + \lambda_n) = f(\lambda_1) \cdots f(\lambda_n)$; then, given the intersection pattern of \ref{torus grassmannian properties}, the fact that $f$ is a homomorphism makes this a factorizable function.
 
 \mbox{$\ref{en:sf mult factorizable} \implies \ref{en:sf mult multiplicative}$}:\kern0.5em Since the $f_n$ are locally constant, the factorization \emph{equalities} extend across the exceptional diagonals.
 
 \mbox{$\ref{en:sf mult multiplicative} \implies \ref{en:sf mult homomorphism}$}:\kern0.5em The equation $f(\lambda + \mu) = f(\lambda) f(\mu)$ follows from the multiplicativity property applied to the map $\on{Gr}_{T,X} \times \on{Gr}_{T,X} \to \on{Gr}_{T,X}^2$, when restricted to the diagonal in $X^2$ (which in the latter is again $\on{Gr}_{T,X}$).
\end{proof}

Unlike for functions, the factorization equivalences do not automatically extend to a multiplicative sf structure for gerbes.  By definition, part of this failure is measured by the associated bilinear form $\kappa$, and in fact, this is the entire obstruction:

\begin{theorem}{prop}{multiplicative factorizable}
 Suppose that $\stack{T}_n$ is an sf $A$-gerbe on $\on{Gr}_{T,X^n}$ whose associated bilinear form is trivial.  Then its factorization equivalences extend to the structure of a multiplicative sf gerbe. Furthermore, if $\stack{T}_n$ is the trivial gerbe, then the associativity constraint is identified with $1 \in A$ and the commutativity constraint of $\stack{T}_1^\lambda \otimes \stack{T}_1^\mu$ with $Q(\lambda)Q(\mu)$.
\end{theorem}

We note that by factorizability, the commutativity of a more general product $\stack{T}_n^{\lambda_1, \dots, \lambda_n} \otimes \stack{T}_n^{\mu_1, \dots, \mu_n}$ is therefore $Q(\lambda_1) \dots Q(\mu_n)$, so it is only necessary to compute the one in the proposition.

\begin{proof}
 That $\stack{T}_1$ is multiplicative comes from \ref{bilinear form multiplicativity}.  Consider the general factorization equivalence $\stack{T}_n \boxtimes \stack{T}_m \cong m_{n,m}^* \stack{T}_{n + m}$ defined away from the exceptional diagonals in the map $m_{n,m}$ of \ref{eq:torus grassmannian multiplication}. On the component $\on{Gr}_n^{\lambda_1, \dots, \lambda_n} \times \on{Gr}_m^{\mu_1, \dots, \mu_m} \cong X^n \times X^m$, let $x_1, \dots, x_n$ and $y_1, \dots, y_m$ be the coordinates and $\Delta_{ij} = \{x_i = y_j\}$ one of the exceptional diagonals; the order of the above equivalence about $\Delta_{ij}$ is $\kappa(\lambda_i, \mu_j)$ by \ref{eq:torus gerbe diagonal orders}.  This shows that the $\stack{T}_n$ are in fact multiplicative factorizable when $\kappa$ is trivial.
 
 To compute the associativity constraint of each $\stack{T}_n$, as in the comment above we need only compute it for $\stack{T}_1$, for which we consider, as in \ref{form is bilinear}, the component $\on{Gr}_3^{\lambda, \mu, \nu} \subset \on{Gr}_3$. By restricting $\stack{T}_3^{\lambda, \mu, \nu}$ successively along the factorization maps
 \begin{equation*}
  \on{Gr}_1^\lambda \times \on{Gr}_1^\mu \times \on{Gr}_1^\nu
   \to \on{Gr}_2^{\lambda,\mu} \times \on{Gr}_1^\nu
   \to \on{Gr}_3^{\lambda,\mu,\nu}
 \end{equation*}
 and applying factorizability of $\stack{T}_3$, we obtain the isomorphism
 \begin{equation*}
  (\stack{T}_1^\lambda \boxtimes \stack{T}_1^\mu) \boxtimes \stack{T}_1^\nu
   \cong \stack{T}_1^{\lambda,\mu,\nu}.
 \end{equation*}
 Likewise, using $\on{Gr}_2^{\mu,\nu}$ in the second stage we get the other bracketing.  But either way the maps are obtained by restriction of the factorization of $\stack{T}_3^{\lambda,\mu,\nu}$ to the smallest diagonal $\Delta_{123} \subset X^3$, and both these restrictions are equal.

 To compute commutativity, we must show that the two equivalences
 \begin{align*}
  \stack{T}_1^\lambda \otimes \stack{T}_1^\mu \cong \stack{T}_1^{\lambda + \mu} &&
  \stack{T}_1^\mu \otimes \stack{T}_1^\lambda \cong \stack{T}_1^{\lambda + \mu}
 \end{align*}
 differ by the constant multiple $Q(\lambda + \mu) = Q(\lambda)Q(\mu)$ (since $\kappa$ is trivial). Each of these is induced by the factorization equivalences
 \begin{align*}
  \stack{T}_1^\lambda \boxtimes \stack{T}_1^\mu \cong \stack{T}_2^{\lambda, \mu} &&
  \stack{T}_1^\mu \boxtimes \stack{T}_1^\lambda \cong \stack{T}_2^{\mu, \lambda}
 \end{align*}
 upon restriction to the diagonal in $X^2$, and these are obtained from each other by swapping coordinates.  We consider, as in \ref{form is quadratic}, the component $\on{Gr}_4^{\lambda, \mu, \lambda, \mu} \subset \on{Gr}_4$, and write the pair of equivalences, each obtained from factorization,
 \begin{equation*}
  \stack{T}_1^\lambda \boxtimes \stack{T}_1^\mu
   \boxtimes \stack{T}_1^\mu \boxtimes \stack{T}_1^\lambda
  \cong \stack{T}_4^{\lambda,\mu,\mu,\lambda}
  \cong \stack{T}_2^{\lambda,\mu} \boxtimes \stack{T}_2^{\mu, \lambda}
 \end{equation*}
 where if we bracket the first and second pairs on the left, the composed equivalence is the product of the two exhibited above.  We restrict the second half of this equation to the diagonal $\Delta_{12,34} = \{x_1 = x_2, x_3 = x_4\}$, obtaining
 \begin{equation*}
  \stack{T}_2^{\lambda + \mu, \lambda + \mu}
   \cong \stack{T}_1^{\lambda + \mu} \boxtimes \stack{T}_1^{\mu + \lambda},
 \end{equation*}
 where of course $\lambda + \mu = \mu + \lambda$.  If $\stack{T}_1$ is trivialized, we may use the first half of the previous equation to trivialize both sides and, comparing with the proof of \ref{covering order quotient}, we see that indeed the coordinate swap makes these trivializations differ by $Q(\lambda + \mu)$.  Since $\kappa$ is trivial, $Q$ is a homomorphism, so this is the same as $Q(\lambda) Q(\mu)$.
\end{proof}

We give an alternative description of such gerbes, beginning by constructing a certain sf sheaf of groups.

\begin{theorem}{defn}{factorizable version of a group}
 For any sheaf of abelian groups $\sh{A}$ on $X$, let $\on{Fact}(\sh{A})_n$ be the following sheaf on $X^n$: for every $U \subset X^n$, $\on{Fact}(\sh{A})_n(U)$ is the subsheaf of $\sh{A}^n$ subject only to the restrictions that if $U \cap \Delta_{ij} \neq \emptyset$, then $\on{Fact}(\sh{A})(U) \subset \Delta_{ij}$ (the first diagonal in $X^n$, the second in $\sh{A}^n$).  If $\sh{A} = \csheaf{A}$ is the constant sheaf on a discrete abelian group $A$, we will just write $\on{Fact}(A)_n$.
\end{theorem}

It should be noted that $\on{Fact}(\sh{A})_n$ is \emph{not} factorizable in the same way that $\on{Gr}_{T,X^n}$ is, as described in \ref{torus grassmannian properties}, since it degenerates to subsets over the diagonals rather than quotients.  In fact, the two are related by the following constructions:

\begin{theorem}{lem}{factorizable group dual}
 Let $T$ be an algebraic torus, ${}^L T$ the dual torus with
 \begin{align*}
  \on{Hom}(\Gm, {}^L T) = \on{Hom}(T, \Gm) && \on{Hom}({}^L T, \Gm) = \on{Hom}(\Gm, T);
 \end{align*}
 then we have
 \begin{equation*}
  \shHom(\on{Fact}({}^L T)_n, \Gm \times X^n) \cong \on{Gr}_{T,X^n}
 \end{equation*}
 as sheaves of groups over $X^n$ (note that ${}^L T$ refers to the sheaf of groups represented by the group scheme ${}^L T \times X$ over $X$, and not the constant sheaf on the abelian group ${}^L T(\C)$.  This is the only time we use a non-discrete group in the $\on{Fact}(\sh{A})_n$ construction).
 
 Conversely, let $A$ be a (discrete) abelian group and denote
 \begin{equation*}
  {}^L T(A) = \on{Hom}(\Lambda_T, A) \cong \Lambda^T \otimes A.  
 \end{equation*}
 Then we can identify the pushforward to $X^n$ of the constant sheaf $\csheaf{A}$ on $\on{Gr}_{T,X^n}$ with $\on{Fact}({}^L T(A))_n$.
\end{theorem}

\begin{proof}
 Let $U \subset X^n$ be the complement of all the diagonals $\Delta_{ij}$; then since we have $\on{Fact}({}^L T)_n|_U \cong {}^L T^n$ and $\on{Hom}({}^L T, \Gm) \cong \Lambda^{{}^L T} = \Lambda_T$, we have
 \begin{equation*}
  \shHom(\on{Fact}({}^L T)_n, \Gm \times X^n)|_U
   \cong (\Lambda^{{}^L T})^n \times X^n
   = \Lambda_T^n \times X^n.
 \end{equation*}
 We use \ref{torus grassmannian properties} to construct an isomorphism with $\on{Gr}_{T,X^n}$ over $U$ and use the same indexing convention for the components.  The first part of the lemma would follow if we could show that the closure of $\shHom(\on{Fact}({}^L T)_n, \Gm \times X^n)|_U^{\lambda_1, \dots, \lambda_n}$ intersects, over $\Delta_{ij}$, all the $\shHom(\on{Fact}({}^L T)_n, \Gm \times X^n)|_U^{\mu_1, \dots, \mu_n}$ with $\lambda_i + \lambda_j = \mu_i + \mu_j$ and all other $\lambda_k = \mu_k$. To see this, let $V$ be any neighborhood intersecting only $\Delta_{ij}$, so that the restriction map
 \begin{equation*}
  {}^L T^{n - 1} \cong \on{Fact}({}^L T)_n|_V \to \on{Fact}({}^L T)|_{V \setminus \Delta_{ij}}
   \cong {}^L T^n
 \end{equation*}
 is the inclusion of ${}^L T^{n - 1}$ as the $ij$'th diagonal of ${}^L T^n$.  The corresponding map $\Lambda_T^n \to \Lambda_T^{n - 1}$ is the summation of the $i$'th and $j$'th coordinates, as desired.
 
 For the second part, we use \ref{torus grassmannian properties} to identify the connected components of $\on{Gr}_{T,X^n}$ over any open set $V \subset X^n$: when $V \subset U$ they correspond to $\Lambda_T^n$, so that locally constant functions on $V$ are identified with $\on{Hom}(\Lambda_T^n, A) = {}^L T(A)^n$, by definition of the latter.  For each diagonal $\Delta_{ij}$ intersecting $V$, components with the same sum of their $i$'th and $j$'th indexes are incident, so that the corresponding functions must have equal $i$'th and $j$'th coordinates in ${}^L T(A)$.  That is, locally constant functions on such $V$ are identified with diagonals in ${}^L T(A)^n$, since as above the sum map for $\Lambda_T$ corresponds to the diagonal map for ${}^L T$, as desired.
\end{proof}

As in the lemma, let $A$ be a (discrete) abelian group, considered as a constant sheaf on $X$. We will need the following notion of \emph{sf comultiplicativity} for gerbes $\stack{Z}_n$ over $\on{Fact}(A)_n$.  Let $p$ be a partition of $[1,n]$ into $m$ parts of sizes $n_i$, and let $X^n_p$ and $\Delta^n_p \cong X^m$ be, as in \ref{grassmannian is factorizable}, the open and diagonal subschemes determined by $p$ in $X^n$; then we have two homomorphisms of sheaves of groups:
\begin{gather}
 \label{eq:sf mult open}
 \psi_p \colon \on{Fact}(A)_n
        \incl  \prod_{i = 1}^m \on{pr}_{X^{n_i}}^* \on{Fact}(A)_{n_i}
 \\
 \label{eq:sf mult diagonal}
 \phi_p \colon             \on{Fact}(A)_m
        \xrightarrow{\sim} \on{Fact}(A)_n|_{\Delta^n_p}
\end{gather}
where $\phi_p$ is in fact an isomorphism, and $\psi_p$ is a generalization of the defining inclusion $\on{Fact}(A)_n \incl \csheaf{A}^n_{X^n}$.  These homomorphisms, as for the factorizability of $\on{Gr}_n$, come with numerous compatibilities when refinements of $p$ are given. Now we may make a definition.

\begin{theorem}{defn}{sf comultiplicative}
 Let the $\stack{Z}_n$ be $\on{Fact}(A)_n$-gerbes on $X^n$ given for all $n$; the structure of \emph{sf comultiplicativity} is the data of equivalences for all partitions $p$ (recall the notation ${}^2 \phi$ for the change-of-group operation on gerbes associated with a group homomorphism of the coefficients, given in \ref{eq:gerbe change of group}):
 \begin{align*}
  {}^2 \psi_p (\stack{Z}_n) \cong \prod_{i = 1}^m \on{pr}_{X^{n_i}}^* (\stack{Z}_{n_i}) &&
  {}^2 \phi_p (\stack{Z}_m) \cong \stack{Z}_n|_{\Delta^n_p}
 \end{align*}
 together with compatibilities as in \ref{sf gerbe}. Likewise, we define sf multiplicative torsors and sections of $\on{Fact}(A)_n$.  
\end{theorem}
 
The concept of sf comultiplicativity is also related to factorizability on the grassmannian:

\begin{theorem}{lem}{factorizable group dual multiplicative}
 The first identification of \ref{factorizable group dual} (taken over all $n$) connects the map \ref{eq:sf mult open} with the multiplication of \ref{eq:torus grassmannian multiplication} and connects \ref{eq:sf mult diagonal} with the diagonal part of the factorization data in \ref{grassmannian is factorizable}.  The second identification (over all $n$) connects multiplicative factorizable $A$-valued functions with sf comultiplicative $\on{Fact}(T(A))_n$-valued functions.
\end{theorem}

\begin{proof}
 Whereas in the previous construction, we used the fact that the sum map on $\Lambda_T$ induced and was induced by the diagonal map on ${}^L T$, here we use the fact that the diagonal map on $\Lambda_T$ induces and is induced by the product map on ${}^L T$.  The details are otherwise the same.
\end{proof}

The main goal of this section is to state and prove the analogue of the second part of \ref{factorizable group dual multiplicative} for gerbes rather than functions.  As is typical of constructions on gerbes, we must pass through torsors and functions as well by way of accounting for higher morphisms.  The statement of the correspondence is straightforward:

\begin{theorem}{defn}{factorizable group dual gerbes}
 Let $\stack{Z}_n$ be an sf comultiplicative $\on{Fact}({}^LT(A))_n$-gerbe on $X^n$.  Then the corresponding multiplicative sf $A$-gerbe on $\on{Gr}_{T,X^n}$ is defined to be trivial above any open set in $X^n$ on which $\stack{Z}_n$ is trivial, with gluing data given by the (recursively defined) map from sf comultiplicative $\on{Fact}({}^LT(A))_n$-torsors to multiplicative sf $A$-torsors on $\on{Gr}_{T,X^n}$, applied to the gluing data of $\stack{Z}_n$.
 
 Conversely, let $\stack{T}_n$ be a multiplicative sf $A$-gerbe on $\on{Gr}_{T,X^n}$.  Since $\Lambda_T$ is finitely generated, there is an open cover of $X^n$ above which $\stack{T}_n$ is trivial; we define the corresponding sf comultiplicative $\on{Fact}({}^L T(A))_n$-gerbe $\stack{Z}_n$ on $X^n$ to be trivial on this cover, with gluing data given by the (recursively defined) map from multiplicative sf $A$-torsors on $\on{Gr}_{T,X^n}$ to sf comultiplicative $\on{Fact}({}^LT(A))_n$-torsors on $X^n$.
\end{theorem}

\begin{theorem}{prop}{multiplicative factorizable lattice gerbes}
 The constructions given in \ref{factorizable group dual gerbes} are inverse equivalences of $2$-categories between multiplicative sf $A$-torsors on $\on{Gr}_{T,X^n}$ and sf comultiplicative $\on{Fact}({}^L T(A))_n$-gerbes on $X^n$.  Furthermore, the $\stack{T}_n$ corresponding to a $\stack{Z}_n$ can be described as follows:
 
 For each section $X^n \cong \on{Gr}_{T,X^n}^{\lambda_1, \dots, \lambda_n} \incl \on{Gr}_{T,X^n}$ as one of the irreducible components, let
 \begin{equation*}
  \bar\lambda \colon \on{Fact}({}^L T)_n \to \Gm \times U
 \end{equation*}
 be the corresponding homomorphism from \ref{factorizable group dual}; we use the same notation for the induced map $\on{Fact}({}^L T(A)) \to \csheaf{A}_{X^n}$.  Then we have
 \begin{equation*}
  \stack{T}_n \cong {}^2 \bar{\lambda} (\stack{Z}_n),
 \end{equation*}
 and the multiplicative factorizable structure is the one obtained from the sf comultiplicative structure of $\stack{Z}_n$ and the first part of \ref{factorizable group dual multiplicative}.
\end{theorem}

\begin{proof}
 We begin by establishing that the second part of the construction is valid.  To expand on it, the claim is as follows: let $\vect{x} \in X^n$ and suppose that it has $k$ distinct coordinates, so that $\on{Gr}_{T,X^n}|_{\vect{x}} \cong \Lambda_T^k$.  We choose finitely many generators $l_1, \dots, l_m$ for $\Lambda_T^k$ and select a neighborhood $V$ about $\vect{x}$, in which all points have at least $k$ distinct coordinates, on which all the $\stack{T}_n^{l_i}$ are trivial; here, the connected components of $\on{Gr}_{T,X^n}|_V$ are indexed by $\Lambda_T^k$ compatibly with multiplication and the $\stack{T}_n^{l_i}$ refer to the parts of $\stack{T}_n$ on these components. Using the multiplicativity of $\stack{T}_n|_V$ over $V$, we obtain trivializations of $\stack{T}_n^l$ for any $l \in \Lambda_T^k$.
 
 For this to unambiguously define a trivialization of $\stack{T}_n|_V$, it is necessary and sufficient that for any $l$ and any representation of $l$ as a linear combination of the $l_i$, the multiplications of the $\stack{T}_n^{l_i}$ are all isomorphic.  For example, the equality $l_i + l_j = l = l_j + l_i$ corresponds to the two multiplications
 \begin{equation*}
  \stack{T}_n^{l_i} \otimes \stack{T}_n^{l_j}
   \cong \stack{T}_n^{l_1 + l_j}
   = \stack{T}_n^l
   = \stack{T}_n^{l_j + l_i}
   \cong \stack{T}_n^{l_j} \otimes \stack{T}_n^{l_i}
 \end{equation*}
 in which both equivalences are the same by commutativity of the multiplicative structure on $\stack{T}_n$.  Similarly, $(l_i + l_j) + l_k = l = l_i + (l_j + l_k)$ requires associativity. Since these are the only constraints on the free abelian group $\Lambda_T^k$, indeed the structure of commutative multiplicativity suffices to give consistent trivializations.  These are by definition commutative multiplicative trivializations, so that the gluing data is indeed a commutative multiplicative torsor, as claimed.
 
 Suppose that we are given $\stack{Z}_n$ for all $n$ forming an sf comultiplicative $\on{Fact}(T(A))_n$-gerbe; we show that the corresponding $A$-gerbes $\stack{T}_n$ are a multiplicative sf $A$-gerbe by proving the claimed formula for them.  Since the change of groups can be applied locally it suffices to prove this recursively for torsors and then functions; in the last case, the statement is merely the second part of \ref{factorizable group dual multiplicative} combined with \ref{sf multiplicative functions}.
 
 Finally, we must show that the constructions invert each other.  Since they are given recursively, it suffices to show this for functions, which is exactly the content of \ref{factorizable group dual multiplicative}.
\end{proof}

In the wake of the apparently facile reduction of the theorem to the trivial case of functions, it must be noted that although multiplicative $A$-valued locally constant functions on $\on{Gr}_{T,X}$ are, by definition, identified with sections of ${}^L T(A) = \on{Fact}({}^L T(A))_1$, it is not true that multiplicative $A$-gerbes on $\on{Gr}_{T,X}$ are identified with ${}^L T(A)$-gerbes on $X$; the reason is that multiplicativity is an additional structure on an $A$-gerbe that is not reflected in any such structure on the corresponding ${}^L T(A)$-gerbe.  (In the above proof, this structure was used in restricting the gluing data of $\stack{Z}_n$ from being an arbitrary torsor to being sf multiplicative, so that the recursive reduction could apply.) The correct analogue of sf comultiplicativity which turns this incorrect theorem into a correct one is simply the removal of factorizability, as it was removed from the multiplicative $A$-gerbe:

\begin{theorem}{defn}{comultiplicative gerbe}
 Let $\sh{A}$ be a sheaf of abelian groups on $X$, as in \ref{factorizable version of a group}. Given an $\sh{A}$-gerbe $\stack{Z}$ on $X$, the structure of \emph{comultiplicativity} is, denoting by $\psi \colon \sh{A} \to \sh{A} \times \sh{A}$ the diagonal map, an equivalence
 \begin{equation*}
  {}^2 \psi (\stack{Z}) \cong \stack{Z} \times \stack{Z}
 \end{equation*}
 together with associativity constraints with the natural compatibilities.  The further structure of \emph{commutativity} is the data of an isomorphism of the composition
 \begin{equation*}
  \stack{Z}_{(1)} \times \stack{Z}_{(2)}
   \cong {}^2 \psi(\stack{Z})
   \cong {}^2 \on{sw} {}^2 \psi (\stack{Z})
   \cong \stack{Z}_{(2)} \times \stack{Z}_{(1)},
 \end{equation*}
 (where the subscripts denote logical labeling of the factors and we denote by $\on{sw} \colon \sh{A} \times \sh{A} \to \sh{A} \times \sh{A}$ the factor-switching map) with the natural auto-equivalence of $\stack{Z} \times \stack{Z}$.  It should have natural compatibilities with itself and with the associativity constraint.  Likewise, we define (commutative) comultiplicative torsors for and sections of $\sh{A}$ (the latter of which are, commutative or not, all sections of $\sh{A}$).
\end{theorem}

\begin{theorem}{lem}{non-sf restriction}
 When $\sh{A} = \csheaf{A}$ is the constant sheaf on an abelian group, if the $\stack{Z}_n$ form an sf comultiplicative gerbe for $\on{Fact}(A)_n$ over $X^n$, then for any $x \in X$, the fiber $\stack{Z} = \stack{Z}_1|_x$ is commutative comultiplicative.
\end{theorem}

\begin{proof}
 The map $\psi_{1,1}$ of \ref{eq:sf mult open} is just the diagonal map $A \to A \times A$ over any
 point of the diagonal in $X^2$.
\end{proof}

In the next proposition, we use the notation ${}^L T(\sh{A}) = \shHom(\Lambda_T, \sh{A}) \cong \Lambda_T \otimes \sh{A}$, which agrees with the previous definition given in \ref{factorizable group dual}.

\begin{theorem}{prop}{multiplicative lattice gerbes}
 For any sheaf of abelian groups $\sh{A}$ on $X$, the constructions of \ref{factorizable group dual gerbes}, for $n = 1$ alone, are inverse equivalences of the $2$-categories of commutative multiplicative $A$-gerbes $\stack{T}$ on $\on{Gr}_{T,X}$ and commutative comultiplicative ${}^L T(A)$-gerbes $\stack{Z}$ on $X$.  Given $\stack{Z}$ and any $\lambda \in \Lambda_T = \on{Hom}({}^L T, \Gm)$, if we also consider $\lambda \in \on{Hom}({}^L T(A), A)$ then the corresponding $\stack{T}^\lambda$ is ${}^2 \lambda(\stack{Z})$; the multiplicative structure is that obtained from $\stack{Z}$.
\end{theorem}

The proof is exactly the same as before.  As a final remark, suppose $\stack{F}$ is a sheaf of categories over $\on{Gr}_{T,X}$ with a $\Lambda_T$-equivariant pairing as in \ref{G-equivariant pairing} which is biequivariant for an action of $A$-torsors on objects of $\stack{F}$.  Thus, if $\stack{T}$ is a commutative multiplicative gerbe, $\stack{F}(\stack{T})$ has a twisted pairing.

Alternatively, we consider $\stack{F}$ as a sheaf of categories $p_* \stack{F}$ (writing $p \colon \on{Gr}_{T,X} \to X$) over $X$ whose objects have $\Lambda_T$-gradings and thus ${}^L T$-actions, and which has a pairing which is biequivariant for an action of ${}^L T(A)$-torsors.  Thus, we may form the twisted category $p_* \stack{F}(\stack{Z})$, where $\stack{Z}$ is the commutative comultiplicative ${}^L T(A)$-gerbe corresponding to $\stack{T}$.  By definition of the product of $\Lambda_T$-gradings, the pairing on $p_* \stack{F}$ can be written as a composition
\begin{equation*}
 p_* \stack{F} \times p_* \stack{F}
  \to \stack{F}_2
  \to p_* \stack{F},
\end{equation*}
where $\stack{F}_2$ consists of $\Lambda_T$-\emph{bigraded} objects; the second map is taking the total grading.  In terms of ${}^L T$-actions, it is restriction along the diagonal map $\psi \colon {}^L T \to {}^L T \times {}^L T$.  By \ref{eq:general twisted pairing}, we can twist the first map
\begin{equation}
 \label{eq:comultiplication twisting}
 p_* \stack{F}(\stack{Z}) \times p_* \stack{F}(\stack{Z})
  \to \stack{F}_2(\stack{Z} \times \stack{Z})
  \cong \stack{F}_2({}^2 \psi \stack{Z})
  \to p_* \stack{F}(\stack{Z})
\end{equation}
by comultiplicativity of $\stack{Z}$.  Thus, we obtain a twisted product on $p_* \stack{F}(\stack{Z})$.

The following proposition is obvious from the proof of \ref{multiplicative lattice gerbes}:

\begin{theorem}{prop}{two kinds of twisting}
 There is a natural equivalence of sheaves of categories $p_* \stack{F}(\stack{Z})$ and $\stack{F}(\stack{T})$ identifying their respective pairings. \qed
\end{theorem}

\section{The affine grassmannian of any group}
\label{s:the affine grassmannian of any group}

In this section we discuss the relationships among the affine grassmannians of more general algebraic groups and their individual structures.

\subsection*{The grassmannian of a torus}

A basic fact about the $\on{Gr}_{G,X^n}$ is that their formation is functorial in $G$, in that any regular homomorphism $G \to H$ induces a map of ind-schemes $\on{Gr}_{G,X^n} \to \on{Gr}_{H,X^n}$ constructed by induction of $G$-torsors to $H$-torsors.  In particular, if we choose a maximal torus $T \subset G$, then we get an induced map
\begin{equation}
 \label{eq:torus section}
 \map{i}{\on{Gr}_{T,X^n}}{\on{Gr}_{G,X^n}}.
\end{equation}
We continue to denote by $\on{Gr}_{T,X^n}^{\lambda_1, \dots, \lambda_n}$ the images of these components under $i$. Given the description of the $\on{Gr}_{T,X}^\lambda$ in \ref{torus grassmannian properties}, it is easy to see that the $\on{Gr}_{T,X^n}^{\lambda_1, \dots, \lambda_n}$ can be described explicitly in $\on{Gr}_{G,X^n}$:

\begin{theorem}{prop}{torus section moduli}
 Let $\lambda_1, \dots, \lambda_n \in \Lambda_T$ be coweights of $G$ and let $p = (\vect{x}, \sh{T}, \phi) \in \on{Gr}_{G,X^n}(S)$ for some scheme $S$.  Then $p \in \on{Gr}^{\lambda_1, \dots, \lambda_n}_{T,X^n}(S)$ if and only if, for every $G$-representation $V$ and every vector $v \in V$ of some weight $\mu \in \Lambda^\vee$, the inclusion over $X_S \setminus \bar{x}$,
 \begin{equation*}
  \genby{v} \otimes \OO_{X_S \setminus \bar{x}} \incl V \otimes \OO_{X_S \setminus \bar{x}}
   \xrightarrow[\sim]{\phi} V_{\sh{T}}|_{X_S \setminus \bar{x}}
 \end{equation*}
 extends to an inclusion of vector bundles $\OO\bigl(\sum_i \langle \mu, \lambda_i \rangle
 \bar{x}_i\bigr) \to V_{\sh{T}}$.
\end{theorem}

\subsection*{The grassmannian of a Borel subgroup}

Now say we choose a Borel subgroup $B$, and let $T = B/N$, where $N$ is the unipotent part of $B$; we do not choose an inclusion of $T$ into $G$.  Then the pair of maps $B \to G$ and $B \to T$ induce a diagram
\begin{equation}
 \label{eq:borel diagram}
 \on{Gr}_{G,X^n} \xleftarrow{b} \on{Gr}_{B,X^n} \xrightarrow{t} \on{Gr}_{T,X^n}
\end{equation}
in which $b$ is surjective (proof: a reduction of a $G$-torsor to $B$ is the same as a trivialization of the induced $G/B$-bundle.  This exists away from a divisor by assumption, so by the valuative criterion of properness for $G/B$, it exists globally) and, on each component of $\on{Gr}_{B,X^n}$, is injective. Although we do not choose an embedding $T \subset G$, for every splitting of $B \to T$, the induced map $i$ makes the above diagram commute.

We will require the following description of the components of $\on{Gr}_G$:

\begin{theorem}{prop}{borel components}
 (\cite{MV_Satake})
 \begin{enumerate}
  \item \label{en:borel orbit components}
  Each inverse image $t^{-1}(\on{Gr}_T^\lambda)$ is connected, so that the components of $\on{Gr}_B$ are also indexed by $\Lambda_T$.

  \item \label{en:borel orbit boundaries}
  There is an embedding of $\on{Gr}_G$ in $\Prj^\infty = \dirlim_n \Prj^n$ such that for any Borel subgroup $B$, each boundary $\partial \on{Gr}_B^{\lambda} = \bigcup_{\mu < \lambda} \on{Gr}_B^\mu$ is a hyperplane section of $\bar{\on{Gr}}_B^\lambda$.  In fact, $\on{Gr}_B^\lambda \cong \Aff^\infty$.

  \item \label{en:borel orbit closures}
  In $\on{Gr}_G$, we have $\bar{\on{Gr}}_B^\lambda = \bigcup_{\mu \leq \lambda} \on{Gr}_B^\mu$.
 \end{enumerate}
\end{theorem}

By the second point, the components $\on{Gr}_B^\lambda$ contained, via $b$, in each component of $\on{Gr}_G$ are indexed by the cosets of the coroot lattice $\Lambda_{T,r}$ of $G$, so that the set of connected components of $\on{Gr}_G$ is identified with $\pi_1(G) = \Lambda_T/\Lambda_{T,r}$.  More precisely, each component of $\on{Gr}_G$ contains exactly one $\on{Gr}_T^\lambda$ where $\lambda$ is minimal dominant with respect to the partial ordering by positive coroots.  These statements hold identically for $\on{Gr}_{G,X}$, etc.\ as well.

\subsection*{The grassmannians of subminimal parabolics}

The preceding constructions for $B$ and $T$ admit natural generalizations where $B$ is replaced by any parabolic subgroup $P$ containing $T$ and $T$ is replaced by its Levi quotient $L$.  Then we have maps
\begin{equation}
 \label{eq:parabolic diagram}
 \on{Gr}_{G,X^n} \xleftarrow{p} \on{Gr}_{P,X^n} \xrightarrow{l} \on{Gr}_{L,X^n}
\end{equation}
and for any splitting of $P \to L$, a map $\map{i}{\on{Gr}_{L,X^n}}{\on{Gr}_{G,X^n}}$.  We will apply this construction when $L$ is a Levi subgroup of semisimple rank $1$ as follows. Inside of $\on{Gr}_{G,X^n}$, the components of $\on{Gr}_{T,X^n}$ are connected by projective line bundles with good factorizability properties.  We will only need these for $n = 1, 2$ and will in fact obtain them first when $G$ has semisimple rank $1$ and then push them forward along the map $i$ as $L$ varies over all the Levi factors of the subminimal parabolics of $G$.

\begin{theorem}{defn}{projective line bundles}
 Let $G = \on{GL}_2$.  We identify $\Lambda_T \cong \Z^2$ and the simple coroot $\check\alpha = (1, -1)$; let $V$ be the standard representation of $\on{GL}_2$, with weights $(1,0)$ and $(0,1)$, and let $V(1,0), V(0,1)$ be its weight spaces.  Also, let $\Omega$ be the determinant representation, with weight $(1,1)$. Let $\lambda = (\lambda_1, \lambda_2)$ and $\mu = (\mu_1, \mu_2)$ be coweights and define the following subfunctors of $\on{Gr}_{\on{GL}_2,X^n}$ ($n = 1,2$):
 \begin{align*}
  \Prj^\lambda_1(S) &= \left\{
   \begin{aligned}
    (&x, \sh{T}, \phi) \\ &\in \on{Gr}_{\on{GL}_2, \rlap{$\scriptstyle X$}\phantom{X^2}}
   \end{aligned}
   \middle\vert
   \begin{aligned}
    V(1,0) \otimes \OO_{X_S} &\incl V_\sh{T}(-\lambda_1\bar{x})  \\
    V(0,1) \otimes \OO_{X_S} &\incl V_\sh{T}(-(\lambda_2 - 1)\bar{x}) \\
    \Omega \otimes \OO_{X_S} &\incl \Omega_{\sh{T}}(-(\lambda_1 + \lambda_2)\bar{x})
   \end{aligned}
  \right\} \\
  \Prj^{\lambda,\mu}_{2,1}(S) &= \left\{
   \begin{aligned}
    (&x,y, \sh{T}, \phi) \\ &\in \on{Gr}_{\on{GL}_2,X^2}
   \end{aligned}
   \middle\vert
   \begin{aligned}
    V(1,0) \otimes \OO_{X_S} &\incl V_\sh{T}(-\lambda_1\bar{x} - (\mu_1 + 1) \bar{y})\\
    V(0,1) \otimes \OO_{X_S} &\incl V_\sh{T}(-(\lambda_2 - 1) \bar{x} - (\mu_2 - 1)\bar{y})\\
    \Omega \otimes \OO_{X_S} &\incl \Omega_{\sh{T}}(-(\lambda_1 + \lambda_2)\bar{x}
                                                   -(\mu_1 + \mu_2) \bar{y})
   \end{aligned}
  \right\}
 \end{align*}
 where each inclusion is an inclusion of \emph{coherent sheaves} (not vector bundles), and the maps refer, like those in \ref{torus section moduli}, to the ones given by the trivialization $\phi$ on $X \setminus \bar{x}$ or $X \setminus (\bar{x} \cup \bar{y})$.
\end{theorem}

We require these subspaces for their excellent properties vis \`a vis factorization, giving $\on{Gr}_{G,X^2}$ some of the flavor of \ref{torus grassmannian properties} which it, in general, lacks.

\begin{theorem}{lem}{projective line bundle properties}
 We have $\Prj_{2,1}^{\lambda,\mu} \cong \Prj_1^\lambda \times \on{Gr}_{T,X}^{\mu + \alpha}$, and this isomorphism agrees with the factorization isomorphism over $X^2 \setminus \Delta$.  Define
 \begin{align*}
  U^0_1 &= \Prj_1^\lambda \cap \on{Gr}_{B,X}^\lambda &
  U^0_{2,1} &= \Prj_{2,1}^{\lambda,\mu} \cap \on{Gr}_{B,X}^{\lambda, \mu + \check\alpha} \\
  U^\infty_1 &= \Prj_1^\lambda \cap \on{Gr}_{B^\op,X}^\lambda &
  U^\infty_{2,1} &= \Prj_{2,1}^{\lambda,\mu} \cap \on{Gr}_{B^\op,X}^{\lambda, \mu + \check\alpha}
 \end{align*}
 (where $B^\op$ is the Borel subgroup opposite to $B$); then $U_1^0 \cong U_1^\infty \cong \Aff^1 \times X$, $U_{2,1}^0 \cong U_{2,1}^\infty \cong \Aff^1 \times X^2$, and these make $\Prj_1^\lambda$ and $\Prj_{2,1}^{\lambda,\mu}$ into $\Prj^1$-bundles on $X$ and $X^2$, respectively. From these formulas, $\Prj_1^\lambda$ has $\on{Gr}_{T,X}^\lambda$ and $\on{Gr}_{T,X}^{\lambda + \check\alpha}$ as zero and infinity sections, respectively, and likewise $\Prj_{2,1}^{\lambda,\mu}$ has $\on{Gr}_{T,X^2}^{\lambda, \mu + \check\alpha}$ and $\on{Gr}_{T,X^2}^{\lambda + \check\alpha, \mu + \check\alpha}$ as zero and infinity sections.

 Finally, suppose $\mu$ is a dominant weight, so $\mu_1 - \mu_2 \geq 0$.  Let
 \begin{equation*}
  \Aff^1 \times (X^2 \setminus \Delta) \to \Aff^1 \times (X^2 \setminus \Delta)
 \end{equation*}
 be the restriction of the factorization map $\Prj_1^\lambda \times \on{Gr}_{T,X}^{\mu + \check\alpha} \to \Prj_{2,1}^{\lambda,\mu}$ to $U_1^0 \times \on{Gr}_{T,X}^{\mu + \check\alpha} \to U_{2,1}^0$.  Then the corresponding map $\Aff^1 \times (X^2 \setminus \Delta) \to \Aff^1$ has a zero of order $(\mu_1 - \mu_2) + 2 = \langle \alpha, \mu \rangle + 2$ along $\Delta$. Likewise, on $U_1^\infty \times \on{Gr}_{T,X}^{\mu + \check\alpha}$ it has a pole of this order.
\end{theorem}

\begin{proof}
 It is obvious that $\Prj_{2,1}^{\lambda,\mu} \cong \Prj_1^\lambda \times \on{Gr}_{T,X}^{\lambda + \check\alpha}$, using \ref{torus section moduli} as a description of the latter, and from the construction of the factorization isomorphism in the proof of \ref{grassmannian is factorizable} that this isomorphism agrees with factorization away from $\Delta$.  For the rest of the computations, we will compute the $\C$-points and show that they have the desired structure, leaving the generalization to $S$-families to the imagination.  We use the fact that a $\on{GL}_2$-torsor $\sh{T}$ is equivalent to its associated vector bundle $V_{\sh{T}}$, with $V$ the standard representation.

 We first prove that $\Prj_1^\lambda$ has the desired form; namely, we will show that if $p = (x,\sh{T}, \phi) \in \Prj_1^\lambda(\C)$, then in any sufficiently small neighborhood of $x$, $V_{\sh{T}}$ may be trivialized so that $\phi$ is given by one (possibly either) of the following forms, and that this trivialization is unique (for each form):
 \begin{equation}
  \label{eq:matrix forms 1}
  A_+ =
   \begin{bmatrix} z_x^{\lambda_1} & s z_x^{\lambda_2 - 1} \\ 0 & z_x^{\lambda_2} \end{bmatrix}
  \qquad
  A_- =
   \begin{bmatrix} z_x^{\lambda_1 + 1} & 0 \\ t z_x^{\lambda_1} & z_x^{\lambda_2 - 1} \end{bmatrix}
 \end{equation}
 for some $s,t \in \C$; here, $z_x$ is a generator of the maximal ideal $\pideal{m}_x$.  We pick any trivialization of $V_{\sh{T}}$, so that the maps $V(1) \otimes \OO_X \incl V_{\sh{T}}$, $V(-1) \otimes \OO_X \incl V_{\sh{T}}(x)$ are given by a matrix with entries in $K(X)$,
 \begin{equation*}
  A = \begin{bmatrix} a & b \\ c & d \end{bmatrix},
 \end{equation*}
 in which $a,c \in \pideal{m}_x^{\lambda_1}$ and $b,d \in \pideal{m}_x^{\lambda_2 - 1}$.  Two such matrices determine equivalent trivializations $\phi$ if they differ by the action of $\on{GL}_2(\OO_X)$ on the left, so we may perform row operations with coefficients in $\OO_X$. Unless both $a,c \in \pideal{m}_x^{\lambda_1 + 1}$, we may perform row operations to put $a \in \pideal{m}_x^{\lambda_1}$, and then put $A$ in the form
 \begin{equation*}
  \begin{bmatrix} a & b \\ 0 & d \end{bmatrix}
  = \begin{bmatrix} z_x^{\lambda_1} & s z_x^{\lambda_2 - 1} \\ 0 & z_x^{\lambda_2} \end{bmatrix}
 \end{equation*}
 where we have applied the determinant condition to find the order of vanishing of $d$ and chosen some regular function $s$ according to $b \in \pideal{m}_x^{\lambda_2 - 1}$.  After one more row operation we may assume $s \in \C$, so $A$ is in the form $A_+$. On the contrary, suppose $a, c \in \pideal{m}_x^{\lambda_1 + 1}$; then we do not have both $b,d \in \pideal{m}_x^{\lambda_2}$ by the determinant condition, so the same logic applies to the second column, and after some row operations $A$ takes the form in $A_-$:
 \begin{equation*}
  \begin{bmatrix} a & 0 \\ c & d \end{bmatrix} =
  \begin{bmatrix} z_x^{\lambda_1 + 1} & 0 \\ tz_x^{\lambda_1} & z_x^{\lambda_2 - 1} \end{bmatrix}.
 \end{equation*}
 No row operation preserves the forms $A_+$ or $A_-$, so the trivialization inducing them is unique. The points $p$ in which $\phi$ admits the form $A_+$ are by definition $U_1^0$ and clearly isomorphic to $\Aff^1 \times X$, while those of the form $A_-$ are $U_1^\infty \cong \Aff^1 \times X$.  Clearly in $U_1^0 \cap U_1^\infty$, we have the relation $s = t^{-1}$ between their coordinates, so that $U_1^0 \cup U_1^\infty \cong \Prj^1 \times X$.

 A similar analysis shows that any point $p = (x,y,\sh{T}, \phi)$ in $\Prj_{2,1}^{\lambda,\mu}$ can, in any sufficiently small neighborhood containing both $x$ and $y$ (if one such exists) and for a unique trivialization of $V_{\sh{T}}$, take either of the forms:
 \begin{equation}
  \label{eq:matrix forms 2}
  A_+ =
   \begin{bmatrix}
    z_x^{\lambda_1} z_y^{\mu_1 + 1} & s z_x^{\lambda_2 - 1} z_y^{\mu_2 - 1} \\
    0 & z_x^{\lambda_2} z_y^{\mu_2 - 1}
   \end{bmatrix} \quad
  A_- =
   \begin{bmatrix}
    z_x^{\lambda_1 + 1} z_y^{\mu_1 + 1} & 0 \\
    t z_x^{\lambda_1} z_y^{\mu_1 + 1} & z_x^{\lambda_2 - 1} z_y^{\mu_2 - 1}
   \end{bmatrix}
 \end{equation}
 with $s,t \in \C$.  As before, choose any trivialization of $V_{\sh{T}}$ to bring $\phi$ into the form $A$ as above, with $a,c \in \OO_X(-\lambda_1 x - (\mu_1 + 1)y) $ and $b,d \in \OO_X(-(\lambda_2 - 1) x -(\mu_2 - 1) y)$. Suppose we do not have both $a,c \in \pideal{m}_x^{\lambda_1 + 1}$; by a suitable linear combination (row operation), we may assume that $a \notin \pideal{m}_x^{\lambda_1 + 1} $ and that $c$ has a higher-order zero at $y$, so that $c/a \in \OO_X$ in a sufficiently small neighborhood of $x,y$.  Thus, row operations bring $A$ into the form
 \begin{equation*}
  \begin{bmatrix} a & b \\ 0 & d \end{bmatrix}
  = \begin{bmatrix} z_x^{\lambda_1} z_y^{\mu_1 + 1} & b \\
                    0 & z_x^{\lambda_2} z_y^{\mu_2 - 1} \end{bmatrix}
 \end{equation*}
 by the determinant condition.  Then $b$ must be of the form $s z_x^{\lambda_2 - 1} z_y^{\mu_2 - 1}$ for some regular function $s$, and further row operations bring $s$ into $\C$, and hence $A$ into the form $A_+$.  If, conversely, we have both $a,c \in \pideal{m}_x$, then the same logic applies to $b,d$ by the determinant condition, and row operations bring $A$ into the form $A_+$. As above, these matrices are obtained by unique trivializations of $V_{\sh{T}}$ and form, respectively, $U_{2,1}^0$ and $U_{2,1}^\infty$, both isomorphic to $\Aff^1 \times X^2$, with $s = t^{-1}$ on $U_{2,1}^0 \cap U_{2,1}^\infty$, so that $\Prj_{2,1}^{\lambda,\mu} \cong \Prj^1 \times X^2$ (locally).

 We now apply the construction of the factorization isomorphism in the proof of \ref{grassmannian is factorizable} to compute it explicitly in these coordinates.  Let $(x, \sh{T}_x, \phi_x)$ and $(y, \sh{T}_y, \phi_y)$ be points in $U_1^0(\C)$ and $\on{Gr}_{T,X}^{\mu + \check\alpha}(\C)$, and let
 \begin{equation*}
  \sh{T} = \sh{T}_x|_{X \setminus \{y\}} \cup_{X \setminus \{x,y\}} \sh{T}_x|_{X \setminus \{x\}}
 \end{equation*}
 be the glued torsor.  Suppose we have chosen trivializations of $V_{\sh{T}_x}$ and $V_{\sh{T}_y}$ such that $\phi_x$ is in the form $A_+$ given in \ref{eq:matrix forms 1}, and $\phi_y$ is given by the diagonal matrix $\on{diag}(z_y^{\mu_1 + 1}, z_y^{\mu_2 - 1})$; denote their bases $e_{1,x}, e_{2,x}$ and $e_{1,y}, e_{2,y}$. We define a global basis $f_1, f_2$ of $V_{\sh{T}}$ such that:
 \begin{align*}
  f_1|_{X \setminus \{y\}} &= z_y^{-(\mu_1 + 1)} e_{1,x} &
  f_2|_{X \setminus \{y\}} &= z_y^{-(\mu_2 + 1)} e_{2,x} \\
  f_1|_{X \setminus \{x\}} &= z_x^{-\lambda_1} e_{1,y} &
  f_2|_{X \setminus \{x\}}
   &= -sz_x^{-(\lambda_1 + 1)} z_y^{(\mu_1 - \mu_1) + 2} e_{1,y} + z_x^{-\lambda_2}e_{2,y}.
 \end{align*}
 Let $e_1, e_2$ be the basis of weight vectors in $V$, so that on $X \setminus \{x,y\}$, accounting for the form of the trivializations $\phi_x, \phi_y$, we have
 \begin{gather*}
  \begin{bmatrix} f_1 \\ f_2 \end{bmatrix} =
  \begin{bmatrix}
   z_x^{-\lambda_1} z_y^{-(\mu_1 + 1)} & -s z_x^{-(\lambda_1 + 1)} z_y^{-(\mu_2 + 1)} \\
   0 & z_x^{-\lambda_2} z_y^{-(\mu_2 + 1)}
  \end{bmatrix}
  \begin{bmatrix} e_1 \\ e_2 \end{bmatrix}
  \\
  \begin{bmatrix} e_1 \\ e_2 \end{bmatrix} =
  \begin{bmatrix}
   z_x^{\lambda_1} z_y^{\mu_1 + 1} & s z_x^{\lambda_2 - 1} z_y^{\mu_1 + 1} \\
   0 & z_x^{\lambda_2} z_y^{\mu_2 + 1} \end{bmatrix}
  \begin{bmatrix} f_1 \\ f_2 \end{bmatrix}
 \end{gather*}
 Comparing this expression to that for $A_+$ in \ref{eq:matrix forms 2}, we find that the $s$ there is given by $s z_y^{\mu_1 - \mu_2 + 2}$.  Let $z_{y - x} = z_y - z_x$ be the uniformizer along the diagonal in $X^2$, so $z_y - z_{y - x} \in \pideal{m}_x$.  After a row operation, we find that the $\Aff^1$-coordinate in \ref{eq:matrix forms 2} has a zero of order $(\mu_1 - \mu_2) + 2$ along $\Delta$, as claimed.
\end{proof}

Having constructed the above projective line bundles for $\on{GL}_2$, we obtain them in particular for $\on{SL}_2$ and $\on{PGL}_2$, as we have an inclusion $\on{Gr}_{\on{SL}_2,X^n} \to \on{Gr}_{\on{GL}_2, X^n}$ and a projection $\on{Gr}_{\on{GL}_2, X^n} \to \on{Gr}_{\on{PGL}_2, X^n}$ which are isomorphisms on each component.  Then for any group $L$ of semisimple rank $1$, let $\bar{L}$ be its semisimple quotient, isomorphic to either one of these two, and the map $\on{Gr}_{L,X^n} \to \on{Gr}_{\bar{L},X^n}$ is an isomorphism on each component, so we define $\Prj_1^\lambda$ and $\smash{\Prj_{2,1}^{\lambda,\mu}}$ in $\on{Gr}_{L,X^n}$ by transport of structure. Finally, let $G$ be any reductive group, $\check\alpha$ any simple coroot, and $L$ the corresponding Levi subgroup of $G$, which has the same torus $T$ and thus the same coweights.  The maps $\on{Gr}_{L,X^n} \to \on{Gr}_{G,X^n}$ are inclusions on each component of the former, so we may define:

\begin{theorem}{defn}{general projective line bundles}
 Let $\Prj_{1;\check\alpha}^\lambda \subset \on{Gr}_{G,X}$ be the image of $\Prj_1^\lambda$ under the map $\on{Gr}_{L,X} \to \on{Gr}_{G,X}$, and likewise $\Prj_{2,1;\check\alpha}^{\lambda, \mu}$. They have the properties given in \ref{projective line bundle properties}.
\end{theorem}

\subsection*{A common factorizable base}

Although $\on{Gr}_{G,X^n}$ is given only over $X^n$, we can boost its base to a larger factorizable scheme resembling $\on{Gr}_{T,X^n}$ more closely.  Let $\pi_1(G)_{X^n} = \pi_1(G)_n$ be the union of copies of $X^n$, indexed by $\pi_1(G)^n$, and intersecting in the pattern described in \ref{torus grassmannian properties}; then it has an obvious factorizable structure which, again as in the proposition, agrees with the group structure of $\pi_1(G)$.  The components of $\on{Gr}_{G,X}$ are naturally indexed by $\pi_1(G)$, so $\on{Gr}_{G,X}$ admits a map to $\pi_1(G)_X$.  Using the factorizable structure of each, we obtain maps $\on{Gr}_{G,X^n} \to \pi_1(G)_{X^n}$.  Since $\pi_1(G) = \Lambda/\Lambda_r$, by the proposition there is also a natural map $\on{Gr}_{T,X^n} \to \pi_1(G)_{X^n}$.

\section{The determinant line bundle and its factorizable gerbe}
\label{s:determinant bundle}

A simpler analogue of sf gerbes is the concept of an \emph{sf line bundle} on the factorizable grassmannian.  The definition is the same as in \ref{sf gerbe} except with all mention of $2$-morphisms omitted. We will use sf line bundles to construct sf gerbes, as follows. Given line bundles $\sh{L}_n$ on each $\on{Gr}_n$, and for $a \in A$, let $\stack{G}_n = \sh{L}_n^{\log a}$.  By \ref{log power functorial}, the factorizable structure of $\stack{L}_n$ becomes such a structure for $\stack{G}_n$.  The basic such sheaves are the determinant line bundles, obtained as follows.

\subsection*{The determinant line bundle}
   
Given any finite-dimensional representation $V$ of $G$, we will define an associated line bundle
$\on{det}(V)_n$ on $\on{Gr}_n$.  Here is the definition; following are supporting lemmas on straightforward facts and sketches of their proofs.

\begin{theorem}{defn}{determinant bundle}
 Let $s = (\vect{x}, \sh{T}, \phi) \in \on{Gr}_{G,X^n}(S)$, and let $V$ be a finite\hyp dimensional representation of $G$.  We write $f \colon \bar{x} \to S$ for the natural map and $\sh{V} = V_{\sh{T}}$ for the induced vector bundle.  The determinant bundle $\det(V)_n$ associated with $V$ is the line bundle on $\on{Gr}_n$ whose fiber at $s \in \on{Gr}_n(S)$ is any of the following naturally isomorphic line bundles on $S$:
 \begin{equation*}
  (\det(V)_n)_s =
   \begin{cases}
    \bigwedge^\text{top} f_ *(\sh{V}/V(a\bar{x}))
    \otimes \bigl(\bigwedge^\text{top} f_*(V(0)/V(a\bar{x}))\bigr)^\vee
     & \text{any $a \ll 0$} \\
    \bigl(\bigwedge^\text{top} f_ *(V(b\bar{x})/\sh{V})\bigr)^\vee
    \otimes \bigwedge^\text{top} f_*(V(b\bar{x})/V(0))
     & \text{any $b \gg 0$}
   \end{cases}
 \end{equation*}
\end{theorem}

The following lemmas support the claims implicit in this definition.

\begin{theorem}{lem}{fractional ideal nesting}
 Let $j \colon X_S \setminus \bar{x} \to X_S$ be the open immersion.  Then there is an inclusion
 \begin{equation*}
  \sh{V} \incl j_* (\OO_U \otimes V) = j_* j^* V(0)
 \end{equation*}
 satisfying the asymptotic containments $V(a\bar{x}) \subset \sh{V} \subset V(b\bar{x})$ for $a \ll 0$ and $b \gg 0$.
\end{theorem}

\begin{proof}[Sketch of proof]
 The map of the inclusion comes by adjunction from the isomorphism induced by the trivialization $\phi$:
 \begin{equation*}
  j^* \sh{V} \cong \OO_U \otimes V,
 \end{equation*}
 and it is an injection because $\sh{V}$ is torsion-free.  The right-hand side contains all the fractional ideal sheaves $V(c \bar{x})$ and the second inclusion follows (after twisting by some irrelevant very ample line bundle) by counting poles of generators of $\sh{V}$.  The first inclusion is obtained by analyzing any subspace $W \subset V$ for which $W(a\bar{x}) \subset \sh{V}$; given any $v \in V$, an application of the second inclusion to $\sh{V} \cap j_* j^*\genby{v}(0)$ shows that it can be enlarged to $W + \genby{v}$.
\end{proof}

\begin{theorem}{lem}{flat quotients}
 The map $f$ is flat, and for any $a, b$ as in \ref{fractional ideal nesting}, the quotients $\sh{V}/V(a\bar{x})$ and $V(b\bar{x})/\sh{V}$ are $f$-flat.
\end{theorem}

\begin{proof}[Sketch of proof]
 Flatness of $f$ follows from the fact that it is surjective with fibers of constant finite length.  Flatness of the first quotient follows from comparison of two flat resolutions (where $c = a - b$):
 \begin{gather*}
  \dots \to \sh{V}(c\bar{x})/\sh{V}(2c\bar{x}) \to V(a\bar{x})/V((a + c)\bar{x}) \to \sh{V}/\sh{V}(c\bar{x})
    \to \sh{V}/V(a\bar{x}) \to 0 \\
  0 \to V(a\bar{x}) \to \sh{V} \to \sh{V}/V(a\bar{x}) \to 0,
 \end{gather*}
 of which the first shows that $\on{Tor}^i_S(\sh{V}/V(a\bar{x}, \farg)$ is (essentially) periodic and the second shows that it is bounded.  The same argument works for the second quotient.
\end{proof}

\begin{theorem}{cor}{vector bundle pushforward}
 $f_* (\sh{V}/V(a\bar{x}))$ and $f_*(V(b\bar{x})/\sh{V})$ are vector bundles on $S$.
\end{theorem}

\begin{proof}
 Note that both $\sh{V}/V(a\bar{x})$ and $V(b\bar{x})/\sh{V}$ are coherent, $f$-flat, and their fibers are flasque (since they live on finite sets), hence acyclic.  Since the Euler characteristic is constant in flat families, the pushforwards have constant rank and are thus vector bundles.
\end{proof}

It is clear that $\det(V)_n$ does not depend on the choice of $a$ or $b$ in \ref{determinant bundle}.  Note that $\det(V)_n$ is a $\Z/(2)$-graded line bundle; i.e.\ it caries a parity depending on those of the two ``top''s, which is the source of the sign in the following theorem:

\begin{theorem}{prop}{determinant factorizable}
 The determinant line bundles $\det(V)_n$ satisfy properties \ref{sf gerbe}(\plainref{en:factorizable gerbe closed}, \plainref{en:factorizable gerbe open}), and possess the equivariance of \ref{en:factorizable gerbe symmetry}.  The compatibility of this equivariance with the natural one given by the tensor product of \ref{en:factorizable gerbe open} holds up to a sign which is constant on components of $\on{Gr}_n$ over $X^n$.
\end{theorem}

\begin{proof}
 For property \ref{en:factorizable gerbe closed}, this is trivial, since the vector bundle $\sh{V}$ and its inclusion $j^* \sh{V} \to \OO$ do not depend on the multiplicity of the components of the divisor $\bar{x}$; likewise, the $S_n$-symmetry of \ref{en:factorizable gerbe symmetry} comes from the fact that these data do not depend on the order of the components of $D$. For property \ref{en:factorizable gerbe open}, we have $\bar{x} = \bigcup \bar{x}_i$, where $\bar{x}_i$ is the union of the graphs of the points in the $X^{n_i}$; let $\sh{V}_i$ be the vector bundles corresponding to these points, with trivializations on $X_S \setminus \bar{x}_i$.  Then we have a map
 \begin{equation*}
  \sh{V}/V(a\bar{x}) \to \bigoplus \sh{V}_i/V(a\bar{x}_i)
 \end{equation*}
 for $a \ll 0$, which is an isomorphism away from the intersections of the $D_i$.  Thus, on $S$ away from these intersections, we have an isomorphism
 \begin{equation*}
  \bigwedge\nolimits^\text{top} f_*(\sh{V}/V(a\bar{x}))
   \cong \bigotimes \bigwedge\nolimits^\text{top} f_*(\sh{V}_i/V(a\bar{x}_i))
 \end{equation*}
 and likewise for the other factor, giving $\det_n \cong \bigotimes_i \det_{n_i}$, as desired. Since the commutativity of this tensor product introduces a sign on the left-hand side, we find that the equivariance of the right-hand side differs from that on the left by this sign, which is constant on connected components of $S$, as claimed.
\end{proof}

\subsection*{Computation of the determinant gerbes}
   
Because the determinant line bundle is not quite an sf line bundle, we are not necessarily able to associate an sf gerbe $\det(V)_n^{\log a}$ with it.  We proceed to investigate the exact identity of the sign in the above theorem so as to specify when this is possible.

It is clear from the definition that if we have a map of groups $G \to H$ inducing a map of factorizable grassmannians $\map{g}{\on{Gr}_{G,X^n}}{\on{Gr}_{H,X^n}}$, then for any $H$-representation $V_H$ considered as a $G$-representation $V_G$, we have $\det(V_H)_n \cong g^* \det(V_G)_n$, and so this is true of the associated gerbes as well.  On this basis, we compute the sign; as in the definition of $Q(\stack{G}_n)$, we begin with the case when $G = T$ is a torus.  For the remainder of this section, we change our notation: elements of $\Lambda_T = X_*(T)$ (coweights) are denoted $\check\lambda$, and elements of $\Lambda^T = X^*(T)$ (weights) are denoted $\lambda$.

\begin{theorem}{lem}{character determinant bundle}
 Let $V = \lambda$ be a character representation of $T$.  Then for any $\check\mu$, if $m = \langle \lambda, \check\mu \rangle$, the determinant bundle on $\on{Gr}_{T,X}$ has component on $\on{Gr}_{T,X}^{\check\mu}$:
 \begin{equation*}
  \det(V)_1^{\check\mu} = \sh{T}_X^{(m^2/2) - (m/2)},
 \end{equation*}
 where $\sh{T}_X$ is the tangent bundle on $X$.  Furthermore, $\det(V)$ is an sf line bundle if and only if $m$ is even.
\end{theorem}

\begin{proof}
 We compute $\det(V)_1$ as a line bundle on $S$ for the $S$-point of $\on{Gr}_{T,X}$ corresponding to $S = X$ with the identity map to $X$ and a $T$-torsor $\sh{T}$ plus trivialization $\sh{T}^0 \to \sh{T}$ on $X^2 \setminus \Delta$ making a point of $\on{Gr}^{\check\mu}_{T,X}$.  By \ref{torus grassmannian properties}\ref{en:torus grassmannian 1}, the corresponding vector bundle is
 \begin{equation*}
  \sh{V} = \OO(-m\Delta).
 \end{equation*}
 This already implies the last claim of the lemma: comparing with the construction in \ref{determinant factorizable}, we see that the $a$ there is equal to $m$.  Since the sign arises from having to move odd-sized blocks past each other in an alternating power, it is equal to $1$ if and only if the blocks, which have size $m$, have even size.
 
 Note that also by definition, the sheaf of differentials $\omega_X$ is $i_\Delta^* \OO(-\Delta)/\OO(-2\Delta)$ and is a line bundle, so that
 \begin{equation*}
  \omega_X^d = i_\Delta^* \OO(-d\Delta)/\OO(-(d + 1)\Delta)
 \end{equation*}
 (for both positive and negative values of $d$; this of course is specific to one-dimensional $X$). By definition of $\det(V)$, when $m > 0$, it is $\bigwedge^\text{top} (\on{pr}_X)_*(\OO(0)/\OO(-m\Delta))^\vee$, where the argument is a chained extension of the $\omega_X^k$ for $k = 0, \dots, m - 1$.  If $m \leq 0$, the quotient is turned around and dualized, but by the above equation, the powers appearing follow the same pattern, the only difference being that the sequence is for $k \in [1, m]$.  Thus, after some arithmetic, $\det(V)_1^{\check\mu} = \sh{T}_X^{m(m - 1)/2}$ in both cases.
\end{proof}

\begin{theorem}{lem}{character determinant factorization}
 Let $V$ be a character representation as before.  On $\on{Gr}_2$, the factorization isomorphism $\det(V)_2 \cong \det(V)_1 \boxtimes \det(V)_1$ over $X^2 \setminus \Delta$ has, on the component $\on{Gr}^{\check\mu, \check\nu}_2$, a pole of order $mn$, where $n = \langle \lambda, \check\nu \rangle$.
\end{theorem}

\begin{proof}
 As in the previous proof, we define $\det(V)_2$ via the point $X^2 \to \on{Gr}_2$ corresponding to the component $\on{Gr}^{\check\mu, \check\nu}_2$ having the data $((\on{pr}_1, \on{pr}_2), \OO(-m\Delta_{12} - n\Delta_{13}), \text{nat.}))$, where $\text{nat.}$ is the natural trivialization away from the graphs $\Delta_{12}$ and $\Delta_{13}$ of the two projections $X^2 \to X$.  Assuming for simplicity that $m,n > 0$, we have a natural map
 \begin{equation}
  \label{eq:chinese remainder map}
  \OO/\OO(-m\Delta_{12} - n\Delta_{13}) \to \OO/\OO(-m\Delta_{12}) \oplus \OO/\OO(-n\Delta_{13})
 \end{equation}
 which away from $\Delta_{12} \cap \Delta_{13}$ is the isomorphism from which the factorization isomorphism is obtained.  To compute the desired pole, it suffices to consider \ref{eq:chinese remainder map} locally, so we assume that $X$ has a coordinate $x$ and that all the vector bundles are trivial. Then $X^3$ has coordinates $(x,y,z)$ and we choose for each of the above sheaves the following bases as $\OO_{X^2}$-modules:
 \begin{align*}
  (x - y)^i, (x - y)^m(x - z)^j
  &&
  (x - y)^i
  &&
  (x - z)^j
 \end{align*}
 where $0 \leq i < m$, $1 \leq j < n$.  Then \ref{eq:chinese remainder map} has the following matrix:
 \begin{equation*}
  \left(
  \begin{array}{c|c}
   \on{id}_{m \times m} & * \\ \hline
   0 & \vphantom{\Bigg|}\Bigl[\bigl((x - y)^m (x - z)^j \bmod (x - z)^n\bigr) [(x - z)^k]\Bigr]_{j,k}
  \end{array}
  \right)
 \end{equation*}
 where the notation $[(x - z)^k]$ means the coefficient of this monomial.  The order of the zero at $\Delta$ of the determinant of this matrix is the desired order of the pole of the factorization map.  Only the lower-right corner need be computed, and it is clear that its determinant is that of the action of $(x - y)^m$ acting on $\OO/\OO(-n\Delta_{13})$.  Since $x - y = (x - z) + (z - y)$, where the former is nilpotent and the latter scalar, this determinant is $(z - y)^{mn}$, as claimed.
\end{proof}

Now we may turn to $\det(V)_n$, where $V$ is any representation of any reductive group $G$.  The above computations immediately imply the following proposition:

\begin{theorem}{prop}{determinant forms}
 Let $G$ be a reductive group with torus $T$ and let $V$ be a representation of $G$ with weights $\lambda$ (so $V_T = \bigoplus \lambda$ is the direct sum of character representations).  Define the $\Z$-valued bilinear form on $\Lambda_T$,
 \begin{equation*}
  K(\check\mu, \check\nu)
   = \sum_\lambda \langle \lambda, \check\mu \rangle \langle \lambda, \check\nu \rangle;
 \end{equation*}
 then $\det(V)_n$ is an sf line bundle if and only if $K$ has values in $2\Z$, and if so, then the corresponding sf gerbe $\det(V)_n^{\log a}$ (for any $a \in A$) has associated quadratic form given by
 \begin{equation*}
  \log_a Q(\check\mu) = R(\check\mu)
    = \frac{1}{2}\sum_\lambda \langle \lambda, \check\mu\rangle^2.
 \end{equation*}
 Furthermore, if this holds, then the half-weight $\zeta = \frac{1}{2} \sum_\lambda \lambda$ is in fact integral, and we have on $\on{Gr}_{T,X}$:
 \begin{equation*}
  \det(V_T)_1^{\check\mu}
   = \omega_X^{\langle \zeta, \check\mu \rangle} \otimes \sh{T}_X^{R(\check\mu)}.
 \end{equation*}
 That is, in reference to the split exact sequence of \ref{torus exact sequence}, the sf gerbe $\det(V_T)_n^{\log a}$ corresponds to the quadratic form $Q$ and the multiplicative gerbe $\omega_X^{\log a^{\langle \zeta, \check\mu \rangle}}$.
\end{theorem}

\begin{proof}
 The appearance of $\zeta$ is due to \ref{character determinant bundle}, after combining all the weight spaces of the $G$-representation $V$.  The only further argument needed is that since the sign in \ref{determinant factorizable} is constant on connected components and compatible with change of group, we may compute it on the copy of  $\on{Gr}_{T,X^n} \subset \on{Gr}_{G,X^n}$, in which case the previous lemmas apply.
\end{proof}

When $G$ is a reductive group of semisimple rank at least $1$, then the adjoint representation always satisfies the conditions of the theorem, since its nonzero weights are the positive \emph{and negative} roots.  When $G$ is a torus, this representation is trivial.

\section{Classification: the general case}
\label{s:the non-split exact sequence}

In this section we pursue the generalization of \ref{torus exact sequence} to arbitrary sf gerbes.  In order to properly analyze the quadratic form, we need to invoke results belonging thematically to \ref{c:equivariance of sf gerbes}, but to preserve the unity of this chapter we do so as forward-references, being careful to avoid circularity.

Ultimately, we will need to restrict to quadratic forms induced from $W$-invariant forms with values in $\Z$, and to begin, we investigate equivalent characterizations of them.

\begin{theorem}{lem}{invariant form and pairing}
 Let $Q \colon \Lambda_T \to k^*$ be a $W$-invariant quadratic form with associated bilinear form $\kappa$.  For any coroot $\check\alpha$ of $G$, there is a homomorphism $\epsilon_{\check\alpha} \colon \Lambda_T \to A_2$ (the 2-torsion in $A$) such that
 \begin{equation*}
  \kappa(\check\alpha, \lambda) =
    \epsilon_{\check\alpha}(\lambda) Q(\check\alpha)^{\langle \alpha, \lambda \rangle},
 \end{equation*}
 which is trivial if $\alpha$ is not twice a weight in $\Lambda^T$.
\end{theorem}

\begin{proof}
 Fix $\lambda_0$ and denote $k = \langle \alpha, \lambda_0 \rangle$.  Then by $W$-invariance of $Q$, we have
 \begin{multline*}
  \kappa(-k\check\alpha, \lambda_0)
   = Q(-k\check\alpha + \lambda_0) Q(-k\check\alpha)^{-1} Q(\lambda_0)^{-1}
   = Q(s_{\check\alpha} \lambda_0) Q(\check\alpha)^{-k^2} Q(\lambda_0)^{-1} \\
   = Q(\check\alpha)^{-k^2}
   = Q(\check\alpha)^{-k\langle \alpha, \lambda_0 \rangle}.
 \end{multline*}
 For any $\lambda$ with $\langle \alpha, \lambda \rangle = nk$, we can write inductively
 \begin{multline*}
  \kappa(-k\alpha, \lambda)
   = \kappa(-k\alpha, \lambda - \lambda_0) \kappa(-k\alpha, \lambda_0)
   = Q(\check\alpha)^{-k\langle \alpha, \lambda - \lambda_0 \rangle}
     Q(\check\alpha)^{-k\langle \alpha, \lambda_0 \rangle} \\
   = Q(\check\alpha)^{-k \langle \alpha, \lambda \rangle}.
 \end{multline*}
 If $\langle \alpha, \lambda \rangle = -nk$ then we can replace $\lambda - \lambda_0$ with $\lambda + \lambda_0$ and conclude the same; thus, the above equality holds for all $n$.
 
 For any coroot $\check\alpha$, it is possible to choose $\lambda_0$ with $k = -2$, and if $\alpha$ is twice a weight then in fact every $\lambda$ is of the form considered above for some $n$; otherwise, we can take $k = -1$.  Thus, in all cases we have
 \begin{equation*}
  \kappa(\check\alpha, \lambda)^2 Q(\check\alpha)^{-2\langle \alpha, \lambda\rangle} = 1;
 \end{equation*}
 therefore $\epsilon_{\check\alpha}$ indeed has values $A_2$ and is trivial if $k = -1$ is possible, as claimed.
\end{proof}

The homomorphisms $\epsilon_{\check\alpha}$ are but a small manifestation of a larger, more irritating phenomenon.  The next lemma describes how to avert it:

\begin{theorem}{lem}{integer-valued forms}
 Let $Q \in Q(\Lambda_T, A)^W$ be any $W$-invariant form.  Then it can be represented as the product of either of the two following kinds of quadratic forms:
 \begin{itemize}
  \item
  Products of Killing forms, namely $a^{Q_i}$, where $a \in B$ for some extension $A \subset B$ and the $Q_i \colon \Lambda_T \to \Z$ are the quadratic forms $Q_i(\lambda) = \frac{1}{2} \sum_\beta \langle \beta, \lambda \rangle^2$ appearing in \ref{determinant forms}, where $\beta$ runs over the roots in an irreducible component (indexed by $i$) of the root system of $G$.
  
  \item
  Denoting $\Lambda_{T,r}^\Q = \Lambda_T \cap \Q\Lambda_{T,r} \subset \Q \otimes \Lambda_T$, a single form $Q$ with associated bilinear form $\kappa$ for which there exists an integer $k$ such that $Q^k(\Lambda_{T,r}^\Q) = 1$ and $\kappa^k(\Lambda_{T,r}^\Q, \Lambda_T) = 1$.
 \end{itemize}
 A quadratic form $Q$ has $k = 1$ in the above expansion if and only if there is an extension $A \subset B$ such that $Q$ lies in the image of $Q(\Lambda_T, \Z)^W \otimes B \to Q(\Lambda_T, B)$. 
\end{theorem}

\begin{proof}
 The values of $Q$ on any element of $\Lambda_{T,r}$ are determined by the $Q(\check\alpha)$ and $\kappa(\check\alpha, \check\beta)$ for coroots $\check\alpha, \check\beta$, and if $Q \in Q(\Lambda_T, A)^W$, then the values $Q(\check\alpha)$ are constant on $W$-orbits among coroots.  By \ref{invariant form and pairing}, these values of $\kappa$ are also determined, up to $2$-torsion, by the values of $Q(\check\alpha)$, so that $Q^2$ is completely determined by the values of $Q$ on $W$-orbits of coroots.  In each irreducible component, there are at most two such orbits, the long and short coroots, of which the former are linear combinations of the latter.  Thus, $Q^2$ is determined by the single values $Q(\check\alpha)$, where $\check\alpha$ runs over representatives of the short roots in the components of the coroot system.
 
 We pick some such coroot $\check\alpha_i$ in the $i$'th component of the coroot system of $G$ and let $a_i = Q(\check\alpha_i)$ and $m = Q_i(\check\alpha_i)$.  If $a_i$ is not an $m$'th power in $A$, we adjoin an $m$'th root in an extension $B$; replace $a_i$ by such a root, so that by the above, $(a_i^{Q_i})^2$ takes the values $Q^2(\check\alpha)$ for \emph{any} $\check\alpha$ in the $i$'th component of the coroot system, and is trivial for any other $\check\alpha$.  Thus, the square of the form
 \begin{equation*}
  R = Q / \prod_i a_i^{Q_i}
 \end{equation*}
 vanishes on the coroot lattice $\Lambda_{T,r}$; by \ref{invariant form and pairing}, the square, $\kappa_R(\Lambda_{T,r}, \Lambda_T)^2$, of its bilinear form also vanishes.  Choose $k_0$ such that $k_0\Lambda_{T,r}^\Q \subset \Lambda_{T,r}$, so that we have
 \begin{equation*}
  \kappa_R(\Lambda_{T,r}^\Q, \Lambda_T)^{2k_0}
   = \kappa_R(k_0\Lambda_{T,r}^\Q, \Lambda_T)^2
   = \kappa_R(\Lambda_{T,r}, \Lambda_T)^2
   = 1.
 \end{equation*}
 Thus, the corresponding power of the quadratic form, $R^{2k_0}$, is a homomorphism on $\Lambda_{T,r}^\Q/\Lambda_{T,r}$, making its $k_0$'th power trivial, so we may take $k = 2k_0^2$.
 
 Every Killing form is, by definition, defined by a $W$-invariant integer-valued form, namely $Q_i$.  For the other case, suppose $Q(\Lambda_{T,r}^\Q) = 1$ and $\kappa(\Lambda_{T,r}^\Q, \Lambda_T) = 1$; then $Q$ and $\kappa$ descend to $\Lambda_T/\Lambda_{T,r}^\Q$, which is a finitely-generated free abelian group.  If we denote by $x_1, \dots, x_r$ some basis, then $\kappa$ is determined by its values $a_{ij} = \kappa(x_i, x_j)$:
 \begin{equation*}
  \kappa(\sum_i n_i x_i, \sum_j m_j x_j)
   = \prod_{i,j} a_{ij}^{n_i m_j}.
 \end{equation*}
 The $n_i$ and $m_j$ are integers, so this expression defines $\kappa$ over $\Z$.  By definition of the bilinear form, we have
 \begin{equation*}
  Q(\sum_i n_i x_i)
   = \prod_i Q(n_i x_i) \prod_{j,k} \kappa(x_j, x_k)^{n_j n_k},
 \end{equation*}
 where we have inductively
 \begin{equation*}
  Q((n + 1) x) = Q(nx) Q(x) \kappa(x, x)^n = Q(x)^{n + 1} \kappa(x,x)^{\binom{n + 1}{2}},
 \end{equation*}
 thus defining $Q$ over $\Z$ as well.
 
 Conversely, suppose we have a $W$-invariant $\Z$-valued form $Q_0$, and apply the decomposition of the first part of the lemma.  Since $\Z$ is torsion-free, we have $k = 1$, as desired.
\end{proof}

We can now state the main theorem of this chapter; here $\cat{H}^2_\text{sf}(G,X,A)$ means, as before, the $2$-category of sf $A$-gerbes on $\on{Gr}_{G,X^n}$.

\begin{theorem}{thm}{sf gerbes exact sequence}
 Let $Q(\Lambda_T, A)_\Z^W \subset Q(\Lambda_T, A)^W$ be the subgroup of $W$-invariant quadratic forms $Q$ coming from $\Z$ in the sense of \ref{integer-valued forms}.  Then there is a split short exact sequence of $2$-Picard categories:
 \begin{equation*}
  1 \to \cat{Hom}(\pi_1(G), \cat{H}^2(X,A))
    \to \cat{H}^2_\text{sf}(G, X, A)
    \to Q(\Lambda_T, A)_\Z^W
    \to 1.
 \end{equation*}
\end{theorem}

The proof occupies the remainder of this section.  We begin by defining a 2-functor using
\ref{eq:borel diagram},
\begin{equation}
 \label{eq:group to torus functor}
 F \colon \cat{H}^2_\text{sf}(G, X, A) \to \cat{H}^2_\text{sf}(T, X, A), \quad
 t^* F(\stack{G}_n) = b^* \stack{G}_n
\end{equation}
where the descent is possible because, fixing any inclusion $T \to G$, we have $\stack{T}_n = i^* \stack{G}_n$ since \ref{eq:torus section} is a section of \ref{eq:borel diagram}.  Using \ref{torus exact sequence}, the theorem then becomes:

\begin{theorem}{thm}{exact sequence restatement}
 An sf gerbe $\stack{T}_n$ on $\on{Gr}_{T,X^n}$ is of the form $F(\stack{G}_n)$ if and only if its associated quadratic form is in $Q(\Lambda_T, A)_\Z^W$ and its associated multiplicative gerbe descends to $\pi_1(G)$ (the quotient of $\Lambda_T$ by the coroot lattice $\Lambda_{T,r}$).  (Recall the notion of a \emph{multiplicative gerbe} on a group scheme, here the discrete groups $\pi_1(G)$ or $\Lambda_T$, from \ref{multiplicative gerbe}.)
\end{theorem}

We will refer to the quadratic form $Q(F(\stack{G}_n))$ as just $Q(\stack{G}_n)$ for an sf gerbe on
$\on{Gr}_{G,X^n}$.  The above theorem thus consists of four independent propositions, which we pursue in turn.

\begin{theorem}{prop}{generalized injectivity}
 For any $\pi_1(G)$-multiplicative gerbe $\stack{G}^\lambda$ ($\lambda \in \pi_1(G)$) and any
 quadratic form $Q$ on $\Lambda_T/\Lambda_{T,r}^\Q$, there is an sf gerbe $\stack{G}_n$ on $\on{Gr}_{G,X^n}$ lifting that on $\on{Gr}_{T,X^n}$ which is associated (via \ref{torus exact sequence}) with the pullback of these data to $\Lambda_T$.
\end{theorem}

\begin{proof}
 Let $Z$ be the torus with $X_*(Z) = \Lambda_T/\Lambda_{T,r}^\Q$ (note that this is a free abelian group of finite rank); let $q_Z \colon \on{Gr}_{T,X^n} \to \on{Gr}_{Z,X^n}$ be the map induced by the quotient $T \to Z$ and let $q \colon \pi_1(G)_{X^n} \to \on{Gr}_{Z,X^n}$, induced by the map $\pi_1(G) = \Lambda_T/\Lambda_{T,r} \to \Lambda_T/\Lambda_{T,r}^\Q$.  Finally, let $p \colon \on{Gr}_{G,X^n} \to \pi_1(G)_{X^n}$ be the map introduced at the end of \ref{s:the affine grassmannian of any group}.
 
 By \ref{torus exact sequence}, there is an sf gerbe $\stack{T}_n$ on $\on{Gr}_{Z,X^n}$ with quadratic form $Q$ (and trivial multiplicative part), and thus $q_Z^* \stack{T}_n$ is an sf gerbe on $\on{Gr}_{T,X^n}$ with form $Q$; then $\stack{G}_n = p^* q^* \stack{T}_n$ is an sf gerbe on $\on{Gr}_{G,X^n}$ whose associated gerbe on $\on{Gr}_{T,X^n}$ is $q_Z^* \stack{T}_n$, as desired.
 
 Similarly, we can construct an sf gerbe $\stack{T}_n$ (with trivial quadratic form) directly on $\on{Gr}_{T,X^n}$ with multiplicative part $\stack{G}^\lambda$, considered as a $\Lambda_T$-multiplicative gerbe, which descends to $\pi_1(G)_{X^n}$ since $\stack{G}^\lambda$ does.  Then the pullback of this descended gerbe along $p$ is the desired sf gerbe on $\on{Gr}_{G,X^n}$.
\end{proof}

All of the remaining claims rely on the following key technical lemma:

\begin{theorem}{lem}{vertical map injective}
 The $2$-functor $\map{F}{\cat{H}^2_\text{sf}(G,X,A)}{\cat{H}^2_\text{sf}(T,X,A)}$ given above is faithful.  This means: for an sf-gerbe $\stack{G}_n$ on $\on{Gr}_{G,X^n}$, any trivialization of $F(\stack{G}_n)$ gives a trivialization of $\stack{G}_n$; for an automorphism $\phi$ of the trivial sf gerbe, a trivialization of $F(\phi)$ gives one of $\phi$; for a $2$-automorphism $a$ of $\id$, if $F(a) = \id$, then $a = \id$.
 
 Furthermore, $F$ identifies $\cat{Aut}(\stack{G}_n^0) \cong \cat{Hom}(\pi_1(G), \cat{H}^1(X,A))$, where $\stack{G}_n^0$ is the trivial sf gerbe.
\end{theorem}

The essential part of the proof of this lemma is an analysis of the gerbe $\stack{G}$ on $\on{Gr}_G$ obtained as the fiber of $\stack{G}_1$ over any point of $X$.  Consider the stratification of $\on{Gr}_G$ by the closed sets $\bar{\on{Gr}}{}_B^\lambda$, using \ref{borel components}.  The interior of this set is $\on{Gr}_B^\lambda \cong \Aff^\infty$, and therefore is cohomologically trivial, so that $\stack{G}$ has a unique trivialization there. The boundary of this set is the union over the simple coroots $\check\alpha$ of the $\bar{\on{Gr}}{}_B^{\lambda - \check\alpha}$, and these are all Cartier divisors, so that for each $\check\alpha$ the unique trivialization has an associated order $m(\lambda,\check\alpha)$ about these boundary components.

\begin{theorem}{lem}{fiber monodromies}
 The numbers $m(\lambda, \check\alpha)$ are independent of $\lambda$.  If $m(\check\alpha)$ is their common value, then we have $m(\check\alpha) = Q(\check\alpha)^{-1}$ for any simple coroot $\check\alpha$, where as before, $Q = Q(\stack{G}_n)$ is the quadratic form associated to the full factorizable gerbe.
\end{theorem}

\begin{proof}
 Let $\lambda, \mu \in \Lambda$ be arbitrary coweights and consider the projective line bundles $\smash{\Prj_{1;\check\alpha}^\lambda}$ and $\Prj_{2,1;\check\alpha}^{\lambda,\mu}$ given in \ref{general projective line bundles}.  We view $\Prj_{1;\check\alpha}^\lambda$ as a $\Prj^1$-bundle on $\on{Gr}_{T,X}^{\lambda + \check\alpha}$, the $\infty$-section, and denote the structure map by $p_1$; since $U_1^\infty$ is an $\smash{\Aff^1}$-bundle, the restriction of $\stack{G}_1$ to it is canonically equivalent to the pullback of its restriction to the $\infty$-section, so
 \begin{equation}
 \label{eq:boundary order 1}
  \stack{G}_1|{\Prj_{1;\check\alpha}^\lambda} \cong
   p_1^* \stack{T}_1^{\lambda + \check\alpha} \otimes
   \OO(\on{Gr}_{T,X}^\lambda)^{\log m(\lambda, \check\alpha)}
 \end{equation}
 by definition of this number. To show that $m(\lambda, \check\alpha)$ is independent of $\lambda$, we do the same on $\Prj_{2,1;\check\alpha}^{\lambda, \mu}$ and apply factorizability on $\Delta$. Since $\Prj_{2,1;\check\alpha}^{\lambda,\mu}$ is a $\Prj^1$-bundle on $\on{Gr}_{T,X^2}^{\lambda + \check\alpha, \mu + \check\alpha}$ (with bundle map $p_2$), by the same reasoning we have
 \begin{equation}
  \label{eq:boundary order 2}
  \stack{G}_1|{\Prj_{2,1;\check\alpha}^{\lambda,\mu}} \cong
   p_2^* \stack{T}_2^{\lambda + \check\alpha, \mu + \check\alpha} \otimes
   \OO(\on{Gr}_{T,X^2}^{\lambda, \mu + \check\alpha})^{\log M}
 \end{equation}
 where $M$ can be determined in two ways.  Over $X^2 \setminus \Delta$, we have $\smash{\Prj_{2,1;\check\alpha}^{\lambda,\mu}} \cong \Prj_{1;\check\alpha}^\lambda \times \on{Gr}_{T,X}^{\mu + \check\alpha}$ by \ref{projective line bundle properties}, so factorizability of $\stack{G}_2$ and \ref{eq:boundary order 1} give $M = m(\lambda, \check\alpha)$.  On the other hand, clearly $\Prj_{2,1;\check\alpha}^{\lambda,\mu}|_\Delta = \Prj_{1;\check\alpha}^{\lambda + \mu + \check\alpha}$, so \ref{eq:boundary order 1} gives $M = m(\lambda + \mu + \check\alpha, \check\alpha)$.  Since $\mu + \check\alpha$ can be any coweight, we conclude that $m(\lambda, \check\alpha) = m(\check\alpha)$ is independent of $\lambda$, as desired.

 We now compute $m(\check\alpha) = Q(\check\alpha)^{-1}$.  Let $\map{m}{\Prj_{1;\check\alpha}^0 \times \on{Gr}_{T,X}^{\check\alpha}}{\Prj_{2,1;\check\alpha}^{0,0}}$ be the extension of the factorization map given in \ref{projective line bundle properties}, having a zero of order $2$ along the divisor $\smash{\tilde\Delta}$, the fiber over $\Delta \subset X^2$.  We give two computations of $m^*(\stack{G}_2|\Prj^{0,0}_{2,1;\check\alpha})$: first, by \ref{eq:boundary order 2}, we have
 \begin{align*}
  m^* \stack{G}_2
   &= m^* (p_2^*\stack{T}_2^{\check\alpha, \check\alpha} \otimes
         \OO(\on{Gr}_{T,X^2}^{0, \check\alpha})^{\log m(\check\alpha)}) \\
   &= p_2^*(m^* \stack{T}_2^{\check\alpha, \check\alpha}) \otimes
         \OO(\on{Gr}_{T,X^2}^{0, \check\alpha})^{\log m(\check\alpha)}
         \otimes \OO(\tilde\Delta)^{\log m(\check\alpha)^2}.
 \end{align*}
 On the other hand, since $m$ extends the factorization map, we have
 \begin{equation*}
  m^* \stack{G}_2
   = \stack{G}_1|{\Prj_{1;\check\alpha}^0} \boxtimes \stack{T}_1^{\check\alpha}
 \end{equation*}
 since along the zero section $\on{Gr}_{T,X^2}^{0,\check\alpha}$, factorization has a order of $\kappa(0, \check\alpha) = 1$.  Inserting \ref{eq:boundary order 1}, we continue:
 \begin{equation*}
  m^* \stack{G}_2
   = (p_1^* \stack{T}_1^{\check\alpha} \otimes \OO(\on{Gr}_{T,X}^0)^{\log m(\check\alpha)})
     \boxtimes \stack{T}_1^{\check\alpha}
   \cong
     p_2^*(\stack{T}_1^{\check\alpha} \boxtimes \stack{T}_1^{\check\alpha}) \otimes
     \OO(\on{Gr}_{T,X^2}^{0,\check\alpha})^{\log m(\check\alpha)}.
 \end{equation*}
 We have $m^* \stack{T}_2^{\check\alpha,\check\alpha} \cong \stack{T}_1^{\check\alpha} \boxtimes \stack{T}_1^{\check\alpha} \otimes \OO(\tilde\Delta)^{\log \kappa(\check\alpha, \check\alpha)}$ by definition of $\kappa(\check\alpha,\check\alpha)$, and therefore comparing monodromies about the zero and diagonal sections in the above two expressions, we find that $m(\check\alpha)^2 = \kappa(\alpha,\alpha)^{-1} = Q(\alpha)^{-2}$.  Since the square is induced via a pullback from a $2$-fold cover branched at $\Delta$ on both sides, we find $m(\check\alpha) = Q(\check\alpha)^{-1}$, as desired.
\end{proof}

\begin{proof}[Proof of \ref{vertical map injective}]
 For 2-morphisms: any factorizable map $\on{Gr}_{G,X^n} \to A$ must be constant on connected components since $A$ is discrete.  This is the same as saying that it descends to $\pi_1(G)_n$.
 
 Now we consider 1-morphisms.  Since each $\on{Gr}_B^\lambda$ is simply-connected, so are their closures, and therefore any $A$-torsor is trivial on each $\bar{\on{Gr}}{}_B^\lambda$.  Since $\on{Gr}_G$ is the union of these nested closed subspaces, any $A$-torsor on $\on{Gr}_G$ is trivial (i.e.\ $\on{Gr}_G$ is simply-connected). Since automorphsims of sf gerbes on $\on{Gr}_{G,X^n}$ are sf $A$-torsors, they are all pullbacks along $\map{p}{\on{Gr}_{G,X^n}}{\pi_1(G)_{X^n}}$, and thus trivializable.  Since 2-automorphisms of sf $A$-torsors are factorizable maps to $A$, the previous paragraph shows that in fact such $A$-torsors are uniquely trivializable (as \emph{sf torsors}).
 
 Now suppose for an sf-gerbe $\stack{G}_n$ that $\stack{T}_n = F(\stack{G}_n)$ is trivial.  Then in particular, it has trivial quadratic form, so by \ref{fiber monodromies}, $\stack{G}_n$ must be trivial on the fibers of $\stack{G}_n \to X^n$; therefore, $\stack{G}_n$ is of the form $p^* \stack{P}_n$, where $\stack{P}_n$ is factorizable on $\pi_1(G)_{X^n}$.  That is, we have $F(\stack{T}_n) = q^* \stack{P}_n$, so $q^* \stack{P}_n$ is trivial on $\on{Gr}_{T,X^n}$, and thus $\stack{P}_n$ is on $X^n$.  Therefore $\stack{G}_n$ is trivial on $\on{Gr}_{G,X^n}$, and by the previous paragraph, that trivialization is unique (up to unique 2-isomorphism).
\end{proof}

Now we may begin the process of lifting $W$-invariant forms to sf gerbes.  The following lemma is a useful technical tool for the general argument, though there are specific cases of interest (for example, when $A = k^*$ with $k$ an algebraically closed field) in which it is unnecessary.

\begin{theorem}{lem}{change of groups}
 Suppose that $A \subset B$ are two abelian groups and that $\stack{G}_n$ is an sf $B$-gerbe such that $F(\stack{G}_n)$ is an sf $A$-gerbe on $\on{Gr}_{T,X^n}$.  Then $\stack{G}_n$ is defined over $A$.
\end{theorem}

\begin{proof}
 It suffices to show that $\stack{G}_n$ is trivial when induced to $C = B/A$.  This is true by hypothesis of $F(\stack{G}_n)$, and by injectivity of the map $\cat{H}^2_\text{sf}(G, X, C) \to \cat{H}^2_\text{sf}(T,X,C)$ (i.e.\ \ref{vertical map injective}), the lemma follows.
\end{proof}

\begin{theorem}{prop}{surjective}
 For every $Q \in Q(\Lambda_T, A)_\Z^W$, there exists an sf $A$-gerbe $\stack{G}_n$ on $\on{Gr}_{G,X^n}$ whose associated quadratic form is $Q$.
\end{theorem}

\begin{proof}
 Using \ref{integer-valued forms}, we produce an extension $A \subset B$ and elements $a_i \in B$ such that $Q$ can be factored as:
 \begin{equation*}
  Q = R \cdot \prod_i a_i^{Q_i},
 \end{equation*}
 with the $Q_i$ being Killing forms coming from adjoint representations $\on{adj}_i$ of $G$ and $R$ descending to $\Lambda_T / \Lambda_{T,r}^\Q$.  By \ref{generalized injectivity}, $R$ comes from some sf gerbe $\stack{R}_n$ on $\on{Gr}_{G,X^n}$, while by \ref{determinant forms}, the $Q_i$ come from the determinant gerbes $\det(\on{adj}_i)^{\log a_i}$.  Nominally, the latter are $B$-gerbes, but since their quadratic forms are $A$-valued and their multiplicative parts are trivial, by \ref{torus exact sequence} the corresponding sf gerbes on $\on{Gr}_{T, X^n}$ are defined over $A$, so by \ref{change of groups}, they themselves are actually defined over $A$.  Then the $A$-gerbe
 \begin{equation*}
  \stack{G}_n = \stack{R}_n \otimes \prod_i \det(\on{adj}_i)^{\log a_i}
 \end{equation*}
 has quadratic form $Q$.
\end{proof}

\begin{theorem}{cor}{multiplicative gerbe fundamental group}
 If $\stack{G}_n$ is an sf gerbe, the $\Lambda_T$-multiplicative gerbe part of $F(\stack{G}_n)$ descends to $\pi_1(G)$.
\end{theorem}

\begin{proof}
 Clearly this is true of $\stack{R}$ and of $\det(\on{adj}_i)_n^{\log a}$ for any $a \in A$ and $G$\hyp representation $V$, so by the preceding construction, we may cancel out the quadratic form of $\stack{G}_n$ and assume, without loss of generality, that it is trivial.  Then by \ref{fiber monodromies}, $\stack{G}_n$ is trivial on the fibers of $\on{Gr}_{G,X^n} \to X^n$ and thus descends to $\pi_1(G)_n$, as desired.
\end{proof}

It remains to show that the functor $F$ actually produces gerbes whose quadratic forms are definable as $W$-invariant forms over $\Z$.  This involves technical tools not to be discussed until the next chapter.

\begin{theorem}{prop}{form is invariant}
 If $\stack{G}_n$ is an sf-gerbe, then $Q(\stack{G}_n) \in Q(\Lambda_T, A)^W_\Z$.
\end{theorem}

\begin{proof}
 We show in \ref{existence of equivariance} that $\stack{G}_n$ possesses a structure of $G(\smash{\hat{\OO}})_n$-equivariance (see \ref{arc and loop groups}).  Let $G_n \subset G(\hat{\OO})_n$ be the factorizable subgroup of points $(\vect{x}, g)$ with $g \in G(\bar{x})$, where we consider $\smash{\hat{X}_{\bar{x}}}$ to cover $\bar{x}$.  Then $\stack{G}_n$ is $G_n$-equivariant and, thus, equivariant for $N_G(T)_n$, where $N_G(T)$ is the normalizer of the torus $T$, and $N_G(T)/T \cong W$ is the Weyl group. This group acts on $\on{Gr}_{T,X^n} \subset \on{Gr}_{G,X^n}$ through its quotient $W_n$, so that $i^* \stack{G}_n$ is necessarily $W_n$-equivariant.
 
 Now we show that $Q(\stack{G}_n) \in Q(\Lambda_T, A)^W_\Z$.  Since it is $W$-invariant, it suffices by  the construction of \ref{surjective} and \ref{integer-valued forms} to show that, if $Q$ is trivial on $\Lambda_{T,r}$, then we can take $k = 1$ in the expression there.  First, consider the image of the grassmannian $\on{Gr}_{G^\SC,X^n}$ in $\on{Gr}_{G,X^n}$, where $G^\SC$ is the simply-connected form of $G$, having weight lattice $\Lambda_{T,r}$; its parts are the irreducible components of the the $\on{Gr}_{G,X^n}$ whose $G(\smash{\hat{\OO}})_n$-orbits (see \ref{s:orbits in the affine grassmannian}) are labelled by $\Lambda_{T,r}^n$.  Thus, since $Q$ vanishes on $\Lambda_{T,r}$, we have by \ref{vertical map injective} that $\stack{G}_n$ is factorizably trivial on these components.
 
 Now we consider the collection of connected components $C$ of $\on{Gr}_{G,X}$ whose $G(\hat{\OO})_1$-orbits $\on{Gr}_{G,X}^\lambda$ have $\lambda \in \smash{\Lambda_{T,r}^\Q}$ (isomorphic to $\on{Gr}_{G'}$, where $G' = [G,G]$ is the derived subgroup of $G$).  In the $k$-fold convolution product (see \ref{convolution product grassmannian})
 \begin{equation*}
  m_{k \cdot 1} \colon \tilde{\on{Gr}}_{k \cdot 1} \to \on{Gr}_k
 \end{equation*}
 the $k$-fold convolution of a single component $C$ maps to $\on{Gr}_{G^\SC, X^n}$, as described above, since $k\smash{\Lambda_{T,r}^\Q} \subset \Lambda_{T,r}$.  By strong factorizable equivariance (\ref{existence of equivariance} and \ref{gerbe strong factorizable equivariance}), we have
 \begin{equation*}
  m_{k \cdot 1}^* \stack{G}_k \cong \stack{G}_1 \tbtimes \dots \tbtimes \stack{G}_1
 \end{equation*}
 on these components; for the twisted product, see \ref{twisted product gerbe}.  Since the former is trivial, so is the latter, which we claim implies that $\stack{G}_1$ is equivariantly trivial.  To show this, consider the definition of the twisted product
 \begin{equation*}
  p^* \stack{G}_1^{\tbtimes k}
    = \pi^* \stack{G}_1^{\tbtimes (k - 1)} \boxtimes_{X^k} \stack{G}_1,
 \end{equation*}
 where $p$ is the $G(\hat{\OO})_1$-torsor defined in the proof of \ref{convolution product twisting} and we write $\pi \colon \tilde{G}(\hat{\KK})_{k \cdot 1} \to \tilde{\on{Gr}}_{(k - 1) \cdot 1}$.  Since $\smash{\stack{G}_1^{\tbtimes k}}$ is trivial, in particular the second factor on the right is $G(\smash{\hat{\OO}})_1$-equivariantly trivial, as desired.
 
 Finally, since $\stack{G}_1$ is factorizably trivial on the components $C$, we have that $Q$ is trivial on $\Lambda_{T,r}^\Q$, and since it is $G(\hat{\OO})_1$-equivariantly trivial, we have by the construction of \ref{existence of equivariance} that $\kappa(\smash{\Lambda_{T,r}^\Q}, \Lambda_T) = 1$; i.e.\ that $k = 1$ in \ref{integer-valued forms}.
\end{proof}

This completes the proof of \ref{sf gerbes exact sequence}.

\chapter{Equivariance of symmetric factorizable gerbes}
\label{c:equivariance of sf gerbes}

In this chapter, we continue to analyze the consequences of the factorizability of a gerbe on the grassmannian.  Unlike in the previous chapter, the principal motivation will be the study of equivariance with respect to a certain group action, which we will see (like our previous theorems) is constructed and constrained simply by the topological data of factorizablity.

\section{Infinitesimal actions on the grassmannian}
\label{s:infinitesimal actions on the grassmannian}

In this section we finally introduce the spaces of infinitesimal loops that the affine grassmannian is usually defined against.  

\subsection*{The formal arc and loop groups}

We use the following notation: for any \emph{affine} scheme $S$ and a point $\vect{x} \in X^n(S)$, the graph $\bar{x} \to X_S$ is a closed affine subscheme of $X_S$ and thus its $n$'th infinitesimal neighborhoods are also affine, with the same underlying topological space $\abs{\bar{x}}$.  Letting $\OO_{\bar{x},n} = \OO_{\bar{x}_n}$ be their rings of global sections on $\abs{\bar{x}}$, we take \begin{equation*}
 \hat{\OO}_{X,\bar{x}} = \invlim_n \OO_{\bar{x},n}
\end{equation*}
(which is also the ring of global sections on the formal completion of $X_S$ at $\bar{x}$) and define
\begin{equation*}
 \hat{X}_{\bar{x}} = \on{Spec} \hat{\OO}_{X,\bar{x}},
\end{equation*}
which we call the \emph{(schemey) formal neighborhood} of $\bar{x}$.  By definition, $\smash{\hat{\OO}}_{X, \bar{x}}$ has $\OO_{\bar{x}}$ as a quotient, so $\smash{\hat{X}}_{\bar{x}}$ has $\bar{x}$ as a closed subscheme, and we define $\hat{X}_{(\bar{x})}$ to be the complement, the \emph{punctured formal neighborhood}.  We reiterate that neither of these coincide with the like-named concepts in formal schemes.

\begin{theorem}{defn}{arc and loop groups}
 We define two functors on affine schemes $S$:
 \begin{align*}
  G(\hat{\OO})_{X^n}(S) &= \{(\vect{x}, g) \mid g \in G(\hat{X}_x)\} \\
  G(\hat{\KK})_{X^n}(S) &= \{(\vect{x}, g) \mid g \in G(\hat{X}_{(x)})\}
 \end{align*}
 We will often write simply $G(\hat{\OO})_n$ and $G(\hat{\KK})_n$.  These are the \emph{formal arc group} and \emph{formal loop group} of $G$.
\end{theorem}

Both of these are, like the $\on{Gr}_{G,X^n}$, representable by ind-schemes (in fact, $G(\hat{\OO})_n$ is actually a scheme, albeit of infinite type).  The following observation is an immediate consequence of the definitions:

\begin{theorem}{lem}{arc and loop group factorizable}
 Both $G(\hat{\OO})_n$ and $G(\hat{\KK})_n$ are sf group ind-schemes.
\end{theorem}

They are closely related to the grassmannians.  The following lemma is a mild generalization of a well-known basic fact, so we only sketch the proof.

\begin{theorem}{lem}{arc and loop group actions}
 We have a transitive action of sf schemes over $X^n$ of $G(\hat{\KK})_n$ on $\on{Gr}_{G,X^n}$, in which the trivial section $(\id, \sh{T}^0, \id)$ has stabilizer $G(\smash{\hat{\OO}})_n$, so we have the quotient as sheaves in the \'etale topology on affine schemes:
 \begin{equation*}
  \on{Gr}_{G,X^n} \cong G(\hat{\KK})_n/G(\hat{\OO})_n.
 \end{equation*}
\end{theorem}

\begin{proof}
 By the Beauville--Laszlo theorem \cite{BL}, any point $p = (\vect{x}, \sh{T}, \phi) \in \on{Gr}_n(S)$ can be equivalently specified by its infinitesimal data:
 \begin{equation*}
  (\vect{x}, \hat{\sh{T}}, \hat{\phi}) \qquad
  \hat{\sh{T}} = \sh{T}|\hat{X}_{\bar{x}}, \quad
  \hat{\phi} = \phi|\hat{X}_{(\bar{x})}.
 \end{equation*}
 Then $(\vect{x}, g) \in G(\hat{\KK})_n$ acts on $p$ by composing $\hat{\phi}$ with $g$ on the right, considering $g$ as an automorphism of the trivial torsor on $\smash{\hat{X}}_{(\bar{x})}$.  The action is transitive because $\smash{\hat{\phi}}$ is identified with a group element when $\sh{T}$ is trivial, which is locally the case.  The stabilizer of the trivial section is $G(\smash{\hat{\OO}})_n$ because up to isomorphism of $\smash{\hat{\sh{T}}}$, any stabilizing $g$ must be the identity, and thus (\emph{not} up to isomorphism) must extend across $\bar{x}$, so is in $G(\smash{\hat{\OO}})_n$.
\end{proof}

When we consider $\on{Gr}_n$ as a quotient of $G(\hat{\KK})_n$, we will denote the projection map by $q_n$.

\subsection*{Outer convolution diagrams}

The most important construction related to the arc groups is that of the following ``convolution diagrams''.  Since it is possible, we give a very general definition, but one should pay particular attention to the case $m = 2$.  These spaces were described in this case in \cite{MV_Satake}*{\S5}.  We use the following notation for partitions: if $p$ is a partition of $n$ into $m$ parts (more properly, a partition of $[1, \dots, n]$ into $m$ subsets), we will denote the $k$'th part by $p_k$, a set, and its cardinality by $\card{p_k}$.

\begin{theorem}{defn}{convolution product grassmannian}
 For any partition $p$ of $n$ into $m$ parts, we have the \emph{convolution diagram} over $X^n$, where $\vect{x}_k$ is the collection of coordinates of $\vect{x}$ indexed by the elements of $p_k$:
 \begin{equation*}
  \tilde{\on{Gr}}_{G,X^p}(S)
   = \left\{\left(
           \begin{gathered}
           \vect{x}, \sh{T}_1, \dots, \sh{T}_m,\\ \phi_1, \dots, \phi_m
           \end{gathered}\right)
      \middle\vert
      \begin{aligned}
       &\text{$\phi_1 \colon \sh{T}^0 \to \sh{T}_1$ is a trivialization on $X \setminus \bar{x}_1 $} \\
       &\text{$\phi_k \colon \sh{T}_{k - 1} \to \sh{T}_k$ is an isomorphism on $X \setminus \bar{x}_k$}
      \end{aligned}
     \right\}
 \end{equation*}
 As usual, we will denote it $\tilde{\on{Gr}}_p$.  We take note of two common special cases: if $p$ is the concatenation of two intervals of sizes $n_1$ and $n_2$, we will write $p = (n_1, n_2)$ and refer to $\smash{\tilde{\on{Gr}}}_{n_1, n_2}$.  If $p$ partitions $mn$ as the concatenation of $m$ intervals of equal size $n$, we will write $p = m \cdot n$.
\end{theorem}

We will have almost no occasion to use convolutions corresponding to partitions other than $p = (n_1, n_2)$ or $p = m \cdot n$, so we will assume tacitly that the parts $p_k$ are consecutive intervals in $[1, \dots, n]$.  Furthermore, in order to reduce the burden of notation, we will take $m = 2$ in all subsequent statements and proofs, with the argument being formally similar for the general case, so that the most general partition we exhibit is $(n_1, n_2)$.

We require the convolution diagrams because of their similarities to and differences from products of grassmannians.  For example, when $p = (n_1, n_2)$, $\smash{\tilde{\on{Gr}}}_{n_1, n_2}$ has maps
\begin{align*}
 \map{m_{n_1, n_2}}{\tilde{\on{Gr}}_{n_1, n_2}}{\on{Gr}_n} &&
 \map{\on{pr}_{n_1, n_2}}{\tilde{\on{Gr}}_{n_1, n_2}}{\tilde{\on{Gr}}_{n_1}}
\end{align*}
The definitions are clear: $m_{n_1, n_2}$ is defined combinatorially as
\begin{equation}
 \label{eq:convolution diagram mult map}
 m_{n_1, n_2}(\vect{x}, \sh{T}_1, \sh{T}_2, \phi_1, \phi_2)
  = (\vect{x}, \sh{T}_2, \phi_2|_{X_S \setminus \bar{x}} \circ \phi_1|_{X_S \setminus \bar{x}}),
\end{equation}
and $\on{pr}_{n_1, n_2}$ is defined by forgetting both $\sh{T}_2$ and $\phi_2$.  More generally, we have maps $m_p \colon \smash{\tilde{\on{Gr}}}_p \to \on{Gr}_n$ and $\on{pr}_p \colon \smash{\tilde{\on{Gr}}}_p \to \tilde{\on{Gr}}_{p'}$, where $p'$ is obtained by removing the last part $p_m$ from $p$.  We proceed with establishing the good properties of these maps.

\begin{theorem}{prop}{convolution diagram factorization}
 On the open set $X^n_p$ of \ref{grassmannian is factorizable}, we have an isomorphism of $\smash{\tilde{\on{Gr}}}_p$ with the product of the $\on{Gr}_{\card p_k}$.  In addition, the map $m_p$ is ind-proper and its restriction to $X^n_p$ is the factorization isomorphism of \ref{grassmannian is factorizable}\ref{en:factorizable open}.
\end{theorem}

\begin{proof}
 We take $p = (n_1, n_2)$. We will see shortly that $\on{pr}_{n_1, n_2}$ realizes $\tilde{\on{Gr}}_{n_1, n_2}$ as a $\on{Gr}_{n_2}$-bundle over $\on{Gr}_{n_1}$; since they are both ind-proper, so is the convolution diagram, and thus so is any map out of it.  Independently of this result, however, we see that we have a trivialization $\psi = \phi_2 \circ \phi_1$ (suitably restricted, as above) of $\sh{T}_2$, on the open set $X_S \setminus \bar{x}$.  Since $X^n_{n_1, n_2}$ has as points those $\vect{x}$ with $\bar{x}_1 \cap \bar{x}_2 = \emptyset$, we may perform a gluing with the trivial torsor,
 \begin{equation*}
  \sh{T}_2' = \sh{T}_2|_{X_S \setminus \bar{x}_1} \sqcup_\psi \sh{T}^0|_{X_S \setminus \bar{x}_2}
 \end{equation*}
 and the two triples $(\vect{x}_1, \sh{T}_1, \phi_1)$ and $(\vect{x}_2, \sh{T}_2', \psi)$ are points of $\on{Gr}_{n_i}$ for $i = 1, 2$, as desired.
 
 Conversely, given two such triples (with the same notation) we construct a point of $\smash{\tilde{\on{Gr}}}_{n_1, n_2}$ by reverse gluing: on their common domain of definition $X_S \setminus (\bar{x}_1 \cup \bar{x}_2) = X_S \setminus \bar{x}$, $\phi_1$ and $\psi$ are both trivializations and, therefore, $\phi_2 = \phi_1 \psi^{-1} \colon \sh{T}_2' \to \sh{T}_1$ is an isomorphism; we define
 \begin{equation*}
  \sh{T}_2 = \sh{T}_2'|_{X_S \setminus \bar{x}_1} \sqcup_{\phi_2} \sh{T}_1|_{X_S \setminus \bar{x}_2},
 \end{equation*}
 and then $(\vect{x}, \sh{T}_1, \sh{T}_2, \phi_1, \phi_2) \in \tilde{\on{Gr}}_{n_1, n_2}(S)$.  It is clear that these two constructions invert each other.
 
 Finally, that $m_{n_1, n_2}$ sends this isomorphism to the factorization map is a matter of comparing the above construction to the proof of \ref{grassmannian is factorizable}.
\end{proof}

We now turn to the relationship between the convolution diagram and products of the grassmannians. They are not equal; however, as already claimed, the convolution diagram is a twisted product.  To properly express this, we introduce another ind-representable functor:

\begin{theorem}{defn}{twisted arc group}
 As before, we write $\vect{x} = (\vect{x}_1, \vect{x}_2)$, with $\vect{x}_i$ having $n_i$ coordinates.  We define the following functor on affine schemes $S$:
 \begin{equation*}
  \tilde{G}(\hat{\KK})_{X^{n_1, n_2}}(S) = 
    \left\{(\vect{x}, \sh{T}, \phi, \psi) \middle\vert
      \begin{aligned}
       &(\vect{x}_1, \sh{T}, \phi) \in \tilde{\on{Gr}}_{n_1} \\
       &\text{$\psi \colon \sh{T}^0 \to \sh{T}$ is a trivialization on $\hat{X}_{\bar{x}_2}$}
      \end{aligned}
    \right\}.
 \end{equation*}
 As usual, we will write just $\tilde{G}(\hat{\KK})_{n_1, n_2}$.
\end{theorem}

There is an obvious map $\tilde{G}(\hat{\KK})_{n_1, n_2} \to \on{Gr}_{n_1}$ given by forgetting $\vect{x}_2$ and $\psi$, and there is also an obvious action of the relative group $X^{n_1} \times G(\smash{\hat{\OO}})_{n_2}$ on $\smash{\tilde{G}(\hat{\KK})}_{n_1, n_2}$, with a point $(\vect{x}_1, \vect{x}_2, g)$ replacing $\psi$ by $\psi g^{-1}$.  This action in fact makes it a torsor for this group over $\on{Gr}_{n_1} \times X^{n_2}$, with the following self-evident trivializations:

\begin{theorem}{lem}{twisted arc group trivializations}
 The $G(\hat{\OO})_{n_2}$-torsor $\tilde{G}(\hat{\KK})_{n_1, n_2}$ is trivialized over $\on{Gr}_{n_1} \times X^{n_2}$ as follows:
 \begin{enumerate}
  \item \label{en:open arc group trivialization}
  On $X^n_{n_1, n_2}$, there is a \emph{natural} projection over $\tilde{\on{Gr}}_{n_1}$ to $G(\hat{\OO})_{n_2}$,
  \begin{equation*}
   (\vect{x}, \sh{T}, \phi, \psi) \mapsto u = \psi^{-1} \circ \phi
  \end{equation*}
  (for partitions with $m > 1$, this trivialization works on the slightly larger open set of $\vect{x}$'s such that $\vect{x}_m$ is disjoint from all $\vect{x}_k$, with $k < m$).
  
  \item \label{en:closed arc group trivialization}
  For \emph{any} $\vect{x}$, there is an \emph{unnatural} projection associated with each choice of trivialization $\psi'$ of $\sh{T}$ on $\hat{X}_{\bar{x}_2}$:
  \begin{equation*}
   (\vect{x}, \sh{T}, \phi, \psi) \mapsto v = \psi^{-1} \circ \psi'. \qed
  \end{equation*}
 \end{enumerate}
\end{theorem}

We note that in the more general situation where $p$ is a larger partition, $\tilde{G}(\hat{\KK})_p$ is a torsor over $\smash{\tilde{\on{Gr}}}_{p'} \times X^{\card{p_m}}$, where $p'$ is the partition obtained by deleting the last part $p_m$ of $p$.  Now we can say that the convolution diagram is the twisting of an ordinary product of grassmannians by this torsor.

\begin{theorem}{prop}{convolution product twisting}
 $\tilde{\on{Gr}}_{n_1, n_2}$ is the bundle associated with $\tilde{G}(\hat{\KK})_{n_1, n_2}$ having fiber $\on{Gr}_{n_2}$.  Explicitly, this means that
 \begin{equation*}
  \tilde{\on{Gr}}_{n_1, n_2}
    = (\tilde{G}(\hat{\KK})_{n_1, n_2} \times_{X^{n_2}} \on{Gr}_{n_2}) / G(\hat{\OO})_{n_2},
 \end{equation*}
 where the latter group acts anti-diagonally on the product. In particular, the trivialization of \ref{twisted arc group trivializations}\ref{en:open arc group trivialization} gives an isomorphism over $X^n_{n_1, n_2}$:
 \begin{equation*}
  \tilde{\on{Gr}}_{n_1, n_2} \cong \on{Gr}_{n_1} \times \on{Gr}_{n_2},
 \end{equation*}
 and this isomorphism agrees with the one given in \ref{convolution diagram factorization}.  For more general partitions $p$, if $p'$ is the partition obtained by removing the last part $p_m$ of $p$, then we have that $\tilde{\on{Gr}}_p$ is a $\on{Gr}_{\card p_m}$-bundle over $\tilde{\on{Gr}}_{p'}$.
\end{theorem}

\begin{proof}
 We begin by defining a map $\pi \colon \tilde{G}(\hat{\KK})_{n_1, n_2} \times_{X^{n_2}} \on{Gr}_{n_2} \to \tilde{\on{Gr}}_{n_1, n_2}$, which on $S$-points looks like
 \begin{equation*}
  \pi \colon
    (\vect{x}_1, \vect{x}_2, \sh{T}_1, \phi_1, \psi), (\vect{x}_2, \sh{T}_2, \phi_2)
  \mapsto (\vect{x}, \sh{T}_1, \sh{T}_2', \phi_1, \phi_2'),
 \end{equation*}
 where $\sh{T}_2'$ is the gluing of $\sh{T}_1|(X \setminus \bar{x}_2)$ with $\sh{T}_2|\hat{X}_{\bar{x}_2}$ along the isomorphism $\phi_2 \circ \psi^{-1}$, and $\phi_2'$ is the natural isomorphism with $\sh{T}_1$ away from $\bar{x}_2$.
 
 Evidently its fibers are invariant under the action of $G(\hat{\OO})_{n_2}$, any element $g$ of which sends $\psi$ to $\psi g$ and $\phi_2$ to $\phi_2 g$, so the $g$ cancels in the gluing.  Conversely, given any tuple in the target, we can locally on $S$ choose a trivialization $\psi$ of $\sh{T}_1$ around $\bar{x}_2$ and define $\sh{T}_2$ by gluing $\sh{T}_2'|\hat{X}_{\bar{x}_2}$ to the trivial torsor via $\phi_2' \circ \psi$.  This shows that $\pi$ is locally surjective and that its fibers are $G(\smash{\hat{\OO}})_{n_2}$-torsors.  The final claim about the isomorphisms follows by chasing the constructions.
\end{proof}

\subsection*{Inner convolution diagrams}

An alternative construction of the convolution diagram that we will use is one which perhaps looks more like the convolution diagram usually defined on $\on{Gr}_G$.

\begin{theorem}{defn}{inner convolution diagram}
 The \emph{inner convolution diagrams} are the following ind-schemes over $X^n$, where $m \geq 1$:
 \begin{equation*}
  \on{Conv}_n^m
   = G(\hat{\KK})_n \times^{G(\hat{\OO})_n}
     \dots \times^{G(\hat{\OO})_n}
     G(\hat{\KK})_n \times^{G(\hat{\OO})_n}
     \on{Gr}_n,
 \end{equation*}
 where there are $m$ terms in all.
\end{theorem}

\begin{theorem}{prop}{inner convolution diagram diagonal}
 Let $p$ be the partition of $mn$ into $m$ copies of $n$ (we write $p = m \cdot n$).  Then we have:
 \begin{enumerate}
  \item \label{en:inner diagram diagonal}
  $\on{Conv}^m_n$ is isomorphic to the restriction of $\tilde{\on{Gr}}_p$ to the diagonal $\Delta
  \subset X^{mn}$ where coordinates $x_i = x_{i + m} = \dots$ for $i = 1, \dots, n$.
  
  \item \label{en:inner diagram cover}
  For $m \geq 2$, the restriction of $\tilde{G}(\hat{\KK})_p$ to $\Delta$ is the twisted product
  \begin{equation*}
   \on{Conv} G(\hat{\KK})_n^m =
   G(\hat{\KK})_n \times^{G(\hat{\OO})_n}
   \dots \times^{G(\hat{\OO})_n}
    G(\hat{\KK})_n
  \end{equation*}
  with $m - 1$ terms in all.
 \end{enumerate}
\end{theorem}

\begin{proof}
 By \ref{arc and loop group actions}, we have $\on{Gr}_n = G(\hat{\KK})_n/G(\hat{\OO})_n$, so that $\on{Conv} G(\hat{\KK})_n^m$ is clearly a $G(\hat{\OO})_n$-torsor over $\on{Conv}_n^m$, and $\on{Conv}_n^m$ is clearly the twisting of $\on{Conv} G(\hat{\KK})_n^m$ with $\on{Gr}_n$.  By \ref{convolution product twisting}, therefore, it suffices to show only \ref{en:inner diagram cover}.

 When $m = 2$, this is the claim that $\tilde{G}(\hat{\KK})_{2 \cdot n}|_\Delta = G(\hat{\KK})_n$.  Indeed, the left-hand side is the moduli space of data $(\vect{x}, \sh{T}_1, \phi_1, \psi)$, where $\vect{x} \in X^n(S)$, $\sh{T}_1$ is a $G$-torsor on $X_S$, $\phi_1$ is its trivialization on $X_S \setminus \bar{x}$, and $\psi$ is its trivialization on $\hat{X}_{\bar{x}}$.  By the Beauville--Laszlo theorem, the latter data is equivalent to giving $\psi^{-1} \phi_1 \in G(\hat{X}_{(\bar{x})})$, as claimed.
 
 In general, we claim that
 \begin{equation*}
  \tilde{G}(\hat{\KK})_{m \cdot n}|_\Delta
    = (\tilde{G}(\hat{\KK})_{(m - 1) \cdot n} \times^{G(\hat{\OO})_n} G(\hat{\KK})_n)|_\Delta
 \end{equation*}
 Since this is also true of $\on{Conv} G(\hat{\KK})_n^m$, the proposition would follow by induction. (In fact, we could remove the $\Delta$'s, but will not need to.) The proof is quite similar to that of \ref{convolution product twisting}; for clarity we exemplify it with the case $m = 3$.  We construct a map from the right-hand side to the left-hand side, starting with a pair of tuples
 \begin{align*}
  (\vect{x}, \sh{T}, \phi, \psi) \in \tilde{G}(\hat{\KK})_{2 \cdot n}(S), &&
  (\vect{x}, g) \in G(\hat{\KK})_n(S)
 \end{align*}
 (having the same $\vect{x}$, since we restrict to $\Delta$).  Thus, $\phi$ trivializes $\sh{T}$ over $X_S \setminus \bar{x}$ and $\psi$ trivializes it on $\smash{\hat{X}}_{\bar{x}}$, while $g \in G(\smash{\hat{X}}_{(\bar{x})})$.  We take $\sh{T}_1 = \sh{T}$ and define $\sh{T}_2$ by gluing, using the Beauville--Laszlo theorem:
 \begin{equation*}
  \sh{T}_2 = \sh{T}_1|_{X_S \setminus \bar{x}} \sqcup_{g\psi^{-1}} \sh{T}^0|_{\hat{X}_{\bar{x}}}.
 \end{equation*}
 Then $(\vect{x}, \sh{T}_1, \sh{T}_2, \phi, \text{can}., \text{can.}) \in \tilde{G}(\hat{\KK})_{3 \cdot n}$, where the maps ``$\text{can.}$'' denote, respectively, the canonical isomorphism of $\sh{T}_2$ with $\sh{T}_1$ on $X_S \setminus \bar{x}$ and the canonical trivialization of $\sh{T}_2|\smash{\hat{X}}_{\bar{x}}$.
 
 Conversely, given a tuple
 \begin{equation*}
  (\vect{x}, \sh{T}_1, \sh{T}_2, \phi_1, \phi_2, \psi) \in \tilde{G}(\hat{\KK})_{3 \cdot n}(S),
 \end{equation*}
 locally on $S$ we may choose a trivialization $\psi'$ of $\sh{T}_1$ on $\hat{X}_{\bar{x}}$ and simply define $g = \psi^{-1} \circ \phi_2 \circ \psi'$.  Then $(\vect{x}, \sh{T}_1, \phi_1, \psi') \in \tilde{G}(\hat{\KK})_{2 \cdot n}(S)$, $(\vect{x}, g) \in G(\hat{\KK})_n$, and $\sh{T}_2$ is by definition the gluing of $\sh{T}_1$ and the trivial torsor as above.  This identifies the fibers of the previous map with $G(\smash{\hat{\OO}})_n$, as desired.
\end{proof}

Using these special convolution diagrams, we can express the maps $m_p$ of \ref{eq:convolution diagram mult map} in terms of the actual multiplication maps on $G(\smash{\hat{\KK}})_n$.  The proof follows by comparing \ref*{eq:convolution diagram mult map} with the above construction.

\begin{theorem}{cor}{inner convolution multiplication}
 After \ref{inner convolution diagram diagonal}\ref{en:inner diagram diagonal}, the map $\on{Conv}^m_n \to \on{Gr}_n$, given by multiplying all the $G(\smash{\hat{\KK}})_n$ coordinates and applying the result to the $\on{Gr}_n$ coordinate, is the same as $m_p|_\Delta$, where $p = m \cdot n$ and $\Delta \subset X^{mn}$ as in the proposition. \qed
\end{theorem}

Although we have exhibited the inner convolution diagrams as special cases of the outer ones, in
fact there is a similar relationship in the other order.  In order to connect with \ref{convolution product grassmannian} in full generality, we consider (for once) arbitrary partitions $p$.

\begin{theorem}{prop}{outer convolution diagram trick}
 Let $p$ be any partition of $n$ with $m$ parts.  Then $\tilde{\on{Gr}}_p$ is the subspace of $\on{Conv}_n^m$ whose $k$'th coordinate in $G(\smash{\hat{\KK}})_n$ (using the product representation of \ref{inner convolution diagram}) parametrizes the pairs $(\vect{x}, g)$ such that $g$ extends to $\smash{\hat{X}}_{\bar{x}} \setminus \bar{x}_k$, where $\vect{x} = (\vect{x}_1, \dots, \vect{x}_m)$ is partitioned according to $p$.  In addition, $\tilde{G}(\smash{\hat{\KK}})_p$ is the restriction of $\on{Conv} G(\smash{\hat{\KK}})_n^m$ to this subspace.
\end{theorem}

\begin{proof}
 Following \ref{inner convolution diagram diagonal}\ref{en:inner diagram diagonal}, we may show that $\tilde{\on{Gr}}_p \subset \tilde{\on{Gr}}_{m \cdot n}|_\Delta$ is identified with the subspace consisting of points $(\vect{x}, \{\sh{T}_k\}, \{\phi_k\})$ such that each isomorphism $\phi_k$ extends to $X_S \setminus \bar{x}_k$.  This is the actual definition, however.
 
 Likewise, for $\tilde{G}(\hat{\KK})_p$, we can show that it is identified with the subspace of $\tilde{G}(\hat{\KK})_{m \cdot n}$ consisting of points $(\vect{x}, \{\sh{T}_k\}_{k < m}, \{\phi_k\}_{k < m}, \psi)$ for which the $\phi_k$'s extend as above and for which $\psi$ extends to $X_S \setminus \bar{x}_m$, which is once again the definition.
\end{proof}

\begin{theorem}{cor}{G(O)_n acts on twisted Gr}
 Let $p$ be any partition of $n$.  There is an action of $G(\hat{\OO})_n$ on $\tilde{\on{Gr}}_p$ such that the map $m_p$ is $G(\hat{\OO})_n$-equivariant.
\end{theorem}

\begin{proof}
 $G(\hat{\OO})_n$ acts on $\on{Conv}^m_n$ by multiplying the first coordinate on the left, and by \ref{inner convolution multiplication}, this makes $m_p|_\Delta$ equivariant when $p = m \cdot n$.  For more general $p$, we note that this action preserves $\smash{\tilde{\on{Gr}}}_p \subset \on{Conv}^m_n$, following the description in \ref{outer convolution diagram trick}.
\end{proof}

Having these variants of the same definition will play a key role in the theorem which we will prove in the next section.

\subsection*{Special properties of the torus}
 
Since when $G = T$ is a torus it is commutative, its torsors admit a tensor product and, as previously, this gives a simplification of the factorizable structure of $\on{Gr}_{T,X^n}$ and allied spaces.

\begin{theorem}{lem}{torus convolution diagram}
 For any partition $p$ into $m$ parts, the factorization of the $T(\hat{\OO})_{n_2}$-bundle of \ref{convolution product twisting},
 \begin{equation*}
  \tilde{T}(\hat{\OO})_p \times_{X^{\card p_m}} \on{Gr}_{T, X^{\card p_m}}
   \xrightarrow{\pi} \tilde{\on{Gr}}_p
 \end{equation*}
 descends to a factorization of $\tilde{\on{Gr}}_p$ extending the natural factorzation over $X^n_p$,
 \begin{equation*}
  \tilde{\on{Gr}}_p \cong \tilde{\on{Gr}}_{p'} \times \on{Gr}_{T,X^{\card p_m}}.
 \end{equation*}
 Furthermore, the map $m_p$ then becomes the multiplication map.
\end{theorem}

\begin{proof}
 The action of $T(\hat{\OO})_{\card p_m}$ on $\on{Gr}_{\card p_m}$ is trivial, since $T$ is commutative, so it acts on the above bundle only through the first factor.  Thus, the product decomposition descends to the base, as claimed.
 
 For simplicity, suppose $p = (n_1, n_2)$.  Composing the constructions of \ref{convolution product twisting} and \ref{convolution diagram factorization}, we see that the point $(\vect{x}, \sh{T}, \phi) \in \on{Gr}_{T,X^{n_2}}$ obtained by applying factorization to $\pi(\vect{x}, \sh{T}_1, \phi_1, \psi, \sh{T}_2, \phi_2)$ is the torsor $\sh{T}$ obtained by gluing $\sh{T}_2|\hat{X}_{\bar{x}_2}$ to the trivial torsor on $X \setminus \bar{x}_2$ along the trivialization $\phi_2 \circ (\psi^{-1} \circ \phi_1)$, and $\phi$ is the natural trivialization on $X \setminus \bar{x}_2$.  Since $\psi^{-1} \circ \phi_1 \in T(\hat{\OO})_{n_2}$, which acts trivially on $\on{Gr}_{T,X^{n_2}}$, we have $(\sh{T}, \phi) \cong (\sh{T}_2, \phi_2)$, so that both definitions of factorization agree.
 
 An alternative way of describing the factorization of the convolution product is as follows, for $(\vect{x}, \sh{T}_1, \sh{T}_2, \phi_1, \phi_2) \in \tilde{\on{Gr}}_{n_1, n_2}(S)$ and $(\vect{x}_i, \sh{T}_i', \phi_i') \in \on{Gr}_{n_i}$ ($i = 1, 2$), we have $(\sh{T}_1', \phi_1') = (\sh{T}_1, \phi_1)$ and
 \begin{gather*}
  (\sh{T}_2', \phi_2') = (\sh{T}_1^{-1} \otimes \sh{T}_2, \id \otimes \phi_2) \\
  (\sh{T}_2, \phi_2) = (\sh{T}_1' \otimes \sh{T}_2', \id \otimes \phi_2').
 \end{gather*}
 Since the multiplication map produces $(\sh{T}, \phi) \in \on{Gr}_{n_1 + n_2}(S)$ with $\sh{T} = \sh{T}_1' \otimes \sh{T}_2'$ and $\phi = \phi_1' \otimes \phi_2'$, and $m_p$ gives $\sh{T} = \sh{T}_2$ with $\phi = \phi_2 \circ \phi_1$, the two agree with the understanding that the composition of maps of $T$-torsors and their product are the same.
\end{proof}

We are inspired to use the notation $\tilde{\on{Gr}}{}_p^{\lambda_1, \dots, \lambda_n}$ to refer to the twisted product of the individual components of $\on{Gr}_n$ indexed by the $\lambda$'s in this order.  We use the same indexing for $\tilde{G}(\hat{\KK})_p$ as for $\tilde{\on{Gr}}_{p'}$ (that is, the ``$G(\hat{\OO})_{\card p_m}$ part'' does not contribute indices).

This lemma is interesting on account of the non-triviality of the torsor $\tilde{T}(\hat{\OO})_p$ even though $T$ is a torus.  In fact, via that space we can give another description of the bilinear form of \ref{quadratic and bilinear forms}.  To do so, we indulge in some generalities; here, the notion of multiplicative torsor is a simplification of that for gerbes.  

\begin{theorem}{lem}{torsors on the torus}
 Let $A$ be an abelian group. An $A$-torsor $\sh{T}$ on $T$ is identified with a homomorphism $\map{\phi}{\Lambda_T}{A}$; all such torsors are multiplicative.  For any $\lambda \in \Lambda_T = \on{Hom}(\Gm, T)$, we have $\lambda^* \sh{T} = \sh{L}_{\phi(\lambda)}$, where for any $a \in A$, $\sh{L}_a$ is, as before \ref{line bundle order}, the local system of rank $1$ on $\Gm$ with monodromy $a$.
\end{theorem}

\begin{proof}
 In general, $A$-torsors on a space are identified with characters of the fundamental group by the above prescription; we identify $\Lambda_T$ with $\pi_1(T, 1)$.  The multiplicativity follows from the fact that all the classes in $\pi_1(T,1)$ are homomorphisms $\Gm \to T$.
\end{proof}

We now give a generalization of the construction of \ref{line bundle order}; as for multiplicative torsors, the notion of twisted equivariance is similar to that for gerbes. Let $\sh{T}$ be a $T$-torsor on a space $Y$, and let $E$ be its total space. Then for any homomorphism $\map{\phi}{\Lambda_T}{A}$, denoting $\sh{L}_\phi$ the associated multiplicative $A$-torsor on $T$, we define $\sh{T}^{\log \phi}$ to be the gerbe whose sections over $U \subset Y$ are the $\sh{L}_\phi$-twisted $T$-equivariant $A$-torsors on $E|_U$. Since $E$ is locally trivial, we see that this is indeed an $A$-gerbe, as before.

Let $D \subset Y$ be a smooth divisor; we will say that a trivialization $\sh{T}^0 \to \sh{T}|_{Y \setminus D}$ has \emph{order $\lambda$} for $\lambda \in \Lambda_T$ if it induces an isomorphism $\sh{T} \cong \OO(\lambda D)$, in the notation of \ref{torus grassmannian properties}\ref{en:torus grassmannian 1}.  The following proposition generalizes \ref{line bundle order} and has, essentially, the same proof (the coweight $\lambda$ enters the computation as the change of trivialization, in place of the parameter $z$ used there).

\begin{theorem}{prop}{order of torus twisting gerbe}
 Suppose we have a trivialization $u$ of order $\lambda$ as above; then the induced trivialization of $\sh{T}^{\log \phi}$ has order $\phi(\lambda)$ in the sense of \ref{2-line bundle of a Cartier divisor}. \qed
\end{theorem}

Now we can characterize a familiar gerbe using the trivializations $u$ and $v$ of \ref{twisted arc group trivializations}.  We consider the not-quite $T$-torsors $\tilde{T}(\hat{\OO})_p$, where the base space is $\tilde{\on{Gr}}_{p'}$ and $\card p_m = 1$.  These are $T(\smash{\hat{\OO}})_1$-torsors, but we have a natural map $T(\smash{\hat{\OO}})_1 \to T \times X$, where for $g \in T(\smash{\hat{X}}_{\bar{x}})$ the $S$-point $(x, g)$ goes to $(x, g|_S)$, with $S \cong \bar{x} \subset \smash{\hat{X}}_{\bar{x}}$ identified with its graph.  We will refer to the $T$-torsor on $\smash{\tilde{\on{Gr}}}_{p'} \times X$ associated with $\smash{\tilde{T}(\hat{\OO})}_p$ as $\sh{T}_p$.

For simplicity, we will consider $p = (n - 1, 1)$ for some integer $n$.  Then for weights $\lambda_1, \dots, \lambda_{n - 1} \in \Lambda_T$, by \ref{torus grassmannian properties} we identify $\smash{\on{Gr}_{n - 1}^{\lambda_1, \dots, \lambda_{n - 1}}} \cong X^{n - 1}$, and thus $\sh{T}_{n - 1, 1}$ with a $T$-torsor on $X^n$.  We call $\map{\pi_{n - 1, 1}}{\smash{\tilde{T}(\hat{\OO})}_{n - 1, 1}}{X^n}$ the structure map.

\begin{theorem}{prop}{twisting gerbe for bilinear form}
 Let $\map{\kappa}{\Lambda_T \otimes \Lambda_T}{A}$ be a bilinear form, and for any $\mu \in \Lambda_T$, let $\kappa_\mu(\farg) = \kappa(\farg, \mu)$.  Then the trivialization $u$ on the complement of the divisors $\Delta_{in}$ for $i < n$ makes $\sh{T}_{n - 1, 1}^{\log \kappa_\mu} \cong \bigotimes_i \OO(\Delta_{in})^{\log \kappa(\lambda_i, \mu)}$.
\end{theorem}

\begin{proof}
 By \ref{torus grassmannian properties}, the trivialization $u$ of $\sh{T}_{n - 1, 1}$ on $X^n \setminus \bigcup \Delta_{in}$ (from \ref{twisted arc group trivializations}) has order $\lambda_i$ about $\Delta_{in}$.  Then by \ref{order of torus twisting gerbe}, we find that $\sh{T}_{n - 1, 1}^{\log \kappa_\mu}$ has order $\kappa_\mu(\lambda_i)$ about $\Delta_{in}$, as desired.
\end{proof}

Returning to the generalities, we note the following fact:

\begin{theorem}{lem}{trivialization for twisting gerbe}
 Let $\map{\pi}{E}{Y}$ be the structure map; then $\pi^* \sh{T}^{\log\phi}$ is trivial.
\end{theorem}

\begin{proof}
 Sections of $\pi^* \sh{T}^{\log\phi}$ are identified with $\sh{L}_\phi$-twisted equivariant torsors on the $T$-bundle $E \times_Y E$ over $E$.  Since $E$ is a $T$-torsor, this is the same as $T \times E$; then $\on{pr}_T^* \sh{L}_\phi$ itself is such a torsor.
\end{proof}

In our particular situation, this implies:

\begin{theorem}{cor}{twisting gerbe arc group trivialization}
 For any $\map{\phi}{\Lambda_T}{A}$, $\pi_{n - 1, 1}^* \sh{T}_{n - 1, 1}^{\log\phi}$ is trivial.
\end{theorem}

\begin{proof}
 There is a natural map over $X^n$, $\tilde{T}(\hat{\OO})_{n - 1, 1} \to \sh{T}_{n - 1, 1}$, and so this follows from \ref{trivialization for twisting gerbe} by pulling back.
\end{proof}

\section{Factorizable gerbes are naturally equivariant}
\label{s:factorizable gerbes are equvariant}  

The main theorem of this section is that factorizable gerbes automatically enjoy an essentially unique structure relating to the infinitesimal actions on the grassmannian.  To describe this structure, we begin with a sequence of definitions.

\begin{theorem}{defn}{gerbe factorizable equivariance}
 Let $\stack{G}_n$ be an sf gerbe.  The structure of \emph{factorizable equivariance} for $G(\smash{\hat{\OO}})_n$ is an equivariance structure on each individual gerbe together with compatibilities of these structures with all the factorization data.  Equivalently, it is a structure of equivariance for the action of the sf group scheme $G(\smash{\hat{\OO}})_n$ on $\on{Gr}_n$.
\end{theorem}

We maintain the expository convention of considering only partitions $p = (n_1, n_2)$ of a number $n$ or $p = m \cdot n$ of the product $mn$.

\begin{theorem}{defn}{twisted product gerbe}
 Let $\stack{G}_n$ be a factorizably equivariant sf gerbe.  The \emph{twisted product} on any $\tilde{\on{Gr}}_{n_1, n_2}$ is defined by the formula
 \begin{equation*}
  \stack{G}_{n_1} \tbtimes \stack{G}_{n_2}
    = \on{pr}_{n_1, n_2}^* \stack{G}_{n_1} \otimes \tilde{\stack{G}}_{n_2},
 \end{equation*}
 where $\on{pr}_{n_1, n_2}$ is, as after \ref{eq:convolution diagram mult map}, the projection $\tilde{\on{Gr}}_{n_1, n_2} \to \on{Gr}_{n_1}$, and the twisted pullback refers to the $G(\OO)_{n_2}$-torsor $\smash{\tilde{G}(\hat{\KK})}_p$ as in \ref{twisted pullback of gerbe}.  By similar measures, we have a twisted product on $\on{Conv}_n^m$, and it is easy to see that they agree under the identification of \ref{outer convolution diagram trick}.  Using this twisted product, or using the generalization of \ref{convolution product twisting} to arbitrary partitions, we can define twisted products for any partition, and either route will give equivalent results for the same reason as above.
\end{theorem}

\begin{theorem}{lem}{twisted product is equivariant}
 With the setup of \ref{twisted product gerbe}, if $p$ is a partition of $n$, then the twisted product is equivariant for $G(\smash{\hat{\OO}})_n$, relative to the action of this group given in \ref{G(O)_n acts on twisted Gr}.
\end{theorem}

\begin{proof}
 By \ref{outer convolution diagram trick} it suffices to verify this for twisted products on $\on{Conv}_n^m$.  The action of $G(\hat{\OO})_n$ on that space is simply on the first factor, whereas $\tilde{\stack{G}}_n$ is the descent of a pullback from the last factor, so unless $m = 1$ it is trivially $G(\hat{\OO})_n$-equivariant; when $m = 1$ the whole construction is trivial.
\end{proof}

Factorizable equivariance is not enough for our purposes because it has a loose end.  To see this,
consider the map $\map{m_{n_1, n_2}}{\tilde{\on{Gr}}_{n_1, n_2}}{\on{Gr}_n}$, which by \ref{convolution diagram factorization}, is naturally identified with the factorization isomorphism $\on{Gr}_n \cong \on{Gr}_{n_1} \times \on{Gr}_{n_2}$ on the open set $X^n_{n_1, n_2}$.  By \ref{convolution product twisting} the projection $\map{\on{pr}_{n_1, n_2}}{\tilde{\on{Gr}}_{n_1, n_2}}{\tilde{\on{Gr}}_{n_1}}$ is a trivial bundle on $X^n_{n_1, n_2}$ (more importantly, by \ref{twisted arc group trivializations}, $\tilde{G}(\hat{\KK})_{n_1, n_2}$ is the trivial $G(\hat{\OO})_{n_2}$-torsor over this set), so that on the open set $X^n_{n_1, n_2}$ we have a natural equivalence
\begin{equation}
 \label{eq:strong factorizability open}
 m_{n_1, n_2}^* \stack{G}_n|_{X^n_{n_1, n_2}}
 \cong
 (\stack{G}_{n_1} \tbtimes \stack{G}_{n_2})|_{X^n_{n_1, n_2}}.
\end{equation}
Similar considerations for more general $p$ establish such an equivalence with higher iterated twisted products.

\begin{theorem}{defn}{gerbe strong factorizable equivariance}
 If $\stack{G}_n$ has a structure of factorizable equivariance, this structure is said to be \emph{strong} if, for every $n$ and every partition $p$, the equivalence of \ref{eq:strong factorizability open} extends from the open set in $\tilde{\on{Gr}}_p$ over $X^n_p$ to all of $\tilde{\on{Gr}}_p$.  (We will show that this is indeed merely a condition rather than additional data.)
\end{theorem}

There remains an additional compatibility to be imposed, which can be seen when considering \ref{multiplicative gerbe descent}.  Recall the quotient maps $q_n \colon G(\hat{\KK})_n \to \on{Gr}_n$ and suppose that the $q_n^* \stack{G}_n$ were given a multiplicative structure for $G(\hat{\KK})_n$ together with a trivialization of that structure for $G(\hat{\OO})_n$.  Then we have a $G(\hat{\OO})_n$-equivariance structure on $\stack{G}_n$.  By \ref{inner convolution multiplication}, we see that we have
\begin{equation*}
 m^* \stack{G}_n
 \cong
 \stack{G}_n \tbtimes \dots \tbtimes \stack{G}_n
\end{equation*}
on $\on{Conv}_n^m$, where the $m$ in $m^*$ is the map $\on{Conv}_n^m \to \on{Gr}_n$.  By \ref{outer convolution diagram trick}, we see that the $\stack{G}_n$ are strongly factorizably equivariant. From the trivialization of $q_n^* \stack{G}_n$ as a multiplicative gerbe on $G(\hat{\OO})_n$, we deduce that the equivariance structure is trivialized on the unit section of $\on{Gr}_{G,X^n}$, a condition we have not yet considered:

\begin{theorem}{defn}{gerbe unital factorizable equivariance}
 The structure of factorizable equivariance on an sf gerbe $\stack{G}_n$ is said to be \emph{unital} if $\stack{G}_n$ is sf equivariantly trivial on the unit sections on $\on{Gr}_{G,X^n}$ (that is, its restriction to the trivial grassmannian $\on{Gr}_{T,X^n}^{0, \dots, 0}$ is trivial as a factorizably equivariant sf gerbe).
\end{theorem}

Having introduced the necessary concepts, the first consequence we will deduce is:

\begin{theorem}{thm}{equivariance and multiplicativity}
 The following structures on an sf $A$-gerbe $\stack{G}_n$ on $\on{Gr}_{G,X^n}$ are equivalent:
 \begin{enumerate}
  \item
  A unital strongly factorizable equivariance for the action of $G(\hat{\OO})_n$;
  
  \item
  A factorizable multiplicative structure for the $q_n^* \stack{G}_n$ which is trivialized on $G(\hat{\OO})_n$.
 \end{enumerate}
\end{theorem}

\begin{proof}
 We have already argued that the second point implies the first.  To see the converse, we take $p = m \cdot n$ and consider the twisted product $\stack{G}_n \tbtimes \cdots \tbtimes \stack{G}_n$, on $\on{Conv}_n^m$. It follows from the definition that when pulled back to $G(\smash{\hat{\KK}})_n \times_{X^n} \dots \times_{X^n} G(\smash{\hat{\KK}})_n$, this gerbe becomes $(q_n^* \stack{G}_n) \boxtimes \dots \boxtimes (q_n^* \stack{G}_n)$.  If the hypothesis of strong factorizability is satisfied, then we deduce from \ref{inner convolution multiplication} that there is an equivalence (here $m$ represents the multiplication map)
 \begin{equation}
  \label{eq:multiplicative structure}
  m^* (q_n^*\stack{G}_n) \cong (q_n^* \stack{G}_n) \boxtimes \dots \boxtimes (q_n^* \stack{G}_n).
 \end{equation}
 It is clear from the expression \ref{inner convolution diagram diagonal}\ref{en:inner diagram cover} and the associativity of tensor products that these isomorphisms are associative among themselves.  This forms an associative multiplication on the $q_n^* \stack{G}_n$; the unit comes from that of the equivariance structure on $\stack{G}_n$.
\end{proof}

We now turn to investigating the existence of strongly factorizable equivariance structures on sf gerbes.  As usual, the first case to consider is that of a torus.

\begin{theorem}{prop}{strong equivariance for a torus}
 When $G = T$ is a torus, an sf $A$-gerbe $\stack{G}_n = \stack{T}_n$ on $\on{Gr}_{T,X^n}$ has a unique structure of unital strongly factorizable equivariance for $T(\hat{\OO})_n$.
\end{theorem}

\begin{proof}
 By \ref{torus convolution diagram} and \ref{quadratic and bilinear forms} we have for $p = (n - 1, 1)$:
 \begin{equation}
  \label{eq:torus strong equivariance}
  m_{n - 1, 1}^* \stack{T}_n^{\lambda_1, \dots, \lambda_{n - 1}, \mu}
  \cong
  \stack{T}_{n - 1}^{\lambda_1, \dots, \lambda_{n - 1}} \boxtimes \stack{T}_1^\mu
  \otimes \bigotimes_i \OO(\Delta_{in})^{\log \kappa(\lambda_i, \mu)}.
 \end{equation}
 where $\kappa$ is, as usual, the bilinear form defined by the sf gerbe.  We must therefore show that the pullback of the latter product to $\tilde{T}(\hat{\OO})_{n - 1, 1}^{\lambda_1, \dots, \lambda_{n - 1}} \times_X \on{Gr}_1^\mu$ is the same as $\stack{T}_{n - 1}^{\lambda_1, \dots, \lambda_{n - 1}} \boxtimes \stack{T}_1^\mu$.
 
 To do this, it is necessary and sufficient only trivialize the second factor on pulling back, since the first factor pulls back to exactly the desired gerbe after using \ref{torus convolution diagram} to identify the factorizations of the convolution product and its cover.  The second factor is indeed trivialized after pulling back: first by \ref{twisting gerbe for bilinear form}, identifying the product with $\stack{T}_{n - 1, 1}^{\log \kappa_\mu}$, and then by \ref{twisting gerbe arc group trivialization}, using the fact that
 \begin{equation*}
  \tilde{T}(\hat{\OO})_{n - 1, 1}^{\lambda_1, \dots, \lambda_{n - 1}} \times_X \on{Gr}_1^\mu
   \cong \tilde{T}(\hat{\OO})_{n - 1, 1}^{\lambda_1, \dots, \lambda_{n - 1}}
 \end{equation*}
 since $\on{Gr}_1^\mu \cong X$.
 
 The descent data of $\stack{T}_{n - 1}^{\lambda_1, \dots, \lambda_{n - 1}} \boxtimes \stack{T}_1^\mu$  is the same as $T(\smash{\hat{\OO}})_1$-equivariance for $\stack{T}_1^\mu$ (and, by agglomeration, all of $\stack{T}_1$); despite $T(\hat{\OO})_1$ acting trivially, this is not the trivial equivariance structure.  Indeed, examining the proof of \ref{trivialization for twisting gerbe}, we see that in fact as we have defined it, the equivariance structure on $\stack{T}_1^\mu$, i.e.\ the equivalence between its pullbacks along the maps
 \begin{equation*}
  T(\hat{\OO})_1 \times_X \on{Gr}_{T,X}^\mu
   \xrightarrow[\on{pr}]{a} \on{Gr}_{T,X}^\mu,
 \end{equation*}
 with $a = \on{pr}$, and therefore an autoequivalence of $\on{pr}^* \stack{T}_1^\mu$, is given by the $A$-torsor $\sh{L}_{\kappa_\mu}$ pulled back to the first cartesian factor.  This follows from \emph{any} of the constructions for particular $\lambda_i$ and is thus, in particular, independent of their choice.
 
 Therefore, the right-hand side of \ref{eq:torus strong equivariance} fulfills the role of $\stack{T}_{n - 1}^{\lambda_1, \dots, \lambda_{n - 1}} \tbtimes \stack{T}_1^\mu$.  Then by \ref{twisted product is equivariant} and induction on \ref{eq:torus strong equivariance} $\stack{T}_n = m_{n - 1, 1}^* \stack{T}_n$ inherits a $T(\smash{\hat{\OO}})_n$-equivariance structure as a twisted product, which by its definition is strongly factorizable; we have just argued that this structure is unique for $n = 1$, and for higher $n$ the construction we have given is the only one possible, hence again unique.  To see that it is unital, it suffices to show that $\stack{T}_n$ is sf-trivial on the components $\on{Gr}_n^{0, \dots, 0}$.  These components form the grassmannian $\on{Gr}_{1, X^n}$ of the trivial torus, and by \ref{torus exact sequence} it indeed has only the trivial sf gerbe on it.
\end{proof}

The details of the above proof do not generalize to groups other than a torus both because \ref{twisting gerbe arc group trivialization} fails for a non-torus and because we no longer have $\tilde{\on{Gr}}_{n_1, n_2} \cong \on{Gr}_{n_1} \times \on{Gr}_{n_2}$, so we proceed first by extending the equivariance structure and then showing separately that it is strongly factorizable.

\begin{theorem}{prop}{equivariance torus to borel}
 Let $\stack{G}_n$ be the part of an sf $A$-gerbe on $\on{Gr}_{G,X^n}$, and let $B$ be a fixed Borel subgroup of $G$ with semisimple quotient $T$; then the $T(\hat{\OO})_n$-equivariance structure of $\stack{T}_n$ on $\on{Gr}_{T,X^n}$ extends uniquely to a $B(\smash{\hat{\OO}})_n$-equivariance structure for $\stack{G}_n$.
\end{theorem}

\begin{proof}
 The map $t \colon \on{Gr}_{B,X^n} \to \on{Gr}_{T,X^n}$ is biequivariant for the actions of $B(\smash{\hat{\OO}})_n$ and $T(\smash{\hat{\OO}})_n$ respectively (along the natural map $B(\smash{\hat{\OO}})_n \to T(\smash{\hat{\OO}})_n$), so $\stack{B}_n = t^* \stack{T}_n$ gets a $B(\smash{\hat{\OO}})_n$-equivariance structure by \ref{pullback of equivariant objects}.  In fact, since $N(\smash{\hat{\OO}})$ is cohomologically trivial, this structure is unique.
 
 We have thus shown that there is a unique equivalence between the pullbacks of $\stack{B}_n$, the restriction of $\stack{G}_n$ to $\on{Gr}_{B,X^n}$, along the two maps
 \begin{equation}
  \label{eq:borel equivariance diagram}
  B(\hat{\OO})_n \times_{X^n} \on{Gr}_{B,X^n} \xrightarrow[\on{pr}]{a} \on{Gr}_{B,X^n}
 \end{equation}
 satisfying the requirements of \ref{equivariant gerbe}.  Now we consider the diagram
 \begin{equation*}
  B(\hat{\OO})_n \times_{X^n} \on{Gr}_{G,X^n} \xrightarrow[\on{pr}]{a} \on{Gr}_{G,X^n}
 \end{equation*}
 and the pullbacks of $\stack{G}_n$ itself.  According to \ref{borel components}, the components of $\on{Gr}_{B,X^n}$ stratify $\on{Gr}_{G,X^n}$, and by \ref{eq:borel equivariance diagram} we have isomorphisms $a^* \stack{G}_n \cong \on{pr}^* \stack{G}_n$ on each one of them.  First fix $\lambda_1, \dots, \lambda_n$ and for $\mu_1, \dots, \mu_n$ with $\mu_i = \lambda_i - \check\alpha_j$ for any simple coroots $\check\alpha_j$, consider the complement of
 \begin{equation*}
  \bar{\on{Gr}}_{B,X^n}^{\mu_1, \dots, \mu_n} \subset
  \bar{\on{Gr}}_{B,X^n}^{\lambda_1, \dots, \lambda_n}.
 \end{equation*}
 It is open and its boundary is a Cartier divisor, so the equivalence of \ref{eq:borel equivariance diagram} has some order along it.  This order is seen to be 1 by considering the transverse section 
 \begin{equation*}
  1_{X^n} \times_{X^n} \on{Gr}_{G,X^n} \subset B(\hat{\OO})_{X^n} \times_{X^n} \on{Gr}_{G,X^n},
 \end{equation*}
 where both pullbacks are canonically identified because the identity section acts trivially and $\stack{B}_n$ has the trivial equivariance structure for this trivial action by the identity axiom of \ref{equivariant gerbe}.  Thus, we have $a^* \stack{G}_n \cong \on{pr}^* \stack{G}_n$ on $\bar{\on{Gr}}_{B,X^n}^{\lambda_1, \dots, \lambda_n}$ for any choice of coweights $\lambda_i$, and that these isomorphisms glue for the inclusions of these strata as the $\lambda_i$ increase.  Thus, $\stack{G}_n$ is $B(\hat{\OO})_n$-equivariant, as desired.
\end{proof}

\begin{theorem}{prop}{equivariance borel to group}
 Let $\stack{G}_n$ be as in \ref{equivariance torus to borel} with its $B(\hat{\OO})_n$-equivariance structure as constructed there.  Then this structure extends uniquely to that of $G(\smash{\hat{\OO}})_n$-equivariance.
\end{theorem}

\begin{proof}
 The choice of Borel is arbitrary, so we also have a $B^\op(\hat{\OO})_n$-equivariance structure, where $B^\op$ is the opposite Borel obtained by swapping the positive and negative roots.  We have $B \cap B^\op = T$, and the ``big cell'' product $B^\op \cdot B = \on{BC}$ is dense and open in $G$.  The same hold of $B(\hat{\OO})_n$ as well.  In the diagram
 \begin{equation*}
  B^\op(\hat{\OO})_n \times_{X^n} B(\hat{\OO})_n \times_{X^n} \on{Gr}_{G,X^n}
   \xrightarrow[\on{pr}]{a} \on{Gr}_{G,X^n},
 \end{equation*}
 where $a$ is obtained by the successive action of the two factors, both pullbacks of $\stack{G}_n$ are identified by their equivariance structures for the two Borel subgroups.  The product of groups covers $\on{BC}(\smash{\hat{\OO}})_n$ as a $T(\smash{\hat{\OO}})_n$-torsor and so this identification descends to $\on{BC}(\smash{\hat{\OO}})_n \times_{X^n} \on{Gr}_{G,X^n}$ via the $T(\hat{\OO})_n$-equivariance structure of $\stack{G}_n$.
 
 Thus, the equivalence $a^* \stack{G}_n \cong \on{pr}^* \stack{G}_n$ for the diagram
 \begin{equation*}
  G(\hat{\OO})_n \times_{X^n} \on{Gr}_{G,X^n} \xrightarrow[\on{pr}]{a} \on{Gr}_{G,X^n}
 \end{equation*}
 holds over a dense open subset of $G(\hat{\OO})_n$ and so has some order along the closed complement (which may indeed have codimension 1, as the example of $G = \on{GL}_2$ shows).  To see that it is trivial we need only show this on some transverse section, for example, the section
 \begin{equation*}
  G(\hat{\OO})_n \times_{X^n} \on{Gr}_{G,X^n}^{0, \dots, 0},
 \end{equation*}
 in the notation of the upcoming \ref{relative orbit properties}, since $\on{Gr}_{G,X^n}^{0,\dots, 0}$, being equal to $G(\hat{\OO})_n / G(\hat{\OO})_n \cong X^n$, has a trivial action of $G(\smash{\hat{\OO}})_n$ and on it, $\stack{G}_n = \stack{B}_n = \stack{T}_n$ has the trivial equivariance structure, by \ref{strong equivariance for a torus}.
 
 Finally, we must show that the above (now globally) defined equivalence admits unique identity and associativity constraints making it an equivariance structure.  The identity constraint is easy, as the restriction of $\phi = (a^* \stack{G}_n \cong \on{pr}^* \stack{G}_n)$ to the identity section of $G(\smash{\hat{\OO}})_n \times_{X^n} \on{Gr}_n$ is the same as the restriction of the original $B(\smash{\hat{\OO}})_n$-equivariance structure, which has such a constraint.  For associativity, it is true \textit{a priori} that the two sides of \ref{eq:equivariant gerbe} differ by some $A$-torsor $\sh{T}$ on $G(\smash{\hat{\OO}})_n \times_{X^n} G(\smash{\hat{\OO}})_n \times_{X^n} \on{Gr}_n$.  Since $\on{Gr}_n$ is simply-connected, $\sh{T}$ descends to the product of groups, and therefore we can check that it is trivial by doing so when the last coordinate is fixed in $\on{Gr}_{G,X^n}^{0,\dots,0}$, on which $\phi$ uniquely extends the trivial equivariance structure for $\smash{B(\hat{\OO})_n}$, and is therefore 
itself trivial.  We must also check that this trivialization of $\sh{T}$ satisfies the higher associativity condition, which the same.
\end{proof}

\begin{theorem}{thm}{existence of equivariance}
 Any sf gerbe admits a unique unital strongly factorizable equivariance structure.
\end{theorem}

\begin{proof}
 After the preceding propositions, it remains only to show that the $G(\hat{\OO})_n$-equivariance structures already constructed are strongly factorizable.  We consider partitions $p = (n_1, n_2)$ for simplicity.  Since the twisted product $\stack{G}_{n_1} \tbtimes \stack{G}_{n_2}$ exists on $\smash{\tilde{\on{Gr}}}_{n_1, n_2}$, and by factorizability and \ref{convolution diagram factorization} is identified with $m_{n_1, n_2}^* \stack{G}_{n_1 + n_2}$ away from $X^n_{n_1, n_2}$, this identification has some order along $X^n_{n_1, n_2}$ which we must show is trivial.  But as the $G(\smash{\hat{\OO}})_{n_i}$-equivariance structures on the $\stack{G}_{n_i}$ extend the $T(\smash{\hat{\OO}})_{n_i}$-equivariance structures on the $\stack{T}_{n_i}$, the twisted product restricted to $\tilde{\on{Gr}}_{T; n_1, n_2}$ (in the obvious notation) is the same when using either equivariance and the claim follows from \ref{strong equivariance for a torus}. 
\end{proof}

\chapter{Relative twisted geometric Satake equivalence}
\label{c:relative twisted geometric Satake}

In this chapter, we consider the following setup: $G$ is a connected, reductive algebraic group over $\C$, $k$ is any field of characteristic zero, and $\{\stack{G}_n\}$ is an sf $k^*$-gerbe.  We consider only the unique strongly factorizable unital $G(\smash{\hat{\OO}})_n$-equivariance of \ref{existence of equivariance} on $\stack{G}_n$; using this, we construct the sheaf of categories of $G(\smash{\hat{\OO}})_n$-equivariant perverse sheaves of $k$-vector spaces on $\on{Gr}_{G,X^n}$ twisted by $\stack{G}_n$.  Our main theorem, \ref{main theorem}, describes this category in terms of the action of a certain ``dual group'' $\smash{\check{G}}_Q$ (defined over $k$ and depending only on the quadratic form $Q = Q(\stack{G}_n)$) on the perverse sheaves on $X^n$, twisted by another gerbe $\stack{Z}_n$ over $\on{Fact}(Z(\check{G}_Q)(k))_n$, which was defined in \ref{factorizable version of a group} and (almost) appeared in \ref{multiplicative factorizable lattice gerbes}.

We will use the following notation: $\Perv_S$ will be the abelian category of perverse sheaves on the scheme $S$, and if $S$ has a gerbe $\stack{G}$, then $\Perv(\stack{G})_S$ will be the category of $\stack{G}$-twisted perverse sheaves.  This notation is not specific to the affine grassmannian. However, we will write $\Sph(\stack{G}_n)$ for the ``spherical'', or $G(\hat{\OO})_n$-equivariant objects in $\Perv(\stack{G}_n)_{\on{Gr}_n}$.

\section{Orbits in the affine grassmannian}
\label{s:orbits in the affine grassmannian}

The last chapter of our study of the affine grassmannian concerns the orbits of the action of $G(\hat{\OO})_n$ on it.  In order to produce perverse sheaves on these orbits we will need precise statements of their dimensions and the dimensions of related spaces.

\subsection*{Dimension of the orbits}

We recall the \emph{absolute grassmannian} 
\begin{equation*}
 \on{Gr}_G = G\lp t \rp/G\lb t \rb \cong \on{Gr}_{G,X}|_x,
\end{equation*}
where $x$ is any closed point of $X$ and $t$ is any choice of local coordinate there.  We write $\OO = \C\lb t \rb$ and $\KK = \C \lp t \rp$, so $\on{Gr}_G \cong G(\smash{\hat{\KK}})/G(\smash{\hat{\OO}})$; we will use the notation
\begin{equation*}
 \map{q}{G(\hat{\KK})}{\on{Gr}_G}
\end{equation*}
for the projection onto this quotient.

Let $B$ be any Borel subgroup of $G$ and $T$ a maximal torus in $B$. Then $\on{Gr}_T \cong \Lambda_T$ by \ref{torus grassmannian properties}, while the components of $\on{Gr}_B$ in $\on{Gr}_G$, indexed by $\Lambda_T$, are the \emph{semi-infinite orbits}.  We write suggestively $t^\lambda$ for $\on{Gr}_T^\lambda$, which is a single point identified with a coweight $\lambda \in \Lambda_T$, and sometimes identify it with an element of $G(\smash{\hat{\KK}})$.  As is customary in representation theory, we write
\begin{equation*}
 2\rho = \sum \alpha,
\end{equation*}
the sum of all the positive roots of $G$, and use the partial ordering on weights or coweights in which positive elements are sums of positive roots.

\begin{theorem}{prop}{absolute orbit properties}
 (\cite{MV_Satake})
 The orbits of $G(\hat{\OO})$ in $\on{Gr}_G$ correspond bijectively to $W$-orbits in $\Lambda_T$ according to the $t^\lambda$ they contain; equivalently, to the dominant coweights. For each dominant coweight $\lambda$, we have the closure relation
 \begin{equation*}
  \bar{\on{Gr}}_G^\lambda = \bigcup_{\mu \leq \lambda} \on{Gr}_G^\mu
 \end{equation*}
 where the union runs over dominant $\mu$, and for any coweight $w_0(\lambda) \leq \mu \leq \lambda$, ($w_0$ being the longest element of $W$) the following intersection has pure dimension:
 \begin{equation*}
  \dim(\on{Gr}_G^\lambda \cap \on{Gr}_B^\mu)
   = \langle \rho, \lambda + \mu \rangle,
 \end{equation*}
 Consequently, $\dim \bar{\on{Gr}}_G^\lambda = \langle 2\rho, \lambda \rangle$; in addition when $\mu = w_0(\lambda)$, the intersection is the single point $\on{Gr}_T^\mu$.  In particular, $\on{Gr}_G$ is stratified by finite-dimensional $G(\hat{\OO})$-orbits. \qed
\end{theorem}

Connected with the $N(\hat{\KK})$-orbits is the following finiteness lemma for $N(\hat{\KK})$
itself.

\begin{theorem}{lem}{unipotent group finiteness}
 For any weight $\lambda$, let $N(\smash{\hat{\KK}})_\lambda = t^{-\lambda} N(\hat{\OO}) t^\lambda$; then when $\mu - \lambda$ is dominant we have $N(\smash{\hat{\KK}})_\lambda \subset N(\smash{\hat{\KK}})_\mu$, and $N(\smash{\hat{\KK}}) = \bigcup_\lambda N(\smash{\hat{\KK}})_\lambda$, the union taken over all dominant weights $\lambda$.  If $Y \subset \on{Gr}_B^0$ is any finite-dimensional subvariety, then $q^{-1}(Y) \subset N(\smash{\hat{\KK}})_\lambda$ for some dominant $\lambda$.
\end{theorem}

\begin{proof}
 First note that when $\lambda$ is \emph{antidominant}, we have $N(\hat{\KK})_\lambda \subset N(\hat{\OO})$; indeed, if $n \in N(\smash{\hat{\OO}})$, then for any highest-weight representation $V^\mu$ with a basis of weight vectors $e^\nu$ (not distinguishing between those of the same weight), we have
 \begin{equation*}
  n \cdot e^\nu = e^\nu + \sum_{\pi > \nu} n_{\nu, \pi} e^\pi,
  \quad
  \on{ord}_t n_{\nu, \pi} \geq 0.
 \end{equation*}
 Thus, we have
 \begin{equation*}
  t^{-\lambda} n t^\lambda \cdot e^\nu
   = e^\nu + \sum_{\pi > \nu} t^{\langle \nu - \pi, \lambda \rangle} n_{\nu, \pi} e^\pi
 \end{equation*}
 where $\langle \nu - \pi, \lambda \rangle \geq 0$ since $\lambda$ is antidominant.  It follows, then, that if $\mu - \lambda$ is dominant then we have
 \begin{equation*}
  N(\hat{\KK})_\lambda
   = t^{-\mu} N(\hat{\KK})_{\lambda - \mu} t^\mu
   \subset t^{-\mu} N(\hat{\OO}) t^\mu
   = N(\hat{\KK})_\mu,
 \end{equation*}
 as claimed.

 As for the union, it suffices to establish it for the matrices of $N(\hat{\KK})$ acting on any faithful representation of $G$.  Then we need only choose $\lambda$ large enough that $\langle \pi - \nu, \lambda \rangle + \on{ord}_t n_{\nu, \pi} \geq 0$ for all weights $\nu, \pi$ of this representation, which since these differences are sums of positive roots we may accomplish for sufficiently large dominant coweights $\lambda$.

 Finally, if $Y \subset \on{Gr}_G$ is finite-dimensional then it intersects only finitely many $\on{Gr}_G^\lambda$, but clearly the $q(N(\smash{\hat{\KK}})_\lambda)$ intersect every one of these orbits. Since $N(\smash{\hat{\KK}})$ is exhausted by the $N(\smash{\hat{\KK}})_\lambda$, we must have $q^{-1}(Y)$ contained in one of them.
\end{proof}

When the orbits are constructed as quotients of $G(\hat{\OO})$, their structure is simple:

\begin{theorem}{lem}{orbit stabilizer is connected}
 The stabilizer in $G(\hat{\OO})$ of $t^\lambda$ is connected and contains an algebraic subgroup $N^\lambda$, normal in $G(\smash{\hat{\OO}})$, such that $G(\smash{\hat{\OO}})/N^\lambda$ is finite-dimensional.
\end{theorem}

\begin{proof}
 Let $H$ be the stabilizer; we construct a surjection $H \to B$ and show that the kernel is unipotent (hence connected, since any finite group is reductive).  This surjection is merely ``evaluation at $t = 0$''; from the moduli description of \ref{torus section moduli} we see that its image must only increase weights in representations of $G$, so lies in $B$.  Conversely, if we view $B \subset G(\hat{\OO})$ then this subgroup is a section of the evaluation map and for the same reason, preserves $t^\lambda$.
 
 The kernel is a closed subgroup of the kernel in $G(\hat{\OO})$ of the evaluation map, so we need only show that the latter kernel $K$ is unipotent.  Since $G(\smash{\hat{\OO}})$ by definition has points
 \begin{equation*}
  G(\hat{\OO})(\on{Spec} R) = G(\on{Spec} R\lb t \rb)
 \end{equation*}
 we also have ``evaluation at $t^n = 0$'' for all $n$, and thus a filtration of $K$:
 \begin{equation*}
  K = K^{(1)} \supset K^{(2)} \supset \dots
 \end{equation*}
 by their kernels.  In other words, $K^{(n)}$ is the subgroup of $\smash{G(\hat{\OO})}$ of points specializing at $t^n = 0$ to the identity in $G$, so that $K^{(1)}/K^{(2)} \cong \lie{g}$, the tangent to $G$ at the identity.  Thus, $K$ has a filtration by affine spaces, so must be unipotent.
 
 To obtain the second statement, we claim that for sufficiently large $n$, we have $K^{(n)} \subset H$.  Since $K^{(n)}$ is the kernel of a specialization homomorphism from $\smash{G(\hat{\OO})}$ and since each successive quotient $K^{(m)}/K^{(m + 1)}$ is finite-dimensional, $G(\hat{\OO})/K^{(n)}$ is a finite-dimensional algebraic group.  The claim itself follows from \ref{torus section moduli}, in that once $n > \langle \lambda, \check\alpha \rangle$ for all simple coroots $\check\alpha$, for $g \in K^{(n)}$ both inclusions
 \begin{align*}
  \genby{v} \otimes \OO(X_S \setminus \bar{x}) \overset{\phi}{\incl}
    V_{\sh{T}}|_{X_S \setminus \bar{x}}
  &&
  \genby{v} \otimes \OO(X_S \setminus \bar{x}) \overset{g\phi}{\incl}
    V_{\sh{T}}|_{X_S \setminus \bar{x}}
 \end{align*}
 (in the notation there) have the same pole about $\bar{x}$.
\end{proof}

By factorizability of $G(\hat{\OO})_n$ and $\on{Gr}_{G,X^n}$, these results (or their products) hold in every fiber over $X^n$, leading to the evident generalization:

\begin{theorem}{prop}{relative orbit properties}
 The orbits of $G(\hat{\OO})_n$ in $\on{Gr}_{G,X^n}$ are in bijection with $\Lambda_T/W$ (equivalently, the dominant coweights), depending on whether they contain the torus grassmannian components $\on{Gr}_{T,X^n}^{\lambda_1, \dots, \lambda_n}$ with fixed $\sum \lambda_i = \lambda$ modulo $W$.  These orbits stratify $\on{Gr}_n$.
 
 If $\on{Gr}_{B,X^n}^\lambda$ is the union of the same components of $\on{Gr}_{B,X^n}$, then the intersection $\on{Gr}_{G,X^n}^\lambda \cap \on{Gr}_{B,X^n}^\mu$ is nonempty only for $w_0(\lambda) \leq \mu \leq w_0(\lambda)$ and then has pure dimension $n + \langle \mu, \lambda \rangle$.  The stabilizer of $\on{Gr}_{T,X^n}^\lambda$ is connected and acts on $\on{Gr}_{G,X^n}^\lambda$ by a finite-dimensional quotient. \qed
\end{theorem}

\subsection*{Convolution diagrams}
   
Since for any $x \in X$ we have $\on{Gr}_G \cong \on{Gr}_{G,X}|_x$, we also have
\begin{equation*}
 \on{Conv}_1^n|_x
  = (G(\hat{\KK})_X \times^{G(\hat{\OO})_X} \dots \times^{G(\hat{\OO})_X} \on{Gr}_{G,X})|_x
  \cong G(\hat{\KK}) \times^{G(\hat{\OO})} \dots \times^{G(\hat{\OO})} \on{Gr}_G.
\end{equation*}
We will care only about the version where $n = 2$ and call it simply $\on{Conv}_G$ or, suggestively, $\on{Gr}_G * \on{Gr}_G$.  We have maps
\begin{equation*}
 \map{\on{pr}, m}{\on{Gr}_G * \on{Gr}_G}{\on{Gr}_G}
\end{equation*}
(projection is onto the first factor; $m$ comes from the action of $G(\hat{\KK})$ on $\on{Gr}_G$). As in \ref{inner convolution diagram diagonal}, this convolution diagram is the moduli space of:
\begin{equation}
 \label{eq:absolute convolution diagram}
 \on{Conv}_G(S) = \left\{(\sh{T}_1, \sh{T}_2, \phi_1, \phi_2) \middle\vert
   \begin{gathered}
    \text{$\sh{T}_i$ are $G$-torsors on $X$,}\\
    \text{$\phi_1$ is a trivialization of $\sh{T}_1$ on $X \setminus \bar{x}$}\\
    \text{$\map{\phi_2}{\sh{T}_1}{\sh{T}_2}$ is an isomorphism on $X \setminus \bar{x}$}
   \end{gathered}
  \right\}
\end{equation}
(for some fixed $x \in X$).  The projection map $\on{pr}$ becomes projection onto $(\sh{T}_1, \phi_1)$ and $m$ becomes projection onto $(\sh{T}_2, \phi_2 \circ \phi_1)$.

Since every $G(\hat{\OO})$-orbit $\on{Gr}_G^\lambda$ is (of course) $G(\hat{\OO})$-invariant, we can also form the diagram
\begin{equation*}
 \on{Conv}_G^{\lambda,\mu}
  = \on{Gr}_G^\lambda * \on{Gr}_G^\mu
  = q^{-1}(\on{Gr}_G^\lambda) \times^{G(\hat{\OO})} \on{Gr}_G^\mu
\end{equation*}
and write $\on{pr}^\lambda$, $m^{\lambda,\mu}$ for the two maps above.  As a twisted product, this convolution has the property that:

\begin{theorem}{lem}{irreducible convolution set}
 $\on{Gr}_G^\lambda * \on{Gr}_G^\mu$ is a $\on{Gr}_G^\mu$-bundle over $\on{Gr}_G^\lambda$ and is thus smooth and irreducible of dimension $\langle 2\rho, \lambda + \mu \rangle$. \qed
\end{theorem}

We will require knowledge of the fibers of $m^{\lambda,\mu}$ in the following ways.

\begin{theorem}{prop}{highest weight convolution fiber}
 The fiber $(m^{\lambda,\mu})^{-1}(t^{\lambda + \mu})$ consists of exactly one point $t^\lambda * t^\mu$.
\end{theorem}

\begin{proof}
 We use the description $\smash{\on{Conv}_G = G(\hat{\KK}) \times^{G(\hat{\OO})} \on{Gr}_G}$ together with the identification $\on{Gr}_G = G(\hat{\KK})/G(\hat{\OO})$, and thus suppose we have a pair $(g, h) \in G(\hat{\KK}) \times G(\hat{\KK})$ such that
 \begin{align*}
  g \in G(\hat{\OO}) t^\lambda G(\hat{\OO}) &&
  h \in G(\hat{\OO}) t^\mu G(\hat{\OO}) &&
  gh \in t^{\lambda + \mu} G(\hat{\OO}).
 \end{align*}
 We must show that $g \in t^\lambda G(\hat{\OO})$, $h \in t^\mu G(\hat{\OO})$.  Since in the pair $(g,h)$ we may pass factors in $\smash{G(\hat{\OO})}$ across the comma, it suffices to show only the first; that is, we may write $g = u t^\lambda$ with $u \in G(\hat{\OO})$, and we want to show that $t^{-\lambda} u t^\lambda \in G(\hat{\OO})$.  Let us also suppose that $h \in N(\hat{\OO}) t^\mu G(\hat{\OO})$. Thus, we have
 \begin{equation*}
  u t^\lambda n t^\mu G(\hat{\OO}) \subset t^{\lambda + \mu} G(\hat{\OO}),
 \end{equation*}
 with $u \in G(\hat{\OO})$ and $n \in N(\hat{\OO})$.  Rearranging, we get
 \begin{equation*}
  (t^{-\lambda} u t^\lambda) \in (t^\mu G(\hat{\OO}) t^{-\mu}) N(\hat{\OO}).
 \end{equation*}
 In any representation of $G$, the matrix of the left-hand side (in a weight basis partially ordered in the usual way) has all of its poles below the diagonal, while any matrix on the right-hand side has all of its poles above.  Thus, the left-hand side is in fact in $\smash{G(\hat{\OO})}$, as desired.
 
 To remove the restriction on $h$, we use the fact that $\smash{\on{Gr}_G^\lambda}$ is the union of its intersections with $\on{Gr}_{B^w}^{w(\lambda)}$, where $B^w = w B w^{-1}$. Since we have equivalently $\on{Gr}_G^\lambda = \on{Gr}_G^{w(\lambda)}$ for any $w$, the proof works for suitable choice of $w$.
\end{proof}

For the next proposition, in the proof we will use the fact that for any subvariety $V \subset \on{Gr}_G$ and any dominant coweight $\lambda$, we may construct the ``convolution diagram'' $V * \on{Gr}_G^\lambda \subset \on{Gr}_G * \on{Gr}_G$ via
\begin{equation*}
 V * \on{Gr}_G^\lambda = q^{-1}(V) \times^{G(\hat{\OO})} \on{Gr}_G^\lambda
\end{equation*}
since by virtue of the pullback the first factor is a $G(\hat{\OO})$-torsor.

\begin{theorem}{prop}{middle weight convolution fiber}
 For any dominant $\lambda \geq 0$, there is some dominant $\mu$ and an irreducible component in the fiber $(m^{\mu, \lambda})^{-1}(t^\mu)$ of dimension at least $\langle \rho, \lambda \rangle = \frac{1}{2} \dim \on{Gr}_G^\lambda$.
\end{theorem}

\begin{proof}
 To do this, we find such a component in a more amenable subspace.  We claim that for $\mu$ sufficiently large, we have
 \begin{equation*}
  m^{-1}(t^\mu) \cap (\on{Gr}_B^\mu * \on{Gr}_G^\lambda) \subset
  m^{-1}(t^\mu) \cap (\on{Gr}_G^\mu * \on{Gr}_G^\lambda).
 \end{equation*}
 Granting this, the following additional equality is obtained by left-multiplying by $t^\mu$:
 \begin{equation*}
  m^{-1}(t^0) \cap (\on{Gr}_B^0 * \on{Gr}_G^\lambda) =
  m^{-1}(t^\mu) \cap (\on{Gr}_B^\mu * \on{Gr}_G^\lambda)
 \end{equation*}
 and identifying $m^{-1}(t^0) \cong \on{Gr}_G$ via $\on{pr}$, the former is identified with $\on{Gr}_B^0 \cap \on{Gr}_G^{-w_0 (\lambda)}$, which has pure dimension $\langle \rho, -w_0(\lambda) \rangle = \langle \rho, \lambda \rangle$ by \ref{absolute orbit properties} since $\lambda \geq 0$, as desired.
 
 To prove the claim, we follow the proof of \cite{Zastava_spaces}*{Proposition 6.4}.  We work with the description $\on{Conv}_G = G(\hat{\KK}) \times^{G(\hat{\OO})} \on{Gr}_G$ and let $\map{q}{G(\hat{\KK}) \times G(\hat{\KK})}{\on{Conv}_G}$ be the quotient map. Then $q^{-1}(\on{Gr}_B^\mu * \on{Gr}_G^\lambda)$ is identified with pairs $(t^\mu n u, v t^\lambda w)$ with $n \in N(\hat{\KK})$ and $u,v,w \in G(\hat{\OO})$, such that
 \begin{equation*}
  t^\mu n u v t^\lambda w = t^\mu x
 \end{equation*}
 with $x \in G(\hat{\OO})$.  In particular, we have $n^{-1} \in q^{-1}(\on{Gr}_G^\lambda)$, or $n \in q^{-1}(\on{Gr}_B^0 \cap \on{Gr}_G^{-w_0(\lambda)})$, where the intersection is finite-dimensional. Applying \ref{unipotent group finiteness}, we have $n \in t^{-\mu} N(\hat{\OO}) t^\mu$ for some $\mu$ depending only on $\lambda$, so $t^\mu n \in N(\hat{\OO}) t^\mu \subset q^{-1}(\on{Gr}_G^\mu)$, as desired.
\end{proof}

\section{Vanishing cycles and gluing}
\label{s:vanishing cycles and gluing}

Before talking about gluing, we need to introduce a form of group action, related to factorizability, which is a key ingredient in the main theorem.  (This is one of the things we will be gluing.)  This was first given by Gaitsgory, \cite{G_deJong}*{\S2.5}, and we state it for sections of any sheaf of $k$-linear categories.

\begin{theorem}{defn}{factorizable action}
 Let $H$ be a group, $X$ a smooth curve, $Y \subset X^n$ be any subvariety, and let $\stack{F}$ be a sheaf of $k$-linear categories on $Y$.  Then a \emph{factorizable action} of $H$ on a section $s \in \stack{F}_Y$ is, simply, an action of $\on{Fact}(H)_n|_Y$ on $s$ (see \ref{factorizable version of a group}).  Explicitly, it is the following data: for every partition $p$ of $n$ into $m$ parts, an action of $H^m$ on $s$ over $Y \cap X^n_p$ such that, for any refinement $p'$ of $p$ with $m'$ parts, the action of $H^{m'}$ on $s$ is the same as the corresponding diagonal action of $H^m$, over $Y \cap X^n_{p'}$.

 For example, when $n = 2$, such an action is the data of an action of $H$ on $s$ as well as an action of $H \times H$ on $s|_{Y \setminus \Delta}$, such that the restriction of the former to the latter set is the action of the diagonal in $H \times H$.
 
 We will denote the sheaf of categories of perverse sheaves with a factorizable $H$-action by $\cat{FRep}_n(H)_Y$.
\end{theorem}

One obvious fact that bears mention is that since $\on{Fact}(H)_n \subset H^n$ by definition, any section with an action of $H^n$ automatically obtains a factorizable action.  In addition, if $Y \subset X^n$ actually lies in one of the diagonals $\cong X^{n - 1}$, then we have a natural equivalence $\cat{FRep}_{n - 1}(H)_Y \cong \cat{FRep}_n(H)_Y$ using \ref{eq:sf mult diagonal}.

We will also want to consider actions of an \emph{algebraic group} on perverse sheaves, which requires a definition since perverse sheaves are not schemes.  The following definition and lemmas are basically trivial.

\begin{theorem}{lem}{algebraic automorphisms}
 Let $\cat{C}$ be a $k$-linear category for some field $k$ and let $x \in \cat{C}$; suppose $\on{End}(x)$ is finite-dimensional.  Then each $\on{Aut}(x)$ has the natural structure of the rational points of an affine algebraic group defined over $k$. \qed
\end{theorem}

\begin{theorem}{defn}{algebraic group action}
 Let $k$ be a field.  An action of an algebraic $k$-group $H$ on an object in a $k$-linear category $\cat{C}$ is a regular homomorphism $H \to \on{Aut}(x)$, the latter considered as an affine algebraic $k$-group.  A map $f \colon x \to y$ of objects with a $H$-action is a map of $H$-actions if both maps $H \to \on{Hom}(x,y)$ obtained by composition with $f$, the latter $k$-vector space considered as an affine $k$-scheme,  are equal.
\end{theorem}

Note that if $x$ has an action of an algebraic $k$-group $H$, then it also has a literal group action of $H(k)$, though the converse is of course not true.  If $\cat{C}$ is the category of finite-dimensional vector spaces, this definition is exactly the same as what is normally meant by the action of an algebraic group.

\begin{theorem}{lem}{group action gluing}
 Let $H$ be an algebraic $k$-group, let $\stack{F}$ be a sheaf of $k$-linear categories in which every $\on{Aut}(s)$ is finite-dimensional, and let $\cat{Rep}(H,\stack{F})$ be the fibered category of sections of $\stack{F}$ with an $H$-action.  Then $\cat{Rep}(H,\stack{F})$ is a sheaf of categories.
\end{theorem}

\begin{proof}
 Let $\{U_i\}$ be an open cover contaning all its finite intersections, let $\{s_i\}$ be sections of the $\cat{Rep}(H,\stack{F}|_{U_i})$, and let $\{\phi_{ij} \colon s_i|_{U_j} \to s_j\}$ be isomorphisms whenever $U_j \subset U_i$.  Then since $\stack{F}$ is a sheaf of categories, the $s_i$ glue to sections of $\stack{F}$, and we must show that the $H$-actions glue as well. Conjugation by $\phi_{ij}$ gives a map $\on{Aut}(s_i) \to \on{Aut}(s_j)$, and these maps form an inverse system indexed by the $U_{ij}$; by definition, $\on{Aut}(s) = \invlim \on{Aut}(s_i)$, and since affine schemes admit all limits, this holds algebraically as well as set-theoretically. Since the $\phi_{ij}$ are maps of $G$-actions, the maps $H \to \on{Aut}(s_i)$ are compatible with the limit and therefore assemble to a map of affine schemes $H \to \on{Aut}(s)$, as desired.
\end{proof}

It follows that when $H$ is an algebraic group, the categories $\cat{FRep}_n(H)$ form a sheaf of categories on $X^n$ (or on $Y \subset X^n$).  Thus, we may also speak of twisting a factorizable action by a $\on{Fact}(H(k))_n$-gerbe, and this will be important in what is to come.

\subsection*{Beilinson's gluing theorem}

One of the most important technical properties of $\Perv_S$ is the following: let $D = f^{-1}(0)$ be a principal Cartier divisor in $S$ (with $\map{f}{S}{\Aff^1}$), with $\map{j}{U = S \setminus D}{S}$, $\map{i}{D}{S}$. For a perverse sheaf $\sh{M}$ on $U$, let $\Psiun_f(\sh{M}) = R\psi_f^\mathrm{un}(\sh{M})[-1]$ be the (unipotent) nearby cycles of $\sh{F}$ along $f$, and for $\sh{F}$ on $S$, let $\Phiun_f(\sh{F}) = R\phi_f^\mathrm{un}(\sh{F})[-1]$ be the vanishing cycles. The monodromy action on $\Psiun_f(\sh{M})$ is given by a unipotent automorphism $\mu$.  The following theorem of Beilinson is discussed in the author's paper \cite{beilinson_notes}.

\begin{theorem*}{thm}
 (Beilinson's gluing theorem)
 \begin{itemize}
  \item
  Let $\sh{L}^n$ be the local system on $\Gm$ of rank $n$ with monodromy given by a unipotent Jordan block.  Then:
  \begin{equation*}
  i_* \Psiun_f(\sh{M}) \cong
   \begin{cases}
    \lim_{n \to \infty}
      \on{ker}\bigl(j_! (\sh{M} \otimes f^* \sh{L}^n) \to j_* (\sh{M} \otimes f^* \sh{L}^n)\bigr) \\
    \lim_{\infty \leftarrow n}
      \on{coker}\bigl(j_! (\sh{M} \otimes f^* \sh{L}^n) \to j_* (\sh{M} \otimes \sh{L}^n)\bigr) \\
    \lim_{n \to \infty} i_* i^* j_{!*}(\sh{M} \otimes f^* \sh{L}^n)[-1]
   \end{cases}
  \end{equation*}
  with the monodromy acting via $\sh{L}^n$; consequently, $\Psiun_f$ commutes with Verdier duality of perverse sheaves.  In particular, the monodromy action on $\Psiun_f(\sh{M})$ is trivial if and only if $\Psiun_f(\sh{M}) = i^* j_{!*}(\sh{F})[-1]$, in which case $\Psiun_f$ depends only on $D$ and may be defined for non-principal Cartier divisors as well.

  \item 
  $\Perv_S$ is equivalent to the \emph{gluing category} of quadruples $(\sh{F}_U, \sh{F}_D, u, v)$, where $\sh{F}_U, \sh{F}_D$ are perverse sheaves on $U$ and $D$ and $\Psiun_f(\sh{F}_U) \xrightarrow{u} \sh{F}_D \xrightarrow{v} \Psiun_f(\sh{F}_U)$ are maps with $v \circ u =  1 - \mu$. The $\sh{F}_D$ term is given by the vanishing cycles functor $\Phiun_f(\sh{F})$, and Verdier duality acts on this data by dualizing all the terms, so that $u$ and $v$ are exchanged by duality.
 \end{itemize}
\end{theorem*}

We claim that this theorem holds also in the following circumstances:
\begin{theorem}{prop}{vanishing cycles gluing}
 \mbox{}
 \begin{enumerate}
  \item \label{en:gerbe twisting gluing}
  Let $Y$ be any scheme and $\stack{G}$ a $k^*$-gerbe on $Y$; then $\Perv(\stack{G})_Y$ has nearby cycles along any coordinate function $f$ and is equivalent to its gluing category.

  \item \label{en:factorizable action gluing}
  Let $\stack{G}$ be a $\on{Fact}(H)_n$-gerbe on $Y$, and for $\sh{F} \in \stack{G} \otimes \cat{FRep}_n(H)_Y$, let $\Psiun_f(\sh{F}|_U)$ and  $\Phiun_f(\sh{F})$ be the perverse sheaves $\Psiun_f(\sh{F}|_U)$ and $\Phiun_f(\sh{F})$ with their induced factorizable actions on $D$. Then $\stack{G} \otimes \cat{FRep}_n(H)_Y$ is equivalent to its gluing category.
  
  \item \label{en:equivariant object gluing}
  Let $Y$ be any scheme, $H$ any group scheme acting on $Y$, and $\stack{G}$ any $H$-equivariant gerbe on $Y$; then the category of $H$-equivariant objects in $\Perv(\stack{G})$ has nearby cycles and is equivalent to its gluing category.
 \end{enumerate}
\end{theorem}

\begin{proof}
 Let $\cat{G}$ be the gluing category of perverse sheaves about a divisor $D$, let $\map{V}{\Perv}{\cat{G}}$ be the ``vanishing cycles'' functor associating a sheaf to its gluing data, and let $G$ be its inverse ``gluing functor''.

 For the first claim, we construct the nearby and vanishing cycles of $\sh{F} \in \Perv(\stack{G})_Y$ by twisting those for ordinary perverse sheaves.  To be precise, let $\stack{G}$ be a $k^*$-gerbe on $Y$ and let $\stack{F}_U, \stack{F}_D$ be its restrictions.  Then we show that the functor $\Psiun_f$ induces one $\Perv(\stack{G}_U) \to \Perv(\stack{G}_D)$ and the functor $\Phiun_f$ induces one $\Perv(\stack{G}) \to \Perv(\stack{G}_D)$.  To construct the first, we may \emph{define} the nearby cycles via $j_!$ and $j_*$ as in Beilinson's construction (taking the limit locally), and then by definition the gluing theorem is satisfied. To construct the second, we may employ the construction of vanishing cycles also given by Beilinson using these operations.  Since these constructions (as shown in \cite{beilinson_notes}) agree with the usual ones for perverse sheaves, we have indeed produced twisted nearby and vanishing cycles functors.

 For the second claim, by functoriality of $V$, the maps $u$ and $v$ respect the factorizable action and so $V(\sh{F})$ is in the gluing category of $\cat{FRep}_n(H, \stack{G})_Y$.  Conversely, suppose we start with a quadruple $\sh{F}_g = (\sh{F}_U, \sh{F}_D, u, v) \in \cat{G}$ where both sheaves have factorizable actions, $\Psiun_f(\sh{F}_U)$ is given the induced action, and $\map{u}{\Psiun_f(\sh{F}_U)}{\sh{F}_D}$ and $v$ preserve these actions.  Then both $\sh{F}_U$ and $\sh{F}_D$ have $\on{Fact}_n(H)$-actions, so $\sh{F}_g$ does.  By functoriality, $G(\sh{F}_g)$ does as well.  By Beilinson's theorem, $VG(\sh{F}_g) \cong \sh{F}_g$ and $GV(\sh{F}) \cong \sh{F}$ as perverse sheaves; we claim that these maps are isomorphisms of factorizable $H$-objects.  For the first this is tautological from the construction of the $H$-action on a quadruple.  For the second, we apply $V$ to both sides; then as just argued, the $H$-actions agree (taking $\sh{F}_g = V(\sh{F})$) and since $V$ is faithful, they agree without 
$V$ as well.
 
 For the third claim, since the equivariant objects form a full subcategory, we only have to check that $V(\sh{F})$ is equivariant when $\sh{F}$ is and that $G(\sh{F}_g)$ is equivariant when the gluing data $\sh{F}_g$ is, which are tautological by functoriality.
\end{proof}

\subsection*{Universally locally acyclic sheaves}

The formalism of nearby cycles neatly complements the following concept of \emph{universal local acyclicity} (ULA).  Recall the Braverman--Gaitsgory formulation \cite{G_Eisenstein}*{\S5.1} of the ULA condition:

\begin{theorem}{defn}{ULA}
 We will say that a closed subscheme $\map{i}{Z}{S}$ of pure codimension $d$ is \emph{smoothly embedded} if $i^! \csheaf{k}_S = \csheaf{k}_Z[-2d]$.  For example, if $Y$ is smooth of pure dimension $d$, and $\map{f}{X}{Y}$ is a map of schemes, then its graph $\map{\Gamma_f}{X}{X \times Y}$ is smoothly embedded.  For any map $\map{g}{T}{S}$, and any complexes of sheaves $\sh{F}_1, \sh{F}_2$ on $S$, there is a natural map $g^* \sh{F}_1 \otimes g^! \sh{F}_2 \to g^!(\sh{F}_1 \otimes \sh{F}_2)$ on $T$, obtained by adjunction from
 \begin{equation*}
  g_!(g^* \sh{F}_1 \otimes g^! \sh{F}_2) \cong \sh{F}_1 \otimes g_! g^! \sh{F}_2
    \to \sh{F}_1 \otimes \sh{F}_2
 \end{equation*}
 where the isomorphism is the projection formula.  In the situation of a map $\map{f}{X}{Y}$, where $Y$ is smooth, if $\sh{F}_X, \sh{F}_Y$ are complexes on $X$ and $Y$, we take $T = Z = X$ (relative to the above notation), $S = X \times Y$, $g = i = \Gamma_f = (\id, f)$, $\sh{F}_1 = \sh{F}_X \boxtimes \sh{F}_Y$, and $\sh{F}_2 = \csheaf{k}_{X \times Y}$; these considerations then give a natural map
 \begin{equation*}
  \sh{F}_X \otimes f^* \sh{F}_Y \to (\sh{F}_X \xtimes f^! \sh{F}_Y)[2\dim Y],
 \end{equation*}
 where $\sh{F}_1 \xtimes \sh{F}_2 = \DD(\DD\sh{F}_1 \otimes \DD\sh{F}_2)$ by definition, and $\sh{F}_X$ is said to be \emph{locally acyclic}, or LA for $f$ if this is an isomorphism for all $\sh{F}_Y$, and \emph{universally locally acyclic}, or ULA if this condition holds after any smooth base change of $Y$.
\end{theorem}

We introduce the superscript ULA, as in $\Sph^\ULA(\stack{G}_n)$, which we will only use with respect to the map $\on{Gr}_{G,X^n} \to X^n$, for the ULA objects in these categories.  ULA sheaves have some properties which are found in \cite{G_Eisenstein}.

\begin{theorem}{prop}{ULA facts}
 If $\sh{F}$ is $f$-ULA, then so is $\DD\sh{F}$, and if $\map{g}{Z}{Y}$ makes $Z$ a $Y$-scheme and $\map{h}{X}{Z}$ is a proper map of $Y$-schemes, then $h_* \sh{F}$ is $g$-ULA.  The ULA property is local in the smooth topology on $Y$ and if $\sh{F}$ is any complex of sheaves, then there exists a nonempty, Zariski-open subset of $Y$ over which $\sh{F}$ becomes ULA. \qed
\end{theorem}

The following property of ULA perverse sheaves explains their importance.

\begin{theorem}{prop}{ULA nearby cycles}
 Let $D \subset Y$ be a smooth Cartier divisor and $E = f^{-1}(D)$ its preimage in $X$, $\map{i}{E}{X}$ with complement $\map{j}{U}{X}$. Then for any $\sh{F} \in \Perv^\ULA_X$, we have $j_{!*}(\sh{F}|_U) = \sh{F}$.  Furthermore, for \emph{any} ULA $\sh{F}$ such that $j^* \sh{F}$ is perverse, monodromy acts trivially on the nearby cycles, so $\Psiun_E(\sh{F}|_U) = i^* \sh{F}[-1] = i^! \sh{F}[1]$.  In particular, any such $\sh{F}$ is necessarily perverse and we have $\Phiun_E(\sh{F}) = 0$.  Finally, when this happens, $i^* \sh{F}[-1] = \Psiun_E(\sh{F}|_U) = i^! \sh{F}[1]$ is ULA for $f|_E$.
\end{theorem}

\begin{proof}
 Let $\csheaf{k}_E = i_* \csheaf{k}$ denote the constant sheaf supported at $E$; using the same notation for the inclusions of $D$ and its complement, $\csheaf{k}_E = f^* \csheaf{k}_D$.  By definition, $i^* \sh{F} = \sh{F} \otimes \csheaf{k}_E$ and $i^! \sh{F} = \sh{F} \xtimes \DD\csheaf{k}_E$, and $j_{!*}(\sh{F}|_U)$ is characterized by the property that it extends $\sh{F}|_U$ and such that its $i^*$ and $i^!$ have vanishing perverse 0-cohomology; we show that this is true of $\sh{F}$. Note that since $D$ is smooth, $\DD\csheaf{k}_D = \csheaf{k}_D[2(\dim Y - 1)]$.  By the ULA condition,
 \begin{multline*}
  i^* \sh{F} = \sh{F} \otimes f^* \csheaf{k}_D
             = (\sh{F} \xtimes f^! \csheaf{k}_D)[2\dim Y]
             = (\sh{F} \xtimes f^! \DD\csheaf{k}_D[2(1 - \dim Y)])[2\dim Y] \\
             = i^! \sh{F}[2].
 \end{multline*}
 By \cite{BBD}*{Corollaire 4.1.10(ii)}, $i^*$ exists only in perverse cohomologies $-1$ and $0$ and $i^!$ in degrees $0$ and $1$; thus, neither has any perverse cohomology in degree $0$, whence the claim.
 
 To compute nearby cycles, we replace $Y$ with an open subset on which $D$ is defined by a single equation (since it will be irrelevant, we refer simply to $\Psiun_E$); let $\sh{L}^n$ denote the pullback to $Y$ from $\Gm$ of the Jordan-block sheaf of Beilinson's theorem.  Then
 \begin{equation*}
  \Psiun_E(\sh{F}|_U)[1]
    = \lim_{n \to \infty}
       \on{ker}(j_!(\sh{F}|_U \otimes f^*\sh{L}^n) \to j_*(\sh{F}|_U \otimes f^*\sh{L}^n)).
 \end{equation*}
 By the projection formula,
 \begin{equation*}
  j_!(\sh{F}|_U \otimes f^* \sh{L}^n) = \sh{F} \otimes j_! f^* \sh{L}^n
    = \sh{F} \otimes f^* j_! \sh{L}^n.
 \end{equation*}
 By the ULA condition for $\sh{F}|_U$,
 \begin{equation*}
  \sh{F}|_U \otimes f^* \sh{L}^n = \sh{F}|_U \xtimes f^! \sh{L}^n[2\dim Y]
 \end{equation*}
 and therefore by the projection formula again
 \begin{multline*}
  j_*(\sh{F}|_U \otimes f^* \sh{L}^n)
    = j_*(\sh{F}|_U \xtimes f^! \sh{L}^n)[2\dim Y]
    = \sh{F} \xtimes j_* f^! \sh{L}^n[2\dim Y] \\
    = \sh{F} \xtimes f^! j_* \sh{L}^n[2\dim Y]
 \end{multline*}
 which by the ULA condition for $\sh{F}$ is just $\sh{F} \otimes f^* j_* \sh{L}^n$.  Thus, the natural map $j_!(\sh{F}|_U \otimes f^* \sh{L}^n) \to j_*(\sh{F}|_U \otimes f^* \sh{L}^n)$ is the same as the natural map
 \begin{equation*}
  \sh{F} \tensor^L f^*(j_! \sh{L}^n \to j_* \sh{L}^n).
 \end{equation*}
 (we have inserted the ${}^L$ to emphasize that tensor product is derived, though we have been neglecting this for the tensor product with locally free sheaves.) Let $C_X^n$ be the cone of this morphism, and let $C_Y^n$ be the cone of the map $j_!\sh{L}^n \to j_* \sh{L}^n$, so $C_X^n = \sh{F} \otimes^L f^* C_Y^n$.  The $C_Y^n$, and consequently the $C_X^n$, form a sequence indexed by $n$, and we claim that the maps in this sequence are all isomorphisms for $C_Y^n$ for all $n$.  Indeed, the perverse cohomologies of $C_Y$ compute the nearby cycles of the constant sheaf on $Y \setminus D$, which has no monodromy. Therefore $C_X^n$ is independent of $n$ as well, whence the result.
 
 If $j^* \sh{F}$ is perverse when $\sh{F}$ is ULA, then we have just shown that $i^* \sh{F}[-1] = \Psiun_E(j^* \sh{F}) = i^! \sh{F}[1]$ are perverse, and therefore $\sh{F}$  is perverse.  To show that $\Phiun_E(\sh{F}) = 0$, we consider the triangle
 \begin{equation*}
  i^* \sh{F} \to \Psiun_E(j^* \sh{F})[1] \to \Phiun_E(\sh{F})[1] \to
 \end{equation*}
 and substitute $\sh{F} = j_{!*} j^* \sh{F}$ to obtain an isomorphism of the first two terms.
 
 Finally, suppose that the isomorphisms $i^* \sh{F}[-1] \cong \Psiun_E(\sh{F}|_U) \cong i^! \sh{F}[1]$ hold.  Then it follows that $\Psiun_E(\sh{F}|_U)$ is ULA over $D$: if $\sh{G}$ is a complex of sheaves on $D$, and if $i$ refers to the inclusions of both $D$ in $Y$ and $E$ in $X$, then we may apply the projection formula and its dual version,
 \begin{align*}
  i_! (i^* \sh{F}_1 \otimes \sh{F}_2) \cong \sh{F}_1 \otimes i_! \sh{F}_2 &&
  i_* (\sh{F}_1 \xtimes i^! \sh{F}_2) \cong i_* \sh{F}_1 \xtimes \sh{F}_2,
 \end{align*}
 where in this case $i_! = i_*$ for a closed immersion, and write the chain of isomorphisms:
 \begin{multline*}
  i_!(i^* \sh{F} \otimes f^* \sh{G})
   \cong \sh{F} \otimes i_! f^* \sh{G}
   \cong \sh{F} \otimes f^* i_! \sh{G} \\
   \cong (\sh{F} \xtimes f^! i_* \sh{G})[2\dim Y]
   \cong (\sh{F} \xtimes i_* f^! \sh{G})[2\dim Y]
   \cong i_*(i^! \sh{F} \xtimes f^! \sh{G})[2\dim Y] \\
   \cong i_*(i^* \sh{F} \xtimes f^! \sh{G})[2(\dim Y - 1)],
 \end{multline*}
 where $\dim Y - 1 = \dim D$.  Applying $i^*$ and shifting by $-1$ to get $\Psiun_E(\sh{F}|_U)$ in place of the (co)restrictions, we have the ULA condition for $\Psiun_E(\sh{F}|_U)$.
\end{proof}

\begin{theorem}{cor}{ULA local systems}
 If $\sh{F}$ is ULA and perverse on $Y$ itself, then $\sh{F}[-\dim Y]$ is a locally constant sheaf
 ($\sh{F}$ is \emph{lisse}).
\end{theorem}

\begin{proof}
 There is some Zariski open set $U$, which we assume to be the complement of a divisor $D$, such that $\sh{F}|_U[-\dim Y]$ is a locally constant sheaf $\sh{L}$.  By \ref{ULA nearby cycles}, the monodromy on $\Psiun_D(\sh{L})$ is trivial, so the monodromy of $\sh{L}$ itself is trivial and thus $\sh{L}$ extends to a local system on $Y$ which we will also call $\sh{L}$.  Since $Y$ is smooth, $\DD(\sh{L}[\dim Y]) = \sh{L}'[\dim Y]$ with $\sh{L}'$ again a locally constant sheaf, and so $\sh{L}[\dim Y]$ verifies the properties of the middle extension of $\sh{F}|_U$, and hence is equal to $\sh{F}$.
\end{proof}

\subsection*{D\'evissage}

The reason to use ULA sheaves is that every perverse sheaf is ``generically'' ULA by \ref{ULA facts}.  In combination with Beilinson's vanishing cycles gluing, this allows us to reduce most theorems about perverse sheaves on $\on{Gr}_n$ to ULA sheaves, for which \ref{ULA nearby cycles} is a powerful tool.  We prepare to state the theorem establishing such a principle by making a few meta-definitions.

\begin{theorem}{defn}{horizontal functor}
 Let $Y$ be any scheme, $\stack{F}$ a sheaf of abelian categories on $Y$.  We will say that it \emph{admits gluing} if it has pushforwards from any open subspace of $Y$ in the sense of \ref{twisted pushforward functor}, as well as pushforwards with compact support, has an action of the tensor category of local systems on $Y$, and if Beilinson's theorem holds.
 
 Let $Y_1$, \dots, $Y_n$ be schemes and $\stack{F}_1$, \dots, $\stack{F}_n$ be sheaves of abelian categories admitting gluing over the $Y_i$; let $\stack{F}_{n + 1}$ be such a sheaf of categories over $Y_{n + 1} = Y_1 \times \dots \times Y_n$. We will say that a functor
 \begin{equation*}
  F \colon \stack{F}_1 \times \dots \times \stack{F}_n \to \stack{F}_{n + 1}
 \end{equation*}
 is \emph{horizontal} if:
 \begin{itemize}
  \item
  It is a functor of categories fibered over $Y_{n + 1}$, so preserves restrictions to all open sets.
  
  \item
  It is exact in each argument and for local systems $\sh{L}_1$, \dots, $\sh{L}_n$ on the $Y_i$ and any sections $s_1$, \dots, $s_n$ in the $\stack{F}_i$, we have
  \begin{equation*}
   \hspace*{\leftmargin}
   F(\sh{L}_1 \otimes s_1, \dots, \sh{L}_n \otimes s_n)
    \cong (\sh{L}_1 \otimes \dots \otimes \sh{L}_n) \otimes F(s_1, \dots, s_n).
  \end{equation*}
  
  \item
  For any Cartier divisors $D_i \subset Y_i$, denoting $j_i \colon U_i = Y_i \setminus D_i \to Y_i$ and setting $D_{n + 1} = D_1 \times \dots \times D_n$ and $j \colon U_{n + 1} = Y_{n + 1} \setminus D_{n + 1} \to Y_{n + 1}$, the following natural maps
  \begin{align*}
   \hspace*{\leftmargin}
   (j_{n + 1})_! F \to F ((j_1)_! \times \dots \times (j_n)_!) &&
   F ((j_1)_* \times \dots \times (j_n)_*) \to (j_{n + 1})_* F,
  \end{align*}
  obtained from the compatibility of $F$ with restriction to $U_{n + 1}$, are isomorphisms.
 \end{itemize}
 We say that a natural transformation of horizontal functors is horizontal if it satisfies the analogous conditions.
\end{theorem}

\begin{theorem}{lem}{compatible with nearby cycles}
 Given divisors as above, if $F$ is horizontal and the $D_i$ are principal, and if $\stack{F}_{n + 1}$ admits gluing, we have
 \begin{align*}
  \Psiun_{D_{n + 1}} F \cong F (\Psiun_{D_1} \times \dots \times \Psiun_{D_n}) &&
  \Phiun_{D_{n + 1}} F \cong F (\Phiun_{D_1} \times \dots \times \Phiun_{D_n})
 \end{align*}
 and $F$ is equal to the functor induced on nearby cycles gluing data by acting on each component of a quadruple.  We will say that any functor for which these equations hold is \emph{compatible with nearby cycles}.  The same holds of natural transformations.
\end{theorem}

\begin{proof}
 The first statement is a direct consequence of Beilinson's construction, after which the second statement is tautological.
\end{proof}

Given the above definition, all of the sheaves of categories described in \ref{vanishing cycles gluing} admit gluing; important examples of horizontal functors will occur later in this chapter.

\begin{theorem}{lem}{horizontal functors}
 Say that a functor is ``left horizontal'' if it satisfies all the hypotheses of being horizontal except that it only commutes with $!$ extensions; likewise, a ``right horizonal'' functor commutes with $*$ extensions.  Suppose $g_i \colon X_i \to Y$ are maps of spaces and that $f \colon X_1 \to X_2$ makes a commutative triangle; then between categories of constructible sheaves, $f_!$ and $f^*$ are left horizontal while $f_*$ and $f^!$ are right horizontal.  The composition of (left-, right-) horizontal functors is again (left-, right-) horizontal.
\end{theorem}

\begin{proof}
 If $U \subset Y$, write $V_i = g_i^{-1}(U)$ and let $j_i$ be their inclusions; then $V_1 = f^{-1}(V_2)$.  Clearly $f_* (j_1)_* = (j_2)_* f_* \, (= (f|_{V_1})_*)$ and likewise for $!$'s, and that $f^! (j_1)_* = (j_2)_* f^!$ and $f^* (j_1)_! = (j_2)_! f^*$ are base change theorems.  The last sentence is obvious.
\end{proof}

Here is the precise definition of the problem we wish to solve using gluing.

\begin{theorem}{defn}{reduces to subcategory}
 Let $\stack{F}_i$ be as in \ref{horizontal functor} and suppose $P_i$ are local properties of sections of these sheaves of categories, with $\stack{F}_i'$ the full subsheaves of sections satisfying these properties.  Let $S$ be a ``structure'' consisting of functors $\stack{F}_1 \times \dots \times \stack{F}_n \to \stack{F}_{n + 1}$, natural transformations among them, and conditions on these data; we will write that $S$ is \emph{compatible with nearby cycles} to mean that all these functors and transformations are. We say that specifying $S$ \emph{reduces to} the $P_i$ if it suffices to do so after replacing the $\stack{F}_i$ by the $\stack{F}_i'$ for $i \leq n$, and \emph{weakly reduces to} the $P_i$ if we also require $\stack{F}_{n + 1}'$.
\end{theorem}

The following multipart metatheorem encapsulates a gluing argument used repeatedly throughout the rest of this work.

\begin{theorem}{lem}{devissage}
 Let $Y_i$ be schemes which are noetherian spaces, $\stack{F}_i$ ($i \leq n$) sheaves of categories admitting gluing, and let $Y_{n + 1} = \prod_i Y_i$, with $\stack{F}_{n + 1}$ any sheaf of categories on it.  Suppose that for each $i \leq n$, every section of $\stack{F}_i$ has property $P_i$ on a Zariski-dense open subset of $Y_i$.  Then the following structures $S$, when compatible with nearby cycles, reduce to the $P_i$ when $\stack{F}_{n + 1}$ admits gluing:
 \begin{itemize}
  \item
  Given a functor $F \colon \stack{F}_1 \times \dots \times \stack{F}_n \to \stack{F}_{n + 1}$, $S$ is the condition that $F$ is exact.
  
  \item
  Given a pair of functors $F, G$ as above, $S$ is the structure of an isomorphism $t \colon F \to G$.
  
  \item
  $S$ is the structure of a functor $F$ as above.
 \end{itemize}
 In addition, suppose instead that $\stack{F}_{n + 1}$ is a twisted derived category of constructible sheaves (thus, admitting nearby cycles) and $F$ is as above.  Then the following structures $S$ also reduce to the $P_i$:
 \begin{itemize}
  \item 
  $S$ is the condition of $F$ taking values in perverse sheaves.
  
  \item
  If all the $\stack{F}_i$ are twisted derived categories, $S$ is the condition of $F$ being t-exact for the perverse t-structure.
 \end{itemize}
 Finally, suppose that all $\stack{F}_i$ ($i \leq n + 1$) admit gluing, that sections of $\stack{F}_{n + 1}$ generically have property $P_{n + 1}$, and that $F$ is as above. Then the following structure $S$ \emph{weakly} reduces to the $P_i$:
 \begin{itemize}
  \item
  $S$ is the structure of a functor $G \colon \stack{F}_{n + 1} \to \prod_{i \leq n} \stack{F}_i$ and isomorphisms $FG \to \id$ and $GF \to \id$, realizing $F$ as an equivalence of categories.
 \end{itemize}
\end{theorem}

\begin{proof}
 We will use the following notation as standard in this proof: for each $i \leq n$, and for each principal Cartier divisor $D_i \subset Y_i$, let $\cat{G}_{i, D_i}$ be the ``gluing category'' consisting of nearby cycles gluing data
 \begin{equation*}
  (s_{U_i}, s_{D_i}, u, v) \qquad \text{where }
   s_{U_i} \in \stack{F}_{i,U_i}',
   s_{D_i} \in \stack{F}_{i,D_i},
   \Psiun_{D_i} (s_{U_i}) \xrightarrow{u} s_{D_i} \xrightarrow{v} \Psiun_{D_i} (s_{u_i})
 \end{equation*}
 with $1 - v \circ u$ equal to the monodromy action on nearby cycles; here we have written $\stack{F}_{i,D_i}$ for the subsheaf of sections supported on $D_i$ in the sense that their restriction to $U_i$ is zero.  Note that $s_{U_i}$ is assumed to have property $P_i$ but $s_{D_i}$ is not.
 
 Beilinson's gluing theorem gives a fully faithful embedding of each $\cat{G}_{i,D_i}$ into $\stack{F}_{i,Y}$ (likewise for any open subscheme of $Y$).  By hypothesis on the $\stack{F}_i$, every section $s \in \stack{F}_{i,Y}$ has property $P_i$ generically, so without loss of generality on the complement of some Cartier divisor, and therefore locally lies in the union of the $\cat{G}_{i, D_i}$.
 
 For the first three points, let $\cat{G}_{n + 1, D_{n + 1}}$ be the gluing category in which $s_{U_i}$ is \emph{not} assumed to have property $P_{n + 1}$; here $D_{n + 1} = \prod_{i \leq n} D_i$ as in \ref{horizontal functor}.  Since all of the above structures $S$ respect nearby cycles and are local on $Y_{n + 1}$, the following strategy suffices to prove the lemma: we assume by noetherian induction that $S$ obtains for the $\stack{F}_i|_{D_i}$ and by hypothesis that it obtains for the $\stack{F}_i'|_{U_i}$, and construct the structure $S$ on the $\cat{G}_{i, D_i}$ ($i \leq n + 1$).  Note that all functors and natural transformations are, by \ref{compatible with nearby cycles}, given term-by-term on gluing data.  Here are the proofs of the first three points:
 \begin{itemize}
  \item
  Morphisms of gluing data are termwise (provided that they form commutative squares with $u$ and $v$) and the abelian structure on a category of gluing data is termwise on such morphisms, so exactness of a functor given termwise on gluing data is determined termwise as well.
  
  \item 
  This is tautological since natural transformations are termwise on gluing data.
  
  \item
  This is tautological since functors are termwise on gluing data.
 \end{itemize}
 
 For the next two points, we need not (and can not) use gluing at all, so we replace the $\cat{G}_{i, D_i}$ ($i \leq n$) with the full subsheaves of the $\stack{F}_i$ whose sections are in $\stack{F}_i'$ when restricted to $U_i$ to which they would otherwise be equivalent.  We continue to assume the noetherian induction hypothesis.
 \begin{itemize}
  \item
  Suppose we have $s_i \in \stack{F}_{i,Y}$ such that $s_i|_{U_i} \in \stack{F}_{i,Y}'$; then
  \begin{equation*}
   F(s_1, \dots, s_n)|_{U_{n + 1}} = F(s_1|_{U_i}, \dots, s_n|_{U_n}),
  \end{equation*}
  where the latter expression is, by hypothesis, perverse.  Thus, since nearby cycles preserve perversity, $\Psiun_{D_{n + 1}} F(\{s_i\})|_{U_{n + 1}}$ is perverse.  Since $F$ also respects vanishing cycles, we have $\Phiun_{D_{i + 1}} F(\{s_i\}) = F(\{\Phiun_{D_i} s_i\})$, which is perverse by noetherian induction.  We now invoke the distinguished triangle
  \begin{equation*}
   \hspace*{\leftmargin}
   \Psiun_{D_{n + 1}} F(\{s_i\})|_{U_{n + 1}} \xrightarrow{u}
    \Phiun_{D_{n + 1}} F(\{s_i\}) \to i^* F(\{s_i\}) \to
  \end{equation*}
  in which the first two terms are both perverse, and so the third term is in $\stack{F}_{n + 1}^{\leqslant 0}$.  By the dual triangle
  \begin{equation*}
   \hspace*{\leftmargin}
   i^! F(\{s_i\}) \to \Phiun_{D_{n + 1}} F(\{s_i\}) \xrightarrow{v}
    \Psiun_{D_{n + 1}} F(\{s_i\})|_{U_{n + 1}} \to
  \end{equation*}
  and the same argument we also have $i^! F(\{s_i\}) \in \stack{F}_{n + 1}^{\geqslant 0}$, so we conclude that $F(\{s_i\})$ is perverse.
  
  \item
  The argument that $F$ is t-exact is exactly the same, since the first triangle shows that $F$ preserves nonpositive perversity and the second one shows that it preserves nonnegative perversity.
 \end{itemize}

 Finally, for the last point we return to the gluing argument but let $\cat{G}_{n + 1, D_{n + 1}}$ contain only those gluing data with $s_{U_{i + 1}} \in \stack{F}_{n + 1, U_{n + 1}}'$.
 \begin{itemize}
  \item
  We simply invoke the third and second points of this lemma to construct $G$ and the two isomorphisms.
 \end{itemize}
 This completes the proof.
\end{proof}

Of course, the hypothesis that every section of $\stack{F}_i$ is locally isomorphic to one in $\stack{F}_i'$ is satisfied by ULA objects in any of the sheaves of categories considered in \ref{vanishing cycles gluing}, which also satisfy nearby cycles gluing.

\section{Convolutions}
\label{s:convolutions}
  
In this section we construct two kinds of convolution products: most generally, we construct an ``outer'' convolution
\begin{equation*}
 \map{*_o}{\Sph(\stack{G}_n) \times \Sph(\stack{G}_m)}{\Sph(\stack{G}_{n + m})}
\end{equation*}
and more specifically, we construct an ``inner'' convolution on the individual ULA categories:
\begin{equation*}
 \map{*_i}{\Sph^\ULA(\stack{G}_n) \times \Sph^\ULA(\stack{G}_n)}{\Sph^\ULA(\stack{G}_n)}.
\end{equation*}
So as to apply \ref{devissage}, we will work not only over products $X^n$ but over any subscheme $Y \subset X^n$.  To keep things neat, we will state definitions and theorems only for $X^n$; they can always be augmented by replacing each copy of an $X^n$ by a subscheme $Y$, and for products $\on{Gr}_n \times \on{Gr}_m$ or twisted products $\tilde{\on{Gr}}_p$ with $p \colon n + m = (n) + (m)$, the base should be replaced by a product $Y \times Y'$ with $Y' \subset X^m$. This partition will arise frequently, so we will denote it simply $(n,m)$.

\subsection*{Inner and outer convolution}

The following notion is explained much more elegantly in \cite{ginzburg}*{3.1}:

\begin{theorem}{defn}{twisted outer product}
 Let $\sh{F}_1 \in \Sph(\stack{G}_n)$, $\sh{F}_2 \in \Sph(\stack{G}_m)$; then by equivariance we have a twisted pullback $\tilde{\sh{F}}_2 \in \Sph(\tilde{\stack{G}}_m)$ on $\tilde{\on{Gr}}_{n,m}$. We define
 \begin{equation*}
  \sh{F}_1 \tbtimes \sh{F}_2 = \on{pr}_{n,m}^* \sh{F}_1 \otimes \tilde{\sh{F}}_2
   \in \Sph(\stack{G}_n \tbtimes \stack{G}_m)
 \end{equation*}
 on $\tilde{\on{Gr}}_{n,m}$, their \emph{twisted outer product} as in \ref{twisted pullback of gerbe}; it is a ($\stack{G}_n \tbtimes \stack{G}_m$)-twisted perverse sheaf by \ref{outer product of twisted sections}, evidently $G(\smash{\hat{\OO}})_{n + m}$-equivariant using \ref{G(O)_n acts on twisted Gr}.
\end{theorem}

Recalling the construction of the twisted product as descending from the product
\begin{equation*}
 \tilde{G}(\hat{\KK})_{n,m} \times_{X^m} \on{Gr}_m
\end{equation*}
by $G(\hat{\OO})_m$-equivariance, and denoting $\on{pr}_n, \on{pr}_m$ the two maps from this space to $\on{Gr}_n$ and $\on{Gr}_m$, we see that $\sh{F}_1 \tbtimes \sh{F}_2$ descends from the outer tensor product
\begin{equation*}
 \on{pr}_n^* \sh{F}_1 \otimes \on{pr}_m^* \sh{F}_2.
\end{equation*}
It must be noted that since $G(\hat{\OO})_n$ has infinite type over $X^n$, the concept of descent is problematic.  This type of problem can be remedied by the following elementary fact:

\begin{theorem}{lem}{finite dimensional quotient}
 Let $\tilde{H}$ be a group scheme, $\pi \colon \tilde{H} \to H$ a quotient group scheme, $T$ a $\tilde{H}$-torsor (over any scheme), and $S$ a scheme on which $\tilde{H}$ acts through $H$.  Then there is a natural isomorphism
 \begin{equation*}
  T \times^{\tilde{H}} S \to {}^1 \pi\, T \times^H S.
 \end{equation*}
\end{theorem}

\begin{proof}
 By definition, ${}^1 \pi\, T \times^H S$ admits a map from ${}^1 \pi\, T \times S$ equalizing the action of $H$, and likewise, $T \times^{\tilde{H}} S$ admits a map from $T \times S$.  We have a $\tilde{H}$-equivariant map $T \to {}^1\pi\, T$ and thus a map $T \times S \to {}^1 \pi\, T \times S$, so by composition a map $T \times S \to {}^1\pi\, T \times^H S$.  By construction, it equalizes the action of $\tilde{H}$, so descends to a map $T \times^{\tilde{H}} S \to {}^1\pi\, T \times^H S$. Locally this is just the map $\tilde{H} \times^{\tilde{H}} S \to H \times^H S$, which is an isomorphism, so it is an isomorphism.
\end{proof}

By \ref{relative orbit properties}, every object of $\Sph(\stack{G}_m)$ is supported on a finite union of $G(\hat{\OO})_m$-orbits on which $\smash{G(\hat{\OO})_m}$ acts through relatively finite-dimensional quotients.  Let $\sh{F}$ be such an object, supported on the finite-dimensional $G(\hat{\OO})_m$-stable subscheme $S$ on which $G(\hat{\OO})_m$ acts through the finite-dimensional quotient $H$, and let $\sh{T}$ be the $H$-torsor over $\smash{\tilde{\on{Gr}}_n}$ obtained from the $G(\hat{\OO})_m$-torsor $\tilde{G}(\hat{\KK})_{n,m}$ by changing groups to $H$.  Then $\tilde{\sh{F}}$ lives on the subset $\tilde{G}(\hat{\OO})_{n,m} \times^{G(\hat{\OO}_m)} S \subset \tilde{\on{Gr}}_{n,m}$, which by \ref{finite dimensional quotient} can be replaced by $\sh{T} \times^H S$, which is a finite-dimensional smooth quotient of $\sh{T} \times S$.  If $q$ is the quotient map, then it has relative dimension $\dim H$ and, by definition, we have
\begin{equation}
 \label{eq:finite dimensional pullback}
 q^* \tilde{\sh{F}} \cong \on{pr}_S^* \sh{F}
\end{equation}
and this is what we mean by the construction of the twisted outer product.  Note that this definition does not depend on the particular choice of $H$.

Since the equivariance structure on the $\stack{G}_n$ is \emph{strongly} factorizable, in that we have an equivalence of \emph{equivariant} gerbes
\begin{equation*}
 m_{n,m}^* \stack{G}_{n + m} \cong \stack{G}_n \tbtimes \stack{G}_m,
\end{equation*}
the following definition is possible:

\begin{theorem}{defn}{convolutions}
 For $\sh{F}_1 \in \Sph(\stack{G}_n)$, $\sh{F}_2 \in \Sph(\stack{G}_m)$, let $\sh{F}_1 \tbtimes \sh{F}_2$ be the twisted product on $\tilde{\on{Gr}}_{n,m}$.  Then the \emph{outer convoluion} is
 \begin{equation*}
  \sh{F}_1 *_o \sh{F}_2 = (m_{n,m})_* (\sh{F}_1 \tbtimes \sh{F}_2).
 \end{equation*}
 When $n = m$, we define the \emph{inner convolution} as:
 \begin{equation*}
  \sh{F}_1 *_i \sh{F}_2 = i_\Delta^* (\sh{F}_1 *_o \sh{F}_2')[-n]
 \end{equation*}
 where $\Delta \subset X^n \times X^n$ is the diagonal embedding of $X^n$.  Equivalently, if we replace $\sh{F}_1 \tbtimes \sh{F}_2$ by its restriction to $\on{Conv}_n^2$, then by proper base change we have $\sh{F}_1 *_i \sh{F}_2 = (m_n)_*(\sh{F}_1 \tbtimes \sh{F}_2)$ in this sense.  \emph{A priori} these exist only in the twisted derived category.
\end{theorem}

It is also possible to make the following definition that more intuitively resembles the concept of convolution.  We write
\begin{align*}
 q_n \colon G(\hat{\KK})_n \to \on{Gr}_n &&
 q_{n,n} \colon G(\hat{\KK})_n \times^{G(\hat{\OO})_n} G(\hat{\KK})_n \to \on{Conv}_n^2
\end{align*}
where the first is the quotient map of \ref{arc and loop group actions} and the second is that of \ref{inner convolution diagram}.  Multiplication in $G(\hat{\KK})_n$ descends along the quotient to give a map
\begin{equation*}
 m_n' \colon
 G(\hat{\KK})_n \times^{G(\hat{\OO})_n} G(\hat{\KK})_n
 \to G(\hat{\KK})_n.
\end{equation*}
Then there are two inner convolution products
\begin{align*}
 \sh{F}_1' *_*' \sh{F}_2'
 = (m_n')_* (\sh{F}_1' \tbtimes \sh{F}_2'),
 &&
 \sh{F}_1' *_!' \sh{F}_2'
 = (m_n')_! (\sh{F}_1' \tbtimes \sh{F}_2'),
 &&
 \bigl(\sh{F}_i' \in \Perv(q_n^* \stack{G}_n)\bigr),
\end{align*}
where $\sh{F}_1' \tbtimes \sh{F}_2'$ is the descent of $\sh{F}_1' \boxtimes \sh{F}_2'$ along the quotient.

In general, these operations are in fact as nice as one could want.

\begin{theorem}{prop}{convolution properties}
 \mbox{}
 \begin{enumerate}
  \item \label{en:convolution is perverse}
  In general, $\sh{F}_1 *_o \sh{F}_2 \in \Sph(\stack{G}_{n + m})$, and if the $\sh{F}_i$ are ULA, we have the following identity, called the ``fusion product'':
  \begin{equation*}
   \sh{F}_1 *_o \sh{F}_2 = j_{!*} (\sh{F}_1 \boxtimes \sh{F}_2)|_U
  \end{equation*}
  where $j$ is the inclusion of $U = X^{n + m}_{n,m}$ and we have used the factorization isomorphism $\on{Gr}_{n + m} \cong \on{Gr}_n \times \on{Gr}_m$ on $U$; both inner and outer convolution are again ULA. Furthermore, both convolutions are exact.
  
  \item \label{en:algebraic constraints}
  For both convolutions there are associativity constraints $\sh{F}_1 * (\sh{F}_2 * \sh{F}_3) \cong (\sh{F}_1 * \sh{F}_2) * \sh{F}_3$, and when $\sh{F}_1, \sh{F}_2 \in \Sph^\ULA(\stack{G}_n)$, there is a commutativity constraint $\sh{F}_1 *_i \sh{F}_2 \cong \sh{F}_2 *_i \sh{F}_1$.
  
  \item \label{en:convolution descent}
  We have
  \begin{equation*}
   q_n^*(\sh{F}_1 *_i \sh{F}_2)
    \cong q_n^* \sh{F}_1 *_*' q_n^* \sh{F}_2
    \cong q_n^* \sh{F}_1 *_!' q_n^* \sh{F}_2
    \overset{\text{def}}= q_n^* \sh{F}_1 *_i' q_n^* \sh{F}_2
  \end{equation*}
 \end{enumerate}
\end{theorem}

\begin{proof}
 We note that $\sh{F}_1 *_o \sh{F}_2$ is equivariant by general nonsense.  To show that convolution is exact and respects perversity, we use \ref{devissage} for the functor $\Sph(\stack{G}_n) \times \Sph(\stack{G}_m) \to \Sph(\stack{G}_{n + m})$.  All three categories admit gluing over, respectively, $X^m$, $X^n$, and $X^{n + m}$, and convolution is compatible with nearby cycles (\ref{compatible with nearby cycles}). Indeed, $(m_{n, m})_* = (m_{n,m})_!$ is a proper pushforward and, by \ref{horizontal functors}, therefore horizontal, while outer tensor product (and therefore twisted outer product, since by \cite{BBD}*{Proposition 4.2.5} smooth pullbacks are faithful) is clearly horizontal. Thus, we suppose that $\sh{F}_1$ and $\sh{F}_2$ are both ULA.
 
 Since the ULA condition is local in the smooth topology and stable under outer tensor products, the twisted product $\sh{F}_1 \tbtimes \sh{F}_2$ is ULA; since $m_{n,m}$ is proper, by \ref{ULA facts} the pushforward $\sh{F}_1 *_o \sh{F}_2$ is ULA, and therefore we have by \ref{ULA nearby cycles}:
 \begin{equation*}
  \sh{F}_1 *_o \sh{F}_2
   = j_{!*} (\sh{F}_1 *_o \sh{F}_2)|_U
   = j_{!*} (\sh{F}_1 \boxtimes \sh{F}_2)|_U,
 \end{equation*}
 using the identification of $\tilde{\on{Gr}}_{n,m}$ with $\on{Gr}_n \times \on{Gr}_m$ over $U$ (\ref{convolution diagram factorization}) and the fact that $m_{n,m}$ is an equality there.
 
 To analyze inner convolution, we apply $i_\Delta^*[-n]$ by applying successive $i_{\Delta_{ij}}^*[-1]$, where the diagonals $\Delta_{ij}$ are chosen so that they intersect in $\Delta$.  By \ref{ULA nearby cycles}, each successive restriction is an application of nearby cycles and preserves perversity and the ULA property.  We obtain the commutativity constraint on $*_i$ by swapping the first and last set of $n$ coordinates in $X^{2n}$ before restricting, obtaining an isomorphism
 \begin{equation*}
  \mathrm{sw}^* j_{!*} (\sh{F}_1 \boxtimes \sh{F}_2)|_U
   \cong j_{!*} (\sh{F}_2 \boxtimes \sh{F}_1)|_U.
 \end{equation*}
 Since $\Delta$ is invariant under the swapping map, both sides become equal to an inner convolution when we apply $i_\Delta^*$, as desired.
 
 Finally, to obtain the associativity constraint it suffices, again by \ref{devissage} (applied to the two functors $(\sh{F}_1 *_o \sh{F}_2) *_o \sh{F}_3$ and $\sh{F}_1 *_o (\sh{F}_2 *_o \sh{F}_3)$, which are both compatible with nearby cycles), to do so for ULA sheaves. Since both orders of association are isomorphic to
 \begin{equation*}
  j_{!*} (\sh{F}_1 \boxtimes \sh{F}_2 \boxtimes \sh{F}_3)|_U
 \end{equation*}
 with $U \subset X^{n + m + l}$ the open set where none of the first $n$, next $m$, and last $l$ coordinates are equal, for ULA sheaves associativity follows from that of the tensor product. This isomorphism is obviously compatible with nearby cycles.  By restriction, we also obtain associativity for inner convolution.
 
 For \ref{en:convolution descent}, it is enough to observe that the ``multiplication'' map $m_{n,m}$ on the twisted grassmannian descends from $m_{n,m}'$ using \ref{inner convolution multiplication} and \ref{outer convolution diagram trick}, and that convolution product of twisted sheaves on $\smash{G(\hat{\KK})_n}$ is possible via the construction \ref{eq:twisted equivariant pairing} using the sf multiplicative structure \ref{equivariance and multiplicativity}.  The $!$ and $*$ versions are equal since $m_n'$ is the base change of the proper map $m_n$.
\end{proof}

\subsection*{Rigidity of inner convolution}

Finally, the inner convolution operation defines the structure of a rigid tensor category on $\Sph^\ULA(\stack{G}_n)$.  We will need the following generalization of the unit sections of \ref{existence of equivariance}: there are closed immersions $\map{i_{n,m;1}}{\on{Gr}_n \times X^m}{\on{Gr}_{n + m}}$ for every $n, m$, corresponding to the data
\begin{equation*}
 ((\vect{x}, \sh{T}, \phi), \vect{y})
  \mapsto (\vect{x} \cup \vect{y}, \sh{T}, \phi|_{X_S \setminus (\bar{x} \cup \bar{y})})
\end{equation*}
where as usual, $\sh{T}$ is a $G$-torsor on $X_S$ and $\phi$ is its trivialization on the complement of the graphs of the coordinates of $\map{\vect{x}}{S}{X^n}$.  In fact, these maps are the same as the restrictions of
\begin{equation*}
 \map{m}{\tilde{\on{Gr}}_{n,m}}{\on{Gr}_{n + m}}
\end{equation*}
to the subset $\on{Gr}_n * \on{Gr}_m^0$, in the notation of \ref{s:orbits in the affine grassmannian}, so that $i_{n,m;1}^* \stack{G}_{n + m}$ and $\stack{G}_n \cong \smash{\stack{G}_n \tbtimes (\stack{G}_n|\on{Gr}_G^0)}$ are naturally equivalent.  Likewise, we have maps $i_{n,m;2}$ inserting $X^n$ in the first $n$ coordinates and $\on{Gr}_m$ in the last $m$.

\begin{theorem}{prop}{convolution is rigid}
 There is an element $1_n \in \Sph^\ULA(\stack{G}_n)$ which is a unit for $*_o$ in the sense that for $\sh{F}_1 \in \Sph(\stack{G}_n)$ and $\sh{F}_2 \in \Sph(\stack{G}_m)$, we have
 \begin{align*}
  \sh{F}_1 *_o 1_m \cong (i_{n,m;1})_* \sh{F}_1, &&
  1_n *_o \sh{F}_2 \cong (i_{n,m;2})_* \sh{F}_2
 \end{align*} 
 The unit sheaf $1_n$ is also an identity for $*_i$.  Each object $\sh{F} \in \Sph^\ULA(\stack{G}_n)$ has a dual object $\sh{F}^*$, in that
 \begin{equation*}
  \on{Hom}(\sh{F}_1 *_i \sh{F}_2, \sh{F}_3) \cong \on{Hom}(\sh{F}_1, \sh{F}_2^* *_i \sh{F}_3);
 \end{equation*}
 we have $(\sh{F}^*)^* \cong \sh{F}$ and $(\sh{F}_1 *_i \sh{F}_2)^* \cong \sh{F}_2^* *_i \sh{F}_1^*$.
\end{theorem}

\begin{proof}
 We remark first that since $\stack{G}_n$ is naturally trivialized on the unit section $X^n \to \on{Gr}_n$ by \ref{existence of equivariance}, the ordinary perverse sheaf $\csheaf{k}[n]$ on $X^n$ is $\stack{G}_n$-twisted and ULA and so if we denote $u = i_{0,n;1} = i_{0,n;2}$, we have $1_n = u_*(\csheaf{k}[n]) \in \Sph^\ULA(\stack{G}_n)$ because a closed immersion is proper (applying \ref{ULA facts}).  Then the twisted product $\sh{F}_1 \tbtimes 1_m$ lives on the convolution diagram
 \begin{equation*}
  \on{Gr}_n * \on{Gr}_m^0 \subset \tilde{\on{Gr}}_{n,m};
 \end{equation*}
 however, this subset is identified by the projection $\on{pr}_n$ with $\on{Gr}_n \times X^m$, and on it the twisted product is just the pullback of $\sh{F}_1$ from the first factor.  Thus, its pushforward along $m$, which is $\sh{F}_1 * 1_m$, is the same as $(i_{n,m;1})_* \sh{F}_1$, as claimed; the same argument works for $1_n * \sh{F}_2$.  To see that $1_n$ is an identity for $*_i$, just restrict to the appropriate diagonal, which converts $\tilde{\on{Gr}}_{n,n}$ into $\on{Conv}_n^2$ by \ref{inner convolution diagram diagonal}\ref{en:inner diagram diagonal}.
 
 We define dualization in $\Sph^\ULA(\stack{G}_n)$ as follows.  Let $\on{Gr}_n^0 \subset \on{Gr}_n$ denote the image of the unit section $i_{0,n;1} = i_{0,n;2}$, and write $\map{m}{\on{Conv}_n^2}{\on{Gr}_n}$ for the multiplication map.  Then $m^{-1}(\on{Gr}_n^0) \cong \on{Gr}_n$; indeed, the isomorphism is given by the projection $\map{\on{pr}}{\on{Conv}_n^2}{\on{Gr}_n}$ and, explicitly, identifies data (recall \ref{convolution product grassmannian})
 \begin{equation*}
  (\vect{x}, \sh{T}_1, \sh{T}_2, \phi_1, \phi_2)
   \cong (\vect{x}, \sh{T}, \phi)
 \end{equation*}
 where the fact that the first datum is in $m^{-1}(\on{Gr}_n^0)$ means, by definition, that $\phi_2 \circ \phi_1$ extends to a trivialization of $\sh{T}_2$.  Thus, we may identify $\phi_1 = \phi_2^{-1} = \phi$ and $\sh{T} = \sh{T}_1$.  Let $\map{i}{m^{-1}(\on{Gr}_n^0)}{\on{Conv}_n^2}$; we set
 \begin{equation*}
  \sh{F}^* = \DD(i^* \tilde{\sh{F}}),
 \end{equation*}
 where as always, $\tilde{\sh{F}}$ is the twisted pullback of $\sh{F}$ to $\on{Conv}_n^2$.  To see that this is perverse, let $\map{\on{inv}_n}{\smash{G(\hat{\KK})_n}}{\smash{G(\hat{\KK})_n}}$ be the inversion map.  Then by the above computation of $m^{-1}(\on{Gr}_n^0)$, we have $q_n^* \sh{F}^* \cong \DD \on{inv}_n^* q_n^* \sh{F}$ (where $q_n \colon \smash{G(\hat{\KK})_n} \to \on{Gr}_n$ is the quotient map, as usual; by descent, $\sh{F}^*$ is perverse.  It also follows from the expression that $(\sh{F}^*)^* \cong \sh{F}$, since $\DD \on{inv}_n^* = \on{inv}_n^* \DD$.  Note for later that $\on{inv}_n^* = (\on{inv}_n)_* = (\on{inv}_n)_!$; we denote their common value just by $\on{inv}_n$.  All of these claims follow from the fact that $\on{inv}_n$ is an automorphism.
 
 We consider the identity $(\sh{F}_1 *_i \sh{F}_2)^* \cong \sh{F}_2^* *_i \sh{F}_1^*$.  To prove it, we use the expression of \ref{convolution properties}\ref{en:convolution descent}:
 \begin{equation*}
  q_n^* (\sh{F}_1 *_i \sh{F}_2)
   \cong q_n^* \sh{F}_1 *_i' q_n^* \sh{F}_2
   \cong (m_n')_{*,!} (q_n^* \sh{F}_1 \tbtimes q_n^* \sh{F}_2).
 \end{equation*}
 To apply the dualization formula above, we note that there is a well-defined automorphism
 \begin{align*}
  \on{inv}_{n,n}
   \colon G(\hat{\KK})_n \times^{G(\hat{\OO})_n} G(\hat{\KK})_n
      \to G(\hat{\KK})_n \times^{G(\hat{\OO})_n} G(\hat{\KK})_n,
  &&
  (g, h) \mapsto (h^{-1}, g^{-1})
 \end{align*}
 so that we have, for any $\sh{F}_i \in \Perv(q_n^* \stack{G}_n)$,
 \begin{equation*}
  \on{inv}_{n,n} (\sh{F}_1' *_i' \sh{F}_2')
   = \on{inv}_n \sh{F}_2' *_i' \on{inv}_n \sh{F}_1',
 \end{equation*}
 using $(\on{inv}_n)_* = (\on{inv}_n)_! = \on{inv}_n$.  Applying duality to both sides and taking $\sh{F}_i' = q_n^* \sh{F}_i$, we get the identity by descent.
 
 Let $\map{u}{\on{Gr}_n^0}{\on{Gr}_n}$ be the inclusion of the ``unit'' section, so $1_n = u_* \csheaf{k}[n]$. Then we have for any $\sh{F}_1, \sh{F}_2 \in \Sph^\ULA(\stack{G}_n)$:
  \begin{multline*}
  \on{Hom}(\sh{F}_1 *_i \sh{F}_2, 1_n) \\
    = \on{Hom}(m_*(\on{pr}^* \sh{F}_1 \otimes \tilde{\sh{F}}_2), 1_n)
    = \on{Hom}(\on{pr}^* \sh{F}_1 \otimes \tilde{\sh{F}}_2, m^! 1_n) \\
    = \on{Hom}(\on{pr}^* \sh{F}_1, \shHom(\tilde{\sh{F}}_2, m^! u_* \csheaf{k}[n])).
 \end{multline*}
 But $m^! u_* \csheaf{k}[n] = i_* m^! \csheaf{k}[n] = i_* \sh{D}$, where $\sh{D}$ is the dualizing sheaf $m^{-1}(\on{Gr}_n^0)$.  Then the right operand is simply
 \begin{equation*}
  \shHom(\tilde{\sh{F}}_2, i_* \sh{D})
   = i_* \shHom(i^* \tilde{\sh{F}}_2, \sh{D})
   = i_* \DD (i^* \tilde{\sh{F}}_2)
   = i_* \sh{F}_2^*.
 \end{equation*}
 Finally, then, we have
 \begin{equation*}
  \on{Hom}(\sh{F}_1 *_i \sh{F}_2, 1_n)
   = \on{Hom}(\on{pr}^* \sh{F}_1, i_* \sh{F}_2^*)
   = \on{Hom}(i^* \on{pr}^* \sh{F}_1, \sh{F}_2^*)
   = \on{Hom}(\sh{F}_1, \sh{F}_2^*),
 \end{equation*}
 as desired.  If we have a third sheaf $\sh{F}_3$, then formally:
 \begin{multline*}
  \on{Hom}(\sh{F}_1 *_i \sh{F}_2, \sh{F}_3)
   \cong \on{Hom}(\sh{F}_1 *_i \sh{F}_2, (\sh{F}_3^*)^*) \\
   \cong \on{Hom}((\sh{F}_1 *_i \sh{F}_2) *_i \sh{F}_3^*, 1_n)
   \cong \on{Hom}(\sh{F}_1 *_i (\sh{F}_2 *_i \sh{F}_3^*), 1_n) \\
   \cong \on{Hom}(\sh{F}_1, (\sh{F}_2 *_i \sh{F}_3^*)^*)
   \cong \on{Hom}(\sh{F}_1, \sh{F}_2^* *_i \sh{F}_3),
 \end{multline*}
 as desired.  
\end{proof}

This definition of the dual (phrased differently) was given in \cite{ginzburg}*{\S2.4} and reprised
in \cite{MV_Satake}*{(11.10)}.

\section{The fiber functor}
\label{s:fiber functor}
   
In this section, we will produce a map from spherical sheaves on $\on{Gr}_{G,X^n}$ to those on $\on{Gr}_{T,X^n}$ which will turn out to retain all of the information about the former.  We recall the fundamental diagram 
\begin{equation}
 \label{eq:fiber functor diagram}
 \on{Gr}_{G,X^n}
 \xleftarrow{b}
 \on{Gr}_{B,X^n}
 \xrightarrow{t}
 \on{Gr}_{T,X^n}
\end{equation}
corresponding to any choice of Borel subgroup $B \subset G$ and its quotient torus $T = B/N$. 
We also recall the indexing of the irreducible components of $\on{Gr}_{B,X^n}$ by $\Lambda_T^n$,
which is the same as that of $\on{Gr}_{T,X^n}$; the \emph{connected} component containing
$\on{Gr}_{B,X^n}^{\lambda_1, \dots, \lambda_n}$ depends only on the sum of the $\lambda_i$. If
$\stack{G}_n$ is an sf gerbe on $\on{Gr}_{G,X^n}$, then we let $\stack{T}_n$ be the unique sf
gerbe on $\on{Gr}_{T,X^n}$ such that $t^* \stack{T}_n \cong b^* \stack{G}_n$, as in \ref{eq:group
to torus functor}.

We use the notation, standard in representation theory,
\begin{equation*}
 2\rho = \sum \alpha,
\end{equation*}
where the $\alpha$'s run over all positive roots of $G$.

\subsection*{The maximal fiber functor}

\begin{theorem}{defn}{fiber functor defn}
 The \emph{fiber functor} $\map{F_n}{\Sph(\stack{G}_n)}{\stack{D}(\stack{T}_n)}$ is either one of the following two functors (which we show are isomorphic):
 \begin{itemize}
  \item
  The functor $(F_n)^*_!$, where for $\sh{F} \in \Sph(\stack{G}_n)$, we let $\sh{F}^\lambda$ be the restriction of $b^* \sh{F}$ to the connected component $\on{Gr}_{B,X^n}^\lambda$, and set
  \begin{equation*}
   (F_n)^*_!(\sh{F})^\lambda
    = t_!^\lambda \sh{F}^\lambda[\langle 2\rho, \lambda \rangle]
  \end{equation*}
  to be the component of $(F_n)^*_!(\sh{F})$ on $\on{Gr}_{T,X^n}^\lambda$.
  
  \item
  The functor $(F_n)^!_*$, replacing $b^*$ with $b^!$ and $t_!$ with $t_*$.
 \end{itemize}
 We observe that for ULA sheaves, we have
 \begin{equation*}
  (F_n)^*_!(\sh{F})^\lambda
   = \DD (F_n)^!_* (\DD \sh{F})^{-\lambda}
   = \DD (F_n)^!_*(\sh{F}^*)^\lambda,
 \end{equation*}
 and set $F = F_*^!$.
\end{theorem}

\begin{theorem}{prop}{fiber functor properties}
 \mbox{}
 \begin{enumerate}
  \item \label{en:fiber independence}
  Neither $(F_n)_!^*$ nor $(F_n)_*^!$ depends on the choice of Borel or torus in $G$, and they are isomorphic.
  
  \item \label{en:fiber exactness}
  The functor $F_n$ is $t$-exact, faithful, and preserves the ULA property.
  
  \item \label{en:fiber convolution}
  There are natural isomorphisms
  \begin{align*}
   F_{n + m}(\sh{F}_1 *_o \sh{F}_2) &\cong F_n(\sh{F}_1) *_o F_m(\sh{F}_2)
   \\
   F_n(\sh{F}_1 *_i \sh{F}_2) &\cong F_n(\sh{F}_1) *_i F_n(\sh{F}_2)
  \end{align*}
  (the latter when the sheaves are ULA) compatible with the associativity constraints.
 \end{enumerate}
\end{theorem}

\begin{proof}
 For \ref{en:fiber independence}, note that alternate choices of $B$ and $T$ are obtained by conjugation; i.e.\ $B' = g^{-1} B g$, $T' = g^{-1} T g$ for some $g \in G(k)$.  On the affine grassmannian, this means that the various possible $b$ are permuted by translation by $G(\C) \times X^n \subset G(\hat{\OO})_n$, and so applied to spherical objects their pullbacks are the same.  The identity $(F_n)_!^* = (F_n)_*^!$ follows from \cite{B_Hyperbolic} (as in \cite{MV_Satake}*{Theorem 3.5}).
  
 To show that it is faithful, the argument in \cite{MV_Satake} also works: if the support of $\sh{F} \in \Sph(\stack{G}_n)$ intersects $\on{Gr}_{B,X^n}^\lambda$ for some $\lambda$ which is minimal with respect to the ordering of coweights by positive coroots, then this intersection is a single point in each fiber over $X^n$ and the stalk of $F_n(\sh{F})^\lambda$ is the stalk of $\sh{F}$ at that point.
 
 To show that $F_n$ preserves the ULA property, we do a formal computation.  For $\sh{G}$ any complex on $X^n$ and $f_G, f_B, f_T$ the structure maps $\on{Gr}_{G,X^n} \to X^n$, etc., we have up to shifts and indexing:
 \begin{align}
  \label{eq:fiber_functor_ula}
  (F_n)_!^*(\sh{F} \otimes f_G^* \sh{G})
    &= t_! b^* (\sh{F} \otimes f_G^* \sh{G})
    = t_! (b^* \sh{F} \otimes f_B^* \sh{G})
    = t_! (b^* \sh{F} \otimes t^* f_T^* \sh{G}) \notag \\
    &= (t_! b^* \sh{F}) \otimes f_T^* \sh{G}
    = (F_n)_!^* \sh{F} \otimes f_T^* \sh{G}.
 \end{align}
 The crucial step was an application of the projection formula for $t$.  An application of the dual projection formula $t_*(b^! \sh{F} \xtimes t^! f_T^! \sh{G}) = (t_* b^! \sh{F}) \xtimes f_T^! \sh{G}$ gives
 \begin{equation}
  \label{eq:fiber_functor_ula'}
  (F_n)_*^!(\sh{F} \xtimes f_G^! \sh{G})
    = (F_n)_*^! \sh{F} \xtimes f_T^! \sh{G}.
 \end{equation}
 Since $(F_n)_*^! = (F_n)_!^*$, shifting this by $2n$ and equating the left sides by the ULA property of $\sh{F}$, one concludes the equation of the ULA property for $F_n(\sh{F})$.
 
 To show that $F_n$ is t-exact, we perform some reductions and invoke a few results that will be proven in later sections (they do not depend on this proposition). First, by \ref{devissage}, it suffices to show that it sends ULA perverse sheaves to perverse sheaves.  Indeed, by \ref{horizontal functors}, $F_n$ (as the composition both of a $!$ pullback and $*$ pushforward and a $*$ pullback and $!$ pushforward) is horizontal, hence compatible with nearby cycles by \ref{compatible with nearby cycles}.  Since it actually preserves the ULA property, by \ref{ULA nearby cycles} it suffices to prove that each $F_n(\sh{F})$ is perverse on the complement of all diagonals in $X^n$.  This is a local question, so by \ref{ULA local equivalence} it suffices to show that $F_n(\sh{F})|_{\vect{x}}[-n]$ is perverse for some (any) point $\vect{x} \in X^n$. Since the fiber of $\on{Gr}_{G,X^n}$ is $\on{Gr}_G^n \cong \on{Gr}_{G^n}$ this is then a consequence of \ref{absolute fiber properties}.
 
 Finally, $F_n$ preserves convolution.  As usual, \ref{devissage} allows us to assume that we are applying it to the convolution of two ULA sheaves, since both $F_n$ and convolution are compatible with nearby cycles.  Now we use the fact that
 \begin{equation*}
  \sh{F}_1 *_o \sh{F}_2 = j_{!*} (\sh{F}_1 \boxtimes \sh{F}_2)|_U
 \end{equation*}
 from \ref{convolution properties}\ref{en:convolution is perverse}, where both sides are ULA. 
 Obviously,
 \begin{equation*}
  F_{n + m}|_U(\sh{F}_1 \boxtimes \sh{F}_2)|_U = (F_n(\sh{F}_1) \boxtimes F_m(\sh{F}_2))|_U
 \end{equation*}
 and since $F_k$ preserves the ULA property, we have
 \begin{multline*}
  F_{n + m}(\sh{F}_1 *_o \sh{F}_2)
    = j_{!*}(F_{n + m}(\sh{F}_1 *_o \sh{F}_2)|_U)
    = j_{!*}(F_n(\sh{F}_1) \boxtimes F_m(\sh{F}_2))|_U \\
    = F_n(\sh{F}_1) *_o F_m(\sh{F}_2).
 \end{multline*}
 This is what we want.
\end{proof}

The proof would likewise show that $F_n$ preserves the commutativity constraint of $*_i$ except that the shift operator in $\stack{D}$, which appears in its definition, is \emph{not} a tensor functor but rather a graded tensor functor, where the grading is $(-1)^{ij}$ for the product of complexes in degrees $i$ and $j$.  In particular, as defined, the commutativity constraints on $\on{Gr}_{G,X^n}$ and on $\on{Gr}_{T,X^n}$ differ by a factor of
\begin{equation*}
 (-1)^{\langle 2\rho, \lambda \rangle \langle 2\rho, \mu \rangle}
\end{equation*}
for the convolution of weight spaces supported on $\on{Gr}_{T,X^n}^\lambda$ and $\on{Gr}_{T,X^n}^\mu$.  Note that the function $\smash{\on{Gr}_{G,X^n}^\lambda} \mapsto \langle 2\rho, \lambda \rangle \pmod 2$ is constant on connected components of $\on{Gr}_{G,X^n}$ by \ref{absolute orbit properties}, so we may modify the natural commutativity constraint on $*_i$ by the above factor, when two sheaves supported on the appropriate components are convolved.  Thus, we conclude:

\begin{theorem}{cor}{fiber functor commutative}
 From this point on, let $\Sph^\ULA(\stack{G}_n)$ denote the same category as usual but whose tensor structure has its commutativity constraint modified as above.  Then the fiber functor
 \begin{equation*}
  \map{F_n}{\Sph^\ULA(\stack{G}_n)}{\Sph^\ULA(\stack{T}_n)}
 \end{equation*}
 is a tensor functor.
\end{theorem}

\subsection*{Parabolic fiber functors}

We will need some generalizations of $F_n$, described in \cite{BD_quantization}*{5.3.27--31}.  Let $P$ be a parabolic subgroup of $G$ with unipotent radical $M$ and Levi quotient $L = P/M$, and recall the diagram \ref{eq:parabolic diagram}:
\begin{equation*}
 \on{Gr}_{G,X^n} \xleftarrow{p} \on{Gr}_{P,X^n} \xrightarrow{l} \on{Gr}_{L,X^n}
\end{equation*}
First of all, it is easily verified that
\begin{equation}
 \label{eq:parabolic pullback}
 \tikzsetnextfilename{parabolic_fiber_square}
 \begin{tikzpicture}[baseline = (current bounding box),row sep = 4em, column sep = 5em]
  \matrix[matrix of math nodes] (objects)
  {
   \on{Gr}_{P, X^n} & \on{Gr}_{B, X^n} \\
   \on{Gr}_{L, X^n} & \on{Gr}_{B_L,X^n} \\
  };
  { [math arrows]
   \draw (objects-1-2) -- node {f}   (objects-1-1);
   \draw (objects-1-2) -- node {g}   (objects-2-2);
   \draw (objects-1-1) -- node {l}   (objects-2-1);
   \draw (objects-2-2) -- node {b_L} (objects-2-1);
  }
 \end{tikzpicture}
\end{equation}
is a pullback diagram, where $B_L = B/N$ is a Borel subgroup in $L$, $b_L$ is the $b$ corresponding to the reductive group $L$, $f$ is the map corresponding to the inclusion $B \to P$, and $g$ to the projection $B \to B_L$. Thus, the components $\on{Gr}_{B,X^n}^\lambda$ sent by $f$ into the same component of $\on{Gr}_{P,X^n}$ are indexed by those $\lambda$ for which $b_L$ sends $\on{Gr}_{B_L,X^n}$ to the same component of $\on{Gr}_{L,X^n}$; namely, they are indexed by $\Lambda/\Lambda_L$, where $\Lambda_L$ is the coroot lattice of $L$.  Thus, we can make the following definition.

\begin{theorem}{defn}{parabolic fiber}
 With notation as above, also let $2\rho_L$ be the sum of the positive roots in $L$ (with respect to $B_L$), and $2\rho_{G/L} = 2\rho - 2\rho_L$.  For $\sh{F} \in \Sph(\stack{G}_n)$, let $p^* \sh{F} = \bigoplus_{\Lambda/\Lambda_L} \sh{F}^{P,\lambda}$ with the summands supported on the $\lambda$-part of $\on{Gr}_{P,X^n}$.  Then the $\lambda$-weight space of the \emph{parabolic fiber functor} $F^P_n$ is
 \begin{equation*}
  (F^P_n)_!^*(\sh{F})_\lambda
    = l_! \sh{F}^{P,\lambda} [\langle 2\rho_{G/L}, \lambda \rangle]
 \end{equation*}
 where the shift does not depend on the representative of $\lambda$ in $\Lambda$.  Likewise, we define $(F^P_n)_*^!$.
\end{theorem}

\begin{theorem}{prop}{parabolic fiber functor properties}
 \mbox{}
 \begin{enumerate}
  \item \label{en:parabolic independence}
  $(F_n)_!^* \circ (F^P_n)_!^* = (F_n)_!^*$ and $(F_n)_*^! \circ (F^P_n)_*^! = (F_n)_*^!$, where the left-hand $F_n$ refers to that defined for the reductive group $L$.  In particular, $(F^P_n)_!^* = (F^P_n)_*^!$. 
  
  \item \label{en:parabolic functorial}
  Let $P_1 \subset P_2$ be two parabolic subgroups of $G$, with corresponding fiber functors $F^{P_1}_n, F^{P_2}_n$, and let $P_L = P_1 \cap L_2$; then $F^{P_1}_n = F^{P_L}_n \circ F^{P_2}_n$.
  
  \item \label{en:parabolic properties}
  $F^P_n$ is t-exact, faithful, and respects convolution and the ULA property.
 \end{enumerate}
\end{theorem}

\begin{proof}
 For \ref{en:parabolic independence}, just use the diagram \ref{eq:parabolic pullback} and the fact that $b_L^* l_! = g_! f^*$ and likewise for $*$ and $!$ reversed to get the first claim. Since $F_n$ is faithful and sends the natural map (defined in \cite{B_Hyperbolic}) $(F^P_n)_!^* \to (F^P_n)_!^*$ to an isomorphism, that map is an isomorphism; likewise, it preserves the isomorphism property of the natural map in \ref{ULA}. Item \ref{en:parabolic functorial} is proven using the evident two-parabolics version of \ref{eq:parabolic pullback} in the same way. For \ref{en:parabolic properties}, the first two claims follow from $F_n \circ F^P_n = F_n$ and the fact that $F_n$ is t-exact and faithful. Since $p$ is proper, the last two are proven exactly as in \ref{fiber functor properties}.
\end{proof}

It is easily verified that the grading of the commutativity constraint which makes $F_n$ a tensor functor also makes each $F^P_n$ a tensor functor.

\section{The main theorem for a torus}
\label{s:main theorem for a torus}

We combine the constructions already given to establish the main theorem in the special case when $G = T$ is a torus; as usual, we let $\stack{T}_n$ be the sf $k^*$-gerbe on $\on{Gr}_{T,X^n}$. First, we complete the correspondence begun in \ref{multiplicative factorizable lattice gerbes} by describing it for $\Sph(\stack{T}_n)$.

\begin{theorem}{prop}{bilinear form triviality}
 There exists an object $\sh{F} \in \Sph^\ULA(\stack{T}_n)$ living on $\on{Gr}_{T,X}^{\lambda_1, \dots, \lambda_n}$ only if we have $\kappa(\lambda_i, \mu) = 1$ for all $\mu \in \Lambda_T$, where $\kappa$ is the bilinear form defined by $\stack{T}_n$.
\end{theorem}

\begin{proof}
 If such a twisted sheaf exists, then by \ref{transitive equivariant objects}, $\stack{T}_n^{\lambda_1, \dots, \lambda_n}$ is equivariantly trivial.  By the construction of \ref{existence of equivariance}, we have $\kappa(\lambda_i, \mu) = 1$ for all $\mu$.
\end{proof}

Let $Q$ be the quadratic form associated with $\stack{T}_n$ and denote by $\Lambda_Q$ the kernel of its associated bilinear form $\kappa$; we define $\smash{\check{T}_Q}$ to be the torus whose weight lattice is $\Lambda_Q$ and $\smash{{}^L \check{T}_Q}$ its dual.  Then by the above we may consider $\Sph(\stack{T}_n)$ to consist of sheaves on $\on{Gr}_{{}^L \check{T}_Q, X^n}$, where $\stack{T}_n$ is equivariantly trivial and so its spherical twisted sheaves coincide simply with $\Perv(\stack{T}_n)$.  By \ref{multiplicative factorizable}, $\stack{T}_n$ is a commutative multiplicative sf $k^*$-gerbe on this restricted grassmannian, and so by \ref{multiplicative factorizable lattice gerbes} corresponds to some sf comultiplicative gerbe $\stack{Z}_n$ for the sheaf of groups $\on{Fact}(\check{T}_Q(k))_n$ (previously called $\check{T}_Q(k^*)$ by abuse of notation, but the definitions are the same).

\begin{theorem}{prop}{main theorem for torus}
 We have $\Sph(\stack{T}_n) \cong \stack{Z}_n \otimes \cat{FRep}_n(\check{T}_Q)$, the equivalence pairing $\Sph^\ULA(\stack{T}_n)$ with lisse sheaves.
\end{theorem}

\begin{proof}
 We will write $\map{p_n}{\on{Gr}_{T,X^n}}{X^n}$ for the structure map of the grassmannian.  To construct an equivalence $\Sph(\stack{T}_n) \to \stack{Z}_n \otimes \cat{FRep}_n(\check{T}_Q)$, it suffices by \ref{devissage} to construct a horizontal functor $\Sph^\ULA(\stack{T}_n) \to \stack{Z}_n \otimes \cat{FRep}_n^\ULA(\check{T}_Q)$ which is an equivalence.  Let $\sh{F} \in \Sph(\stack{T}_n)$, so that the topological structure of $\on{Gr}_{T,X^n}$ given in \ref{torus grassmannian properties} allows us to write, for any partition $p$ of $n$ with $m$ parts, a direct sum decomposition
 \begin{equation*}
  \sh{F}|_{X^n_p}
    \cong \bigoplus_{\lambda_i \in X^*(\check{T}_Q)} \sh{F}^{\lambda_1, \dots, \lambda_m}_p
 \end{equation*}
 such that if $p'$ refines $p$, then its corresponding decomposition refines that of $p$ over $X^n_{p'}$.  However, if $\sh{F}$ is ULA, then by \ref{ULA nearby cycles} it is functorially determined by its restriction to the complement of all diagonals in $X^n$, and thus we can simply write
 \begin{equation*}
  \sh{F} \cong \bigoplus \sh{F}^{\lambda_1, \dots, \lambda_n}
 \end{equation*}
 with each component sheaf supported on $\on{Gr}_{T,X^n}^{\lambda_1, \dots, \lambda_n}$.  This gives $\sh{F}$ an action of $(\check{T}_Q)^n$ via a grading by its weight lattice.  Furthermore, since $\check{T}_Q^n$ acts on $\sh{F}^{\lambda_1, \dots, \lambda_n}$ through the character $(\lambda_1, \dots, \lambda_n$), the $\stack{T}_n$-twisting on $\sh{F}$ becomes a $\stack{Z}_n$-twisting according to the construction given in \ref{multiplicative factorizable lattice gerbes}, and so we have the desired functor, equal to $(p_n)_*$:
 \begin{equation*}
  \Sph^\ULA(\stack{T}_n)
    \to \stack{Z}_n \otimes \cat{Rep}^\ULA(\check{T}_Q^n)
    =\stack{Z}_n \otimes \cat{FRep}_n^\ULA(\check{T}_Q).
 \end{equation*}
 By definition, it is an equivalence, and by \ref{horizontal functors} it is horizontal because $p_n$ is ind-finite (so ind-proper), so we are done.
\end{proof}

To finish the picture, we must show how convolution is carried by this equivalence.  First, we have a comparison of convolution on $\on{Gr}_{T,X^n}$ with the ordinary tensor product.  We recall that since objects of $\Sph(\stack{T}_n)$ are supported on $\on{Gr}_{{}^L \check{T}_Q, X^n}$, where $\stack{T}_n$ is multiplicative factorizable, we have natural equivalences
\begin{equation*}
 \stack{T}_n^{\lambda_1, \dots, \lambda_n}
  \boxtimes \stack{T}_m^{\mu_1, \dots, \mu_m}
 \cong \stack{T}_{n + m}^{\lambda_1, \dots, \lambda_n, \mu_1, \dots, \mu_m}
\end{equation*}
so that the twisted outer tensor product on components of $\on{Gr}_{{}^L \check{T}_Q, X^n}$ and $\on{Gr}_{{}^L \check{T}_Q, X^m}$:
\begin{equation*}
 \Perv(\stack{T}_n^{\lambda_1, \dots, \lambda_n})
  \times \Perv(\stack{T}_m^{\mu_1, \dots, \mu_m})
 \to \Perv(\stack{T}_n^{\lambda_1, \dots, \lambda_n} \boxtimes \stack{T}_m^{\mu_1, \dots, \mu_m})
\end{equation*}
takes its values in $\Perv(\stack{T}_{n + m}^{\lambda_1, \dots, \lambda_n, \mu_1, \dots, \mu_m})$. We will refer to this as simply the \emph{outer tensor product}.  There is also a natural \emph{inner tensor product} of ULA sheaves obtained when $n = m$ by restricting to the diagonal copy of $X^n$.  

\begin{theorem}{prop}{convolution and tensor product}
 Outer and inner convolutions coincide with outer and inner tensor products.
\end{theorem}

\begin{proof}
 As always, we work on $\on{Gr}_{{}^L \check{T}_Q, X^n}$.  Then, by \ref{torus convolution diagram}, the convolution diagrams are identified with ordinary products and the multiplication maps become identities on each component; since the $G(\hat{\OO})$-equivariance structures on both gerbe and sheaves are trivial, the twisted products become ordinary outer tensor products.  This establishes the first claim, and the second follows by definition of inner convolution.
\end{proof}

There is also a tensor product operation in $\stack{Z}_n \otimes \cat{FRep}(\check{T}_Q)$.  

\begin{theorem}{defn}{convolution of factorizable actions}
 Let $H$ be a group.  The \emph{outer tensor product} on $\cat{FRep}_n(H)$ and $\cat{FRep}_m(H)$ is the composition of the outer tensor product of perverse sheaves,
 \begin{equation*}
  \cat{FRep}_n(H) \times \cat{FRep}_m(H)
   \to (\on{Fact}(H)_n \boxtimes \on{Fact}(H)_m)\cat{\mhyphen Mod}
 \end{equation*}
 (in perverse sheaves on $X^{n + m}$) with the forgetful map along the diagonal \ref{eq:sf mult open}
 \begin{equation*}
  \psi_{m,n} \colon \on{Fact}(H)_{n + m} \incl \on{Fact}(H)_n \boxtimes \on{Fact}(H)_m
 \end{equation*}
 into $\cat{FRep}_{n + m}(H)$.  The \emph{outer tensor product} of twisted factorizable representations is
 \begin{multline*}
  (\stack{Z}_n \otimes \cat{FRep}_n(H)) \times (\stack{Z}_m \otimes \cat{FRep}_m(H)) \\
    \to (\stack{Z}_n \boxtimes \stack{Z}_m) \otimes (\on{Fact}(H)_n \boxtimes
         \on{Fact}(H)_m)\cat{\mhyphen Mod} \\
    \cong {}^2 \psi_{m,n}(\stack{Z}_{n + m}) \otimes (\on{Fact}(H)_n \boxtimes
         \on{Fact}(H)_m)\cat{\mhyphen Mod} \\
    \to \stack{Z}_{n + m} \otimes \cat{FRep}_{n + m}(H)
 \end{multline*}
 (compare with \ref{eq:comultiplication twisting}). Likewise, we have an \emph{inner tensor product} of twisted lisse sheaves obtained by restriction to the diagonal copy of $X^n$ in $X^n \times X^n$.
\end{theorem}

Then we can strengthen \ref{main theorem for torus}:

\begin{theorem}{prop}{pushforward is tensor}
 The equivalence of \ref{main theorem for torus} is a tensor functor.
\end{theorem}

\begin{proof}
 By \ref{devissage}, it is enough to prove this for ULA sheaves,. But it is clear that $(p_n)_*$ sends one notion of tensor product to the other over the open subset of $X^n$ away from the diagonals, so by \ref{ULA nearby cycles} it preserves the tensor product on all of $X^n$.
\end{proof}

\section{Absolute twisted Satake: semisimplicity}
\label{s:absolute twisted Satake semisimplicity}

In this section we apply the preceding constructions to the category of twisted perverse sheaves on the absolute grassmannian $\on{Gr}_G$.  

\subsection*{Convolution on the absolute grassmannian}

As is usual, we will compare this category to perverse sheaves on $\on{Gr}_{G,X}$ for sufficiently small $X$.  Thus, let $x \in X$ be a fixed point and replace $X$ by a small (contractible) disk neighborhood of $x$.  We identify $\on{Gr}_G = \on{Gr}_{G,X}|_x$; then we actually have an identification $\on{Gr}_{G,X} \cong \on{Gr}_G \times X$, since the coordinate on $X$ identifies all the fibers $\on{Gr}_{G,X}|_y \cong G(\hat{\KK}_y)/G(\hat{\OO}_y)$ with the one at $x$.  We set $\stack{G} = \stack{G}_1|_x$; it inherits a $\smash{G(\hat{\OO}) = G(\hat{\OO})_X|_x}$-equivariance structure, and so we can define $\Sph(\stack{G})$, the category of $G(\hat{\OO})$-equivariant $\stack{G}$-twisted perverse sheaves on $\on{Gr}_G$.

\begin{theorem}{defn}{absolute convolution}
 For $\sh{F}_1, \sh{F}_2 \in \Sph(\stack{G})$, let $\sh{F}_1 \tbtimes \sh{F}_2$ be their twisted product on $\on{Conv}_G$, defined as in \ref{inner convolution diagram} to be $G(\hat{\KK}) \times^{G(\hat{\OO})} \on{Gr}_G = \on{Conv}^2_1|_x$, and set
 \begin{equation*}
  \sh{F}_1 * \sh{F}_2 = m_*(\sh{F}_1 \tbtimes \sh{F}_2),
 \end{equation*}
 relative to the proper map $\map{m}{\on{Conv}_G}{\on{Gr}_G}$.
\end{theorem}

From the generalities above, it is easy to deduce the good properties of this convolution. To see this, let $\map{\on{Spr}}{\Sph(\stack{G})}{\Sph^\ULA(\stack{G}_1)}$ be the ``spreading'' functor defined by
\begin{equation*}
 \on{Spr}(\sh{F}) = \on{pr}_{\on{Gr}_G}^* \sh{F}[1],
\end{equation*}
using $\on{Gr}_{G,X} \cong \on{Gr}_G \times X$.  Of course, this depends on the decomposition but this will not matter.  Clearly, $\on{Spr}$ is left-inverted by the functor $\map{\on{Res}}{\Sph^\ULA(\stack{G}_1)}{\Sph(\stack{G})}$ defined by
\begin{equation*}
 \on{Res}(\sh{F}) = \sh{F}|_x[-1].
\end{equation*}

\begin{theorem}{prop}{absolute relative convolution}
 Let $\sh{F}_1, \sh{F}_2 \in \Sph(\stack{G})$; then we have $\on{Spr}(\sh{F}_1) *_i \on{Spr}(\sh{F}_2) = \on{Spr}(\sh{F}_1 * \sh{F}_2)$.  In particular, $\sh{F}_1 * \sh{F}_2$ is perverse and defines a rigid tensor structure on $\Sph(\stack{G})$.
\end{theorem}

\begin{proof}
 The product $\on{Gr}_{G,X} \cong \on{Gr}_G \times X$ induces an identification $\on{Conv}_1^2 \cong \on{Conv}_G \times X$ compatible with the multiplication maps.  The first statement then follows from the definitions of the convolutions and proper base change.  The second statement then follows by applying $\on{Res}$ to \ref{convolution properties} and \ref{convolution is rigid}.
\end{proof}

\subsection*{Simple objects}

We use this structure to show that $\Sph(\stack{G})$ is in fact a semisimple abelian category. First, we identify the irreducible objects: let $\sh{J}$ be any simple object of $\Sph(\stack{G})$; by general principles, it must be of the form $j_{!*}(\sh{L})$, where $\sh{L}$ is a locally constant sheaf of $k$-vector spaces supported on a locally closed subset $\map{j}{S}{\on{Gr}_G}$; since $\sh{J}$ is $\smash{G(\hat{\OO})}$-equivariant, $S$ must be a union of orbits, and so by \ref{absolute orbit properties}, must have some $\on{Gr}_G^\lambda$ as a dense open subset; we assume therefore that $S = \on{Gr}_G^\lambda$.  By \ref{orbit stabilizer is connected}, \ref{transitive equivariant objects} applies and so $\stack{G}|\on{Gr}_G^\lambda$ and $\sh{L}$ are equivariantly trivial.  We conclude:

\begin{theorem}{prop}{absolute simple objects}
 The simple objects of $\Sph(\stack{G})$ are all minimal extensions $\sh{J}(\lambda)$ from some $\on{Gr}_G^\lambda$ of the constant sheaves $\csheaf{k}[\dim \on{Gr}_G^\lambda]$. \qed
\end{theorem}

We can say more about which coweights $\lambda$ occur.  By \ref{bilinear form triviality} applied to $\on{Spr}(\sh{J}(\lambda))$, we have $F_1(\on{Spr}(\sh{J}(\lambda)))^\lambda = 0$ if $\lambda \notin \Lambda_Q$, the kernel of the bilinear form $\kappa$ associated with $Q$.  By \ref{fiber functor properties}\ref{en:fiber exactness}, we have $\on{Spr}(\sh{J}(\lambda)) = 0$, and therefore $\sh{J}(\lambda) = 0$ as well.  We arrive at the following more precise statement:

\begin{theorem}{prop}{precise simple objects}
 We have $\sh{J}(\lambda) = 0$ if $\lambda \notin \Lambda_Q$.  Furthermore, for every coroot $\check\alpha$ of $G$ such that $Q(\check\alpha)$ has infinite order in $A$, if for some $\lambda$ we have $\kappa(\check\alpha, \lambda) = 1$, then $\langle \alpha, \lambda \rangle = 0$.  That is, $\Lambda_Q$ is contained in the rational span of the coroots $\check\alpha$ such that $Q(\check\alpha)$ has finite order.
\end{theorem}

\begin{proof}
 The first statement was already argued.  For the second statement, we need only apply \ref{invariant form and pairing} with $\epsilon_{\check\alpha} = 1$ according to \ref{sf gerbes exact sequence}:
 \begin{equation*}
  \kappa(\check\alpha, \lambda) = Q(\check\alpha)^{\langle \alpha, \lambda \rangle}.
 \end{equation*}
 If the base is not a root of unity, this is $1$ if and only if the exponent is zero.
\end{proof}

This leads to a finiteness result on the gerbe itself.

\begin{theorem}{cor}{gerbe finite order}
 Suppose $\lambda \in \Lambda_Q$; then $\stack{G}$ has finite order on $\bar{\on{Gr}}_G^\lambda$.
\end{theorem}

\begin{proof}
 We use \ref{fiber monodromies}; that is, we stratify $\bar{\on{Gr}}_G^\lambda$ by the intersections with the $\on{Gr}_B^\mu$ such that $w_0(\lambda) \leq \mu \leq \lambda$, on each of which $\stack{G}$ is trivialized with the order of trivialization being $Q(\check\alpha)^{-1}$ around each boundary divisor $\on{Gr}_B^{\mu - \check\alpha}$.  By \ref{precise simple objects}, if $Q(\check\alpha)$ does not have finite order, then $\lambda$, being in $\Lambda_Q$, must be orthogonal to $\alpha$, and therefore, if $\on{Gr}_B^\mu \cap \smash{\bar{\on{Gr}}}_G^\lambda \neq \emptyset$, we also have $\mu$ orthogonal to $\alpha$.  Thus, $\mu$ lies in the rational span of those $\check\alpha$ such that $Q(\check\alpha)$ has finite order.  In particular, the numbers $Q(\check\alpha)^{-1}$ which occur are all of finite order, so $\stack{G}^n$ is trivial on $\bar{\on{Gr}}_G$ for some $n$, as desired.
\end{proof}

We begin to analyze the basic properties of extensions of these simple objects.

\begin{theorem}{lem}{self-extensions}
 We have $\on{Ext}^1(\sh{J}(\lambda), \sh{J}(\lambda)) = 0$; i.e.\ there are no nontrivial extensions of $\sh{J}(\lambda)$ by itself.
\end{theorem}

\begin{proof}
 Denote by $\map{j}{\on{Gr}_G^{\lambda}}{\smash{\bar{\on{Gr}}_G}^\lambda}$ the inclusion map, $i$ the inclusion of the complement.  The property
 \begin{equation}
  \label{eq:minimal extension}
  {}^p H^0(i^* j_{!*} \sh{F}) = 0 = {}^p H^0(i^! j_{!*} \sh{F}).
 \end{equation}
 uniquely characterizes the functor $j_{!*}$ and is preserved under extensions, so that we have
 \begin{equation*}
  \on{Ext}^1(j_{!*}\sh{F}_1, j_{!*} \sh{F}_2) \cong \on{Ext}^1(\sh{F}_1, \sh{F}_2).
 \end{equation*}
 When $\sh{F}_1 = \sh{F}_2 = \csheaf{k}[\dim \on{Gr}_G^\lambda]$, the latter is zero since $\on{Gr}_G^\lambda$ is simply connected.
\end{proof}

We now turn to more general extensions.  In identifying simple factors of a spherical sheaf, the following lemma is key:

\begin{theorem}{lem}{composition factors}
 Let $\sh{F} \in \cat{Sph}(\stack{G})$ have composition factors $\sh{J}(\lambda_i)$ for various dominant coweights $\lambda_i$; then the orbits $\on{Gr}_G^{\lambda_i}$ are exactly those such that (denoting by $j$ their inclusions into $\on{Gr}_G$) we have ${}^p H^0(j^* \sh{F}) \neq 0$. More precisely, $\sh{J}(\nu)$ occurs $n$ times as a composition factor of $\sh{F}$ if and only if for $\map{i}{\{t^{\nu}\}}{\on{Gr}_G}$, we have
 \begin{equation}
  \label{eq:fiber cohomology criterion}
  {}^p H^0(i^* \sh{F}[-\dim \on{Gr}_G^{\nu}]) \cong k^n.
 \end{equation}
\end{theorem}

\begin{proof}
 The first statement follows from \ref{eq:minimal extension}.  For the second, we know that the sheaf ${}^p H^0(j^* \sh{F})$ is $\smash{G(\hat{\OO})}$-equivariant and therefore constant on $\on{Gr}_G^\nu$, so vanishes if and only if its stalk at $t^{\nu}$ does.  The more refined count comes from the fact that $j^* \sh{J}(\nu) \cong k[\dim \on{Gr}_G^\nu]$.
\end{proof}

In order to produce such a nonvanishing stalk, we will use the following fact about the top cohomology of a proper scheme.

\begin{theorem}{lem}{proper top cohomology}
 Let $\map{p}{P}{\on{Spec} \C}$ be the structure map of a connected proper scheme of dimension $d$ and let $\map{j}{U}{P}$ be the inclusion of a smooth open set whose complement $\map{i}{Z}{P}$ has codimension at least $2$.  If $A^\bullet \in {}^p \stack{D}^{\leqslant 0}$ and if $j^* A^\bullet \cong \csheaf{k}[d]$, then ${}^p H^d p_*(A^\bullet) \cong k$.
\end{theorem}

Note that the lemma concerns untwisted perverse sheaves, but that if we are given a trivial (if not \emph{trivialized}) gerbe on $\on{Spec} \C$, the same statement holds of twisted sheaves as well, since in fact they are equivalent to untwisted ones.  There is no twisted version of this lemma because in order to apply cohomology (that is, pushforward), the gerbe on $P$ must be the pullback of that on $\on{Spec} \C$, which is always trivial.

\begin{proof}
 We apply $p_* = p_!$ to the canonical triangle:
 \begin{equation*}
  j_! j^* A^\bullet \to A^\bullet \to i_* i^* A^\bullet \to
  \quad\implies\quad
  (p|_U)_! j^* A^\bullet \to p_* A^\bullet \to (p|_Z)_* i^* A^\bullet \to.
 \end{equation*}
 By hypothesis, $\dim Z \leq d - 2$ and $i^* A^\bullet \in {}^p \stack{D}^{\leqslant 0}$, so by \cite{BBD}*{\S4.2.4} and the long exact sequence of perverse cohomology we have, respectively:
 \begin{equation*}
  {}^p H^d (p|_Z)_* i^* A^\bullet = {}^p H^{d - 1} (p|_Z)_* i^* A^\bullet = 0
  \quad\implies\quad
  {}^p H^d p_* (A^\bullet) \cong {}^p H^d (p|_U)_! (j^* A^\bullet).
 \end{equation*}
 The last term is the dual of $H^{-d} (p|_U)_* (j^* \DD A^\bullet)$ on $\on{Spec} \C$, and we have
 \begin{equation*}
  H^{-d} (p|_U)_* (j^* \DD A^\bullet)
   = H^{-d} (p|_U)_* (\DD \csheaf{k}[d])
   = H^0 (p|_U)_* \csheaf{k}
   = \Gamma(U, \csheaf{k})
   = k,
 \end{equation*}
 using $\DD (\csheaf{k}[d]) = \csheaf{k}[d]$ since $U$ is smooth.
\end{proof}

\subsection*{Summands of a convolution}

In particular, we may apply this to obtain composition factors of a convolution.  For simplicity, we use the notation $l = \dim \on{Gr}_G^{\lambda}$, $m = \dim \on{Gr}_G^{\mu}$, $n = \dim \on{Gr}_G^{\nu}$.

\begin{theorem}{lem}{fiber dimension criterion}
 For any dominant $\nu$, let $U = m^{-1}(t^{\nu}) \cap (\on{Gr}_G^\lambda * \on{Gr}_G^\mu)$.  Then
 \begin{equation}
  \label{eq:fiber dimension bound}
  \dim U \leq \frac{1}{2}(l + m - n),
 \end{equation}
 with equality if and only if $\sh{J}(\nu)$ is a composition factor of $\sh{J}(\lambda) * \sh{J}(\mu)$. Furthermore, the multiplicity of this summand is the number of components of $U$ having this dimension.
\end{theorem}

\begin{proof}
 Note that \ref{eq:fiber dimension bound} is equivalent, using \ref{absolute orbit properties} and \ref{irreducible convolution set}, to:
 \begin{equation}
  \label{eq:new fiber dimension bound}
  \on{codim} U - \dim U \geq n
 \end{equation}
 For brevity, let $\sh{F} = \sh{J}(\lambda) \tbtimes \sh{J}(\mu)$; it lives on the closure $C$ of $\on{Gr}_G^\lambda * \on{Gr}_G^\mu$ and is constant on this convolution product of orbits
 itself; we will denote
 \begin{equation*}
  P = \bar{U} = m^{-1}(t^\nu) \cap C.
 \end{equation*}
 To evaluate \ref{eq:fiber cohomology criterion} for $m_*(\sh{F})$, we apply proper base change and then \ref{proper top cohomology} with $A^\bullet = \sh{F}|_P[-\on{codim} P]$.  Its restriction to $U$ is constant in degree $-\dim P$ and the boundary of $U$ in $P$ has codimension $2$ by \ref{absolute orbit properties}, so the lemma applies and we conclude that the extremal case of \ref{eq:new fiber dimension bound} is the precise condition necessary for \ref{composition factors} to apply.  If the left side were decreased, then \ref{proper top cohomology} would produce positive-degree cohomology sheaves of $i^* m_* (\sh{F})[-n]$ and therefore of $m_* (\sh{F})$ and, finally, of $\sh{J}(\lambda) * \sh{J}(\mu)$, in contradiction to the fact that this is perverse.  This gives the inequality of \ref{eq:fiber dimension bound}.  The statement on multiplicity follows from \ref{composition factors}.
\end{proof}

\begin{theorem}{cor}{convolution highest weight}
 The convolution $\sh{J}(\lambda) * \sh{J}(\mu)$ has exactly one copy of $\sh{J}(\lambda + \mu)$ as a composition factor, and all other ones $\sh{J}(\nu)$ have $\nu \leq \lambda + \mu$.
\end{theorem}

\begin{proof}
 This follows directly from \ref{fiber dimension criterion} and \ref{highest weight convolution fiber}.
\end{proof}

Finally, the last step on the way to general semisimplicity is to show that all the composition factors of a convolution are actually summands.  This lemma, with the same proof, appeared as \cite{ginzburg}*{Proposition 2.2.1} but was absent from \cite{MV_Satake} since the latter paper proves semisimplicity of $\Sph$ by a different, more computational route.

\begin{theorem}{lem}{semisimple convolution}
 The convolution of two $\sh{J}(\lambda)$'s is semisimple in $\Sph(\stack{G})$.
\end{theorem}

\begin{proof}
 We have $\sh{J}(\lambda) * \sh{J}(\mu) = m_*(\sh{J}(\lambda) \tbtimes \sh{J}(\mu))$, where the twisted product is on $\on{Conv}_G$.  It is the minimal extension from $\on{Gr}_G^\lambda * \on{Gr}_G^\mu \subset \on{Conv}_G$ of the constant sheaf, and therefore simple.  We would like to apply the decomposition theorem \cite{BBD}*{Th\'eor\`eme 6.2.5}, for which we must return to untwisted sheaves.
 
 First, we obviously have nothing to do if $\sh{J}(\lambda) = 0$ or $\sh{J}(\mu) = 0$.  If both are nonzero, then by \ref{convolution highest weight}, $\sh{J}(\lambda + \mu) \neq 0$ as well, and therefore $\stack{G}$ is trivial on $U = \on{Gr}_G^{\lambda + \mu}$, since it possesses a twisted $k^*$-torsor there.  By the same result, $\sh{J}(\lambda) * \sh{J}(\mu)$ is supported on its closure, and since $U$ is a dense open subspace of it, there is some $\stack{G}$-trivializing open cover $\{V_i\}$ of that closure all of whose elements intersect $U$, and we choose trivializations.
 
 The trivializations on the $V_i$ entail trivializations of $m^*(\stack{G}) \cong \stack{G} \tbtimes \stack{G}$ on the $m^{-1}(V_i)$, which differ by $k^*$-torsors on the intersections.  Since all of them intersect $m^{-1}(U)$, they all intersect the set $\smash{\tilde{U}} = \on{Gr}_G^\lambda * \on{Gr}_G^\mu$ on which the twisted product $\sh{F} = \sh{J}(\lambda) \tbtimes \sh{J}(\mu)$ is a constant sheaf, so it is a local system on each $m^{-1}(V_i) \cap \smash{\tilde{U}}$.  Thus, on $m^{-1}(V_i)$, we have that $\sh{F}$ is the minimal extension of a local system, and thus is of geometric origin if that local system has finite monodromy.  Indeed, by \ref{gerbe finite order}, $\stack{G}$ itself has finite order on $\bar{\on{Gr}}_G^{\lambda + \mu}$, and in particular, all the transition torsors of $m^*(\stack{G})$ are of finite order, and thus have finite monodromy as desired.
 
 It follows that on each $m^{-1}(V_i)$, the twisted product $\sh{F} = \sh{J}(\lambda) \tbtimes \sh{J}(\mu)$ is simple of geometric origin, so since $m$ is proper, the decomposition theorem \emph{does} apply (this could have been done more directly using Kashiwara's conjecture, proved in \cite{D_Kashiwara}).  Therefore, $\sh{J}(\lambda) * \sh{J}(\mu)$ is semisimple on each $V_i$.  Its summands are isomorphic up to the product with local systems on the $V_i \cap V_j$, so the convolution itself is a direct sum of twisted sheaves, each of which is locally on the $V_i$ a direct sum of multiple copies of $\sh{J}(\nu)$ for \emph{one} coweight $\nu$.  But by \ref{self-extensions}, there are no extensions of $\sh{J}(\nu)$ by itself, so each of these summands is itself semisimple.
\end{proof}

The preceding results allow us to prove the analogue for perverse sheaves of the following proposition, true for representations of any reductive group $H$.

\begin{theorem*}{lem}
 If $\lambda$ is a dominant weight in the root lattice of $H$, then there is some $\mu$ such that $V^\lambda$ is a direct summand of $(V^\mu)^* \otimes V^\mu$.
\end{theorem*}

\begin{theorem}{lem}{dual direct summand}
 Let $\lambda$ be a dominant coweight and a sum of simple coroots.  Then there exists a $\mu$ such that $\sh{J}(\lambda)$ is a direct summand of $\sh{J}(\mu)^* * \sh{J}(\mu)$.
\end{theorem}

\begin{proof}
 Since the convolution is semisimple, this is equivalent to
 \begin{equation*}
  0 \neq \on{Hom}(\sh{J}(\lambda), \sh{J}(\mu)^* * \sh{J}(\mu))
    = \on{Hom}(\sh{J}(\mu) * \sh{J}(\lambda), \sh{J}(\mu))
 \end{equation*}
 and therefore to finding a copy of $\sh{J}(\mu)$ as a summand of $\sh{J}(\mu) * \sh{J}(\lambda)$, for some $\mu$.  By the criterion of \ref{fiber dimension criterion}, we need only find an irreducible component of the fiber of $\on{Gr}_G^\mu * \on{Gr}_G^\lambda$ over $t^{\mu}$ with dimension at least (hence equal to) $l/2$.  This was already provided in \ref{middle weight convolution fiber}.
\end{proof}

\subsection*{Semisimplicity and consequences}

We give a slick demonstration of the semisimplicity of $\cat{Sph}$ by an argument of Frenkel, Gaitsgory, and Vilonen from \cite{FGV_Whittaker}*{\S6.1}.

\begin{theorem}{prop}{absolute semisimple}
 $\Sph(\stack{G})$ is semisimple.
\end{theorem}

\begin{proof}
 It suffices to show that there are no nontrivial extensions of the irreducible objects, so we must show that $\on{Ext}^1(\sh{J}(\lambda), \sh{J}(\mu)) = 0$ always.  Since $\on{Ext}^1$ is a derived functor of $\on{Hom}$, by the properties of the dual we have
 \begin{equation*}
  \on{Ext}^1(\sh{J}(\lambda), \sh{J}(\mu))
    = \on{Ext}^1(1, \sh{J}(\lambda)^* * \sh{J}(\mu)).
 \end{equation*}
 By \ref{semisimple convolution}, the latter sheaf is semisimple, so we may assume it is just of the form $\sh{J}(\lambda)$.  If $\lambda$ is not in the coroot lattice, then $\sh{J}(\lambda)$ and $1$ are supported on different connected components of $\on{Gr}_G$, so of course have no nontrivial extensions.  Otherwise, \ref{dual direct summand} applies and it suffices to replace the right-hand side with $\sh{J}(\mu)^* * \sh{J}(\mu)$.  Then:
 \begin{equation*}
  \on{Ext}^1(1, \sh{J}(\mu)^* * \sh{J}(\mu))
    = \on{Ext}^1(\sh{J}(\mu), \sh{J}(\mu))
    = 0,
 \end{equation*}
 by \ref{self-extensions}.
\end{proof}

\begin{theorem}{cor}{simple object extensions}
 Let $\lambda \in \Lambda_T$ be dominant and denote by $j \colon \on{Gr}_G^\lambda \to \bar{\on{Gr}}{}_G^\lambda$ the open inclusion, and let $l = \dim \on{Gr}_G^\lambda$.  Then the natural maps are isomorphisms:
 \begin{equation*}
  \sh{J}_!(\lambda) = j_!(\csheaf{k}[l]) \to \sh{J}(\lambda) = j_{!*}(\csheaf{k}[l])
    \to \sh{J}_*(\lambda) = j_*(\csheaf{k}[l]).
 \end{equation*}
\end{theorem}

\begin{proof}
 The complex $\sh{J}_!(\lambda)$ is \textit{a priori} perverse in nonpositive degrees.  Starting from the natural map, we get a distinguished triangle
 \begin{equation*}
  A^\bullet \to \sh{J}_!(\lambda) \to \sh{J}(\lambda) \to
 \end{equation*}
 and since $\sh{J}(\lambda)$ \emph{is} perverse, the long exact sequence of cohomology gives a short exact sequence in degree zero
 \begin{equation*}
  0 \to {}^p H^0(A^\bullet) \to {}^p H^0(\sh{J}_!(\lambda)) \to \sh{J}(\lambda).
 \end{equation*}
 The last map is surjective by definition, as well as $G(\hat{\OO})$-equivariant by functoriality, so the first term is in $\Sph(\stack{G})$ and thus, by semisimplicity, the exact sequence splits. Applying $j^*$, we find that ${}^p H^0(A^\bullet)$ is in the image of $i_*$, but after applying $i^*$ we find, since $i^* j_! = 0$ and $i^*$ is right t-exact, that it vanishes (we use the splitting here).  Thus, $\sh{J}(\lambda) \cong {}^p H^0(\sh{J}_!(\lambda))$.
 
 Going back to the distinguished triangle, application of $j^*$ shows that $A^\bullet$ is in the image of $i_*$, and application of $i^*$ gives that $A^\bullet \cong i_* i^* \sh{J}(\lambda)[-1]$, so we have
 \begin{equation*}
  A^\bullet
   \cong i_* i^* {}^p H^0(\sh{J}_!(\lambda))[-1]
   \cong i_* {}^p H^0(i^* \sh{J}_!(\lambda))[-1]
   = 0.
 \end{equation*}
 Therefore, $\sh{J}(\lambda) \cong \sh{J}_!(\lambda)$, as desired.  The other isomorphism follows from duality.
\end{proof}

In the coda, we compare $\Sph(\stack{G})$ to $\Sph^\ULA(\stack{G}_1)$.

\begin{theorem}{lem}{relative simple objects}
 The simple objects of $\Sph^\ULA(\stack{G}_1)$ are minimal extensions from some $\on{Gr}_{G,X}^\lambda$ of locally constant sheaves $\sh{L}[\dim \on{Gr}_{G,X}^\lambda]$ which are trivial on the fibers over $X$.
\end{theorem}

\begin{proof}
 The same argument as for $\on{Gr}_G$ shows that they are all of the form $j_{!*}(\sh{F})$, where $j$ is the inclusion of one of these orbits and $\sh{F} \in \Sph^\ULA(\stack{G}_1)$ is a simple sheaf on $\on{Gr}_G^\lambda$.  Since $\sh{F}|\on{Gr}_{T,X}^\lambda[-\dim_X \on{Gr}_{G,X}^\lambda]$ is again ULA, it must be lisse on $\on{Gr}_{T,X}^\lambda \cong X$ by \ref{ULA local systems}. Considering $\on{Gr}_{T,X}^\lambda \cong X$, by $G(\hat{\OO})_X$-equivariance $\sh{F}|\on{Gr}_{G,X}^\lambda$ is the pullback from it of a lisse sheaf, hence lisse, and so \ref{transitive equivariant objects} applies as in \ref{absolute simple objects}.
\end{proof}

In particular, when $X$ is cohomologically trivial, we must have $\sh{L} = \csheaf{k}$; let $\sh{J}(\lambda)_X$ be the corresponding irreducible objects.  \ref{semisimple convolution} holds for these objects with the same proof, so the semisimple abelian category generated by these objects is closed under $*_i$.  The obvious equivalence (by \ref{absolute semisimple}) with $\Sph(\stack{G})$ which identifies their irreducible objects is by definition an equivalence of tensor categories that does not depend on the choice of isomorphism $\on{Gr}_{G,X} \cong \on{Gr}_G \times X$.  In particular, the tensor structure (in particular, the commutativity constraint) on $\Sph(\stack{G})$ is independent of this isomorphism, of the choice of $x \in X$, and of the neighborhood containing $x$ on which the isomorphism was given.

Finally, we can completely describe the ULA sheaves in a neighborhood of a point.

\begin{theorem}{prop}{ULA local equivalence}
 Let $X$ be a small (contractible) neighborhood of one of its points; then we have a natural equivalence $\Sph^\ULA(\stack{G}_1) \cong \Sph(\stack{G})$ of tensor abelian categories.  Likewise, in any small neighborhood in $X^n$ not intersecting any diagonal, we have $\Sph^\ULA(\stack{G}_n) \cong \Sph(\stack{G}^{\boxtimes n})$.
\end{theorem}

\begin{proof}
 First, we observe that $\Sph^\ULA(\stack{G}_1)$ is semisimple.  To accomplish this, the same proof of \ref{absolute semisimple} will work if we can establish \ref{dual direct summand} for the $\sh{J}(\lambda)_X$.  But the existence of a summand such as the lemma claims can be detected after application of $\on{Res}$ by semisimplicity of inner convolution, so the lemma reduces to that on $\on{Gr}_G$, which is already proven.

 Therefore both $\Sph^\ULA(\stack{G}_1)$ and $\Sph(\stack{G})$ are semisimple tensor abelian categories whose simple objects are identified in such a way that their convolutions agree.  Thus, the equivalence which sends one set of irreducibles to the other is a tensor equivalence.
 
 For the last statement we simply observe that on such a neighborhood, $\on{Gr}_{G,X^n} \cong X^n \times \on{Gr}_G^n$, with $X^n$ contractible and $\on{Gr}_G^n \cong \on{Gr}_{G^n}$, so that the same proof applies.
\end{proof}

\section{Absolute twisted Satake: root data}
\label{s:absolute twisted Satake root data}

We turn to the twisted Satake equivalence on $\on{Gr}_G$.  The results of this section generalize and simplify those of \cite{FL_twistedSatake}.

\subsection*{The absolute fiber functor}

First, we define the fiber functor for $\Sph(\stack{G})$.  As for $\Sph(\stack{G}_n)$, we have the fundamental diagram
\begin{equation*}
 \on{Gr}_G \xleftarrow{b} \on{Gr}_B \xrightarrow{t} \on{Gr}_T
\end{equation*}
and the sf $k^*$-gerbe $\stack{T}$ on $\on{Gr}_T$ such that $t^* \stack{T} \cong b^* \stack{G}$; in the notation of the previous section, we also have $\stack{T} = \stack{T}_1|_x$.  It is noncanonically trivial as a gerbe, but not as a multiplicative gerbe, so we continue to use it in the notation.

\begin{theorem}{defn}{absolute fiber functor}
 The \emph{fiber functor} $\map{F}{\Sph(\stack{G})}{\stack{D}(\stack{T})}$ is the following functor: for $\sh{F} \in \Sph(\stack{G})$, let $\sh{F}^\lambda$ be the restriction of $b^* \sh{F}$ to the connected component $\on{Gr}_B^\lambda$, and set
 \begin{equation*}
  F(\sh{F})^\lambda = F_!^*(\sh{F})^\lambda
   = t_!^\lambda \sh{F}^\lambda[\langle 2\rho, \lambda \rangle]
 \end{equation*}
 to be the component of $F(\sh{F})$ on $\on{Gr}_T^\lambda$.
\end{theorem}

Equivalently, it is clear that (again in the previous notation) we have
\begin{equation*}
 \on{Res} F_1(\on{Spr} \sh{F}) = F(\sh{F}),
\end{equation*}
so that by \ref{fiber functor properties} we have:

\begin{theorem}{prop}{absolute fiber properties}
 The functor $F$ is t-exact, swaps tensor duality and Verdier duality, and is a faithful tensor functor $\Sph(\stack{G}) \to \Sph(\stack{T})$.
\end{theorem}

\begin{proof}
 The t-exactness statement in \ref{fiber functor properties} was reduced to this theorem, so we complete the proof here.  By \ref{absolute semisimple} it is enough to show that $F(\sh{J}(\lambda))$ is perverse for any $\lambda$.  We will show that $F(\sh{J}(\lambda)) \in {}^p \stack{D}(\stack{T})^{\leqslant 0}$; the dual claim is established by replacing $\sh{J}(\lambda)$ with its convolution dual.
 
 $\sh{J}(\lambda)$ is supported on $\bar{\on{Gr}}{}_G^\lambda$; for any coweight $\mu$, let $I^{\lambda,\mu} = \bar{\on{Gr}}{}_G^\lambda \cap \on{Gr}_B^\mu$; temporarily denote
 \begin{align*}
  b \colon I^{\lambda, \mu} \to \bar{\on{Gr}}{}_G^\lambda &&
  t \colon I^{\lambda, \mu} \to \on{Gr}_T^\mu \cong \on{Spec} \C;
 \end{align*}
 then $F(\sh{J}(\lambda))^\mu = t_! b^* \sh{J}(\lambda) [\langle 2\rho, \mu \rangle]$.  By \cite{BBD}*{\S4.2.4}, the functor $t_![\on{dim} I^{\lambda, \mu}]$ is right t-exact.  We claim that $b^*(\sh{J}(\lambda))[-\on{codim} I^{\lambda,\mu}]$ is perverse.  To express this, we make the additional notation
 \begin{align*}
  \bar{b} \colon \bar{\on{Gr}}{}_B^\mu \cap \bar{\on{Gr}}{}_G^\lambda \to \bar{\on{Gr}}{}_G^\lambda
  &&
  j \colon \on{Gr}_B^\mu \to \bar{\on{Gr}}{}_B^\mu,
 \end{align*}
 so that $j$ is an open immersion and it suffices to show that $\bar{b}{}^*(\sh{J}(\lambda))[-\on{codim} I^{\lambda,\mu}]$ is perverse.  If $\lambda - \mu$ is not a sum of simple coroots then $I^{\lambda, \mu}$ is empty; otherwise, there is a chain of inequalities $\mu = \mu_n < \mu_{n - 1} < \dots < \lambda = \mu_0$ such that each $\mu_k - \mu_{k + 1}$ is a simple coroot.  The corresponding $\bar{\on{Gr}}{}_B^{\mu_k}$ form a chain of successive Cartier divisors by \ref{borel components}\ref{en:borel orbit boundaries}, and thus so do their intersections with $\bar{\on{Gr}}{}_G^\lambda$ (all nonempty and all intersecting $\on{Gr}_G^\lambda$ as well).  The claim then follows from \ref{successive divisors}, using \ref{simple object extensions}.
 
 We have thus shown that $t_! b^* [\on{dim} I^{\lambda, \mu} - \on{codim} I^{\lambda, \mu}]$ is right t-exact on $\cat{Sph}(\stack{G})$.  By \ref{absolute orbit properties}, we have
 \begin{align*}
  \on{dim} I^{\lambda, \mu} = \langle \rho, \lambda + \mu \rangle &&
  \on{codim} I^{\lambda, \mu} = \langle \rho, \lambda - \mu\rangle
 \end{align*}
 whose difference is indeed $\langle 2\rho, \mu \rangle$.
\end{proof}

\begin{theorem}{lem}{successive divisors}
 Let $Y$ be a scheme. $Y = Y_0 \supset Y_1 \supset \dots \supset Y_n$ be a sequence of successive Cartier divisors, and denote $i \colon Y_n \to Y$.  Then $i^*[-n]$ is left t-exact. Furthermore, let $j_U \colon U \to Y$ be a smooth dense open subscheme and let $\sh{F}$ be lisse on $U$, perverse on $Y$, and suppose $(j_U)_! \sh{F}|_U \cong \sh{F}$.  Then if $U \cap Y_i \neq \emptyset$ for all $i < n$, then $i^*(\sh{F})[-n]$ is perverse.
\end{theorem}

\begin{proof}
 If we write $i_k$ for the inclusion of $Y_k$ in $Y_{k + 1}$, then $i$ is the composition of all the $Y_k$'s.  By \cite{BBD}*{Corollaire 4.1.10(ii)}, each $i_k$ has cohomological amplitude $[-1, 0]$, so their composition has amplitude $[-n, 0]$, as desired.
 
 For the second claim, we also apply induction.  Indeed, let $i_{n - 1} \colon Y_{n - 1} \to Y$, so since $U \cap Y_{n - 1} \neq \emptyset$, we have $i_{n - 1}^* (j_U)_! \cong (j_{U \cap Y_{n - 1}})_! i_{n - 1}^*$.  In particular, the natural map
 \begin{equation*}
  (j_{U \cap Y_{n - 1}})_! (i_{n - 1}^* \sh{F})|_{U \cap Y_{n - 1}}
   \to i_{n - 1}^* \sh{F}
 \end{equation*}
 is again an isomorphism.  By induction, it is an isomorphism of perverse sheaves after shifting by $n - 1$.  Thus, we may assume that $n = 1$.  But then the natural six-functors triangle
 \begin{equation*}
  (j_U)_! j_U^* \sh{F} \to \sh{F} \to i_* i^* \sh{F} \to
 \end{equation*}
 shows that ${}^p H^0 (i^* \sh{F}) = 0$, and we have already shown above that $i^* \sh{F}[-1]$ has cohomological amplitude $[0,1]$, so it is perverse.
\end{proof}

\subsection*{Tannakian duality}

We will show that $\cat{Perv}(\stack{G})$ is equivalent to representations of the following group:

\begin{theorem}{defn}{dual root data}
 Let $G$ be a complex reductive group, $T \subset G$ a maximal torus with Weyl group $W$, and $Q \in Q(\Lambda_T, k^*)^W_\Z$ as in \ref{sf gerbes exact sequence}; as above, let $\kappa$ be its symmetric bilinear form and $\Lambda_Q = \on{ker}(\kappa)$.  We define the root data of the \emph{twisted dual group} $\check{G}_Q$ by the following description: it has a maximal torus $\check{T}_Q$ with $X^*(\check{T}_Q) = \Lambda_Q$, and its roots are all multiples of the coroots of $G$, and vice-versa, as follows.  For every coroot $\check\alpha$, let $r = \on{ord} Q(\check\alpha)$ if this is finite.  Then we have the corresponding root $\alpha_Q = r\check\alpha$ and coroot $\check\alpha_Q = \frac{1}{r}\alpha$.
\end{theorem}

\begin{theorem}{prop}{dual is root data}
 The root datum defined above is in fact a root datum.
\end{theorem}

\begin{proof}
 Using the equation of \ref{invariant form and pairing} with $\epsilon_{\check\alpha} = 1$, we see that $Q(\check\alpha)^{\langle r\check\alpha, \lambda \rangle} = 1$ for all $\lambda$ indeed implies $r\check\alpha \in \Lambda_Q$.  Conversely, if $\lambda \in \Lambda_Q$, then in particular we have $\kappa(\check\alpha, \lambda) = 1$, so $r \mid \langle \alpha, \lambda \rangle$ and $\frac{1}{r} \alpha \in (\Lambda_Q)^\vee$.  Finally, note that if $Q(\check\alpha)$ has infinite order, then no multiple of $\check\alpha$ is in $\Lambda_Q$.
 
 It is obvious that $\langle \alpha_Q, \check\alpha_Q \rangle = \langle \alpha, \check\alpha \rangle = 2$, and in particular, $s_{\alpha_Q} = s_{\check\alpha}|_{\Lambda_Q}$.  Since $Q$ is $W$-invariant, we have $Q(s_{\check\alpha}\check\beta) = Q(\check\beta)$ for any coroots $\check\alpha, \check\beta$, so the multiplier for a coroot is constant in $W$-orbits; it then follows that $s_{\alpha_Q}(\beta_Q)$ is another root in $\Lambda_Q$.
\end{proof}

Since the codomain of $F$ is not in fact the category of vector spaces we must make some preparations before applying Tannakian duality.  It follows from the combination of \ref{ULA local equivalence} and \ref{bilinear form triviality} that the values of $F$ are supported on $\on{Gr}_{{}^L \check{T}_Q} \subset \on{Gr}_T$, and in fact, this restriction is sharp:

\begin{theorem}{lem}{full irreducible objects}
 Let $\lambda \in \Lambda_Q$ be any coweight which is dominant in $\Lambda$.  Then $\stack{G}$ is equivariantly trivial on $\on{Gr}_G^\lambda$ and, thus, we have an irreducible object $\sh{J}(\lambda) \in \Sph(\stack{G})$.
\end{theorem}

\begin{proof}
 Let $U \subset \on{Gr}_G^\lambda$ be the intersection $\on{Gr}_G^\lambda \cap \on{Gr}_B^\lambda$, which by \ref{absolute orbit properties} is a dense open subset containing $\smash{t^\lambda = \on{Gr}_T^\lambda}$; since the Weyl group $W$ permutes the Borel subgroups, we also have open subsets $U^w$ for $w \in W$ contaning $t^{w(\lambda)}$, and these cover $\on{Gr}_G^\lambda$.  On each, $\stack{G}$ is trivial, since it descends to $\on{Gr}_T$ from $\on{Gr}_B$.  Furthermore, clearly $\smash{B(\hat{\OO})}$ acts on $U$, so $\stack{G}|_U$ is trivially $\smash{B(\hat{\OO})}$-equivariant. Likewise for the action of $\smash{wBw^{-1}(\hat{\OO})}$ on $U^w$.
 
 Thus, the sheaf $\sh{F}_{U^w} = \csheaf{k}[-\dim \on{Gr}_G^\lambda]$ on $U^w$ is perverse and $\stack{G}|_{U^w}$-twisted, and so we have the $\stack{G}$-twisted perverse sheaf
 \begin{equation*}
  \sh{F} = \bigoplus_W j^w_{!*}(\sh{F}_{U^w}),
 \end{equation*}
 where $\map{j^w}{U^w}{\on{Gr}_G^\lambda}$ is the open immersion.  For any $b \in B(\hat{\OO})$ and any $w \in W$, we have $b U^w \cap U^w \neq \emptyset$ since $\smash{\on{Gr}_G^\lambda}$ is irreducible, so there is a natural $\smash{B(\hat{\OO})}$-equivariance on $\sh{F}$ obtained by matching the summands of $b^* \sh{F}$ to those of $\sh{F}$ on these open intersections.  Since clearly it is $W$-equivariant, by the Bruhat decomposition we have natural isomorphisms $g^* \sh{F} \cong \sh{F}$ for any $g \in G(\hat{\OO})$, giving $\sh{F}$ a $G(\hat{\OO})$-equivariance structure.  By \ref{transitive equivariant objects}, $\stack{G}$ is trivial on $\on{Gr}_G^\lambda$. 
\end{proof}

Identifying perverse sheaves on this space with $\Lambda_Q$-graded, finite-dimensional vector spaces (the analogue of \ref{main theorem for torus}), the lemma shows that every weight of $\check{T}_Q$ can be obtained from an object of $\cat{Sph}(\stack{G})$.  However, it is not the case that $F$ is a \emph{tensor} functor to $\check{T}_Q$-representations.

\begin{theorem}{lem}{variant tensor category}
 We have natural identifications
 \begin{equation*}
  \cat{Sph}(\stack{T})
   \cong
  \stack{Z} \otimes \cat{Rep}(\check{T}_Q)
   \cong
  \cat{Rep}(\check{T}_Q)',
 \end{equation*}
 where $\stack{Z} = \stack{Z}_1|_x$ (with $x \in X$ as before) is the fiber of the sf comultiplicative gerbe considered in \ref{main theorem for torus}, and thus a comultiplicative $A$-gerbe on $X$ (\ref{non-sf restriction}), and $\cat{Rep}(\check{T}_Q)'$ means altering the commutativity constraint in $\cat{Rep}(\check{T}_Q)$ such that the product of two spaces of weights $\lambda$ and $\mu$ (with $\lambda, \mu \in \Lambda_Q \subset \Lambda_T$) has the sign $Q(\lambda) Q(\mu)$.
\end{theorem}

\begin{proof}
 We use \ref{two kinds of twisting} to obtain the first equivalence; the equivalence of the first and third categories is the content of \ref{twisted G-equivariant pairing}, where the actual number was computed in \ref{multiplicative factorizable}.
\end{proof}

We note that, in the lemma, the gerbe $\stack{Z}$ is trivial (since it is a gerbe on $\on{Spec} \C$), but not comultiplicatively trivial, in an instance of the principle suggested above \ref{G-equivariant pairing}.  We also note that, according to \ref{dual root data} and the number appearing in the lemma, $\stack{Z}$ is in fact a gerbe not for $\check{T}_Q(k)$ but for $Z(\check{G}_Q)(k)$, since $Q(\lambda) = 0$ by definition if $\lambda$ is in the root lattice of $\check{G}_Q$.

We thus have $F \colon \cat{Sph}(\stack{G}) \to \stack{Z} \otimes \cat{Rep}(\check{T}_Q)$, and we will want to move the $\stack{Z}$ from the right to the left.  This is achieved by another small computation:

\begin{theorem}{lem}{fundamental group action}
 Every object of $\cat{Sph}(\stack{G})$ admits an action of $Z(\check{G}_Q)$, giving $\cat{Sph}(\stack{G})$ an action of the ``sheaf of categories'' $\stack{H}^1(\on{Spec} \C, Z(\check{G}_Q))$.  With this structure, the functor
 \begin{equation*}
  F \colon \cat{Sph}(\stack{G}) \to \stack{Z} \otimes \cat{Rep}(\check{T}_Q)
 \end{equation*}
 is $Z(\check{G}_Q)(k)$-equivariant.
\end{theorem}

\begin{proof}
 Such an action of a torus is the same as a grading by its weight lattice. Every object of $\cat{Sph}(\stack{G})$ has a natural grading by $\pi_1(G)$, since the components of $\on{Gr}_G$ are identified with this set; in fact, this grading is compatible with convolution, as follows from the fusion product (\ref{convolution properties}\ref{en:convolution is perverse}). Since $\pi_1(G)$ is a quotient of $\Lambda_T$ and $\Lambda_Q$ is a subgroup, and since the root lattice of $\check{G}_Q$ is contained in the coroot lattice of $G$, we obtain a grading by $X^*(Z(\check{G}_Q))$.
 
 Ignoring the twisting by the trivial gerbe $\stack{Z}$, objects of the category $\cat{Rep}(\check{T}_Q)$ have a natural grading by $Z(\check{G}_Q)$, simply by grouping all the graded parts of each object whose degrees differ by sums of coroots of $G$.  According to \ref{borel components}\ref{en:borel orbit closures} and \ref{absolute fiber functor}, this makes $F$ compatible with the gradings, i.e.\ equivariant.
\end{proof}

Multiplying $F$ by the inverse of $\stack{Z}$, we get a new tensor functor
\begin{equation}
 \label{eq:twisted fiber functor}
 F' \colon \stack{Z}^{-1} \otimes \cat{Sph}(\stack{G}) \to \cat{Rep}(\check{T}_Q).
\end{equation}
Note that since $\stack{Z}$ is trivial, the left-hand side merely has its commutativity constraint altered as well (by the same factor as above).  By the Tannakian duality theorem of \cite{DM_Tannakian}, the functor $F'$ induces an equivalence
\begin{equation*}
 \stack{Z}^{-1} \otimes \Sph(\stack{G}) \cong \cat{Rep}(\check{G}),
\end{equation*}
for some pro-algebraic $k$-group $\check{G}$ with a map from $\check{T}_Q$.

\subsection*{Properties of the dual group}

Having shown that $\check{G}_Q$ is indeed the dual group for the twisted Satake equivalence, we proceed to identify it precisely.

\begin{theorem}{prop}{connected and algebraic}
 The group $\check{G}$ is algebraic, connected, and reductive.
\end{theorem}

\begin{proof}
 We apply, respectively, the criteria of \cite{DM_Tannakian}*{Prop. 2.20(b), 2.22, 2.23}.  For algebraic, we choose any (finite) set of generators $\lambda_i$ of $\Lambda_Q$ and, by \ref{full irreducible objects}, construct the irreducible objects $\sh{J}(\lambda_i)$.  Then by \ref{convolution highest weight}, every object of $\Sph(\stack{G})$ is a subquotient of some tensor polynomial in $\bigoplus \sh{J}(\lambda_i)$, so this is a tensor generator.  For connected, we verify that no abelian subcategory generated by finitely many $\sh{J}(\lambda)$'s is closed under convolution, which follows from \ref{convolution highest weight} again since $\Lambda_Q$ is torsion-free.  Finally, for reductive, we have \ref{absolute semisimple}.
\end{proof}

The remainder of the section is devoted to the identification of $\check{G}$ as a reductive group. We will produce an isomorphism $\check{G}_Q \to \check{G}$ according to the following strategy: we will identify a Borel subgroup of $\smash{\check{G}}$ and show that its corresponding dominant weights are identified with those coweights in $\Lambda_Q$ which are dominant for $G$, and hence $\check{G}_Q$. Then we will pick out the simple roots of $\smash{\check{G}}$ and apply the following lemma:

\begin{theorem}{lem}{levi identification}
 Let $H$ and $H'$ be reductive groups with maximal tori $U, U'$ that are isomorphic; suppose further that $C$ and $C'$ are Borel subgroups containing these tori, the choice of which identifies the dominant weights in $X^*(U)$ with those in $X^*(U')$ under this isomorphism. Suppose that for every simple root $\alpha$ of $H$, with corresponding Levi factor $L$ (whose only simple root is $\alpha$), there is a commutative diagram
 \begin{equation} 
  \label{eq:levi diagram}
   \tikzsetnextfilename{levi_diagram}
   \begin{tikzpicture}[baseline = (current bounding box.center), row sep = 2.5em, column sep = 3em]
    \matrix [matrix of math nodes] (objects)
    {
     L & H\smash' \\ U & U\smash' \\
    };
    { [math arrows]
     \draw (objects-1-1) -- (objects-1-2);
     \draw (objects-2-1) -- (objects-1-1);
     \draw (objects-2-2) -- (objects-1-2);
    }
    \draw[equal] (objects-2-1) -- (objects-2-2);
   \end{tikzpicture}
 \end{equation}
 Then there is a unique isomorphism $H \to H'$ extending these maps.
\end{theorem}

To prove this, we need an even smaller lemma on algebraic groups, which proves itself.

\begin{theorem}{lem}{groups with the same torus}
 Let $K,L$ be reductive groups with maximal tori $S, U$.  Let $\map{f}{L}{K}$ be an algebraic group homomorphism such that $f|_S$ is an isomorphism of $S$ with $U$.  Let $\alpha$ be a root of $L$ and in the Lie algebra $\lie{l}$, let $v$ be a weight vector for the adjoint action of $L$, with weight $\alpha \in X^*(U) = X^*(S)$. Then $df(v)$ is a weight vector with weight $\alpha$ for the adjoint action of $K$ on $\lie{k}$, so $\alpha$ is a root of $K$. \qed
\end{theorem}

\begin{proof}[Proof of \ref{levi identification}]
 We apply \ref{groups with the same torus} to $H'$ and $L$, concluding that $\alpha$ is a root of $H'$ for any simple root $\alpha$ of $H$.  The collection of all the $\alpha$ determine the set of dominant weights in $X^*(U) = X^*(U')$ as those weights $\lambda$ such that $\langle \lambda, \check\alpha \rangle \geq 0$ for every $\alpha$ which is a simple root of $H$.  But that means that $\{\alpha\}$ determines the Weyl chamber of weights corresponding to $C'$, and therefore to the basis determined by $C'$.  Since $\{\alpha\}$ is a basis for the weight lattice, it is in fact the basis for the root sytem corresponding to $C'$.
 
 Thus, $H$ and $H'$ have the same simple roots; we claim that they have the same coroots as well. Indeed, each Levi $L$ corresponding to $\alpha$ has $\check\alpha$ as a simple coroot, and the map $L \to H'$ sends $\check\alpha$ to some coroot of $H'$ (which is of course equal to $\check\alpha$, since the tori in $L$ and $H'$ are identified).  By definition of the simple reflections, $\check\alpha$ is negated by $s_\alpha$, which means that it is a multiple of the coroot $\check\beta$ of $H'$ dual to $\alpha$, and since $\langle \alpha, \check\alpha \rangle = 2 = \langle \alpha, \check\beta \rangle$, that multiple must be $1$.
 
 Therefore there is an isomorphism between the root data of $H$ and $H'$, so there is a unique group isomorphism, identifying the Levi factors, that induces it.
\end{proof}

We have already argued that $\check{T}_Q$ admits a map into $\check{G}$ because $F'$ takes values in $\Lambda_Q$-graded vector spaces.  By \ref{full irreducible objects}, in fact the image of $F'$ generates $\cat{Rep}(\check{T}_Q)$, so by \cite{DM_Tannakian}*{Prop. 2.21(b)} this map is a closed immersion.  To show it is maximal, we consider the Borel subgroup.

By the Pl\"ucker relations, a Borel subgroup in $\check{G}$ is specified by giving, for every irreducible $V^\lambda \in \cat{Rep}(\smash{\check{G}})$, a line $\ell^\lambda$ in it, such that for any other irreducible $V^\mu$, the line $\ell^\lambda \otimes \ell^\mu \in V^\lambda \otimes V^\mu$ agrees with the line $\ell^{\lambda + \mu}$ lying in the (unique) summand $V^{\lambda + \mu}$ inside the tensor product.  We take $\ell^\lambda$ to be the weight space $F'(\sh{J}(\lambda))^\lambda$ and apply \ref{convolution highest weight} and the fact that $F'$ is a tensor functor to obtain this for $\Sph(\stack{G})$.  We will denote this Borel subgroup by $\check{B}_Q$; by definition, the dominant weights of $\check{G}$ with respect to $\check{B}_Q$ are exactly the dominant $\lambda$ which lie in $\Lambda_Q$.

Since every $\ell^\lambda$ is a $\check{T}_Q$-weight line, $\check{T}_Q \subset \check{B}_Q$ is a closed subgroup.  It is in fact a maximal torus, since the $\lambda$-weight line has multiplicity one. Thus, the first two criteria of \ref{levi identification} are satisfied, and we need only produce the maps of Levi factors.

\subsection*{Semisimple rank 1}

We now appeal to the ``parabolic fiber functors'' $F^P_n$ defined in \ref{parabolic fiber}, following the apparently original appearance of this technique in \cite{BD_quantization}*{5.3.27--31}. Let $\alpha$ be any simple root of $G$ and let $P$ be the parabolic subgroup corresponding to $\alpha$, with Levi quotient $L$, a reductive group of semisimple rank $1$ having the same maximal torus $T$.  We denote by $\stack{L}$ the gerbe on $\on{Gr}_L$ induced by $\stack{G}$. As for the plain fiber functor $F$, by \ref{parabolic fiber functor properties} it induces a faithful, exact tensor functor
\begin{equation*}
 \map{F^P}{\Sph(\stack{G})}{\Sph(\stack{L})}
\end{equation*}
factoring the fiber functor.  By the Tannakian formalism we have a group homomorphism $\check{L} \to \check{G}$.  It remains only to prove that $\check{L} \cong \check{L}_Q$ in order to apply \ref{levi identification}, so we now assume that $G$ has semisimple rank $1$.  The following lemma allows us to focus on the \emph{semisimple} groups of rank $1$:

\begin{theorem}{lem}{root multiples}
 The root lattice of $\check{G}$ in $\Lambda_Q$ must be contained in the coroot lattice of $G$, so that when $G$ has semisimple rank $1$, the roots of $\check{G}$ are multiples of the coroots of $G$.
\end{theorem}

\begin{proof}
 A highest-weight representation $V^\lambda$ of a reductive group has its highest weight $\lambda$ in the root lattice of that group if and only if its weight space $V^\lambda(0) \neq 0$.  By \ref{absolute fiber functor}, this holds in $\Sph(\stack{G})$ only if $\on{Gr}_B^0 = \on{Gr}_T^0$ intersects the support of $\sh{J}(\lambda)$, so $\on{Gr}_G^\lambda$ must be in the connected component of $\on{Gr}_G$ which contains $\on{Gr}_T^0$, so $\lambda$ is a coroot multiple.
\end{proof}

To pass from a general group $G$ of semisimple rank $1$ to a semisimple one, we replace $G$ by its (semisimple) derived subgroup $G' = [G,G]$, whose coweight lattice is the rational coroot lattice $\Lambda_{T,r}^\Q$ as defined in \ref{integer-valued forms}.  Suppose we write, similarly to that lemma but using \ref{sf gerbes exact sequence},
\begin{equation*}
 \stack{G}_n = \stack{K}_n \otimes \stack{M}_n,
\end{equation*}
where $\stack{K}_n$ is of ``Killing type'', in that its quadratic form is a product of Killing forms and its associated multiplicative gerbe is trivial, and $\stack{M}_n$ is of ``multiplicative type'', in that its quadratic form descends to $\Lambda_T / \Lambda_{T,r}^\Q$.  Therefore, we have that
\begin{align*}
 Q_{\stack{G}} = Q_{\stack{K}} Q_{\stack{M}} &&
 \kappa_{\stack{G}} = \kappa_{\stack{K}} \kappa_{\stack{M}}
\end{align*}
and since $Q_{\stack{M}}(\Lambda_{T,r}^\Q) = 1$ and $\kappa_{\stack{M}}(\Lambda_{T,r}^\Q, \Lambda_T) = 1$, we have that $\stack{M}$ is $G(\hat{\OO})$-equivariantly trivial on $\on{Gr}_{G'}$, and thus that
\begin{equation*}
 \Sph(\stack{G}|\on{Gr}_{G'}) \cong \Sph(\stack{K}|\on{Gr}_{G'}),
\end{equation*}
that the latter category is independent of whether we consider $G(\hat{\OO})$-equivariant sheaves or $G'(\smash{\hat{\OO}})$-equivariant sheaves (since if $\langle \alpha, \lambda \rangle = 1$ for all roots $\alpha$, we also have $\kappa_{\stack{K}}(\alpha, \lambda) = 1$; thus, the presence of the central coweights does not affect $\Lambda_{Q_{\stack{K}}}$), and finally, that the induced map
\begin{equation*}
 \check{G} \to (G')^\vee,
\end{equation*}
which is a surjection by \cite{DM_Tannakian}*{Proposition 2.21(a)}, induces an isomorphism on root lattices.  Since $\Lambda_{T,r}^\Q$ contains (by construction) all fractional multiples of the coroots of $G$, the multipliers appearing in \ref{dual root data} are the same for $G$ as for $G'$, so we may replace $G$ with $G'$.

Therefore, we may assume that $G = \on{SL}_2, \on{PGL}_2$ is one of the two semisimple groups of rank~$1$.  We will denote by $\stack{G}_{\on{SL}_2}$ or $\stack{G}_{\on{PGL}_2}$ the sf gerbe on their grassmannians.  Then $\Lambda_T \cong \Z$ for each, the nontrivial element of $W$ acts as negation, and the quadratic form $Q$ is necessarily of the form $Q(n) = q_0^{n^2}$ for some $q_0 \in k^*$. Clearly, the natural map $\on{SL}_2 \to \on{PGL}_2$ induces an inclusion $\Lambda_{\on{SL}_2} = 2\Lambda_{\on{PGL}_2}$, and if we identify the latter with $\Z$, sends $Q$ (for $\on{PGL}_2$) to the quadratic form on $\Lambda_{\on{SL}_2}$, $n \mapsto q_0^{4n^2}$.  The following result is a quick computation:

\begin{theorem}{lem}{primed lattice in rank 1}
 Suppose $G = \on{PGL}_2$, $\tilde{G} = \on{SL}_2$ and identify the lattice of the former with $\Z$, that of the latter with $2\Z$; let $q_0 = Q(1)$ have order $r_0 \in \N \cup \{\infty\}$. Then $\Lambda_{Q,G} = \Lambda_{Q,\smash{\tilde{G}}}$ is zero if $r_0$ is infinite, and otherwise:
 \begin{itemize}
  \item If $r_0$ is odd, then $\Lambda_{Q,\smash{\tilde{G}}} = r_0 \Z = 2\Lambda_{Q,G}$.
  \item If $\on{ord}_2 r_0 = 1$, then $\Lambda_{Q,\smash{\tilde{G}}} = r_0 \Z = 2\Lambda_{Q,G}$.
  \item If $\on{ord}_2 r_0 = 2$, then $\Lambda_{Q,G} = (r_0/2)\Z = \Lambda_{Q,\smash{\tilde{G}}}$.
  \item If $\on{ord}_2 r_0 \geq 3$, then $\Lambda_{Q,G} = (r_0/2)\Z = 2\Lambda_{Q,\smash{\tilde{G}}}$. \qed
 \end{itemize}
\end{theorem}

Now we are prepared to compute $\check{G}$ when $G$ has semisimple rank $1$.

\begin{theorem}{prop}{roots for semisimple rank 1}
 Let $G$ be a semisimple group of rank $1$ with simple coroot $\check\alpha$; let $r$ be the order of $q = Q(\check\alpha) \in k^*$.  Then the positive root of $\check{G}$ is $r \check\alpha$, if $r$ is finite, and otherwise $\check{G}$ is trivial.
\end{theorem}

\begin{proof}
 We consider the map $\on{SL}_2 \to \on{PGL}_2$, which induces a map $\on{Gr}_{\on{SL}_2} \to \on{Gr}_{\on{PGL}_2}$ identifying the former with the connected component of the latter which contains $\on{Gr}_T^0$.  Let $\stack{G}_{\on{PGL}_2}$ be the fiber of an sf gerbe on the latter, with quadratic form $Q$, and let $\stack{G}_{\on{SL}_2}$ be its restriction to this component, inducing a form on the weight lattice of $\on{SL}_2$.  Since we have $\on{SL}_2 \to \on{PGL}_2$, the natural restriction makes $\stack{G}_{\on{PGL}_2}$ an $\on{SL}_2(\hat{\OO})$-equivariant gerbe on the common component and we have $\Sph(\stack{G}_{\on{PGL}_2})^\circ \subset \Sph(\stack{G}_{\on{SL}_2})$, the circle denoting objects supported on $\on{Gr}_{\on{SL}_2}$.

 If $r$ is infinite, then $\kappa$ is nondegenerate on $\Lambda_T$ (for either $\on{SL}_2$ or $\on{PGL}_2$), so $\Lambda_Q = 0$ and therefore, since $F$ is faithful, $\Sph(\stack{G}_{\on{SL}_2}) = \Sph(\stack{G}_{\on{PGL}_2}) \cong \cat{Vect}_k$ (generated by the trivial sheaf supported on $\on{Gr}_G^0$).  Thus, $\check{G} = 1$ in either case.
 
 If not, then we have a few possibilities.  First, note that when $\Lambda_Q \neq 0$, the dual group $\smash{\check{\on{SL}}_2}$ is not a torus, since in fact any $\sh{J}(\lambda)$ with $\lambda \neq 0$ is irreducible but has at least two nonzero weight spaces $F(\sh{J}(\lambda))^{\pm \lambda}$.
 
 Now suppose $r_0$ is odd; then, by \ref{primed lattice in rank 1}, we have $\Lambda_{Q, \on{SL}_2} \subset \Lambda_{Q, \on{PGL}_2}$, so by faithfulness of $F$, we also have $\Sph(\stack{G}_{\on{SL}_2}) \subset \Sph(\stack{G}_{\on{PGL}_2})$, so we have a surjection $\check{\on{SL}}_2 \to \check{\on{PGL}}_2$ of which the first, as noted, is not a torus so neither is the second.  They are thus both semisimple of rank $1$; since the map from one to the other sends roots to roots by \ref{groups with the same torus}, the computation in \ref{primed lattice in rank 1} shows that the common simple root must be $2r_0 = r_0 \check\alpha$.  Since $q_0$ has odd order, $q = q_0^4$ has the same order $r = r_0$.
 
 If $\on{ord}_2 r_0 = 1$, then we have the same inclusion of lattices, and therefore the common root is $r_0 = (r_0/2) \check\alpha$.  However, note that since $q = q_0^4 = (q_0^2)^2$, we actually have $r = r_0/2$, so again the root is $r \check\alpha$.
 
 If $\on{ord}_2 r_0 \geq 3$, then the inclusion of lattices is reversed.  Thus, $\Sph(\stack{G}_{\on{PGL}_2}) \subset \Sph(\stack{G}_{\on{SL}_2})$ and we have $\Lambda_{Q, \on{PGL}_2} = 2\Lambda_{Q, \on{SL}_2}$, so that by the computations of the lemma the common root is necessarily $r_0/2 = (r_0/4) \check\alpha$.  Since $q = q_0^4$ and $q_0$ has order a multiple of $4$, in fact $r = r_0/4$, so again the root is $r\check\alpha$.
 
 Now suppose $\on{ord}_2 r_0 = 2$.  Then by the computations of \ref{primed lattice in rank 1}, we have $\Lambda_{Q, \on{SL}_2} = \Lambda_{Q, \on{PGL}_2}$, so $\check{\on{SL}}_2 \cong \check{\on{PGL}}_2$ but we cannot extract the root from this comparison.  The quadratic form on $\Lambda_{\on{SL}_2}$ is defined by $q = q_0^4$, which has odd order $r = r_0/4$, and is therefore the fourth power of some other $q_1$ with the same order (namely, $q_1 = q^a$, where $4a + rb = 1$).  We replace $\stack{G}_{\on{PGL}_2}$ by the fiber of the sf gerbe defined by the quadratic form corresponding to $q_1$; then we are back in the case where $r_0$ is odd, and so $\check{\on{SL}}_2 = \on{PGL}_2$, so the root is the smallest element of $\Lambda_{Q, \on{SL}_2} = \Lambda_{Q, \on{PGL}_2}$, namely $r_0/2 = 2(r_0/4) = r\check\alpha$ once again.
 
 In all of the above computations, we have shown that when $\stack{G}_{\on{SL}_2}$ is inherited from $\on{PGL}_2$, then the proposition holds.  Suppose we are given only an sf $k^*$-gerbe $\stack{G}_{\on{SL}_2,n}$ with its quadratic form, determined by the number $q$.  Let $\map{\phi}{k}{l}$ be an inclusion of fields in which $q = q_0^4$ with $q_0 \in l$, and let $\stack{H}_{\on{SL}_2,n} = {}^2 \phi(\stack{G}_{\on{SL}_2})$ be the induced sf $l^*$-gerbe.  Let $X$ be a small disk and, by \ref{sf gerbes exact sequence}, let $\stack{H}_{\on{PGL}_2, n}$ be the sf gerbe corresponding to the $l^*$-valued quadratic form defined by $q_0$.  Then both gerbes are entirely determined by their quadratic form, so in particular $\stack{H}_{\on{SL}_2}$ really is determined by $\stack{H}_{\on{PGL}_2}$ by the map $\on{Gr}_{\on{SL}_2} \to \on{Gr}_{\on{PGL}_2}$.  Then by the above computations, the split $l^*$-group $\check{G}_l$ obtained from $\on{SL}_2$ from $\stack{H}_{\on{SL}_2}$ has the correct root data.  However, clearly 
we have $\Sph(\stack{H}_{\on{SL}_2}) = \Sph(\stack{G}_{\on{SL}_2}) \otimes_k l$, so $\check{G}_l \cong \check{\on{SL}}_2 \otimes_{\on{Spec} k} \on{Spec} l$ and so $\check{\on{SL}}_2$ has the correct root system as well.  This completes the proof.
\end{proof}

Going back to \ref{eq:twisted fiber functor}, we can restate our result in the form originally intended:

\begin{theorem}{thm}{absolute twisted satake}
 The fiber functor $F$ induces an equivalence
 \begin{equation*}
  \Sph(\stack{G}) \cong \stack{Z} \otimes \cat{Rep}(\check{G}_Q) \cong \cat{Rep}(\check{G}_Q)',
 \end{equation*}
 where the prime denotes modifying the commutativity constraint of the tensor product of spaces of weight $\lambda$ and $\mu$ by $Q(\lambda) Q(\mu)$. \qed
\end{theorem}

\section{Relative twisted Satake}
\label{s:relative twisted satake}

All that has already come is sufficient preparation to easily deduce non-local versions of the results of the previous section.

\begin{theorem}{prop}{gerbe is defined over center}
 Let $\stack{G}_n$ be an sf $k^*$-gerbe on $\on{Gr}_{G,X^n}$, and let $\stack{Z}_n$ be the sf comultiplicative $\on{Fact}(\smash{\check{T}_Q}(k))_n$-gerbe on $X^n$ which has already appeared in \ref{main theorem for torus}.  Then $\stack{Z}_n$ is in fact defined over $\on{Fact}(\smash{Z(\check{G}}_Q)(k))_n$.
\end{theorem}

\begin{proof}
 Let $\stack{T}_n$ be the sf gerbe on $\on{Gr}_{T,X^n}$ corresponding to $\stack{G}_n$; by \ref{multiplicative factorizable lattice gerbes}, we must show that $\stack{T}_n$ is sf-trivial on the sf-scheme within $\on{Gr}_{T,X^n}$ consisting of components $\on{Gr}_{T,X^n}^{\lambda_1, \dots, \lambda_n} \cong X^n$ in which all the $\lambda_i$ are in the root lattice of $\check{G}_Q$.  By \ref{torus exact sequence}, it suffices to show that the quadratic form $Q$ (which is a homomorphism on $\Lambda_Q$ since its associated bilinear form is trivial) and the associated commutative multiplicative gerbe $\stack{M}$ are trivialized on the roots of $\check{G}_Q$.  This is tautologically true for $Q$ by \ref{dual root data}.  Similarly, $\stack{M}$, as a $\Lambda_T$-multiplicative gerbe, is trivialized on the coroot lattice $\Lambda_{T,r}$ (that is, it is multiplicative for $\pi_1(G) = \Lambda_T/\Lambda_{T,r}$), and therefore also on the root lattice of $\check{G}_Q$, again by definition.
\end{proof}

\begin{theorem}{thm}{ULA main theorem}
 With notation as in \ref{gerbe is defined over center}, $F_n$ induces an equivalence of rigid tensor abelian categories
 \begin{equation*}
  \Sph^\ULA(\stack{G}_n) \cong \stack{Z}_n \otimes \cat{Rep}^\ULA(\check{G}_Q^n),
 \end{equation*}
 where the representations are taken in local systems (i.e.\ ULA perverse sheaves) on $X^n$.
\end{theorem}

\begin{proof}
 We have by definition
 \begin{equation*}
  F_n \colon \Sph^\ULA(\stack{G}_n) \to \Sph^\ULA(\stack{T}_n),
 \end{equation*}
 where $\stack{T}_n$ is, as usual, the sf gerbe on $\on{Gr}_{T,X^n}$ corresponding to $\stack{G}_n$. By \ref{main theorem for torus} and \ref{pushforward is tensor}, we have $\Sph^\ULA(\stack{T}_n) \cong \stack{Z}_n \otimes \cat{FRep}^\ULA_n(\check{T}_Q)$ as tensor categories, the latter identified with factorizable representations of $\check{T}_Q$ in locally constant sheaves, so we collapse the notation and write
 \begin{equation*}
  F_n \colon \Sph^\ULA(\stack{G}_n) \to \stack{Z}_n \otimes \cat{FRep}^\ULA_n(\check{T}_Q).
 \end{equation*}

 The first step is to prove the analogue of \ref{fundamental group action} for $\Sph^\ULA(\stack{G}_n)$ and $F_n$ so that we can move $\stack{Z}_n$ from one side to the other. The proof is the same: over any open set $U \subset X^n$ having $m$ distinct coordinates, the connected components of $\on{Gr}_{G,X^n}|_U$ are indexed by $\pi_1(G)^n$ and thus objects of $\Sph^\ULA(\stack{G}_n|_U)$ are graded by $\pi_1(G)^m$, and thus degenerately by $Z(\check{G}_Q)^m$. Since this indexing of components corresponds to the inclusion of the various $\on{Gr}_{T,X^n}^{\lambda_1, \dots, \lambda_n}$, \ref{torus grassmannian properties} shows that the corresponding actions of $Z(\check{G}_Q)^m$ assemble to an action of $\on{Fact}(Z(\smash{\check{G}_Q}))_n$. As in \ref*{fundamental group action}, one sees that $F_n$ is equivariant for this action, so we may consider the functor
 \begin{equation*}
  F_n' \colon \stack{Z}_n^{-1} \otimes \Sph^\ULA(\stack{G}_n) \to \cat{FRep}^\ULA(\check{T}_Q^n).
 \end{equation*}
 By \ref{fiber functor properties}, $F_n$ is faithful and exact and, by \ref{fiber functor commutative}, a tensor functor, so these are true of $F_n'$ as well; the last one, which is not tautological, follows from \ref{twisted pairing} using the sf multiplicativity of $\stack{Z}_n$ (\ref{multiplicative factorizable lattice gerbes}). 
 
 By \ref{ULA local equivalence} and \ref{absolute twisted satake}, in a small neighborhood $D$ of any $\vect{x} \in X^n$ with distinct coordinates, $F_n'$ induces an equivalence
 \begin{equation*}
  \stack{Z}_n^{-1}|_D \otimes \Sph(\stack{G}_n|_D) \cong \cat{Rep}^\ULA(\check{G}_Q^n).
 \end{equation*}
 These equivalences are compatible with restriction to smaller neighborhoods and so, by \ref{group action gluing}, induce a $\check{G}_Q^n$-action on every $\sh{F} \in \stack{Z}_n^{-1} \otimes \Sph^\ULA(\stack{G})$.
 
 If $U \subset X^n$ is the open set with distinct coordinates, then
 \begin{equation*}
  F_n \colon \Sph^\ULA(\stack{G}_n|_U) \to \stack{Z}_n|_U \otimes \cat{FRep}^\ULA_n(\check{T}_Q)
 \end{equation*}
 factors through $\stack{Z}_n|_U \otimes \cat{Rep}^\ULA_n(\check{G}_Q^n)_U$, where
 \begin{equation*}
  \cat{Rep}^\ULA_n(\check{G}_Q^n)_U \cong \cat{FRep}^\ULA_n(\check{G}_Q)_U.  
 \end{equation*}
 Since restriction of ULA sheaves to $U$ is inverted by minimal extension by \ref{ULA nearby cycles}, in fact $F_n$ factors through globally as well.
\end{proof}

We now state the main result of this work.

\begin{theorem}{fancythm}{main theorem}
 With notation as in \ref{gerbe is defined over center}, the fiber functor $F_n$ induces an equivalence
 \begin{equation*}
  \Sph(\stack{G}_n) \cong \stack{Z}_n \otimes \cat{FRep}_n(\check{G}_Q),
 \end{equation*}
 compatibly with their respective outer convolutions.
\end{theorem}

\begin{proof}
 The combination of the definition of $F_n$ and \ref{main theorem for torus} gives a map from $\Sph(\stack{G}_n)$ to $\stack{Z}_n \otimes \cat{FRep}_n(\check{T}_Q)$, and \ref{fiber functor properties} shows that it respects convolution.  \ref{devissage} may be applied since $F_n$ is horizontal and preserves ULA sheaves, so we successively construct a functor $\Sph(\stack{G}_n) \to \stack{Z}_n \otimes \cat{FRep}(\check{G}_Q)$, an isomorphism verifying that it factors $F_n$ through the forgetful functor to $\stack{Z}_n \otimes \cat{FRep}(\check{T}_Q)$, and an inverse functor making it an equvalence.
\end{proof}

\chapter{Connections}
\label{c:connections}

The main result of this work, \ref{main theorem}, is (at present) the endpoint of a long chain of development.  Thus, we close with a few remarks on how it fits into the larger scheme of things.

\section{Relation to the result of Finkelberg--Lysenko}
\label{s:relation to the result of Finkelberg--Lysenko}
 
The main result of Finkelberg--Lysenko \cite{FL_twistedSatake} can be stated in our language as follows, assuming as usual that $G$ is a complex reductive algebraic group and $T$ a maximal torus (however, that paper considers more general algebraically closed fields of any characteristic).

\begin{theorem*}{thm}
 (Finkelberg--Lysenko, \cite{FL_twistedSatake}*{Theorem 1})
 Let $d$ be fixed as described in \textit{loc.\ cit.}\ following their Proposition 1, let $\check{h}$ be the dual Coxeter number of $G$, and choose a positive integer $N$.  Denote by $E^c$ a $2\check{h}/d$'th root of the determinant line bundle $\det(\on{adj})$ on the \emph{ordinary} affine grassmannian $\on{Gr}_G$ and let $q$ be an $N$'th root of unity.  Write $\stack{G} = (E^c)^{\log q}$; then we have
 \begin{equation*}
  \on{Sph}(\stack{G}) \cong \cat{Rep}(\check{G}),
 \end{equation*}
 where $\check{G}$ is the reductive algebraic group with the following root datum.  Fix the notation
 \begin{align*}
  \iota \colon \Lambda_T \to \Lambda^T, &&
  \iota(\lambda) &= \frac{1}{2\check{h}} \sum_\alpha \langle \alpha, \lambda \rangle \alpha \\
  (\farg, \farg) \colon \Lambda_T \otimes \Lambda_T \to \Z, &&
  (\lambda, \mu) &= \langle \iota(\lambda), \mu \rangle
 \end{align*}
 (the sum running over all roots $\alpha$).  Then we have for $\check{G}$:
 \begin{enumerate}
  \item \label{en:fl weight lattice}
  Its weight lattice consists of those $\lambda \in \Lambda_T$ such that $d \iota(\lambda) \in
  N \Lambda^T$;
  
  \item \label{en:fl roots}
  For each coroot $\check\alpha$ of $G$, the multiple $r\alpha$ is a root, where $r$ is the denominator of the fraction $d(\alpha, \alpha)/2N$ written in lowest terms.
 \end{enumerate}
 The coweights of $\check{G}$ are dual to its weights, and the coroots are the multiples $\alpha/r$.
\end{theorem*}

It should be noted that the statement of this theorem is incorrect in one particular: it lacks the modified commutativity constraint given in \ref{absolute twisted satake}.  This error can be traced to the beginning of \cite{FL_twistedSatake}*{\S4.2}, where Tannakian duality is prematurely invoked without fully computing the tensor category structure of the left-hand side.  However, the dual group itself agrees with ours:

\begin{theorem}{prop}{finkelberg dual group}
 Let $\stack{G}$ have quadratic form $Q$; then $\check{G} = \check{G}_Q$.
\end{theorem}

\begin{proof}
 Note that, since $(E^c)^{2\check{h}/d} = \det(\on{adj})$, we can write just as well that $\stack{G} = \det(\on{adj})^{d/2\check{h}N}$.  By \ref{determinant forms}, the associated quadratic form $Q$ is the exponential of a primitive $1^{d/2\check{h}N}$ by the integer-valued form
 \begin{equation}
  \label{eq:fl quadratic form}
  \frac{1}{2} \sum_\alpha \langle \alpha, \lambda \rangle^2
   = \check{h} \langle \iota(\lambda), \lambda \rangle
   = \check{h} (\lambda, \lambda),
 \end{equation}
 the sum running over all roots $\alpha$.  Its associated bilinear form $\kappa$ is likewise the exponential of this primitive root of unity by the integer-valued form
 \begin{equation}
  \label{eq:fl bilinear form}
  \sum_\alpha \langle \alpha, \lambda \rangle \langle \alpha, \mu \rangle
   = 2\check{h} \langle \iota(\lambda), \mu \rangle
   = 2\check{h} (\lambda, \mu).
 \end{equation}
 Thus, $\lambda \in \on{ker}(\kappa) = \Lambda^{\check{T}_Q}$ if and only if $2\check{h}N/d$ divides the value in \ref{eq:fl bilinear form} for all $\mu \in \Lambda_T$, or if $d \iota(\lambda) \in N \Lambda^T$, which is the condition of point \ref{en:fl weight lattice}.  In addition, we choose the roots in this lattice to be multiples $r\check\alpha$ of the coroots $\check\alpha$, where $r = \on{ord} Q(\check\alpha)$. According to \ref{eq:fl quadratic form}, $r$ is the least integer such that $2\check{h} r(\check\alpha, \check\alpha)$ is a multiple of $2\check{h}N/d$; i.e.\ the denominator of $d(\check\alpha, \check\alpha)/2N$, as given in point \ref{en:fl roots}.
\end{proof}

\section{Relation to Lusztig's quantum groups}
\label{s:relation to Lusztig's quantum groups}

In \cite{L_Quantum}*{\S2.2.5}, Lusztig transforms a root datum $(X_*, X^*, \Phi_*, \Phi^*)$ by means of a specified positive integer $l$ and a choice of associated ``Cartan datum'', defined as follows:

\begin{theorem*}{defn}
 (Lusztig \cite{L_Quantum}*{\S1.1.1})
 A \emph{Cartan datum} is a finite set $I$ together with a symmetric bilinear pairing
 \begin{equation*}
  \farg \cdot \farg \colon \Z[I] \otimes \Z[I] \to \Z
 \end{equation*}
 such that $i \cdot i$ is even and positive for all $i \in I$ and $2(i \cdot j)/(j \cdot j)$ is a nonpositive integer for all distinct $i,j \in I$.  A \emph{root datum of type I} is a quadruple $(X_*, X^*, \phi_*, \phi^*)$ consisting of a pair of dual finitely-generated free abelian groups $X_*, X^*$ and embeddings
 \begin{align*}
  \phi_* \colon I \to X_* &&
  \phi^* \colon I \to X^*
 \end{align*}
 such that we have $\langle \phi^*(i), \phi_*(j)\rangle = 2(i \cdot j)/(j \cdot j)$ for all $i,j$.
\end{theorem*}

However, since we begin with a root datum $(X_*, X^*, \Phi_*, \Phi^*)$ defined in the usual way, we prefer to consider this definition differently:

\begin{theorem}{defn}{cartan datum variant}
 Let $D = (X_*, X^*, \Phi_*, \Phi^*)$ be a root datum (that is, a pair of dual free abelian groups of finite type together with roots and coroots in them satisfying the usual axioms).  To specify a \emph{Cartan datum attached to $D$} is to give:
 \begin{enumerate}
  \item
  A base for $D$; that is, a set of positive roots and their corresponding coroots.  We denote by $I$ the abstract set of positive roots, $\phi_*$, $\phi^*$ its inclusions into $X_*$ and $X^*$.
  
  \item
  A $W$-invariant integer-valued quadratic form $f$ on $\Z[I]$ (which is identified via the $\phi$'s with  both the root and coroot lattices) taking positive values on $I$.
 \end{enumerate}
\end{theorem}

It is not hard to see that these are actually the same definition:

\begin{theorem}{lem}{equivalence of cartan data}
 \ref{cartan datum variant} agrees with Lusztig's definition.
\end{theorem}

\begin{proof}
 Let $b$ be the bilinear form associated with $f$ (that is, $f(i + j) = f(i) + b(i,j) + f(j)$) and define $i \cdot j = b(i,j)$.  Then $i \cdot i = b(i,i) = 2 f(i)$ is even (\textit{a fortiori}) and positive (\textit{a priori}).  That $2 (i \cdot j)/(j \cdot j)$ is nonpositive will follow from
 \begin{equation*}
  f(j) \langle \phi_*(i), \phi^*(j) \rangle = b(i,j),
 \end{equation*}
 since the pairing between two simple roots is always nonpositive.  In fact, this is just a restatement of $W$-invariance of $f$:
 \begin{align*}
  f(j)
   &= f(s_i(j))
   = f\bigl(j - \langle \phi^*(i), \phi_*(j) \rangle i\bigr) \\
   &= f(j) - \langle \phi^*(i), \phi_*(j) \rangle b(j,i)
          + \langle \phi^*(i), \phi_*(j) \rangle^2 f(i),
 \end{align*}
 so trivial algebra gives the desired equation when $\langle \phi^*(i), \phi_*(j) \rangle \neq 0$.  Otherwise, it follows from \ref{invariant form and pairing}, since $\Z$ is torsion-free.
\end{proof}

Note that in \ref{cartan datum variant}, we have $f(i) = (i \cdot i)/2$.  It is apparent
that, root datum in hand, the only substantial choice in the definition of an associated Cartan datum is that of the quadratic form $f$; we will continue to refer to the whole Cartan datum just using the letter $f$.  Following Lusztig, we make the following definition of another root datum:

\begin{theorem*}{defn}
 Fix a root datum $D$ and associated Cartan datum $f$ as in \ref{cartan datum variant}.   Let $l$ be any positive integer and for each $i$, let $l_i = l/\!\on{gcd}(l, f(i))$.  Then the dual root datum constructed from $f$ and $l$ is $D_{f,l} = (Y_*, Y^*, \Psi_*, \Psi^*)$, where
 \begin{enumerate}
  \item \label{en:lusztig weight lattice}
  $Y^* = \{\lambda \in X_* \mid \forall i \in I, \langle \phi^*(i), \lambda \rangle \in l_i \Z \}$.
  
  \item \label{en:lusztig roots}
  $\psi^*(i) = l_i \phi_*(i)$.
 \end{enumerate}
 Likewise, $Y_*$ is the dual of $Y_*$ and $\psi_*(i) = \phi^*(i)/l_i$.
\end{theorem*}

We will show that this construction almost coincides with a special case of that given in \ref{dual root data}.  To this end, let $q$ be a primitive $l$'th root of unity in $k$ (an algebraically closed field of characteristic zero, which can in fact be any ring for this construction) and define a $k^*$-valued quadratic form on $\Z[I]$ (equivalently, the coroot lattice in $X_*$) \begin{equation*}
 Q(\lambda) = q^{f(\lambda)}.
\end{equation*}
Since $Q$ is not defined on all of $X_*$ it is not possible to apply our construction literally. However, we can perform one much like it, as follows.  We extend the bilinear form $b$ defined by $f$ to $\Z[I] \otimes X_*$ by taking
\begin{equation*}
 b(i,\lambda) = f(i) \langle \phi^*(i), \lambda \rangle
\end{equation*}
and extending by linearity.  Note that $b(i,\phi_*(j)) = f(i) \langle \phi^*(i), \phi_*(j) \rangle = b(i,j)$ as previously defined/proven in \ref{equivalence of cartan data}.  If we take
\begin{equation*}
 \kappa(i, \lambda) = q^{b(i,\lambda)},
\end{equation*}
a $k^*$-valued bilinear form, then its restriction to $\Z[I] \otimes \phi_*(\Z[I])$ is the bilinear form defined by $Q$.  Furthermore, its right kernel in $X_*$ is the set of all $\lambda$ such that, for each $i \in I$, we have
\begin{equation*}
 f(i) \langle \phi^*(i), \lambda \rangle = b(i, \lambda) \in l\Z,
\end{equation*}
or equivalently, $\langle \phi^*(i), \lambda \rangle \in l_i \Z$, as required by \ref{en:lusztig weight lattice}.  Likewise, since $l_i$ is the least positive integer such that $l_i f(i) \in l\Z$, it is the order of $Q(i)$ and therefore taking the roots in this subspace to be $l_i \phi_*(i)$ agrees with \ref{en:lusztig roots}.  In summary:

\begin{theorem}{prop}{root data agreement}
 Let $G$ be a reductive group with root datum $D$, $\Lambda_T$ its coweight lattice and $\Lambda_{T,r}$ the coroot lattice. For any $k^*$-valued bilinear form $\kappa$ on $\Lambda_{T,r} \times \Lambda_T$ and  quadratic form $Q$ defining its restriction to $\Lambda_{T,r} \times \Lambda_{T,r}$, each coming from $W$-invariant forms valued in $\Z$ as in \ref{integer-valued forms}, we define dual root data $D_{Q, \kappa}$ as in \ref{dual root data} but taking $\Lambda^{\check{T}_Q}$ to be the right kernel of $\kappa$.  Then if $Q = q^f$ with $f$ an integer-valued quadratic form and $q$ an $l$'th root of unity, we have $D_{Q, \kappa} = D_{f, l}$. \qed
\end{theorem}

The generality of this construction is incomparable to that of \ref{dual root data}, though if $Q$ extends to all of $\Lambda$ and the resulting bilinear form has the same kernel as $\kappa$, the two give the same result.  We do not know of a geometric interpretation of such quadratic forms, though \ref{integer-valued forms} is relevant for bilinear forms defined on $\Lambda_{T,r} \times \Lambda_T$.  

In addition, although it is possible (as in the above proposition) to weaken the hypotheses of \ref{dual root data} without losing the ability to perform the construction of the dual group, the full generality of a quadratic form defined on the whole of $\Lambda$ is necessary in order to obtain the modified commutativity constraint exhibited in \ref{absolute twisted satake}. This modification thus does not occur in Lusztig's work.

\section{Alteration of the commutativity constraint}
\label{s:alteration of the commutativity constraint}

On the subject of the commutativity constraint, it should be noted that we have introduced two modifications: one in \ref{absolute twisted satake} and depending on the gerbe $\stack{G}_n$, and the other in \ref{fiber functor commutative} and depending only on $\on{Gr}_{G,X^n}$.  In fact, by judicious twisting, one can reduce the latter to the former, thus tuning a jarring note in both our exposition (see \ref{s:fiber functor}) and that of \cite{MV_Satake}.  The appropriate gerbe is in fact quite significant:

\begin{theorem}{lem}{half-forms gerbe}
 Let $\on{adj}$ be the adjoint representation of $G$, whose weights are exactly the roots of $G$ (and zero).  Then $\det(\on{adj})^{\log(-1)}$ has quadratic form
 \begin{equation*}
  Q(\lambda) = (-1)^{\langle 2\rho, \lambda \rangle}
 \end{equation*}
 and trivial bilinear form.  Furthermore, this gerbe is trivial (as a gerbe, but not as an sf gerbe).
\end{theorem}

\begin{proof}
 Since that quadratic form is valued in $2$-torsion elements, the bilinear form is automatically trivial, so it suffices to establish just the equation.  Then we have, denoting roots by $\alpha$,
 \begin{equation*}
  \log_{-1} Q(\lambda)
   = \frac{1}{2} \sum_\alpha \langle \alpha, \lambda \rangle^2
   = \sum_{\alpha > 0} \langle \alpha, \lambda \rangle^2.
 \end{equation*}
 Since by definition $2\rho = \sum_{\alpha > 0} \alpha$, we see that
 \begin{equation*}
  \langle 2\rho, \lambda \rangle^2 \equiv \sum_{\alpha > 0} \langle \alpha, \lambda \rangle^2
   \mod 2,
 \end{equation*}
 and of course we have $\langle 2\rho, \lambda \rangle^2 \equiv \langle 2\rho, \lambda \rangle \mod 2$.  To show that $\det(\on{adj})^{\log(-1)}$ is trivial on $\on{Gr}_{T,X^n}$, it suffices by \ref{determinant forms} to show that $\omega_X^{\log(-1)}$ is trivial on $X$.  Choose a theta-characteristic on $X$; i.e.\ a square root $\omega_X^{1/2}$ of the canonical sheaf.  Then the total space of $\omega_X^{1/2}$ is a two-fold covering of that of $\omega$, and its sheaf of sections is thus \textit{a fortiori} a $\sh{L}_{-1}$-twisted local system on $\omega_X$, trivializing $\omega_X^{\log(-1)}$ according to \ref{line bundle order}.
\end{proof}

We note that the dual group $\check{G}_Q$ obtained from the above quadratic form is isomorphic to ${}^L G$, since $\langle 2\rho, \alpha \rangle = 2$ for any coroot $\alpha$.  Therefore, if we let $\stack{G} = \det(\on{adj})^{\log(-1)}$ on $\on{Gr}_G$, we can replace the usual statement of the geometric Satake equivalence with the one that $\Sph(\stack{G}) \cong \cat{Rep}({}^L G)$, where the left-hand side has its \emph{natural} commutativity constraint, i.e.\ not modified, and is equivalent as a monoidal category to $\Sph_G$. This suggests that the Satake equivalence naturally concerns $\stack{G}$-twisted perverse sheaves even in the ``untwisted'' case.  From the proof of the lemma, we also see that this twisting is precisely the ``critical twisting'' considered everywhere in \cite{BD_quantization}.

\section{Relation to the quantum Langlands correspondence}
\label{s:relation to the quantum Langlands correspondence}

It is possible to relate sf gerbes for $G$ to those for ${}^L G$, the Langlands dual, under some hypotheses.  Suppose that we are given a nondegenerate $\C$-valued (additive) $W$-invariant bilinear form $\map{b}{\Lambda_T \times \Lambda_T}{\C}$ (note that this automatically comes from a $\Z$-valued form as in \ref{integer-valued forms}, since $\C$ is torsion-free); this induces a $\C^*$-valued (multiplicative) form $\kappa(\lambda, \mu) = \exp(2\pi i\ b(\lambda, \mu))$ and quadratic form $Q(\lambda) = \exp(\pi i\; b(\lambda, \lambda))$. Then $b$ is equivalent to an isomorphism $\map{f}{\C \otimes \Lambda_T}{\C \otimes \Lambda^T}$; $f^{-1}$ is equivalent to a nondegenerate bilinear form ${}^L b$ on $\C \otimes \Lambda^T = \C \otimes \Lambda_{{}^L T}$.  Let ${}^L Q(\lambda) = \exp(\pi i\; {}^L b(\lambda,\lambda))$.

\begin{theorem}{prop}{dual group isomorphism}
 The dual group of \ref{dual root data} associated with the pair $(G, Q)$ is isomorphic to that of the pair $({}^L G, {}^L Q)$, where ${}^L G$ is, as usual, the Langlands dual of $G$.
\end{theorem}

\begin{proof}
 We must show that there is an isomorphism of their weight lattices identifying their roots, and that the dual isomorphism identifies the coroots.
 
 For $\check{G}_Q$, we have $\Lambda^{\check{T}_Q} = \on{ker}(\kappa) = f^{-1}(\Lambda^T) \cap \Lambda_T$; likewise, for $({}^L G)_{{}^L Q}^\vee$, we have
 \begin{equation*}
  \Lambda^{({}^L T)_{{}^L Q}^\vee} = \Lambda^T \cap f(\Lambda_T).
 \end{equation*}
 Thus, we choose the isomorphism
 \begin{equation*}
  \map{f}{\Lambda^{\check{T}_Q}}{\Lambda^{({}^L T)_{{}^L Q}^\vee}}.
 \end{equation*}
 Now let $\alpha \in \Lambda^T$ be a root of $G$ and $\check\alpha$ the corresponding coroot; then the $W$-invariance of $b$ implies the $W$-equivariance of $f$ and, therefore, that $f(\check\alpha) \in \C \alpha$; likewise $f^{-1}(\alpha) \in \C \check\alpha$.  It follows that $Q(\check\alpha)$ has finite order if and only if $f(\check\alpha) \in \Q\alpha$, and likewise ${}^L Q(\alpha)$ has finite order if and only if $f^{-1}(\alpha) \in \Q\check\alpha$, which is equivalent.  More precisely, if $f(\check\alpha) = (k/n) \alpha$ with $\on{gcd}(k,n) = 1$, then the corresponding root in $\Lambda^{\check{T}_Q}$ is $n \check\alpha$ by definition.  We may write this as
 \begin{equation*}
  f(n \check\alpha) = k\alpha,
 \end{equation*}
 which is also equivalent to $f^{-1}(\alpha) = (n/k) \check\alpha$, where $\on{gcd}(n,k) = 1$ and thus, ${}^L Q(\alpha)$ has order $k$; so the corresponding root of $({}^L G)_{{}^L Q}$ is $k\alpha$, and the preceding equation shows that $f$ sends the root of $\check{G}_Q$ to that of $({}^L G)_{{}^L Q}^\vee$.
 
 Finally, since the new coroots are respectively $\alpha/n$ and $\check\alpha/k$, the map on coweights induced by $f$ (that is, $f^{-1}$) also sends roots to roots, and thus $f$ is an isomorphism of root data.
\end{proof}
 
Given this setup, let $\stack{G}_n$ be the sf $\C^*$-gerbe corresponding to the quadratic form ${}^L Q$, and let ${}^L \stack{G}_n$ be the one corresponding to ${}^L Q$, with their multiplicative-gerbe parts trivial.  Denote by $\stack{Z}_n$ and ${}^L \stack{Z}_n$ the gerbes for $\on{Fact}(Z(\check{G}_Q)(k))_n \cong \on{Fact}(Z\bigl(({}^L G)^\vee_{{}^L Q}\bigr)(k))_n$ appearing in \ref{main theorem}.

\begin{theorem}{lem}{central gerbe equality}
 We have $\stack{Z}_n \cong {}^L \stack{Z}_n$.
\end{theorem}

\begin{proof}
 Let $\stack{T}_n$ and ${}^L \stack{T}_n$ be the corresponding gerbes on $\on{Gr}_T$ and $\on{Gr}_{{}^L T}$.  According to \ref{multiplicative factorizable lattice gerbes}, each of the $\stack{Z}$ gerbes corresponds to the restrictions of $\stack{T}_n$ and ${}^L \stack{T}_n$ to $\on{Gr}_{{}^L \check{T}_Q}$ and $\on{Gr}_{{}^L ({}^L T)^\vee_{{}^L Q}}$, respectively, as multiplicative factorizable gerbes.  The multiplicative structure is uniquely determined by the factorizable structure according to \ref{multiplicative factorizable}, and this, in turn, is uniquely determined by the restrictions of the quadratic forms $Q$ and $({}^L Q)$ to $\Lambda_{{}^L \check{T}_Q}$ and $\Lambda_{{}^L ({}^L T)^\vee_{{}^L Q}}$ (respectively), by \ref{torus exact sequence}. These are identified via $f$ as in the proof of \ref{dual group isomorphism}.
\end{proof}

Therefore, applying \ref{main theorem} to both $\cat{Sph}(\stack{G}_n)$ and $\cat{Sph}({}^L \stack{G}_n)$ (the former a category of sheaves on $\on{Gr}_{G,X^n}$ and the latter of sheaves on $\on{Gr}_{{}^L G, X^n}$, we find that they are equivalent. This is a local version of the quantum Langlands correspondence.


\end{document}